\documentclass[letterpaper,oneside,english]{amsart}
\usepackage[T1]{fontenc}
\usepackage[latin9]{inputenc}
\usepackage{units}
\usepackage{mathrsfs}
\usepackage{amsthm}
\usepackage{amsbsy}
\usepackage{amstext}
\usepackage{amssymb}
\usepackage{esint}

\makeatletter
\theoremstyle{plain}
\newtheorem{thm}{\protect\theoremname}[section]
  \theoremstyle{plain}
  \newtheorem{lem}[thm]{\protect\lemmaname}
  \theoremstyle{plain}
  \newtheorem{prop}[thm]{\protect\propositionname}
  \theoremstyle{remark}
  \newtheorem{rem}[thm]{\protect\remarkname}
  \theoremstyle{plain}
  \newtheorem{cor}[thm]{\protect\corollaryname}

\makeatother

\usepackage{babel}
  \providecommand{\corollaryname}{Corollary}
  \providecommand{\lemmaname}{Lemma}
  \providecommand{\propositionname}{Proposition}
  \providecommand{\remarkname}{Remark}
\providecommand{\theoremname}{Theorem}

\numberwithin{equation}{section} 
\numberwithin{thm}{section}
\date{}

\begin{document}

\title[Analysis of Vel$\acute{\textrm{A}}$zquez's solution to the MCF]{Analysis of Vel$\acute{\textrm{A}}$zquez's solution to the mean
curvature flow with a type $\mathrm{II}$ singularity } 

\author{Siao-Hao Guo and Natasa Sesum}
\maketitle
\begin{abstract}
J.J.L. Vel$\acute{a}$zquez in 1994 used the degree theory to show
that there is a perturbation of Simons' cone, starting from which
the mean curvature flow develops a type $\mathrm{II}$ singularity
at the origin. He also showed that under a proper time-dependent rescaling
of the solution around the origin, the rescaled flow converges in
the $C^{0}$ sense to a minimal hypersurface which is tangent to Simons'
cone at infinity. In this paper, we prove that the rescaled flow actually
converges locally smoothly to the minimal hypersurface, which appears
to be the singularity model of the type $\mathrm{II}$ singularity.
In addition, we show that the mean curvature of the solution blows
up near the origin at a rate which is smaller than that of the second
fundamental form.
\end{abstract}

\footnote{Natasa Sesum thanks NSF for the support in DMS-1056387.}

\section{Introduction}

J.J.L. Vel$\acute{a}$zquez in \cite{V} constructed a solution to
the mean curvature flow which develops a type $\mathrm{II}$ singularity.
Below is his result:
\begin{thm}
\label{Velazquez thm}Let $n\geq4$ be a positive integer. If $t_{0}<0$
and $\left|t_{0}\right|\ll1$ (depending on $n$), then there is a
$O\left(n\right)\times O\left(n\right)$ symmetric mean curvature
flow $\left\{ \Sigma_{t}\right\} _{t_{0}\leq t<0}$ so that

1. $\left\{ \Sigma_{t}\right\} _{t_{0}\leq t<0}$ develops a type
$\mathrm{II}$ singularity at $O$ as $t\nearrow0$ in the sense that
there is $0<\sigma=\sigma\left(n\right)<\frac{1}{2}$ (see (\ref{sigma}))
so that the second fundamental form of $\Sigma_{t}$ satisfies 
\[
\limsup_{t\nearrow0}\sup_{\Sigma_{t}\cap B\left(O;\,\sqrt{-t}\right)}\left(-t\right)^{\frac{1}{2}+\sigma}\left|A_{\Sigma_{t}}\right|\:>0
\]

2. The type $\mathrm{I}$ rescaled hypersurfaces 
\[
\left\{ \Pi_{s}=\left.\frac{1}{\sqrt{-t}}\Sigma_{t}\right|_{_{t=-e^{-s}}}\right\} _{-\ln\left(-t_{0}\right)\leq s<\infty}
\]
$C^{2}$-converge to Simons' cone $\mathcal{C}$ in any fixed annulus
centered at $O$ (i.e. $B\left(O;\,\mathrm{R}\right)\setminus B\left(O;\,\mathrm{r}\right)$
with $0<\mathrm{r}<\mathrm{R}<\infty$) as $s\nearrow\infty$.

3. The type $\mathrm{II}$ rescaled hypersurfaces 
\[
\left\{ \Gamma_{\tau}=\left.\frac{1}{\left(-t\right)^{\frac{1}{2}+\sigma}}\Sigma_{t}\right|_{t=-\left(2\sigma\tau\right)^{\frac{-1}{2\sigma}}}\right\} _{\frac{1}{2\sigma\left(-t_{0}\right)^{2\sigma}}\leq\tau<\infty}
\]
 locally $C^{0}$-converges to a minimal hypersurface $\mathcal{M}_{k}$
(see Section \ref{minimal}), which is tangent to Simons' cone $\mathcal{C}$
at infinity.
\end{thm}
Vel$\acute{a}$zquez's idea is to find a $O\left(n\right)\times O\left(n\right)$
symmetric solution to the ``normalized mean curvature flow'' $\left\{ \Pi_{s}\right\} _{s_{0}\leq s<\infty}$
which exists for a long time and converges (locally and away from
$O$) to Simons' cone $\mathcal{C}$ as $s\nearrow\infty$. Note that
the minimal cone $\mathcal{C}$ is a self-shrinker with a singularity
at the origin and that this singularity of $\mathcal{C}$ forces the
normalized mean curvature flow $\left\{ \Pi_{s}\right\} _{s_{0}\leq s<\infty}$
to develop a singularity at $O$ as $s\nearrow\infty$. Consequently,
the corresponding mean curvature flow $\left\{ \Sigma_{t}\right\} _{t_{0}\leq t<0}$
develop a type $\mathrm{II}$ singularity at $O$ in finite time (as
$t\nearrow0$). In addition, he used the comparison principle to show
that the type $\mathrm{II}$ rescaled hypersurfaces convergers locally
uniformly, in the $C^{0}$ sense, to a minimal hypersurface $\mathcal{M}_{k}$.

The motivation of studying Vel$\acute{a}$zquez's solution comes from
two natural questions. The first one is whether the minimal hypersurface
$\mathcal{M}_{k}$ is the singularity model of the type $\mathrm{II}$
singularity at $O$? Note that the minimal hypersurface is stationary,
which is a special case of the ``translating mean curvature flow''.
Vel$\acute{a}$zquez's result make us believe that this is true. However,
we cannot be assured by his result since he only show that the type
$\mathrm{II}$ rescaled hypersurfaces converges to $\mathcal{M}_{k}$
in the $C^{0}$ sense. Secondly, we would like to know whether the
mean curvature of Vel$\acute{a}$zquez's solution blows up as $t\nearrow0$
or not. There is a long-lasting question in the study of mean curvature
flow: ``Does the mean curvature blow up at the first singular time?''
The answer is positive under a variety of hypotheses. For instance,
if the mean curvature flow is rotationally symmetric or its singularities
belong to type $\mathrm{I}$, then the mean curvature must blow up
(see \cite{K} and \cite{LS}). People believe this is true in general
for low-dimensional mean curvature flow, and it has been verified
by Li and Wang (see \cite{LW}) for the 2-dimensional case. However,
people are skeptical about this for high-dimensional mean curvature
flow, and they think Vel$\acute{a}$zquez's solution might be a counterexample.
Heuristically speaking, the type $\mathrm{II}$ rescaling of Vel$\acute{a}$zquez's
solution converges to a ``minimal hypersurface'', so it seems that
there is a chance for the mean curvature of Vel$\acute{a}$zquez's
solution to stay bounded upto the first singular time. 

In this paper, we answer both of the above questions. More explicitly,
we show the following:
\begin{thm}
Let $\left\{ \Sigma_{t}\right\} _{t_{0}\leq t<0}$ be Vel$\acute{a}$zquez's
solution in Theorem \ref{Velazquez thm} with $n\geq5$. By choosing
proper initial data outside a small ball centered at $O$, the origin
is the only singularity of the solution at the first singular time
$t=0$. Moreover, the type $\mathrm{II}$ rescaled hypersurfaces $\left\{ \Gamma_{\tau}\right\} _{\frac{1}{2\sigma\left(-t_{0}\right)^{2\sigma}}\leq\tau<\infty}$
converges locally smoothly to the minimal hypersurface $\mathcal{M}_{k}$
as $\tau\nearrow\infty$. It follows that the second fundamental form
of $\Sigma_{t}$ satisfies 
\[
0<\:\limsup_{t\nearrow0}\,\sup_{\Sigma_{t}}\left(-t\right)^{\frac{1}{2}+\sigma}\left|A_{\Sigma_{t}}\right|\:<\infty
\]
In addition, the mean curvature of $\Sigma_{t}$ blows up as $t\nearrow0$
at a rate which smaller than that of the second fundamental form.
More precisely, there hold 
\[
\limsup_{t\nearrow0}\sup_{\Sigma_{t}\cap B\left(O;\,C\left(n\right)\left(-t\right)^{\frac{1}{2}+\sigma}\right)}\left(-t\right)^{\frac{1}{2}-\sigma}\left|H_{\Sigma_{t}}\right|\:>0
\]
\[
\limsup_{t\nearrow0}\,\sup_{\Sigma_{t}}\left(-t\right)^{\frac{1}{2}+\left(1-2\varrho\right)\sigma}\left|H_{\Sigma_{t}}\right|\:<\infty
\]
for some constant $0<\varrho=\varrho\left(n\right)<1$.
\end{thm}
\begin{proof}
The smooth convergence of the type $\mathrm{II}$ rescaled hypersurfaces
$\left\{ \Gamma_{\tau}\right\} $ to $\mathcal{M}_{k}$ as $\tau\nearrow\infty$
and the fact that the origin is the only singularity of $\left\{ \Sigma_{t}\right\} $
at $t=0$ follow from Theorem \ref{Velazquez} (see also Remark \ref{remark on convergence}).
The blow-up rates of the second fundamental form $A_{\Sigma_{t}}$
and mean curvature $H_{\Sigma_{t}}$ can be found in Proposition \ref{blow-up rate outer},
Proposition \ref{blow-up rate intermediate}, Proposition \ref{blow-up rate tip}
and Proposition \ref{blow-up mean curvature}.
\end{proof}
To improve the convergence of the type $\mathrm{II}$ rescaled flow,
all we need is to derive some smooth estimates (see Proposition \ref{degree}
and Proposition \ref{uniform estimates}). One of the key ingredients
to achieve that is to use the curvature estimates in \cite{EH}. As
for the blow-up of the mean curvature, it follows from the smooth
convergence of type $\mathrm{II}$ rescaled flow and L'H$\hat{o}$pital's
rule. Moreover, by modifying Vel$\acute{a}$zquez's estimates, we
show that the blow-up rate of the mean curvature is smaller than that
of the second fundamental form.

The paper is organized as follows. In Section \ref{minimal}, we introduce
the minimal hypersurface $\mathcal{M}_{k}$ found by Vel$\acute{a}$zquez
and then derive some smooth estimates for it. In Section \ref{admissible},
we specify the set up for constructing Vel$\acute{a}$zquez's solution
and define various regions and rescalings for analyzing the solution.
In Section \ref{existence}, we state the key a priori estimates (Proposition
\ref{degree} and Proposition \ref{uniform estimates}) and explain
how to use them to construct Vel$\acute{a}$zquez's solution (for
the sake of completeness) and to see the behavior of the solution
in different regions (see Theorem \ref{Velazquez}). In Section \ref{mean curvature blow-up},
we explain why the mean curvature blows up and why its blow-up rate
is smaller than that of the second fundamental form. Lastly, in Section
\ref{degree C^0}, Section \ref{degree C^infty} and Section \ref{degree Lambda}
we prove Proposition \ref{degree} and Proposition \ref{uniform estimates}
for completion of the argument. 

\section{\label{minimal}Minimal hypersurfaces tangent to Simons' cone at
infinity}

Let 
\[
\mathcal{C}=\left\{ \left(r\nu,\,r\omega\right)\left|\,r>0;\,\nu,\,\omega\in\mathbb{S}^{n-1}\right.\right\} 
\]
be Simons' cone, where $n\geq4$ is a positive integer and $\mathbb{S}^{n-1}$
is the unit sphere in $\mathbb{R}^{n}$. It is shown in \cite{V}
that there is a smooth minimal hypersurface 
\[
\mathcal{M}=\left\{ \left.\left(r\nu,\,\hat{\psi}\left(r\right)\omega\right)\right|\,r\geq0;\,\nu,\,\omega\in\mathbb{S}^{n-1}\right\} 
\]
in $\mathbb{R}^{2n}$ which is tangent to $\mathcal{C}$ at infinity,
and that the function $\hat{\psi}\left(r\right)$ satisfies 
\[
\frac{\partial_{rr}^{2}\hat{\psi}}{1+\left(\partial_{r}\hat{\psi}\right)^{2}}+\left(n-1\right)\left(\frac{\partial_{r}\hat{\psi}}{r}-\frac{1}{\hat{\psi}}\right)=0
\]
and
\[
\left\{ \begin{array}{c}
\partial_{rr}^{2}\hat{\psi}\left(r\right)>0\\
\\
\partial_{r}\hat{\psi}\left(0\right)=0,\quad\lim_{r\nearrow\infty}\frac{\partial_{r}\hat{\psi}\left(r\right)-1}{r^{\alpha-1}}=\alpha\,2^{\frac{\alpha+1}{2}}\\
\\
\hat{\psi}\left(r\right)>r,\quad\lim_{r\nearrow\infty}\frac{\hat{\psi}\left(r\right)-r}{r^{\alpha}}=2^{\frac{\alpha+1}{2}}
\end{array}\right.
\]
where 
\[
\alpha=\frac{-\left(2n-3\right)+\sqrt{4n^{2}-20n+17}}{2}\,\in\left[-2,\,-1\right)
\]
is a root of the quadratic polynomial 
\begin{equation}
\alpha\left(\alpha-1\right)+2\left(n-1\right)\left(\alpha+1\right)=0\label{alpha}
\end{equation}
By symmetry, studying $\mathcal{M}$ is equivalent to analyzing the
projected curves
\[
\mathcal{\bar{M}}=\left\{ \left.\left(r,\,\hat{\psi}\left(r\right)\right)\right|\,r\geq0\right\} 
\]
 
\begin{equation}
\mathcal{\bar{C}}=\left\{ \left.\left(r,\,r\right)\right|\,r>0\right\} \label{projected cone}
\end{equation}
Note that $\mathcal{\bar{M}}$ is a convex curve which lies above
$\mathcal{\bar{C}}$ (i.e. $\hat{\psi}\left(r\right)>r$ for $r\geq0$);
moreover, $\mathcal{\bar{M}}$ intersects orthogonally with the vertical
ray $\left\{ \left.\left(0,\,r\right)\right|\,r>0\right\} $ (i.e.
$\partial_{r}\hat{\psi}\left(0\right)=0$) and is asymptotic to $\mathcal{\bar{C}}$
at infinity (i.e. $\hat{\psi}\left(r\right)=r+O\left(r^{\alpha}\right)$
as $r\nearrow\infty$). Therefore, $\mathcal{\bar{M}}$ is a graph
over $\mathcal{\bar{C}}$; more precisely,
\[
\mathcal{\bar{M}}=\left\{ r\left(\frac{1}{\sqrt{2}},\,\frac{1}{\sqrt{2}}\right)+\psi\left(r\right)\left(\frac{-1}{\sqrt{2}},\,\frac{1}{\sqrt{2}}\right)\left|\,r\geq\,\frac{\hat{\psi}\left(0\right)}{\sqrt{2}}\right.\right\} 
\]
\[
=\left\{ \left(\left(r-\psi\left(r\right)\right)\frac{1}{\sqrt{2}},\,\left(r+\psi\left(r\right)\right)\frac{1}{\sqrt{2}}\right)\left|\,r\geq\,\frac{\hat{\psi}\left(0\right)}{\sqrt{2}}\right.\right\} 
\]
Vel$\acute{a}$zquez in \cite{V} showed that the function $\psi\left(r\right)$
satisfies 
\[
\frac{\partial_{rr}^{2}\psi}{1+\left(\partial_{r}\psi\right)^{2}}+2\left(n-1\right)\frac{r\,\partial_{r}\psi\,+\,\psi}{r^{2}-\psi^{2}}=0
\]
and
\[
\left\{ \begin{array}{c}
\partial_{rr}^{2}\psi\left(r\right)>0\\
\\
\partial_{r}\psi\left(\frac{\hat{\psi}\left(0\right)}{\sqrt{2}}\right)=-1,\quad\lim_{r\nearrow\infty}\frac{\partial_{r}\psi\left(r\right)}{r^{\alpha-1}}=\alpha\\
\\
\psi\left(\frac{\hat{\psi}\left(0\right)}{\sqrt{2}}\right)=\frac{\hat{\psi}\left(0\right)}{\sqrt{2}},\quad\lim_{r\nearrow\infty}\frac{\psi\left(r\right)}{r^{\alpha}}=1
\end{array}\right.
\]
More generally, for each $k>0$, we can define 
\[
\mathcal{M}_{k}\,=\,k^{\frac{1}{1-\alpha}}\mathcal{M}
\]
Then $\mathcal{M}_{k}$ is also a minimal hypersurface in $\mathbb{R}^{2n}$
which is tangent to $\mathcal{C}$ at infinity. Notice that 
\[
\mathcal{M}_{k}\,=\left\{ \left.\left(r\,\nu,\,\hat{\psi}_{k}\left(r\right)\,\omega\right)\right|\,r\geq0;\,\nu,\,\omega\in\mathbb{S}^{n-1}\right\} 
\]
where 
\begin{equation}
\hat{\psi}_{k}\left(r\right)=k^{\frac{1}{1-\alpha}}\,\,\hat{\psi}\left(k^{\frac{-1}{1-\alpha}}\,r\right)\label{psi'}
\end{equation}
By rescaling, we deduce that
\begin{equation}
\frac{\partial_{rr}^{2}\hat{\psi}_{k}}{1+\left(\partial_{r}\hat{\psi}_{k}\right)^{2}}+\left(n-1\right)\left(\frac{\partial_{r}\hat{\psi}_{k}}{r}-\frac{1}{\hat{\psi}_{k}}\right)=0\label{eq psi'}
\end{equation}
\[
\left\{ \begin{array}{c}
\partial_{rr}^{2}\hat{\psi}_{k}\left(r\right)>0\\
\\
\partial_{r}\hat{\psi}_{k}\left(0\right)=0,\quad\lim_{r\nearrow\infty}\frac{\partial_{r}\hat{\psi}_{k}\left(r\right)-1}{r^{\alpha-1}}=k\alpha\,2^{\frac{\alpha+1}{2}}\\
\\
\hat{\psi}_{k}\left(r\right)>r,\quad\lim_{r\nearrow\infty}\frac{\hat{\psi}_{k}\left(r\right)-r}{r^{\alpha}}=k\,2^{\frac{\alpha+1}{2}}
\end{array}\right.
\]
Moreover, there holds a ``monotonic'' property of the rescaling
family, i.e. $\hat{\psi}_{k_{1}}\left(r\right)<\hat{\psi}_{k_{2}}\left(r\right)$
whenever $0<k_{1}<k_{2}<\infty$. To see that, let's first derive
the following lemma.
\begin{lem}
The function $\hat{\psi}_{k}\left(r\right)$ satisfies 
\begin{equation}
\hat{\psi}_{k}\left(r\right)-r\,\partial_{r}\hat{\psi}_{k}\left(r\right)>0\label{positivity}
\end{equation}
for $r\geq0$. In addition, there holds
\begin{equation}
\lim_{r\nearrow\infty}\frac{\hat{\psi}\left(r\right)-r\,\partial_{r}\hat{\psi}\left(r\right)}{r^{\alpha}}=\left(1-\alpha\right)\,2^{\frac{\alpha+1}{2}}\label{asymptotic psi'}
\end{equation}
\end{lem}
\begin{proof}
Notice that 
\[
\partial_{r}\left(\hat{\psi}\left(r\right)-r\,\partial_{r}\hat{\psi}\left(r\right)\right)\,=\,-r\,\partial_{rr}^{2}\hat{\psi}\,<0
\]
which means the function $\hat{\psi}\left(r\right)-r\,\partial_{r}\hat{\psi}$
is decreasing. Furthermore, we have 
\[
\lim_{r\nearrow\infty}\frac{\hat{\psi}\left(r\right)-r\,\partial_{r}\hat{\psi}\left(r\right)}{r^{\alpha}}=\lim_{r\nearrow\infty}\left(\frac{\hat{\psi}\left(r\right)-r}{r^{\alpha}}+\frac{1-\partial_{r}\hat{\psi}\left(r\right)}{r^{\alpha-1}}\right)=\left(1-\alpha\right)\,2^{\frac{\alpha+1}{2}}\,>0
\]
which implies
\[
\hat{\psi}\left(r\right)-r\,\partial_{r}\hat{\psi}\left(r\right)>0
\]
for $r\gg1$. The conclusions follow immediately.
\end{proof}
Now we show the monotonic property of the rescaling family.
\begin{lem}
\label{monotonicity}There holds 
\[
\partial_{k}\,\hat{\psi}_{k}>0
\]
In other words, $\hat{\psi}_{k}$ is monotonically increasing in $k$.
\end{lem}
\begin{proof}
By definition, we have 
\[
\partial_{k}\,\hat{\psi}_{k}\left(z\right)=\partial_{k}\left(k^{\frac{1}{1-\alpha}}\,\,\hat{\psi}\left(k^{\frac{-1}{1-\alpha}}\,z\right)\right)
\]
\[
=\partial_{k}\,k^{\frac{1}{1-\alpha}}\,\left.\left(\hat{\psi}\left(r\right)-r\,\partial_{r}\hat{\psi}\left(r\right)\right)\right|_{r=k^{\frac{-1}{1-\alpha}}\,z}\,>0
\]
\end{proof}
On the other hand, notice that the projected curve of $\mathcal{M}_{k}$
is also a graph over over $\mathcal{\bar{C}}$, i.e.
\begin{equation}
\mathcal{\bar{M}}_{k}=\left\{ \left.\left(r,\,\hat{\psi}_{k}\left(r\right)\right)\right|\,r\geq0\right\} \label{projected minimal hypersurface}
\end{equation}
\[
=\left\{ \left(\left(r-\psi_{k}\left(r\right)\right)\frac{1}{\sqrt{2}},\,\left(r+\psi_{k}\left(r\right)\right)\frac{1}{\sqrt{2}}\right)\left|\,r\geq\,\frac{\hat{\psi}_{k}\left(0\right)}{\sqrt{2}}\right.\right\} 
\]
where 
\begin{equation}
\psi_{k}\left(r\right)=k^{\frac{1}{1-\alpha}}\,\,\psi\left(k^{\frac{-1}{1-\alpha}}\,r\right)\label{psi}
\end{equation}
By rescaling, the function $\psi_{k}\left(r\right)$ satisfies 
\begin{equation}
\frac{\partial_{rr}^{2}\psi_{k}}{1+\left(\partial_{r}\psi_{k}\right)^{2}}+2\left(n-1\right)\frac{r\,\partial_{r}\psi_{k}\,+\,\psi_{k}}{r^{2}-\psi_{k}^{2}}=0\label{eq psi}
\end{equation}
\[
\left\{ \begin{array}{c}
\partial_{rr}^{2}\psi_{k}\left(r\right)>0\\
\\
\partial_{r}\psi_{k}\left(\frac{\hat{\psi}_{k}\left(0\right)}{\sqrt{2}}\right)=-1,\quad\lim_{r\nearrow\infty}\frac{\partial_{r}\psi_{k}\left(r\right)}{r^{\alpha-1}}=k\alpha\\
\\
\psi_{k}\left(\frac{\hat{\psi}_{k}\left(0\right)}{\sqrt{2}}\right)=\frac{\hat{\psi}_{k}\left(0\right)}{\sqrt{2}},\quad\lim_{r\nearrow\infty}\frac{\psi_{k}\left(r\right)}{r^{\alpha}}=k
\end{array}\right.
\]
Note that $\psi_{k}\left(r\right)\searrow0$ as $r\nearrow\infty$.
Below we have the decay estimates for $\psi_{k}\left(r\right)$.
\begin{lem}
\label{order psi}For any $m\in\mathbb{Z}_{+}$, there holds 
\[
\left|\partial_{r}^{m}\psi_{k}\left(r\right)\right|\,\leq\,C\left(n,\,m\right)kr^{\alpha-m}
\]
for $r\geq\frac{\hat{\psi}_{k}\left(0\right)}{\sqrt{2}}$.
\end{lem}
\begin{proof}
By rescaling, it is sufficient to check for $k=1$. 

From
\[
\lim_{r\rightarrow\infty}\frac{\psi\left(r\right)}{r^{\alpha}}\,=\,1\,=\,\lim_{r\rightarrow\infty}\frac{\partial_{r}\psi\left(r\right)}{\alpha r^{\alpha-1}}
\]
we have
\[
\max\left\{ \left|\frac{\psi\left(r\right)}{r}\right|,\,\left|\partial_{r}\psi\left(r\right)\right|\right\} \,\leq\,C\left(n\right)r^{\alpha-1}
\]
for $r\geq\frac{\hat{\psi}\left(0\right)}{\sqrt{2}}$. In particular,
there is $R\gg1$ (depending on $n$) so that 
\[
\max\left\{ \left|\frac{\psi\left(r\right)}{r}\right|,\,\left|\partial_{r}\psi\left(r\right)\right|\right\} \,\leq\,\frac{1}{3}
\]
for $r\geq R$. By (\ref{eq psi}), we have
\[
\partial_{rr}^{2}\psi\left(r\right)=-2\left(n-1\right)\left(1+\left(\partial_{r}\psi\left(r\right)\right)^{2}\right)\frac{r\,\partial_{r}\psi\left(r\right)\,+\,\psi\left(r\right)}{r^{2}-\psi^{2}\left(r\right)}
\]
It follows that 
\[
\left|\partial_{rr}^{2}\psi\left(r\right)\right|\leq C\left(n\right)r^{\alpha-2}
\]
for $r\geq R$. Continuing differentiating the equation of $\psi\left(r\right)$
and using induction yields 
\[
\left|\partial_{r}^{m}\psi\left(r\right)\right|\,\leq\,C\left(n,\,m\right)r^{\alpha-m}
\]
for $r\geq R$, $m\in\mathbb{Z}_{+}$.

On the other hand, by the above choice of $R=R\left(n\right)$, we
have 
\[
\sup_{\frac{\hat{\psi}\left(0\right)}{\sqrt{2}}\leq r\leq R}r^{m-\alpha}\left|\partial_{r}^{m}\psi\left(r\right)\right|\,\leq\,R^{m-\alpha}\sup_{\frac{\hat{\psi}\left(0\right)}{\sqrt{2}}\leq r\leq R}\left|\partial_{r}^{m}\psi\left(r\right)\right|\,\leq\,C\left(n,\,m\right)
\]
for any $m\in\mathbb{Z}_{+}$. Therefore, we conclude that for any
$m\in\mathbb{Z}_{+}$
\[
\left|\partial_{r}^{m}\psi\left(r\right)\right|\,\leq\,C\left(n,\,m\right)r^{\alpha-m}
\]
for $r\geq\frac{\hat{\psi}\left(0\right)}{\sqrt{2}}$. 
\end{proof}
As a corollary, we have the following decay estimates for the higher
order derivatives of $\hat{\psi}_{k}\left(r\right)$.
\begin{lem}
\label{order psi'}For any $m\geq2$, there holds
\[
\left|\partial_{r}^{m}\hat{\psi}_{k}\left(r\right)\right|\,\leq\,C\left(n,\,m\right)kr^{\alpha-m}
\]
for $r\geq0$.
\end{lem}
\begin{proof}
By rescaling, it is sufficient to check for $k=1$. 

Let's first parametrize the projected curve $\mathcal{\bar{M}}$ by
\[
\mathcal{Z}=\left(\left(r-\psi_{k}\left(r\right)\right)\frac{1}{\sqrt{2}},\,\left(r+\psi_{k}\left(r\right)\right)\frac{1}{\sqrt{2}}\right)
\]
In this parametrization, the normal curvature of $\mathcal{\bar{M}}$
is given by 
\[
A_{\mathcal{\bar{M}}}=\frac{\partial_{rr}^{2}\psi\left(r\right)}{\left(1+\left(\partial_{r}\psi\left(r\right)\right)^{2}\right)^{\frac{3}{2}}}
\]
Let $\nabla_{\mathcal{\bar{M}}}$ be the covariant derivative of $\mathcal{\bar{M}}$,
i.e. 
\[
\nabla_{\mathcal{\bar{M}}}\,f\,=\frac{\partial_{r}f\left(r\right)}{\sqrt{1+\left(\partial_{r}\psi\left(r\right)\right)^{2}}}\qquad\textrm{for}\;\,f\in C^{1}\left(\mathcal{\bar{M}}\right)
\]
By Lemma \ref{order psi}, there is $R\gg1$ (depending on $n$) so
that 
\[
\max\left\{ \left|\frac{\psi\left(r\right)}{r}\right|,\,\left|\partial_{r}\psi\left(r\right)\right|\right\} \,\leq\,\frac{1}{3}
\]
and
\begin{equation}
\left|\mathcal{Z}\right|^{m}\left|\nabla_{\mathcal{\bar{M}}}^{m}\,\,A_{\mathcal{\bar{M}}}\right|\leq C\left(n,\,m\right)\left|\mathcal{Z}\right|^{\alpha-2}\label{order normal curvature}
\end{equation}
for $r\geq R$, $m\in\mathbb{Z}_{+}$. Notice that 
\[
\left|\mathcal{Z}\right|=\sqrt{r^{2}+\psi^{2}\left(r\right)}
\]
is comparible with $r$ for $r\geq R$.

Next, let's reparametrize $\mathcal{\bar{M}}$ by 
\begin{equation}
\mathcal{Z}=\left(r,\,\hat{\psi}\left(r\right)\right)\label{order psi'1}
\end{equation}
In this parametrization, the normal curvature is given by 
\begin{equation}
A_{\mathcal{\bar{M}}}=\frac{\partial_{rr}^{2}\hat{\psi}\left(r\right)}{\left(1+\left(\partial_{r}\hat{\psi}\left(r\right)\right)^{2}\right)^{\frac{3}{2}}}\label{order psi'2}
\end{equation}
and the covariant derivative is defined by
\begin{equation}
\nabla_{\mathcal{\bar{M}}}\,f\,=\frac{\partial_{r}f\left(r\right)}{\sqrt{1+\left(\partial_{r}\hat{\psi}\left(r\right)\right)^{2}}}\qquad\textrm{for}\;\,f\in C^{1}\left(\mathcal{\bar{M}}\right)\label{order psi'3}
\end{equation}
Note also that by (\ref{eq psi'}), we have 
\begin{equation}
\begin{array}{c}
0\,\leq\,\frac{\hat{\psi}\left(r\right)}{r}\,\leq\,C\left(n\right)\\
\\
0\,\leq\,\partial_{r}\hat{\psi}\left(r\right)\,\leq\,1
\end{array}\label{order psi'4}
\end{equation}
 for $r\geq R=R\left(n\right)$. Then by (\ref{order normal curvature}),
(\ref{order psi'1}), (\ref{order psi'2}), (\ref{order psi'3}) and
(\ref{order psi'4}), we infer that 
\[
\left|\partial_{r}^{m}\hat{\psi}\left(r\right)\right|\,\leq\,C\left(n,\,m\right)r^{\alpha-m}
\]
for $r\geq2R$, $m\geq2$. 

On the other hand, by the above choice of $R=R\left(n\right)$, there
holds
\[
\sup_{0\leq r\leq2R}\,r^{m-\alpha}\left|\partial_{r}^{m}\hat{\psi}\left(r\right)\right|\,\leq\,\left(2R\right)^{m-\alpha}\sup_{0\leq r\leq2R}\,\left|\partial_{r}^{m}\hat{\psi}\left(r\right)\right|\,\leq\,C\left(n,\,m\right)
\]
for any $m\geq2$. Consequently, we get 
\[
\left|\partial_{r}^{m}\hat{\psi}\left(r\right)\right|\,\leq\,C\left(n,\,m\right)r^{\alpha-m}
\]
for $r\geq0$, $m\geq2$. 
\end{proof}
Lastly, we conclude this section by estimating the difference between
$\psi_{k}$ and its asymptotic function appeared in (\ref{eq psi}).
\begin{lem}
\label{asymptotic psi}The function $\psi_{k}\left(r\right)$ satisfies
\[
\left|\psi_{k}\left(r\right)-kr^{\alpha}\right|\,\leq\,C\left(n\right)\,k^{3}\,r^{3\alpha-2}
\]
 
\[
\left|\partial_{r}\psi_{k}\left(r\right)-k\alpha r^{\alpha-1}\right|\,\leq\,C\left(n\right)\,k^{3}\,r^{3\alpha-3}
\]
for $r\geq\frac{\hat{\psi}_{k}\left(0\right)}{\sqrt{2}}$.
\end{lem}
\begin{proof}
Without loss of generality, we may assume $k=1$. 

First, let's rewrite the equation of $\psi\left(r\right)$ as 
\begin{equation}
r\,\partial_{rr}^{2}\psi=-2\left(n-1\right)\frac{1+\left(\partial_{r}\psi\right)^{2}}{1-\left(\frac{\psi}{r}\right)^{2}}\left(\partial_{r}\psi\,+\,\frac{\psi}{r}\right)\label{eq psi revised}
\end{equation}
Let 
\[
P=\partial_{r}\psi\left(r\right),\qquad Q=\frac{\psi\left(r\right)}{r}
\]
and
\[
\mathfrak{h}=\ln\left(r\right)
\]
Then from (\ref{eq psi revised}), we deduce 
\begin{equation}
\left\{ \begin{array}{c}
\partial_{\mathfrak{h}}P=-2\left(n-1\right)\frac{1+P^{2}}{1-Q^{2}}\left(P+Q\right)\\
\\
\partial_{\mathfrak{h}}Q=P-Q
\end{array}\right.\label{asymptotic psi1}
\end{equation}
On the other hand, by (\ref{alpha}), we can also deduce that
\[
r\,\partial_{rr}^{2}r^{\alpha}=-2\left(n-1\right)\left(\partial_{r}r^{\alpha}\,+\,\frac{r^{\alpha}}{r}\right)
\]
Let 
\[
P_{*}=\partial_{r}r^{\alpha}=\alpha r^{\alpha-1},\qquad Q_{*}=\frac{r^{\alpha}}{r}=r^{\alpha-1}
\]
and 
\[
\mathfrak{h}=\ln\left(r\right)
\]
Similarly, there holds
\begin{equation}
\left\{ \begin{array}{c}
\partial_{\mathfrak{h}}P_{*}=-2\left(n-1\right)\left(P_{*}+Q_{*}\right)\\
\\
\partial_{\mathfrak{h}}Q_{*}=P_{*}-Q_{*}
\end{array}\right.\label{asymptotic psi2}
\end{equation}
Now subtract (\ref{asymptotic psi2}) from (\ref{asymptotic psi1})
to get 
\[
\left\{ \begin{array}{c}
\partial_{\mathfrak{h}}\left(P-P_{*}\right)=-2\left(n-1\right)\left(\left(P-P_{*}\right)+\left(Q-Q_{*}\right)\right)-2\left(n-1\right)\frac{\left(P^{2}+Q^{2}\right)\left(P+Q\right)}{1-Q^{2}}\\
\\
\partial_{\mathfrak{h}}\left(Q-Q_{*}\right)=\left(P-P_{*}\right)-\left(Q-Q_{*}\right)
\end{array}\right.
\]
Note that by (\ref{eq psi}) we have 
\[
\lim_{r\rightarrow\infty}\frac{\psi\left(r\right)-r^{\alpha}}{r^{\alpha}}=0=\lim_{r\rightarrow\infty}\frac{\partial_{r}\psi\left(r\right)-\alpha r^{\alpha-1}}{r^{\alpha-1}}
\]
which implies
\[
\left\{ \begin{array}{c}
P-P_{*}=\partial_{r}\psi\left(r\right)-\alpha r^{\alpha-1}=o\left(r^{\alpha-1}\right)=o\left(e^{\left(\alpha-1\right)\mathfrak{h}}\right)\\
\\
Q-Q_{*}=\frac{\psi\left(r\right)}{r}-r^{\alpha-1}=o\left(r^{\alpha-1}\right)=o\left(e^{\left(\alpha-1\right)\mathfrak{h}}\right)
\end{array}\right.
\]
as $\mathfrak{h}\rightarrow\infty$. Now let 
\[
\Theta=\left(\begin{array}{c}
P-P_{*}\\
\\
Q-Q_{*}
\end{array}\right),\quad\boldsymbol{f}\left(\mathfrak{h}\right)=\left(\begin{array}{c}
-2\left(n-1\right)\frac{\left(P^{2}+Q^{2}\right)\left(P+Q\right)}{1-Q^{2}}\\
\\
0
\end{array}\right)
\]
and
\[
\boldsymbol{L}=\left(\begin{array}{ccc}
2\left(n-1\right) &  & 2\left(n-1\right)\\
\\
-1 &  & 1
\end{array}\right)
\]
 Then we have 
\begin{equation}
\left\{ \begin{array}{c}
\partial_{\mathfrak{h}}\Theta+\boldsymbol{L}\Theta=\boldsymbol{f}\\
\\
\Theta\left(\mathfrak{h}\right)=o\left(e^{\left(\alpha-1\right)\mathfrak{h}}\right)\quad\textrm{as}\;\,\mathfrak{h}\rightarrow\infty
\end{array}\right.\label{asymptotic psi3}
\end{equation}
 Notice that
\[
\boldsymbol{L}=\left(\begin{array}{ccc}
\alpha &  & \bar{\alpha}\\
\\
1 &  & 1
\end{array}\right)\left(\begin{array}{ccc}
-\alpha+1 &  & 0\\
\\
0 &  & -\bar{\alpha}+1
\end{array}\right)\left(\begin{array}{ccc}
\alpha &  & \bar{\alpha}\\
\\
1 &  & 1
\end{array}\right)^{-1}
\]
where 
\[
\bar{\alpha}=\frac{-\left(2n-3\right)-\sqrt{4n^{2}-20n+17}}{2}\,<\,\alpha
\]
and 
\[
\left|\boldsymbol{f}\left(\mathfrak{h}\right)\right|\,\leq\,C\left(n\right)e^{3\left(\alpha-1\right)\mathfrak{h}}\qquad\textrm{for}\;\,\mathfrak{h}\geq\ln\left(\frac{\hat{\psi}\left(0\right)}{\sqrt{2}}\right)
\]
It follows that for any $R>\mathfrak{h}\geq\ln\left(\frac{\hat{\psi}\left(0\right)}{\sqrt{2}}\right)$,
\[
\left|\Theta\left(\mathfrak{h}\right)\right|\,\leq\,e^{\left(R-\mathfrak{h}\right)\left(-\alpha+1\right)}\left|\Theta\left(R\right)\right|\,+\,\int_{\mathfrak{h}}^{R}e^{\left(\xi-\mathfrak{h}\right)\left(-\alpha+1\right)}\left|\boldsymbol{f}\left(\xi\right)\right|d\xi
\]
\[
\leq\,\left(e^{\left(-\alpha+1\right)R}\left|\Theta\left(R\right)\right|\right)e^{\left(\alpha-1\right)\mathfrak{h}}\,+\,C\left(n\right)e^{3\left(\alpha-1\right)\mathfrak{h}}
\]
Note that 
\[
\Theta\left(R\right)=o\left(e^{\left(\alpha-1\right)R}\right)
\]
as $R\rightarrow\infty$ by (\ref{asymptotic psi3}). Let $R\nearrow\infty$
to get 
\[
\left|\Theta\left(\mathfrak{h}\right)\right|\,\leq\,C\left(n\right)e^{3\left(\alpha-1\right)\mathfrak{h}}\qquad\textrm{for}\;\,\mathfrak{h}\geq\ln\left(\frac{\hat{\psi}\left(0\right)}{\sqrt{2}}\right)
\]
which yields 
\[
\left|\partial_{r}\psi\left(r\right)-\alpha r^{\alpha-1}\right|\,+\,\left|\frac{\psi\left(r\right)}{r}-r^{\alpha-1}\right|\,\leq\,C\left(n\right)r^{3\left(\alpha-1\right)}\quad\textrm{for}\;\,r\geq\frac{\hat{\psi}\left(0\right)}{\sqrt{2}}
\]
\end{proof}

\section{\label{admissible}Admissible mean curvature flow}

Let $n\geq5$ be a positive integer and $\Lambda=\Lambda\left(n\right)\gg1$,
$0<\rho\ll1\ll\beta$ (depending on $n$, $\Lambda$), $t_{0}<0$
with $\left|t_{0}\right|\ll1$ (depending on $n$, $\Lambda$, $\rho$,
$\beta$) be constants to be determined. Recall that an one-parameter
family of smooth hypersurfaces $\left\{ \Sigma_{t}\right\} _{t_{0}\leq t\leq\mathring{t}}$
in $\mathbb{R}^{2n}$, where $\mathring{t}<0$ is a constant, is called
a mean curvature flow (MCF) provided that 
\begin{equation}
\partial_{t}X_{t}\cdot N_{\Sigma_{t}}=H_{\Sigma_{t}}\label{MCF}
\end{equation}
where $X_{t}$ is the position vector, $N_{\Sigma_{t}}$ and $H_{\Sigma_{t}}$
are the unit normal vector and mean curvature of $\Sigma_{t}$, respectively.
We define the MCF $\left\{ \Sigma_{t}\right\} _{t_{0}\leq t\leq\mathring{t}}$
to be $\mathbf{admissible}$ if every time-sclice $\Sigma_{t}$ is
a complete, embedded and smooth hypersurface which satisfies
\begin{enumerate}
\item $\Sigma_{t}$ is $O\left(n\right)\times O\left(n\right)$ symmetric
and it can be parametrized as 
\begin{equation}
\Sigma_{t}=\left\{ \left(x\,\nu,\,\hat{u}\left(x,\,t\right)\omega\right)\left|\,x\geq0;\,\nu,\,\omega\in\mathbb{S}^{n-1}\right.\right\} \label{u'}
\end{equation}
where $\hat{u}\left(x,\,t\right)$ is a smooth function which satisfies
\begin{equation}
\partial_{t}\,\hat{u}=\frac{\partial_{xx}^{2}\hat{u}}{1+\left(\partial_{x}\hat{u}\right)^{2}}+\left(n-1\right)\left(\frac{\partial_{x}\hat{u}}{x}-\frac{1}{\hat{u}}\right)\label{eq u'}
\end{equation}
\[
\hat{u}\left(0,\,t\right)>0,\quad\partial_{x}\hat{u}\left(0,\,t\right)=0
\]
for $t_{0}\leq t\leq\mathring{t}$. Note that the above condition
means that the projected curve 
\begin{equation}
\bar{\Sigma}_{t}=\left\{ \left.\left(x,\,\hat{u}\left(x,\,t\right)\right)\right|\,x\geq0\right\} \label{projected Sigma}
\end{equation}
lives in the first quadrant and intersects orthogonally with the vertical
ray $\left\{ \left.\left(0,\,x\right)\right|\,x>0\right\} $.
\item The projected curve $\bar{\Sigma}_{t}$ is a graph over $\mathcal{\bar{C}}$
outside $B\left(O;\,\beta\left(-t\right)^{\frac{1}{2}+\sigma}\right)$,
where 
\begin{equation}
\sigma=-\frac{1}{2}+\frac{2}{1-\alpha}\,\in\left[\frac{1}{6},\,\frac{1}{2}\right)\label{sigma}
\end{equation}
Equivalently, this is saying that $\Sigma_{t}$ is a normal graph
over $\mathcal{C}$ outside $B\left(O;\,\beta\left(-t\right)^{\frac{1}{2}+\sigma}\right)$.
In other words, we can reparametrize $\Sigma_{t}$ by 
\begin{equation}
X_{t}\left(x,\,\nu,\,\omega\right)=\left(\left(x-u\left(x,\,t\right)\right)\frac{\nu}{\sqrt{2}},\,\left(x+u\left(x,\,t\right)\right)\frac{\omega}{\sqrt{2}}\right)\label{u}
\end{equation}
for $x\geq\beta\left(-t\right)^{\frac{1}{2}+\sigma}$, $\nu,\,\omega\in\mathbb{S}^{n-1}$,
where $u\left(x,\,t\right)$ is a smooth function satisfying 
\begin{equation}
\partial_{t}u=\frac{\partial_{xx}^{2}u}{1+\left(\partial_{x}u\right)^{2}}+2\left(n-1\right)\frac{x\,\partial_{x}u+u}{x^{2}-u^{2}}\label{eq u}
\end{equation}
\item For the function $u\left(x,\,t\right)$, there holds 
\begin{equation}
x^{i}\left|\partial_{x}^{i}u\left(x,\,t\right)\right|<\Lambda\left(\left(-t\right)^{2}x^{\alpha}+x^{2\lambda_{2}+1}\right),\qquad i\in\left\{ 0,\,1,\,2\right\} \label{a priori bound u}
\end{equation}
for $\beta\left(-t\right)^{\frac{1}{2}+\sigma}\leq x\leq\rho$, $t_{0}\leq t\leq\mathring{t}$,
where $\lambda_{2}=\frac{1}{2}\left(\alpha+3\right)$ is a constant
(see Proposition \ref{linear operator}). 
\end{enumerate}
In order to analyze an admissible MCF, below we divide the space into
three (time-dependent) regions and do proper rescaling for small regions. 
\begin{itemize}
\item The $\mathbf{outer}$ $\mathbf{region}$ \textendash{} $\Sigma_{t}\setminus B\left(O;\,\sqrt{-t}\right)$ 
\item The $\mathbf{intermediate}$ $\mathbf{region}$ \textendash{} $\Sigma_{t}\cap\left(B\left(O;\,\sqrt{-t}\right)\setminus B\left(O;\,\beta\left(-t\right)^{\frac{1}{2}+\sigma}\right)\right)$:
here we perform the ``type $\mathrm{I}$'' rescaling 
\begin{equation}
\Pi_{s}=\left.\frac{1}{\sqrt{-t}}\Sigma_{t}\right|_{t=-e^{-s}}\label{Pi}
\end{equation}
By this rescaling, the intermediate region is then dilated to become
\[
\Pi_{s}\cap\left(B\left(O;\,1\right)\setminus B\left(O;\,\beta e^{-\sigma s}\right)\right)
\]
for $s_{0}\leq s\leq\mathring{s}$, where $s_{0}=-\ln\left(-t_{0}\right)$
and $\mathring{s}=-\ln\left(-\mathring{t}\right)$. Note that $s_{0}\gg1$
iff $\left|t_{0}\right|\ll1$.
\item The $\mathbf{tip}$ $\mathbf{region}$ \textendash{} $\Sigma_{t}\cap B\left(O;\,\beta\left(-t\right)^{\frac{1}{2}+\sigma}\right)$:
here we perform the ``type $\mathrm{II}$'' rescaling 
\begin{equation}
\Gamma_{\tau}=\left.\frac{1}{\left(-t\right)^{\frac{1}{2}+\sigma}}\Sigma_{t}\right|_{t=-\left(2\sigma\tau\right)^{\frac{-1}{2\sigma}}}\label{Gamma}
\end{equation}
By this rescaling, the intermediate region is dilated to become
\[
\Gamma_{\tau}\cap B\left(O;\,\beta\right)
\]
for $\tau_{0}\leq\tau\leq\mathring{\tau}$, where $\tau_{0}=\frac{1}{2\sigma\left(-t_{0}\right)^{2\sigma}}$,
$\mathring{\tau}=\frac{1}{2\sigma\left(-\mathring{t}\right)^{2\sigma}}$.
Note that $\tau_{0}\gg1$ iff $\left|t_{0}\right|\ll1$. 
\end{itemize}
In the outer region, we parametrize $\Sigma_{t}$ by 
\[
X_{t}\left(x,\,\nu,\,\omega\right)=\left(\left(x-u\left(x,\,t\right)\right)\frac{\nu}{\sqrt{2}},\,\left(x+u\left(x,\,t\right)\right)\frac{\omega}{\sqrt{2}}\right)
\]
and study the function $u\left(x,\,t\right)$ via (\ref{eq u}). In
$B\left(O;\,\rho\right)\setminus B\left(O;\,\sqrt{-t}\right)$, Vel$\acute{a}$zquez
showed that by choosing suitable initial data (see Section \ref{existence}),
there holds
\[
u\left(x,\,t\right)\sim x^{2\lambda_{2}+1}
\]
However, the behavior outside $B\left(O;\,\rho\right)$ was not clear
in \cite{V}. In this paper we complete this part by providing smooth
estimate for $\Sigma_{t}\setminus B\left(O;\,\rho\right)$.

In the intermediate region, we first do the type $\mathrm{I}$ rescaling
and parametrize the rescaled hypersurface $\Pi_{s}$ by 
\begin{equation}
Y_{s}\left(y,\,\nu,\,\omega\right)=\left(\left(y-v\left(y,\,s\right)\right)\frac{\nu}{\sqrt{2}},\,\left(y+v\left(y,\,s\right)\right)\frac{\omega}{\sqrt{2}}\right)\label{v}
\end{equation}
where 
\begin{equation}
v\left(y,\,s\right)=\left.\frac{1}{\sqrt{-t}}\,u\left(\sqrt{-t}\,y,\,t\right)\right|_{t=-e^{-s}}\label{uv}
\end{equation}
From (\ref{eq u}), we derive 
\begin{equation}
\partial_{s}v=\frac{\partial_{yy}^{2}v}{1+\left(\partial_{y}v\right)^{2}}+2\left(n-1\right)\frac{y\,\partial_{y}v+v}{y^{2}-v^{2}}+\frac{1}{2}\left(-y\,\partial_{y}v+v\right)\label{eq v}
\end{equation}
Notice that (\ref{a priori bound u}) is equivalent to
\begin{equation}
y^{i}\left|\partial_{y}^{i}v\left(y,\,s\right)\right|<\Lambda e^{-\lambda_{2}s}\left(y^{\alpha}+y^{2\lambda_{2}+1}\right),\qquad i\in\left\{ 0,\,1,\,2\right\} \label{a priori bound v}
\end{equation}
for $\beta e^{-\sigma s}\leq y\leq\rho e^{\frac{s}{2}}$, $s_{0}\leq s\leq\mathring{s}$.
To study the function $v\left(y,\,s\right)$, Vel$\acute{a}$zquez
linearized (\ref{eq v}) and showed that
\[
v\left(y,\,s\right)\sim e^{-\lambda_{2}s}\varphi_{2}\left(y\right)
\]
by (\ref{a priori bound v}) and the choice of initial data (see Section
\ref{existence}), where $\lambda_{2}$ and $\varphi_{2}\left(y\right)$
are the first positive eigenvalue and eigenfunction of the linearized
operator (see Proposition \ref{linear operator}). More precisely,
(\ref{eq v}) can be rewritten as 
\begin{equation}
\partial_{s}v=-\mathcal{L}v+\mathcal{Q}v\label{linearize eq v}
\end{equation}
where 
\begin{equation}
\mathcal{L}v=-\left(\partial_{yy}^{2}v+2\left(n-1\right)\frac{y\,\partial_{y}v+v}{y^{2}}+\frac{1}{2}\left(-y\,\partial_{y}v+v\right)\right)\label{L}
\end{equation}
\[
=-\left(y^{2\left(n-1\right)}e^{-\frac{y^{2}}{4}}\right)^{-1}\partial_{y}\left(y^{2\left(n-1\right)}e^{-\frac{y^{2}}{4}}\partial_{y}v\right)-\left(\frac{2\left(n-1\right)}{y^{2}}+\frac{1}{2}\right)v
\]
is the (negative) linearization of the RHS of (\ref{eq v}), and 
\begin{equation}
\mathcal{Q}v=-\frac{\left(\partial_{y}v\right)^{2}}{1+\left(\partial_{y}v\right)^{2}}\,\partial_{yy}^{2}v+2\left(n-1\right)\frac{\left(\frac{v}{y}\right)^{2}}{1-\left(\frac{v}{y}\right)^{2}}\left(\frac{\partial_{y}v}{y}+\frac{v}{y^{2}}\right)\label{Q}
\end{equation}
is the remaining (quadratic) parts. Vel$\acute{a}$zquez showed that
the linear differential operator $\mathcal{L}$ has the following
properties (see \cite{V}):
\begin{prop}
\label{linear operator}Define an inner product 
\[
\left\langle \mathsf{v}_{1},\,\mathsf{v}_{2}\right\rangle =\int_{0}^{\infty}\mathsf{v}_{1}\left(y\right)\mathsf{v}_{2}\left(y\right)\,y^{2\left(n-1\right)}e^{-\frac{y^{2}}{4}}dy
\]
and the associated norm 
\[
\left\Vert \mathsf{v}\right\Vert \,=\,\sqrt{\left\langle \mathsf{v},\,\mathsf{v}\right\rangle }
\]
Let $\boldsymbol{\mathrm{H}}$ be the Hilbert space formed by the
completion of $C_{c}^{\infty}\left(\mathbb{R}_{+}\right)$ with respect
to the following inner product: 
\[
\left(\mathsf{v}_{1},\,\mathsf{v}_{2}\right)\equiv\left\langle \partial_{y}\mathsf{v}_{1},\,\partial_{y}\mathsf{v}_{2}\right\rangle +\left\langle \mathsf{v}_{1},\,\mathsf{v}_{2}\right\rangle 
\]
Then we have 
\[
\left\Vert \frac{\mathsf{v}}{y}\right\Vert ^{2}\,\leq\,\frac{4}{\left(2n-3\right)^{2}}\left\Vert \partial_{y}\mathsf{v}\right\Vert ^{2}\,+\,\frac{1}{2n-3}\left\Vert \mathsf{v}\right\Vert ^{2}
\]
and $\mathcal{L}$ is a bounded linear operator in $\boldsymbol{\mathrm{H}}$,
which satisfies 
\[
\left\langle \mathcal{L}\mathsf{v}_{1},\,\mathsf{v}_{2}\right\rangle =\left\langle \partial_{y}\mathsf{v}_{1},\,\partial_{y}\mathsf{v}_{2}\right\rangle -2\left(n-1\right)\left\langle \frac{\mathsf{v}_{1}}{y},\,\frac{\mathsf{v}_{2}}{y}\right\rangle -\frac{1}{2}\left\langle \mathsf{v}_{1},\,\mathsf{v}_{2}\right\rangle 
\]
 
\begin{equation}
\left\langle \mathcal{L}\mathsf{v},\,\mathsf{v}\right\rangle \,\geq\,\frac{4n^{2}-20n+17}{\left(2n-3\right)^{2}}\left\Vert \partial_{y}\mathsf{v}\right\Vert ^{2}\,-\,\frac{6n-7}{2\left(2n-3\right)}\left\Vert \mathsf{v}\right\Vert ^{2}\label{coercivity}
\end{equation}
Note that $4n^{2}-20n+17\geq1$ if $n\geq4$. 

Moreover, the eigenvalues and eigenfunctions of $\mathcal{L}$ are
given by 
\begin{equation}
\lambda_{i}=-\frac{1}{2}\left(1-\alpha\right)+i,\;\;\;\textrm{for}\;\;i=0,\,1,\,2,\cdots\label{eigenvalues}
\end{equation}
and
\[
\varphi_{i}\left(y\right)=c_{i}\,y^{\alpha}\,M\left(-i,\,n+\alpha-\frac{1}{2};\,\frac{y^{2}}{4}\right)
\]
respectively, where $c_{i}>0$ is the normalized constant so that
\[
\left\Vert \varphi_{i}\right\Vert \,=\,\sqrt{\bigl\langle\varphi_{i},\,\varphi_{i}\bigr\rangle}\,=1
\]
and $M\left(\mathrm{a},\,\mathrm{b};\,\xi\right)$ is the Kummer's
function defined by 
\[
M\left(\mathrm{a},\,\mathrm{b};\,\xi\right)=1\,+\,\sum_{j=1}^{\infty}\,\frac{\mathrm{a}\left(\mathrm{a}+1\right)\cdots\left(\mathrm{a}+j-1\right)}{\mathrm{b}\left(\mathrm{b}+1\right)\cdots\left(\mathrm{b}+j-1\right)}\frac{\xi^{j}}{j!}
\]
and satisfying 
\[
\xi\,\partial_{\xi\xi}^{2}M\left(\mathrm{a},\,\mathrm{b};\,\xi\right)\,+\,\left(\mathrm{b}-\xi\right)\,\partial_{\xi}M\left(\mathrm{a},\,\mathrm{b};\,\xi\right)\,-\,\mathrm{a}\,M\left(\mathrm{a},\,\mathrm{b};\,\xi\right)=0
\]
In addition, the family of eigenfunctions $\left\{ \varphi_{i}\right\} _{i=0,\,1,\,2,\cdots}$
forms a complete orthonormal set in $\boldsymbol{\mathrm{H}}$, and
$\lambda_{2}$ is the first positive eigenvalue of $\mathcal{L}$,
i.e. 
\[
\lambda_{0},\,\lambda_{1}<0,\qquad\lambda_{2}>0
\]
\end{prop}
\begin{rem}
\label{parameters}The first three eigenfunctions of $\mathcal{L}$
are given by 
\[
\varphi_{0}\left(y\right)=c_{0}\,y^{\alpha}
\]
 
\[
\varphi_{1}\left(y\right)=c_{1}\,y^{\alpha}\left(1+\varUpsilon_{1}y^{2}\right)
\]
 
\[
\varphi_{2}\left(y\right)=c_{2}\,y^{\alpha}\left(1+2\varUpsilon_{1}y^{2}+\varUpsilon_{2}y^{4}\right)
\]
where 
\[
\varUpsilon_{1}=\frac{-1}{4\left(n+\alpha-\frac{1}{2}\right)},\qquad\varUpsilon_{2}=\frac{1}{16\left(n+\alpha-\frac{1}{2}\right)\left(n+\alpha+\frac{1}{2}\right)}
\]
Note that 
\[
\partial_{yy}^{2}\varphi_{2}\left(y\right)=c_{2}\,y^{\alpha-2}\left(\alpha\left(\alpha-1\right)+2\varUpsilon_{1}\left(\alpha+2\right)\left(\alpha+1\right)y^{2}+\varUpsilon_{2}\left(\alpha+4\right)\left(\alpha+3\right)y^{4}\right)>0
\]
for $y>0$. In addition, for those constants, there hold
\[
\alpha+4=2\lambda_{2}+1
\]
 
\[
\sigma=\frac{\lambda_{2}}{1-\alpha}
\]
Furthermore, when $n\gg1$, we have 
\[
\alpha\approx-1-\frac{1}{n},\qquad\sigma\approx\frac{1}{2}-\frac{1}{2n}
\]
\[
\lambda_{0}\approx-1-\frac{1}{2n},\quad\lambda_{1}\approx-\frac{1}{2n},\quad\lambda_{2}\approx1-\frac{1}{2n}
\]
\end{rem}
Lastly, in the tip region, we do the type $\mathrm{II}$ rescaling
to get
\begin{equation}
\Gamma_{\tau}=\left\{ \left(z\,\nu,\,\hat{w}\left(z,\,\tau\right)\omega\right)\left|\,z\geq0;\,\nu,\,\omega\in\mathbb{S}^{n-1}\right.\right\} \label{w'}
\end{equation}
where 
\begin{equation}
\hat{w}\left(z,\,\tau\right)=\left.\frac{1}{\left(-t\right)^{\frac{1}{2}+\sigma}}\,\hat{u}\left(\left(-t\right)^{\frac{1}{2}+\sigma}z,\,t\right)\right|_{t=-\left(2\sigma\tau\right)^{\frac{-1}{2\sigma}}}\label{u'w'}
\end{equation}
From (\ref{eq u'}) we derive 
\begin{equation}
\partial_{\tau}\hat{w}=\frac{\partial_{zz}^{2}\hat{w}}{1+\left(\partial_{z}\hat{w}\right)^{2}}+\left(n-1\right)\left(\frac{\partial_{z}\hat{w}}{z}-\frac{1}{\hat{w}}\right)+\frac{\frac{1}{2}+\sigma}{2\sigma\tau}\left(-z\,\partial_{z}\hat{w}+\hat{w}\right)\label{eq w'}
\end{equation}
\[
\hat{w}\left(0,\,\tau\right)>0,\quad\partial_{z}\hat{w}\left(0,\,\tau\right)=0
\]
for $\tau_{0}\leq\tau\leq\mathring{\tau}$. Vel$\acute{a}$zquez showed
that by chooing suitable initial data (see Section \ref{existence}),
there holds
\[
\hat{w}\left(z,\,\tau\right)\,\stackrel{C_{loc}^{0}}{\longrightarrow}\,\hat{\psi}_{k}\left(z\right)
\]
for some $k\approx1$, where $\hat{\psi}_{k}$ is the function defined
in Section \ref{minimal}. On the other hand, by the admissible condition
and rescaling, we can regard the rescaled projected curve 
\begin{equation}
\bar{\Gamma}_{\tau}=\left\{ \left(z,\,\hat{w}\left(z,\,\tau\right)\right)\left|\,z\geq0\right.\right\} \label{projected Gamma}
\end{equation}
as a graph over $\mathcal{\bar{C}}$ outside $B\left(O;\,\beta\right)$.
In other words, $\Gamma_{\tau}$ can be reparametrized as a normal
graph over $\mathcal{C}$ outside $B\left(O;\,\beta\right)$, say
\begin{equation}
Z_{\tau}\left(z,\,\nu,\,\omega\right)=\left(\left(z-w\left(z,\,\tau\right)\right)\frac{\nu}{\sqrt{2}},\,\left(z+w\left(z,\,\tau\right)\right)\frac{\omega}{\sqrt{2}}\right)\label{w}
\end{equation}
for $z\geq\beta$, where 
\[
w\left(z,\,\tau\right)=\left.\frac{1}{\left(-t\right)^{\frac{1}{2}+\sigma}}\,u\left(\left(-t\right)^{\frac{1}{2}+\sigma}z,\,t\right)\right|_{t=-\left(2\sigma\tau\right)^{-\frac{1}{2\sigma}}}
\]
 
\begin{equation}
=\left.e^{\sigma s}\,\,v\left(e^{-\sigma s}z,\,s\right)\right|_{s=\frac{1}{2\sigma}\ln\left(2\sigma\tau\right)}\label{uw}
\end{equation}
From (\ref{eq u}) we derive 
\begin{equation}
\partial_{\tau}w=\frac{\partial_{zz}^{2}w}{1+\left(\partial_{z}w\right)^{2}}+2\left(n-1\right)\frac{z\,\partial_{z}w+w}{z^{2}-w^{2}}+\frac{\frac{1}{2}+\sigma}{2\sigma\tau}\left(-z\,\partial_{z}w+w\right)\label{eq w}
\end{equation}
Notice that (\ref{a priori bound u}) is equivalent to 
\begin{equation}
z^{i}\left|\partial_{z}^{i}w\left(z,\,\tau\right)\right|<\Lambda\left(z^{\alpha}+\frac{z^{2\lambda_{2}+1}}{\left(2\sigma\tau\right)^{2}}\right),\qquad i\in\left\{ 0,\,1,\,2\right\} \label{a priori bound w}
\end{equation}
for $\beta\leq z\leq\rho\left(2\sigma\tau\right)^{\frac{1}{2}+\frac{1}{4\sigma}}$,
$\tau_{0}\leq\tau\leq\mathring{\tau}$. 

\section{\label{existence}Construction of Vel$\acute{\textrm{A}}$zquez's
solution }

For readers' convenience and also for the sake of the completeness
of the argument, in this section we show how Vel$\acute{a}$zquez's
solution is constructed. We basically follow Vel$\acute{a}$zquez's
idea in \cite{V} and modify his proofs and estimates. Also, our setting
is slightly different from that in \cite{V} since we assume more
condtions in order to get better results. The key step is Proposition
\ref{degree} and Proposition \ref{uniform estimates}. The main theorem
in this section is Theorem \ref{Velazquez}.

The idea is as follows. At the initial time $t_{0}$, we would choose
a bunch of ``initial hypersurfaces'' $\left\{ \Sigma_{t_{0}}^{\left(a_{0},\,a_{1}\right)}\right\} _{\left(a_{0},\,a_{1}\right)}$
(as candidates) and move each of them by the mean curvature vector.
We then manage to show that for each $\mathring{t}\in\left[t_{0},\,0\right)$,
there is an index $\left(a_{0},\,a_{1}\right)$ for which the corresponding
mean curvature flow $\left\{ \Sigma_{t}^{\left(a_{0},\,a_{1}\right)}\right\} _{t\geq t_{0}}$
exits and is admissible up to time $\mathring{t}$. In addition, we
would establish uniform estimates for these solutions. Lastly, by
the compactness theory, we then get a solution to the MCF which exists
and is admissible for $t_{0}\leq t<0$ and also admits those uniform
estimates.

Let's start with choosing a proper family of $\mathbf{initial}$ $\mathbf{hypersurfaces}$.
Let 
\[
\left\{ \Sigma_{t_{0}}^{\left(a_{0},\,a_{1}\right)}\left|\,\left(a_{0},\,a_{1}\right)\in\overline{B}^{2}\left(O;\,\beta^{2\left(\alpha-1\right)}\right)\right.\right\} 
\]
be a continuous two-parameters family of complete, embedded and smooth
hypersurfaces so that each element $\Sigma_{t_{0}}^{\left(a_{0},\,a_{1}\right)}$
is $\mathbf{admissible}$ $\mathbf{at}$ $\mathbf{time}$ $t_{0}$
and satisfies
\begin{enumerate}
\item The funtion $v\left(y,\,s_{0}\right)=v^{\left(a_{0},\,a_{1}\right)}\left(y,\,s_{0}\right)$
(defined in (\ref{v})) of the type $\mathrm{I}$ rescaled hypersurface
\[
\Pi_{s_{0}}^{\left(a_{0},\,a_{1}\right)}=\frac{1}{\sqrt{-t_{0}}}\Sigma_{t_{0}}^{\left(a_{0},\,a_{1}\right)}
\]
is given by
\begin{equation}
v\left(y,\,s_{0}\right)=e^{-\lambda_{2}s_{0}}\left(\frac{1}{c_{2}}\varphi_{2}\left(y\right)+\frac{a_{1}}{c_{1}}\varphi_{1}\left(y\right)+\frac{a_{0}}{c_{0}}\varphi_{0}\left(y\right)\right)\label{initial v}
\end{equation}
 
\[
=e^{-\lambda_{2}s_{0}}y^{\alpha}\left(1+a_{1}+a_{0}+\left(2+a_{1}\right)\varUpsilon_{1}y^{2}+\varUpsilon_{2}y^{4}\right)
\]
for $\frac{1}{2}\beta e^{-\sigma s_{0}}\leq y\leq2\rho e^{\frac{s_{0}}{2}}$
(see Proposition \ref{linear operator} and Remark \ref{parameters}).
\item The function $u(x,\,t_{0})=u^{\left(a_{0},\,a_{1}\right)}(x,\,t_{0})$
(defined in (\ref{u})) of $\Sigma_{t_{0}}^{\left(a_{0},\,a_{1}\right)}$
is chosen to be
\[
u(x,\,t_{0})\,\approx\,\frac{\varUpsilon_{2}x^{2\lambda_{2}+1}}{1+x^{4}}
\]
for $x\gtrsim\rho$ so that
\begin{equation}
\left\{ \begin{array}{c}
\left|u\left(x,\,t_{0}\right)\right|\leq\frac{1}{5}\,\min\left\{ x,\,1\right\} \\
\\
\left|\partial_{x}u\left(x,\,t_{0}\right)\right|\leq\frac{1}{5}\\
\\
\left|\partial_{xx}^{2}u\left(x,\,t_{0}\right)\right|\leq C\left(n,\,\rho\right)
\end{array}\right.\label{initial u}
\end{equation}
for $x\geq\frac{1}{6}\rho$.
\item The function $\hat{w}\left(\cdot,\,\tau_{0}\right)=\hat{w}^{\left(a_{0},\,a_{1}\right)}\left(\cdot,\,\tau_{0}\right)$
(defined in (\ref{w'})) of the type $\mathrm{II}$ rescaled hypersurface
\[
\Gamma_{\tau_{0}}^{\left(a_{0},\,a_{1}\right)}=\frac{1}{\left(-t_{0}\right)^{\frac{1}{2}+\sigma}}\Sigma_{t_{0}}^{\left(a_{0},\,a_{1}\right)}
\]
is chosen to be
\[
\hat{w}\left(z,\,\tau_{0}\right)\approx\hat{\psi}_{1+a_{1}+a_{0}}\left(z\right)
\]
 for $0\leq z\lesssim\beta$ so that
\begin{equation}
\left\{ \begin{array}{c}
\hat{\psi}_{1-\beta^{\frac{3}{2}\alpha-\frac{5}{2}}}\left(z\right)\,<\,\hat{w}\left(z,\,\tau_{0}\right)\,<\,\hat{\psi}_{1+\beta^{\frac{3}{2}\alpha-\frac{5}{2}}}\left(z\right)\\
\\
0\,=\,\partial_{z}\hat{w}\left(0,\,\tau_{0}\right)\,\leq\,\partial_{z}\hat{w}\left(z,\,\tau_{0}\right)\,<\,1\\
\\
0\,<\,\partial_{zz}^{2}\hat{w}\left(z,\,\tau_{0}\right)\,\leq\,C\left(n\right)
\end{array}\right.\label{initial w'}
\end{equation}
for $0\leq z\leq5\beta$. Furthermore, if we reparametrize the projected
curve $\bar{\Gamma}_{\tau_{0}}^{\left(a_{0},\,a_{1}\right)}$ as a
graph over $\mathcal{\bar{C}}$, the function $w^{\left(a_{0},\,a_{1}\right)}\left(z,\,\tau_{0}\right)=w\left(z,\,\tau_{0}\right)$
(defined in (\ref{w})) satisfies
\[
w\left(z,\,\tau_{0}\right)\,\approx\,\psi_{1+a_{1}+a_{0}}\left(z\right)
\]
for $1\lesssim z\lesssim\beta$ so that 
\begin{equation}
\left\{ \begin{array}{c}
0\,\leq\,w\left(z,\,\tau_{0}\right)\,\leq\,C\left(n\right)z^{\alpha}\\
\\
\left|\partial_{z}w\left(z,\,\tau_{0}\right)\right|\,\leq\,C\left(n\right)z^{\alpha-1}\\
\\
0\,<\,\partial_{zz}^{2}w\left(z,\,\tau_{0}\right)\,\leq\,C\left(n\right)z^{\alpha-2}
\end{array}\right.\label{initial w}
\end{equation}
for $\frac{\hat{\psi}_{2}\left(0\right)}{\sqrt{2}}\leq z\leq5\beta$, 
\end{enumerate}
The following remark shows that (\ref{initial v}) fits in with the
admissible condition and is compatible with (\ref{initial u}).
\begin{rem}
By (\ref{uv}) and Remark \ref{parameters}, (\ref{initial v}) is
equivalent to 
\[
u\left(x,\,t_{0}\right)=\left(-t\right)^{\lambda_{2}+\frac{1}{2}}\left(\frac{1}{c_{2}}\varphi_{2}\left(\frac{x}{\sqrt{-t}}\right)+\frac{a_{1}}{c_{1}}\varphi_{1}\left(\frac{x}{\sqrt{-t}}\right)+\frac{a_{0}}{c_{0}}\varphi_{0}\left(\frac{x}{\sqrt{-t}}\right)\right)
\]
 
\[
=\left(1+a_{1}+a_{0}\right)\left(-t_{0}\right)^{2}x^{\alpha}+\left(2+a_{1}\right)\varUpsilon_{1}\left(-t_{0}\right)x^{\alpha+2}+\varUpsilon_{2}x^{2\lambda_{2}+1}
\]
 
\begin{equation}
=x^{2\lambda_{2}+1}\left(\varUpsilon_{2}+\left(2+a_{1}\right)\varUpsilon_{1}\left(\frac{-t_{0}}{x^{2}}\right)+\left(1+a_{1}+a_{0}\right)\left(\frac{-t_{0}}{x^{2}}\right)^{2}\right)\label{initial u intermediate}
\end{equation}
for $\frac{1}{2}\beta\left(-t_{0}\right)^{\frac{1}{2}+\sigma}\leq x\leq2\rho$.
In particular, there hold
\[
x^{i}\left|\partial_{x}^{i}u\left(x,\,t\right)\right|\,\leq\,C\left(n\right)\left(\left(-t\right)^{2}x^{\alpha}+x^{2\lambda_{2}+1}\right),\qquad i\in\left\{ 0,\,1,\,2\right\} 
\]
 
\begin{equation}
\left|\frac{u\left(x,\,t_{0}\right)}{x}\right|\,\leq\,C\left(n\right)\left(\beta^{\alpha-1}+\rho^{2\lambda_{2}}\right)\label{ratio}
\end{equation}
for $\frac{1}{2}\beta\left(-t_{0}\right)^{\frac{1}{2}+\sigma}\leq x\leq2\rho$.
Thus, we may assume that 
\[
x^{i}\left|\partial_{x}^{i}u\left(x,\,t\right)\right|\,\leq\,\frac{\Lambda}{3}\left(\left(-t\right)^{2}x^{\alpha}+x^{2\lambda_{2}+1}\right),\qquad i\in\left\{ 0,\,1,\,2\right\} 
\]
for $\beta\left(-t_{0}\right)^{\frac{1}{2}+\sigma}\leq x\leq\rho$,
provided that $\Lambda\gg1$ (depending on $n$). Also by (\ref{initial u}),
(\ref{initial w'}) and (\ref{ratio}), we may assume that
\[
\hat{u}\left(x,\,t_{0}\right)>0
\]
for $x\geq0$, provided that $0<\rho\ll1\ll\beta$ (depending on $n$).
Furthermore, by (\ref{initial u intermediate}) we have 
\[
u\left(x,\,t_{0}\right)=x^{2\lambda_{2}+1}\left(\varUpsilon_{2}+O\left(\frac{-t_{0}}{x^{2}}\right)\right)
\]
for $\sqrt{-t_{0}}\lesssim x\leq2\rho$, which is comparible with
(\ref{initial u}) provided that $0<\rho\ll1$ (depending on $n$)
and $\left|t_{0}\right|\ll1$ (depending on $n$, $\rho$).
\end{rem}
The following remark shows that (\ref{initial v}), (\ref{initial w'})
and (\ref{initial w}) are compatible. 
\begin{rem}
By (\ref{initial w'}), $\bar{\Gamma}_{\tau_{0}}^{\left(a_{0},\,a_{1}\right)}$
(see (\ref{projected Gamma})) is a convex curve which lies between
$\mathcal{\bar{M}}_{1-\beta^{\frac{3}{2}\alpha-\frac{5}{2}}}$ and
$\mathcal{\bar{M}}_{1+\beta^{\frac{3}{2}\alpha-\frac{5}{2}}}$ (see
(\ref{projected minimal hypersurface})) and intersects orthogonally
with the vertical ray $\left\{ \left.\left(0,\,z\right)\right|\,z>0\right\} $.
Hence, if we reparametrize $\bar{\Gamma}_{\tau_{0}}^{\left(a_{0},\,a_{1}\right)}$
as a graph over $\mathcal{\bar{C}}$, it follows that 
\[
\psi_{1-\beta^{\frac{3}{2}\alpha-\frac{5}{2}}}\left(z\right)\,<\,w\left(z,\,\tau_{0}\right)\,<\,\psi_{1+\beta^{\frac{3}{2}\alpha-\frac{5}{2}}}\left(z\right)
\]
Then (\ref{initial w}) is compatible with (\ref{initial w'}) in
view of Lemma \ref{order psi}.

On the other hand, by (\ref{uw}) and Remark \ref{parameters}, (\ref{initial v})
is equivalent to 
\[
w\left(z,\,\tau_{0}\right)=\left(2\sigma\tau_{0}\right)^{\frac{\alpha}{2}}\left(\frac{1}{c_{2}}\varphi_{2}\left(\frac{z}{\sqrt{2\sigma\tau_{0}}}\right)+\sum_{i=0}^{1}\,\frac{a_{i}}{c_{i}}\varphi_{i}\left(\frac{z}{\sqrt{2\sigma\tau_{0}}}\right)\right)
\]
\begin{equation}
=z^{\alpha}\left(1+a_{1}+a_{0}+\left(2+a_{1}\right)\varUpsilon_{1}\frac{z^{2}}{2\sigma\tau_{0}}+\varUpsilon_{2}\left(\frac{z^{2}}{2\sigma\tau_{0}}\right)^{2}\right)\label{initial w intermediate}
\end{equation}
for $\frac{1}{2}\beta\leq z\leq2\rho\left(2\sigma\tau_{0}\right)^{\frac{1}{2}+\frac{1}{4\sigma}}$,
which means
\[
w\left(z,\,\tau_{0}\right)=\left(1+a_{1}+a_{0}+O\left(\frac{z^{2}}{2\sigma\tau_{0}}\right)\right)z^{\alpha}
\]
for $\frac{1}{2}\beta\leq z\leq\sqrt{2\sigma\tau_{0}}$. By Lemma
\ref{asymptotic psi}, we then get

\[
\left|w\left(z,\,\tau_{0}\right)-\psi\left(z\right)\right|\,\leq\,\left|w\left(z,\,\tau_{0}\right)-z^{\alpha}\right|\,+\,\left|z^{\alpha}-\psi\left(z\right)\right|
\]
\[
\leq\,\left(\left|a_{0}\right|+\left|a_{1}\right|+C\left(n\right)\left(\frac{z^{2}}{2\sigma\tau_{0}}+z^{2\left(\alpha-1\right)}\right)\right)z^{\alpha}\,\leq\,C\left(n\right)\beta^{2\left(\alpha-1\right)}z^{\alpha}
\]
for $\frac{1}{2}\beta\leq z\leq\left(2\sigma\tau_{0}\right)^{\frac{1}{3}}$
, provided that $\beta\gg1$ (depending on $n$) and $\tau_{0}\gg1$
(depending on $n$, $\beta$). Note also that Lemma \ref{asymptotic psi}
yields 
\[
\psi_{1\pm\beta^{\frac{3}{2}\alpha-\frac{5}{2}}}\left(z\right)-\psi\left(z\right)=\left(\pm\beta^{\frac{3}{2}\alpha-\frac{5}{2}}+O\left(z^{2\left(\alpha-1\right)}\right)\right)z^{\alpha}
\]
in which we have 
\[
\frac{3}{2}\alpha-\frac{5}{2}\,>\,2\left(\alpha-1\right)
\]
Consequently, we get
\[
\psi_{1-\beta^{\frac{3}{2}\alpha-\frac{5}{2}}}\left(z\right)\,<\,w\left(z,\,\tau_{0}\right)\,<\,\psi_{1+\beta^{\frac{3}{2}\alpha-\frac{5}{2}}}\left(z\right)
\]
for $\frac{1}{2}\beta\leq z\leq\left(2\sigma\tau_{0}\right)^{\frac{1}{3}}$,
provided that $\beta\gg1$ (depending on $n$) and $\tau_{0}\gg1$
(depending on $n$, $\beta$).
\end{rem}
Next, for each $\left(a_{0},\,a_{1}\right)\in\overline{B}^{2}\left(O;\,\beta^{2\left(\alpha-1\right)}\right)$,
by \cite{EH} $\Sigma_{t_{0}}^{\left(a_{0},\,a_{1}\right)}$ can be
flowed by (\ref{MCF}) for a short period of time. Let's denote the
corresponding solution by $\left\{ \Sigma_{t}^{\left(a_{0},\,a_{1}\right)}\right\} $.
Given $\mathring{t}\in\left[t_{0},\,0\right)$, let $\mathcal{O}_{\mathring{t}}$
be a set consisting of all $\left(a_{0},\,a_{1}\right)\in B^{2}\left(O;\,\beta^{2\left(\alpha-1\right)}\right)$
for which 
\begin{itemize}
\item The corresponding mean curvature flow $\left\{ \Sigma_{t}^{\left(a_{0},\,a_{1}\right)}\right\} $
exists for $t_{0}\,\leq\,t\,\leq\,\mathring{t}$ and can be extended
beyond time $\mathring{t}$.
\item $\left\{ \Sigma_{t}^{\left(a_{0},\,a_{1}\right)}\right\} $ is $\mathbf{admissible}$
for $t_{0}\,\leq\,t\,\leq\,\mathring{t}$.
\end{itemize}
Clearly, 
\[
\mathcal{O}_{t_{0}}=B^{2}\left(O;\,\beta^{2\left(\alpha-1\right)}\right)
\]
 and $\mathcal{O}_{\mathring{t}}$ is non-increasing in $\mathring{t}$. 

Now let $\zeta\left(r\right)$ be a smooth, non-decreasing function
so that 
\begin{equation}
\zeta\left(r\right)=\left\{ \begin{array}{c}
0,\qquad\textrm{for}\;\,r\leq0\\
\\
1,\qquad\textrm{for}\;\,r\geq1
\end{array}\right.\label{zeta}
\end{equation}
For each $t\geq t_{0},$ we define a map $\Phi_{t}:\,\mathcal{\overline{O}}_{t}\rightarrow\mathbb{R}^{2}$
by 
\begin{equation}
\Phi_{t}\left(a_{0},\,a_{1}\right)=\left.\left(\begin{array}{c}
\left\langle \zeta\left(e^{\sigma s}y-\beta\right)\,\zeta\left(\rho e^{\frac{s}{2}}-y\right)v\left(\cdot,\,s\right),\,\,c_{0}\,\varphi_{0}\right\rangle \\
\\
\left\langle \zeta\left(e^{\sigma s}y-\beta\right)\,\zeta\left(\rho e^{\frac{s}{2}}-y\right)v\left(\cdot,\,s\right),\,\,c_{1}\,\varphi_{1}\right\rangle 
\end{array}\right)\right|_{s=-\ln\left(-t\right)}\label{Phi}
\end{equation}
where the inner product $\left\langle \cdot,\cdot\right\rangle $
is defined in Proposition \ref{linear operator} and $v\left(y,\,s\right)=v^{\left(a_{0},\,a_{1}\right)}\left(y,\,s\right)$
is the function of $\Pi_{s}^{\left(a_{0},\,a_{1}\right)}$ defined
in (\ref{v}) with $s=-\ln\left(-t\right)$. Note that the localized
function 
\[
\tilde{v}\left(y,\,s\right)=\zeta\left(e^{\sigma s}y-\beta\right)\,\zeta\left(\rho e^{\frac{s}{2}}-y\right)v\left(y,\,s\right)
\]
appeared in (\ref{Phi}) is supported in $\left[\beta e^{-\sigma s},\,\rho e^{\frac{s}{2}}\right]$
and would be studied carefully in Proposition \ref{C^0 v}. When $t=t_{0}$,
we have the following lemma. 
\begin{lem}
\label{cut-off}If $s_{0}\gg1$ (depending on $n$, $\rho$, $\beta$),
there hold 
\[
\left|\left\langle \zeta\left(e^{\sigma s_{0}}y-\beta\right)\,\zeta\left(\rho e^{\frac{s_{0}}{2}}-y\right)\varphi_{i},\,\varphi_{j}\right\rangle -\delta_{ij}\right|\,\leq\,C\left(n\right)e^{-2\left(n+\alpha-\frac{1}{2}\right)\sigma s_{0}}
\]

\[
\left\Vert \left(1-\zeta\left(e^{\sigma s_{0}}y-\beta\right)\,\zeta\left(\rho e^{\frac{s_{0}}{2}}-y\right)\right)\varphi_{i}\right\Vert \,\leq\,C\left(n\right)e^{-\left(n+\alpha-\frac{1}{2}\right)\sigma s_{0}}
\]
for $i,\,j\in\left\{ 0,\,1,\,2\right\} $, where $s_{0}=-\ln\left(-t_{0}\right)$
and $\varphi_{i}$ is the $i^{\textrm{th}}$ eigenfunction of $\mathcal{L}$
(see Proposition \ref{linear operator}).
\end{lem}
\begin{proof}
Notice that 
\[
\left\langle \varphi_{i},\,\varphi_{j}\right\rangle =\delta_{ij}
\]
 and 
\[
\zeta\left(e^{\sigma s_{0}}y-\beta\right)\,\zeta\left(\rho e^{\frac{s_{0}}{2}}-y\right)\rightarrow1\qquad\textrm{as}\;\,s_{0}\nearrow\infty
\]
Then we compute 
\[
\left|\left\langle \zeta\left(e^{\sigma s_{0}}y-\beta\right)\,\zeta\left(\rho e^{\frac{s_{0}}{2}}-y\right)\varphi_{i},\,\varphi_{j}\right\rangle -\delta_{ij}\right|
\]
\[
=\left|\left\langle \left(1-\zeta\left(e^{\sigma s_{0}}y-\beta\right)\,\zeta\left(\rho e^{\frac{s_{0}}{2}}-y\right)\right)\varphi_{i},\,\varphi_{j}\right\rangle \right|
\]
\[
\leq\int_{0}^{\left(\beta+1\right)e^{-\sigma s_{0}}}\left|\varphi_{i}\varphi_{j}\right|y^{2\left(n-1\right)}e^{-\frac{y^{2}}{4}}dy\,+\,\int_{\rho e^{\frac{s_{0}}{2}}-1}^{\infty}\left|\varphi_{i}\varphi_{j}\right|y^{2\left(n-1\right)}e^{-\frac{y^{2}}{4}}dy
\]
\[
\leq C\left(n\right)\left(\int_{0}^{\left(\beta+1\right)e^{-\sigma s_{0}}}y^{2\alpha}y^{2\left(n-1\right)}dy\,+\,\int_{\rho e^{\frac{s_{0}}{2}}-1}^{\infty}y^{2\lambda_{i}+2\lambda_{j}+2}y^{2\left(n-1\right)}e^{-\frac{y^{2}}{4}}dy\right)
\]
\[
\leq C\left(n\right)e^{-2\left(n+\alpha-\frac{1}{2}\right)\sigma s_{0}}
\]
It follows that 
\[
\left\Vert \left(1-\zeta\left(e^{\sigma s_{0}}y-\beta\right)\,\zeta\left(\rho e^{\frac{s_{0}}{2}}-y\right)\right)\varphi_{i}\right\Vert ^{2}
\]
\[
=\left\langle \left(1-\zeta\left(e^{\sigma s_{0}}y-\beta\right)\,\zeta\left(\rho e^{\frac{s_{0}}{2}}-y\right)\right)\varphi_{i},\,\left(1-\zeta\left(e^{\sigma s_{0}}y-\beta\right)\,\zeta\left(\rho e^{\frac{s_{0}}{2}}-y\right)\right)\varphi_{i}\right\rangle 
\]
\[
\leq\left\langle \left(1-\zeta\left(e^{\sigma s_{0}}y-\beta\right)\,\zeta\left(\rho e^{\frac{s_{0}}{2}}-y\right)\right)\varphi_{i},\,\varphi_{i}\right\rangle 
\]
\[
\leq C\left(n\right)e^{-2\left(n+\alpha-\frac{1}{2}\right)\sigma s_{0}}
\]
\end{proof}
By (\ref{initial v}) and Lemma \ref{cut-off}, the function $\Phi_{t_{0}}$
converges uniformly to the identity map in $\overline{B}^{2}\left(O;\,\beta^{2\left(\alpha-1\right)}\right)$
as $t_{0}\nearrow0$. Thus, if $\left|t_{0}\right|\ll1$ (depending
on $n$, $\beta$), we have 
\[
\left(0,\,0\right)\notin\Phi_{t_{0}}\left(\partial\,\overline{B}^{2}\left(O;\,\beta^{2\left(\alpha-1\right)}\right)\right)
\]
and 
\[
1=\deg\left(\textrm{Id},\,B^{2}\left(O;\,\beta^{2\left(\alpha-1\right)}\right),\,\left(0,\,0\right)\right)=\deg\left(\Phi_{t_{0}},\,B^{2}\left(O;\,\beta^{2\left(\alpha-1\right)}\right),\,\left(0,\,0\right)\right)
\]
\begin{equation}
=\deg\left(\Phi_{t_{0}},\,\mathcal{O}_{t_{0}},\,\left(0,\,0\right)\right)\label{induction initial}
\end{equation}
In addition, notice that $\mathcal{O}_{t}$ is an open subset of $B^{2}\left(O;\,\beta^{2\left(\alpha-1\right)}\right)$
(by the continuous dependence on the initial data), and that $\Phi_{t}$
is continuous in the parameter $t$. Then we consider the following
index set 
\[
\mathcal{I}=\left\{ t\in\left[t_{0},\,0\right)\left|\,\,\deg\left(\Phi_{t},\,\mathcal{O}_{t},\,\left(0,\,0\right)\right)=1\right.\right\} 
\]
Below are crucial a priori estimates of $\left\{ \Sigma_{t}^{\left(a_{0},\,a_{1}\right)}\right\} _{t_{0}\leq t\leq t_{1}}$
for which 
\[
\Phi_{t_{1}}\left(a_{0},\,a_{1}\right)=\left(0,\,0\right)
\]
We leave the proof in Section \ref{degree C^0}, Section \ref{degree C^infty}
and Section \ref{degree Lambda}.
\begin{prop}
\label{degree}Let $n\geq5$ be a positive integer and choose $\varsigma=\varsigma\left(n\right)>0$,
$\vartheta=\vartheta\left(n\right)\in\left(0,\,1\right)$ so that
\begin{equation}
0<\varsigma<\min\left\{ \frac{n+\alpha-\frac{5}{2}}{1-\alpha},\,\frac{1}{\lambda_{2}}\right\} \label{varsigma}
\end{equation}
\begin{equation}
\frac{-1-\alpha}{1-\alpha}<\vartheta<\min\left\{ \frac{\left(1-\alpha\right)\varsigma}{n+\alpha+\frac{3}{2}},\,\frac{1-\alpha}{2-\alpha},\,\frac{1}{2\sigma}\right\} \label{vartheta}
\end{equation}
Assume that $\left(a_{0},\,a_{1}\right)\in\mathcal{\overline{O}}_{t_{1}}$
for which 
\[
\Phi_{t_{1}}\left(a_{0},\,a_{1}\right)=\left(0,\,0\right)
\]
where $t_{1}\in\left[t_{0},\,0\right)$ is a constant. Suppose that
\[
\left(a_{0},\,a_{1}\right)\in\mathcal{\overline{O}}_{\mathring{t}}
\]
for some $\mathring{t}\in\left[t_{1},\,e^{-1}t_{1}\right]$. Then
if $\Lambda\gg1$ (depending on $n$), $0<\rho\ll1\ll\beta$ (depending
on $n$, $\Lambda$) and $\left|t_{0}\right|\ll1$ (depending on $n$,
$\Lambda$, $\rho$, $\beta$), we have the following estimates.

1. The function $\hat{u}\left(x,\,t\right)$ defined in (\ref{u'})
satisfies
\begin{equation}
\partial_{xx}^{2}\hat{u}\left(x,\,t\right)\geq0\label{convexity}
\end{equation}
for $0\leq x\leq\rho$, $t_{0}\leq t\leq\mathring{t}$.

2. The function $u\left(x,\,t\right)$ defined in (\ref{u}) satisfies
\begin{equation}
\left\{ \begin{array}{c}
\left|u\left(x,\,t\right)\right|\,\leq\,\frac{1}{3}\,\min\left\{ x,\,1\right\} \\
\\
\left|\partial_{x}u\left(x,\,t\right)\right|\,\leq\,\frac{1}{3}\\
\\
\left|\partial_{xx}^{2}u\left(x,\,t\right)\right|\,\leq\,C\left(n,\,\rho\right)
\end{array}\right.\label{C^2 outside u bound}
\end{equation}
for $x\geq\frac{1}{3}\rho$, $t_{0}\leq t\leq\mathring{t}$, and
\begin{equation}
x^{i}\left|\partial_{x}^{i}u\left(x,\,t\right)\right|\,\leq\,\frac{\Lambda}{2}\left(\left(-t\right)^{2}x^{\alpha}+x^{2\lambda_{2}+1}\right),\qquad i\in\left\{ 0,\,1,\,2\right\} \label{smaller a priori bound u}
\end{equation}
for $\beta\left(-t\right)^{\frac{1}{2}+\sigma}\leq x\leq\rho$, $t_{0}\leq t\leq\mathring{t}$. 

3. In the tip region, if we do the type $\mathrm{\mathrm{II}}$ rescaling,
the rescaled function $\hat{w}\left(z,\,\tau\right)$ defined in (\ref{u'w'})
satisfies
\begin{equation}
\left\{ \begin{array}{c}
\hat{\psi}_{1-2\beta^{\alpha-3}}\left(z\right)\,<\,\hat{w}\left(z,\,\tau\right)\,<\,\hat{\psi}_{1+2\beta^{\alpha-3}}\left(z\right)\\
\\
0\,\leq\,\partial_{z}\hat{w}\left(z,\,\tau\right)\,\leq\,1+\beta^{\alpha-2}\\
\\
\left|\partial_{zz}^{2}\hat{w}\left(z,\,\tau\right)\right|\,\leq\,C\left(n\right)
\end{array}\right.\label{C^2 w' bound}
\end{equation}
for $0\leq z\leq3\beta$, $\tau_{0}\leq\tau\leq\mathring{\tau}$,
where $\mathring{\tau}=\frac{1}{2\sigma\left(-\mathring{t}\right)^{2\sigma}}$. 
\end{prop}
Furthermore, we have the following asymptotic formulas and smooth
estimates for the solution in Proposition \ref{degree}.
\begin{prop}
\label{uniform estimates}Under the hypothesis of Proposition \ref{degree},
there is 
\[
k\in\left(1-C\left(n,\,\Lambda,\,\rho,\,\beta\right)\left(-t_{0}\right)^{\varsigma\lambda_{2}},\,1+C\left(n,\,\Lambda,\,\rho,\,\beta\right)\left(-t_{0}\right)^{\varsigma\lambda_{2}}\right)
\]
so that for any given $0<\delta\ll1$, $m,\,l\in\mathbb{Z}_{+}$,
the following smooth estimates hold.

1. In the $\mathbf{outer}$ $\mathbf{region}$, the function $u(x,\,t)$
of $\Sigma_{t}^{\left(a_{0},\,a_{1}\right)}$ defined in (\ref{u})
satisfies 
\begin{equation}
\left|\partial_{x}^{m}\partial_{t}^{l}u(x,\,t)\right|\leq C\left(n,\,\rho,\,\delta,\,m,\,l\right)\label{C^infty outside u bound}
\end{equation}
for $x\geq\frac{1}{2}\rho$, $t_{0}+\delta^{2}\leq t\leq\mathring{t}$,
and 
\begin{equation}
x^{m+2l}\left|\partial_{x}^{m}\partial_{t}^{l}\left(u\left(x,\,t\right)-\frac{k}{c_{2}}\left(-t\right)^{\lambda_{2}+\frac{1}{2}}\varphi_{2}\left(\frac{x}{\sqrt{-t}}\right)\right)\right|\leq C\left(n,\,\Lambda,\,\delta,\,m,\,l\right)\rho^{4\lambda_{2}}\,x^{2\lambda_{2}+1}\label{C^infty u bound}
\end{equation}
for $\left(x,\,t\right)$ satisfying $\frac{1}{2}\sqrt{-t}\leq x\leq\frac{3}{4}\rho$,
$t_{0}+\delta^{2}x^{2}\leq t\leq\mathring{t}$. Note that 
\[
\frac{k}{c_{2}}\left(-t\right)^{\lambda_{2}+\frac{1}{2}}\varphi_{2}\left(\frac{x}{\sqrt{-t}}\right)=kx^{2\lambda_{2}+1}\left(\varUpsilon_{2}+2\varUpsilon_{1}\frac{-t}{x^{2}}+\left(\frac{-t}{x^{2}}\right)^{2}\right)
\]
(see Proposition \ref{linear operator} and Remark \ref{parameters}).

2. In the $\mathbf{intermediate}$ $\mathbf{region}$, if we rescale
the hypersurface by the type $\mathrm{I}$ rescaling (see (\ref{Pi})),
then the function $v\left(y,\,s\right)$ of the rescaled hypersurface
$\Pi_{s}^{\left(a_{0},\,a_{1}\right)}$ defined in (\ref{v}) satisfies
\begin{equation}
y^{m+2l}\left|\partial_{y}^{m}\partial_{s}^{l}\left(v\left(y,\,s\right)-\frac{k}{c_{2}}e^{-\lambda_{2}s}\varphi_{2}\left(y\right)\right)\right|\leq C\left(n,\,\Lambda,\,\delta,\,m,\,l\right)e^{-\varkappa s}e^{-\lambda_{2}s}y^{\alpha+2}\label{C^infty v bound intermediate}
\end{equation}
for $\left(y,\,s\right)$ satisfying $e^{-\vartheta\sigma s}\leq y\leq2$,
$s_{0}+\delta^{2}y^{2}\leq s\leq\mathring{s}$, and 
\begin{equation}
y^{m+2l}\left|\partial_{y}^{m}\partial_{s}^{l}\left(v\left(y,\,s\right)-e^{-\sigma s}\,\psi_{k}\left(e^{\sigma s}y\right)\right)\right|\leq C\left(n,\,\Lambda,\,\delta,\,m,\,l\right)\beta^{\alpha-3}e^{-2\varrho\sigma\left(s-s_{0}\right)}e^{-\lambda_{2}s}y^{\alpha}\label{C^infty v bound tip}
\end{equation}
for $\left(y,\,s\right)$ satisfying $\frac{3}{2}\beta e^{-\sigma s}\leq y\leq e^{-\vartheta\sigma s}$,
$s_{0}+\delta^{2}y^{2}\leq s\leq\mathring{s}$, where $\mathring{s}=-\ln\left(-\mathring{t}\right)$
and
\begin{equation}
\varkappa=\min\left\{ \varsigma\lambda_{2}-\vartheta\sigma\left(n+\alpha+\frac{3}{2}\right),\,\frac{\varsigma\lambda_{2}}{2},\,2\left(\lambda_{2}+\left(\alpha-2\right)\vartheta\sigma\right)\right\} >0\label{varkappa}
\end{equation}
\begin{equation}
\varrho=1-\frac{1}{2}\left(1-\alpha\right)\left(1-\vartheta\right)\in\left(0,\,\vartheta\right)\label{varrho}
\end{equation}
are constants. Note that 
\[
\frac{k}{c_{2}}e^{-\lambda_{2}s}\varphi_{2}\left(y\right)=ke^{-\lambda_{2}s}y^{\alpha}\left(1+2\varUpsilon_{1}y^{2}+\varUpsilon_{2}y^{4}\right)
\]
 
\[
e^{-\sigma s}\,\psi_{k}\left(e^{\sigma s}y\right)=ke^{-\lambda_{2}s}y^{\alpha}\left(1+O\left(\left(e^{\sigma s}y\right)^{-2\left(1-\alpha\right)}\right)\right)
\]
 (see Proposition \ref{linear operator} and (\ref{psi}) ). 

3. In the $\mathbf{tip}$ $\mathbf{region}$, if we rescale the hypersurface
by the type $\mathrm{II}$ rescaling (see (\ref{Gamma})), then the
function $\hat{w}\left(z,\,\tau\right)$ of the rescaled hypersurface
$\Gamma_{\tau}^{\left(a_{0},\,a_{1}\right)}$ defined in (\ref{w'})
satisfies
\begin{equation}
\delta^{m+2l}\left|\partial_{z}^{m}\partial_{\tau}^{l}\left(\hat{w}\left(z,\,\tau\right)-\hat{\psi}_{k}\left(z\right)\right)\right|\leq C\left(n,\,m,\,l\right)\beta^{\alpha-3}\left(\frac{\tau}{\tau_{0}}\right)^{-\varrho}\label{C^infty w' bound}
\end{equation}
 for $0\leq z\leq2\beta$, $\tau_{0}+\delta^{2}\leq\tau\leq\mathring{\tau}$,
where $\mathring{\tau}=\frac{1}{2\sigma\left(-\mathring{t}\right)^{2\sigma}}$.
\end{prop}
\begin{rem}
\label{induction}By Proposition \ref{degree}, Proposition \ref{uniform estimates}
and \cite{EH}, we may infer that if $\left(a_{0},\,a_{1}\right)\in\overline{\mathcal{O}}_{t_{1}}$
and
\[
\Phi_{t_{1}}\left(a_{0},\,a_{1}\right)=\left(0,\,0\right)
\]
 then $\left(a_{0},\,a_{1}\right)\in\mathcal{O}_{e^{-1}t_{1}}$. In
other words, $\Sigma_{t_{0}}^{\left(a_{0},\,a_{1}\right)}$ is a ``good''
candidate of initial hypersurfaces to flow.
\end{rem}
We then have the following corollary.
\begin{cor}
\label{index set}If $\left|t_{0}\right|\ll1$ (depending on $n$),
then we have $\mathcal{I}=\left[t_{0},\,0\right)$.
\end{cor}
\begin{proof}
Notice that by (\ref{induction initial}) we have $t_{0}\in\mathcal{I}$.
Then we would like to prove the corollary by induction. 

Assume that $t_{1}\in\mathcal{I}$. The goal is to show that $t_{2}\in\mathcal{I}$
for any $t_{2}\in\left[t_{1},\,e^{-1}t_{1}\right]$. By definition,
there holds 
\[
\deg\left(\Phi_{t_{1}},\,\mathcal{O}_{t_{1}},\,\left(0,\,0\right)\right)=1
\]
It follows that there is $\left(a_{0},\,a_{1}\right)\in\mathcal{O}_{t_{1}}$
for which 
\[
\Phi_{t_{1}}\left(a_{0},\,a_{1}\right)=\left(0,\,0\right)
\]
By Remark \ref{induction}, we then have $\left(a_{0},\,a_{1}\right)\in\mathcal{O}_{t_{2}}$
and $\left(0,\,0\right)\notin\Phi_{t}\left(\partial\mathcal{O}_{t_{2}}\right)$
for all $t_{1}\leq t\leq t_{2}$. Consequently, $\mathcal{O}_{t_{2}}$
is non-empty and the degree of $\Phi_{t}$ at $\left(0,\,0\right)$
is well defined in $\mathcal{O}_{t_{2}}$ for each $t_{1}\leq t\leq t_{2}$.
Since $\Phi_{t}$ is continuous in $t$, by the homotopy invariance
of degree, there holds
\[
\deg\left(\Phi_{t_{2}},\,\mathcal{O}_{t_{2}},\,\left(0,\,0\right)\right)=\deg\left(\Phi_{t_{1}},\,\mathcal{O}_{t_{2}},\,\left(0,\,0\right)\right)
\]
In addition, by Remark \ref{induction}, $\left(0,\,0\right)\notin\Phi_{t_{1}}\left(\mathcal{O}_{t_{1}}\setminus\mathcal{O}_{t_{2}}\right)$,
which, by the excision property of degree, implies that 
\[
\deg\left(\Phi_{t_{1}},\,\mathcal{O}_{t_{2}},\,\left(0,\,0\right)\right)=\deg\left(\Phi_{t_{1}},\,\mathcal{O}_{t_{1}},\,\left(0,\,0\right)\right)=1
\]
Therefore, we get $t_{2}\in\mathcal{I}$.
\end{proof}
Now we are ready to prove the existence theorem of Vel$\acute{a}$zquez's
solution.
\begin{thm}
\label{Velazquez}Let $n\geq5$ be a positive integer. If $\left|t_{0}\right|\ll1$
(depending on $n$), there is an $\mathbf{admissible}$ mean curvature
flow $\left\{ \Sigma_{t}\right\} _{t_{0}\leq t<0}$ (see Section \ref{admissible})
for which the the functions $\hat{u}(x,\,t)$ and $u(x,\,t)$ (defined
in (\ref{u'}) and (\ref{u}), respectively) satisfy (\ref{convexity})
and (\ref{C^2 outside u bound}). Besides, in the tip region, if we
perform the type $\mathrm{II}$ rescaling, the rescaled function $\hat{w}\left(\cdot,\,\tau\right)$
(defined in (\ref{u'w'})) satisfies (\ref{C^2 w' bound}). 

In addition, there is
\[
k\in\left(1-C\left(n\right)\left(-t_{0}\right)^{\varsigma\lambda_{2}},\,1+C\left(n\right)\left(-t_{0}\right)^{\varsigma\lambda_{2}}\right)
\]
so that for any given $0<\delta\ll1$, $m,\,l\in\mathbb{Z}_{+}$,
there hold

1. In the $\mathbf{outer}$ $\mathbf{region}$, the function $u(x,\,t)$
of $\Sigma_{t}$ defined in (\ref{u}) satisfies (\ref{C^infty outside u bound})
and (\ref{C^infty u bound}).

2. In the $\mathbf{intermediate}$ $\mathbf{region}$, if we do the
type $\mathrm{I}$ rescaling, the function $v\left(y,\,s\right)$
of the rescaled hypersurface $\Pi_{s}$ defined in (\ref{v}) satisfies
(\ref{C^infty v bound intermediate}) and (\ref{C^infty v bound tip}).

3. In the $\mathbf{tip}$ $\mathbf{region}$, if we do the type $\mathrm{II}$
rescaling, the function $\hat{w}\left(\cdot,\,\tau\right)$ of the
rescaled hypersurface $\Gamma_{\tau}$ defined in (\ref{w'}) satisfies
(\ref{C^infty w' bound}).
\end{thm}
\begin{proof}
Let $t_{i}>t_{0}$ be a sequence so that $t_{i}\nearrow0$. By Corollary
\ref{index set}, there is $\left(a_{0}^{i},\,a_{1}^{i}\right)\in\mathcal{O}_{t_{i}}$
for which 
\[
\Phi_{t_{i}}\left(a_{0}^{i},\,a_{1}^{i}\right)=\left(0,\,0\right)
\]
By the uniform estimates in Proposition \ref{degree} and Proposition
\ref{uniform estimates}, we may assume (by passing to a subsequence)
that as $i\rightarrow\infty$, 
\[
k^{\left(a_{0}^{i},\,a_{1}^{i}\right)}\rightarrow k
\]
and the functions $\left\{ \hat{u}^{\left(a_{0}^{i},\,a_{1}^{i}\right)}\left(x,\,t\right)\right\} $
and $\left\{ u^{\left(a_{0}^{i},\,a_{1}^{i}\right)}\left(x,\,t\right)\right\} $
of $\Sigma_{t}^{\left(a_{0}^{i},\,a_{1}^{i}\right)}$ (defined in
(\ref{u'}) and (\ref{u})) converge locally smoothly to $\hat{u}\left(x,\,t\right)$
and $u\left(x,\,t\right)$, respectively. The conclusion follows immediately
by passing the uniform estimates (in Proposition \ref{degree} and
Proposition \ref{uniform estimates}) to limit. 
\end{proof}
\begin{rem}
\label{remark on convergence}Let $\left\{ \Sigma_{t}\right\} _{t_{0}\leq t<0}$
be Vel$\acute{a}$zquez's solution in Theorem \ref{Velazquez}. From
(\ref{v}), (\ref{uv}), (\ref{C^infty u bound}) and (\ref{C^infty v bound intermediate}),
the type $\mathrm{I}$ rescaled hypersurfaces $\Pi_{s}$ (see (\ref{Pi}))
converges smoothly to $\mathcal{C}$ on any fixed annulus centered
at $O$, i.e. for any $0<r<R<\infty$,
\[
\Pi_{s}\,\stackrel{C^{\infty}}{\longrightarrow}\,\mathcal{C}\qquad\textrm{in}\;\,B\left(O;\,R\right)\setminus B\left(O;\,r\right)
\]
as $s\nearrow\infty$. Likewise, from (\ref{w'}), (\ref{w}), (\ref{uw}),
(\ref{C^infty v bound tip}) and (\ref{C^infty w' bound}), the type
$\mathrm{II}$ rescaled hypersurfaces $\Gamma_{\tau}$ (see (\ref{Gamma}))
converges to $\mathcal{M}_{k}$ locally smoothly, i.e. 
\[
\Gamma_{\tau}\,\stackrel{C_{loc}^{\infty}}{\longrightarrow}\,\mathcal{M}_{k}
\]
In addition, by the admissible conditions, the projected curve $\bar{\Sigma}_{t}$
(see (\ref{projected Sigma})) is a graph over $\mathcal{\bar{C}}$
outside $B\left(O;\,\beta\left(-t\right)^{\frac{1}{2}+\sigma}\right)$.
By (\ref{convexity}) and the admissible conditions, we know that
inside $B\left(O;\,\beta\left(-t\right)^{\frac{1}{2}+\sigma}\right)$,
$\bar{\Sigma}_{t}$ is a convex curve which intersects orthogonally
with the vertical ray $\left\{ \left.\left(0,\,x\right)\right|\,x>0\right\} $;
moreover, if we zoom in at $O$ by the type $\mathrm{II}$ rescaling,
by (\ref{eq psi'}) and (\ref{C^0 w' bound}), the rescaled curve
$\bar{\Gamma}_{\tau}$ (see \ref{projected Gamma}) lies above $\mathcal{\bar{C}}$
and tends to it for $z\nearrow\beta$. Therefore, $\bar{\Gamma}_{\tau}$
is a graph over $\mathcal{\bar{C}}$ inside $B\left(O;\,\beta\right)$,
which in turn implies that $\bar{\Sigma}_{t}$ is also graph over
$\mathcal{\bar{C}}$ inside $B\left(O;\,\beta\left(-t\right)^{\frac{1}{2}+\sigma}\right)$.
Hence, we get
\[
\bar{\Sigma}_{t}=\left\{ \left.\left(x,\,\hat{u}\left(x,\,t\right)\right)\right|\,x\geq0\right\} 
\]
 
\[
=\left\{ \left.\left(\left(x-u\left(x,\,t\right)\right)\frac{1}{\sqrt{2}},\,\left(x+u\left(x,\,t\right)\right)\frac{1}{\sqrt{2}}\right)\right|\,x\geq\frac{\hat{u}\left(0,\,t\right)}{\sqrt{2}}\right\} 
\]
\end{rem}

\section{\label{mean curvature blow-up}Type $\mathrm{II}$ singularity and
blow-up of the mean curvature}

In this section we explain why Vel$\acute{a}$zquez's solution (see
Theorem \ref{Velazquez}) develops a type $\mathrm{II}$ singularity
at the origin and why its mean curvature blows up as $t\nearrow0$.
The lower bound for the blow-up rate of the second fundamental form
is already shown in \cite{V}, while the upper bound (of the second
fundamental form) and the blow-up of the mean curvature are new results. 

To estimate the second fundamental form and mean curvature, we would
use the asymptotic formulas in Theorem \ref{Velazquez} to examine
the solution in each region separately. Let's start with analyzing
the outer region by (\ref{u}), (\ref{C^2 outside u bound}) and (\ref{smaller a priori bound u}).
\begin{prop}
\label{blow-up rate outer}Let $\left\{ \Sigma_{t}\right\} _{t_{0}\leq t<0}$
be Vel$\acute{a}$zquez's solution in Theorem \ref{Velazquez}. In
the $\mathbf{outer}$ $\mathbf{region}$, the second fundamental form
of $\Sigma_{t}$ is bounded by
\[
\sqrt{-t}\left|A_{\Sigma_{t}}\right|\,\leq\,C\left(n\right)
\]
 for $\frac{1}{2}t_{0}\leq t<0$.
\end{prop}
\begin{proof}
In the outer region, we parametrize $\Sigma_{t}$ by (\ref{u}). The
second fundamental form is then given by 
\[
A_{\Sigma_{t}}=\frac{1}{\sqrt{1+\left(\partial_{x}u\right)^{2}}}\left(\begin{array}{ccc}
\frac{\partial_{xx}^{2}u}{1+\left(\partial_{x}u\right)^{2}}\\
 & \frac{1+\partial_{x}u}{x-u}\,I_{n-1}\\
 &  & \frac{-1+\partial_{x}u}{x+u}\,I_{n-1}
\end{array}\right)
\]
By (\ref{C^2 outside u bound}) and (\ref{smaller a priori bound u}),
we have
\[
\left\{ \begin{array}{c}
\max\left\{ \left|\frac{u\left(x,\,t\right)}{x}\right|,\,\left|\partial_{x}u\left(x,\,t\right)\right|\right\} \leq\frac{1}{3}\\
\\
\left|\partial_{xx}^{2}u\left(x,\,t\right)\right|\leq C\left(n\right)
\end{array}\right.
\]
for $x\geq\sqrt{-t}$, $\frac{1}{2}t_{0}\leq t<0$. The conclusion
follows immediately.
\end{proof}
In the intermediate region, we first do the type $\mathrm{I}$ rescaling
and study the rescaled hypersurface by (\ref{v}), (\ref{uv}), (\ref{smaller a priori bound u}),
(\ref{C^infty v bound intermediate}) and (\ref{C^infty v bound tip}).
Then we undo the rescaling to get the estimates for the solution.
\begin{prop}
\label{blow-up rate intermediate}Let $\left\{ \Sigma_{t}\right\} _{t_{0}\leq t<0}$
be Vel$\acute{a}$zquez's solution in Theorem \ref{Velazquez}. In
the $\mathbf{intermediate}$ $\mathbf{region}$, the second fundamental
form and the mean curvature of $\Sigma_{t}$ are bounded by
\[
\left(-t\right)^{\frac{1}{2}+\sigma}\left|A_{\Sigma_{t}}\right|\,\leq\,C\left(n\right)
\]
 
\[
\left(-t\right)^{\frac{1}{2}+\left(1-2\varrho\right)\sigma}\left|H_{\Sigma_{t}}\right|\,\leq\,C\left(n,\,t_{0}\right)
\]
for $\frac{1}{2}t_{0}\leq t<0$, where $0<\sigma<\frac{1}{2}$ and
$0<\varrho<1$ are constants defined in (\ref{sigma}) and (\ref{varrho}),
respectively.
\end{prop}
\begin{proof}
In the intermediate region, we rescale Vel$\acute{a}$zquez's solution
by 
\[
\Pi_{s}=\left.\frac{1}{\sqrt{-t}}\Sigma_{t}\right|_{t=-e^{-s}}
\]
which can be parametrized by (\ref{v}). The second fundamental form
and the mean curvature of $\Pi_{s}$ are then given by 
\[
A_{\Pi_{s}}=\frac{1}{\sqrt{1+\left(\partial_{y}v\right)^{2}}}\left(\begin{array}{ccc}
\frac{\partial_{yy}^{2}v}{1+\left(\partial_{y}v\right)^{2}}\\
 & \frac{1+\partial_{y}v}{y-v}\,I_{n-1}\\
 &  & \frac{-1+\partial_{y}v}{y+v}\,I_{n-1}
\end{array}\right)
\]
 
\[
H_{\Pi_{s}}=\frac{1}{\sqrt{1+\left(\partial_{y}v\right)^{2}}}\left(\frac{\partial_{yy}^{2}v}{1+\left(\partial_{y}v\right)^{2}}+2\left(n-1\right)\frac{y\,\partial_{y}v+v}{y^{2}-v^{2}}\right)
\]
\[
=\frac{1}{\sqrt{1+\left(\partial_{y}v\right)^{2}}}\left(\partial_{s}v-\frac{1}{2}\left(-y\,\partial_{y}v+v\right)\right)
\]
By (\ref{uv}) and (\ref{smaller a priori bound u}), we have
\[
\left\{ \begin{array}{c}
\max\left\{ \left|\frac{v\left(y,\,t\right)}{y}\right|,\,\left|\partial_{y}v\left(y,\,s\right)\right|\right\} \,\leq\,C\left(n\right)e^{-\lambda_{2}s}y^{\alpha-1}\,\leq\,\frac{1}{3}\\
\\
\left|\partial_{yy}^{2}v\left(y,\,s\right)\right|\,\leq\,C\left(n\right)\left(e^{-\lambda_{2}s}y^{\alpha-1}\right)y^{-1}\,\leq\,C\left(n\right)e^{\sigma s}
\end{array}\right.
\]
for $\beta e^{-\sigma s}\leq y\leq1$, $-\ln\left(-\frac{1}{2}t_{0}\right)\leq s<\infty$.
Thus, we get 
\[
\left|A_{\Pi_{s}}\right|\,\leq\,C\left(n\right)e^{\sigma s}
\]
in the intermediate region for $-\ln\left(-\frac{1}{2}t_{0}\right)\leq s<\infty$. 

As for the mean curvature, notice that
\[
v\left(y,\,s\right)\approx\left\{ \begin{array}{c}
\frac{k}{c_{2}}\,e^{-\lambda_{2}s}\,\varphi_{2}\left(y\right)\qquad\textrm{for}\;\,e^{-\vartheta\sigma s}\leq y\leq1\\
\\
e^{-\sigma s}\,\psi_{k}\left(e^{\sigma s}y\right)\qquad\textrm{for}\;\,\beta e^{-\sigma s}\leq y\leq e^{-\vartheta\sigma s}
\end{array}\right.
\]
We then compute
\[
\left(\partial_{s}+\frac{y}{2}\,\partial_{y}-\frac{1}{2}\right)\left(\frac{k}{c_{2}}e^{-\lambda_{2}s}\varphi_{2}\left(y\right)\right)
\]
\[
=\left(\partial_{s}+\frac{y}{2}\,\partial_{y}-\frac{1}{2}\right)\left(ke^{-\lambda_{2}s}y^{\alpha}\left(1+2\varUpsilon_{1}y^{2}+\varUpsilon_{2}y^{4}\right)\right)
\]
\[
=-2ke^{-\lambda_{2}s}y^{\alpha}\left(1+\varUpsilon_{1}y^{2}\right)
\]
and 
\[
\left(\partial_{s}+\frac{y}{2}\,\partial_{y}-\frac{1}{2}\right)\left(e^{-\sigma s}\,\psi_{k}\left(e^{\sigma s}y\right)\right)
\]
\[
=-\left.\left(\frac{1}{2}+\sigma\right)e^{-\sigma s}\left(\psi_{k}\left(z\right)-z\,\partial_{z}\psi_{k}\left(z\right)\right)\right|_{z=e^{\sigma s}y}
\]
\[
=-\left(\frac{1}{2}+\sigma\right)e^{-\sigma s}\left(\left(1-\alpha\right)k\left(e^{\sigma s}y\right)^{\alpha}+O\left(\left(e^{\sigma s}y\right)^{3\alpha-2}\right)\right)
\]
\[
=-2ke^{-\lambda_{2}s}y^{\alpha}\left(1+O\left(\left(e^{\sigma s}y\right)^{-2\left(1-\alpha\right)}\right)\right)
\]
It follows, by (\ref{C^infty v bound intermediate}) and (\ref{C^infty v bound tip}),
that 
\[
\left|\left(\partial_{s}+\frac{y}{2}\,\partial_{y}-\frac{1}{2}\right)v\left(y,\,s\right)\right|
\]
\[
\leq\left|\left(\partial_{s}+\frac{y}{2}\,\partial_{y}-\frac{1}{2}\right)\left(\frac{k}{c_{2}}e^{-\lambda_{2}s}\varphi_{2}\left(y\right)\right)\right|\,+\,C\left(n,\,t_{0}\right)e^{-\varkappa s}\left(e^{-\lambda_{2}s}y^{\alpha}\right)
\]
\[
\leq\left|-2ke^{-\lambda_{2}s}y^{\alpha}\left(1+\varUpsilon_{1}y^{2}\right)\right|\,+\,C\left(n,\,t_{0}\right)e^{-\varkappa s}\left(e^{-\lambda_{2}s}y^{\alpha}\right)
\]
\[
\leq C\left(n,\,t_{0}\right)e^{-\lambda_{2}s}y^{\alpha}\,\leq\,C\left(n,\,t_{0}\right)
\]
for $e^{-\vartheta\sigma s}\leq y\leq1$, $-\ln\left(-\frac{1}{2}t_{0}\right)\leq s<\infty$,
and
\[
\left|\left(\partial_{s}+\frac{y}{2}\,\partial_{y}-\frac{1}{2}\right)v\left(y,\,s\right)\right|
\]
\[
\leq\left|\left(\partial_{s}+\frac{y}{2}\,\partial_{y}-\frac{1}{2}\right)\left(e^{-\sigma s}\,\psi_{k}\left(e^{\sigma s}y\right)\right)\right|\,+\,C\left(n,\,t_{0}\right)e^{-2\varrho\sigma s}\left(e^{-\lambda_{2}s}y^{\alpha-2}\right)
\]
\[
\leq\left|-2ke^{-\lambda_{2}s}y^{\alpha}\left(1+O\left(\left(e^{\sigma s}y\right)^{-2\left(1-\alpha\right)}\right)\right)\right|\,+\,C\left(n,\,t_{0}\right)\left(e^{-\lambda_{2}s}y^{\alpha-1}\right)\left(e^{-2\varrho\sigma s}y^{-1}\right)
\]
\[
\leq C\left(n,\,t_{0}\right)e^{-\lambda_{2}s}y^{\alpha-1}\left(y\,+\,e^{-2\varrho\sigma s}y^{-1}\right)\,\leq\,C\left(n,\,t_{0}\right)e^{\left(1-2\varrho\right)\sigma s}
\]
for $\beta e^{-\sigma s}\leq y\leq e^{-\vartheta\sigma s}$, $-\ln\left(-\frac{1}{2}t_{0}\right)\leq s<\infty$.
Consequently, 
\[
\left|H_{\Pi_{s}}\right|=\frac{\left|\partial_{s}v-\frac{1}{2}\left(-y\,\partial_{y}v+v\right)\right|}{\sqrt{1+\left|\partial_{y}v\right|^{2}}}\,\leq\,C\left(n,\,t_{0}\right)e^{\left(1-2\varrho\right)\sigma s}
\]
Lastly, by the relation 
\[
A_{\Pi_{s}}\left(y\right)=\left.\sqrt{-t}\,\,A_{\Sigma_{t}}\left(\sqrt{-t}\,y\right)\right|_{t=-e^{-s}}
\]
 
\[
H_{\Pi_{s}}\left(y\right)=\left.\sqrt{-t}\,\,H_{\Sigma_{t}}\left(\sqrt{-t}\,y\right)\right|_{t=-e^{-s}}
\]
the conclusion follow easily.
\end{proof}
In the tip region, we do the type $\mathrm{II}$ rescaling and study
the rescaled hypersurface by (\ref{w'}), (\ref{C^2 w' bound}) and
(\ref{C^infty w' bound}). Then we undo the rescaling to get estimates
of the solution. 
\begin{prop}
\label{blow-up rate tip}Let $\left\{ \Sigma_{t}\right\} _{t_{0}\leq t<0}$
be Vel$\acute{a}$zquez's solution in Theorem \ref{Velazquez}. In
the $\mathbf{tip}$ $\mathbf{region}$, the second fundamental form
and the mean curvature of $\Sigma_{t}$ satisfy 
\[
\frac{1}{C\left(n\right)}\,\leq\,\left(-t\right)^{\frac{1}{2}+\sigma}\left|A_{\Sigma_{t}}\right|\,\leq\,C\left(n\right)
\]
 
\[
\left(-t\right)^{\frac{1}{2}+\left(1-2\varrho\right)\sigma}\left|H_{\Sigma_{t}}\right|\,\leq\,C\left(n,\,t_{0}\right)
\]
for $\frac{1}{2}t_{0}\leq t<0$, where $0<\sigma<\frac{1}{2}$ and
$0<\varrho<1$ are constants defined in (\ref{sigma}) and (\ref{varrho}),
respectively.
\end{prop}
\begin{proof}
In the tip region, we first rescale\textbf{ }Vel$\acute{a}$zquez's
solution by 
\[
\Gamma_{\tau}=\left.\frac{1}{\left(-t\right)^{\frac{1}{2}+\sigma}}\Sigma_{t}\right|_{t=-\left(2\sigma\tau\right)^{\frac{-1}{2\sigma}}}
\]
which can be parametrized by (\ref{w'}). Then the second fundamental
form and the mean curvature of $\Gamma_{\tau}$ are given by

\[
A_{\Gamma_{\tau}}=\frac{1}{\sqrt{1+\left|\partial_{z}\hat{w}\right|^{2}}}\left(\begin{array}{ccc}
\frac{\partial_{zz}^{2}\hat{w}}{1+\left|\partial_{z}\hat{w}\right|^{2}}\\
 & \frac{\partial_{z}\hat{w}}{z}\,I_{n-1}\\
 &  & \frac{-1}{\hat{w}}\,I_{n-1}
\end{array}\right)
\]
 
\[
H_{\Gamma_{\tau}}=\frac{1}{\sqrt{1+\left(\partial_{z}\hat{w}\right)^{2}}}\left(\frac{\partial_{zz}^{2}\hat{w}}{1+\left(\partial_{z}\hat{w}\right)^{2}}+\left(n-1\right)\left(\frac{\partial_{z}\hat{w}}{z}-\frac{1}{\hat{w}}\right)\right)
\]
\[
=\frac{1}{\sqrt{1+\left(\partial_{z}\hat{w}\right)^{2}}}\left(\partial_{\tau}\hat{w}-\frac{\frac{1}{2}+\sigma}{2\sigma\tau}\left(-z\,\partial_{z}\hat{w}+\hat{w}\right)\right)
\]
By (\ref{C^2 w' bound}), we have 
\[
\frac{1}{C\left(n\right)}\,\leq\,\hat{w}\left(z,\,\tau\right)\,\leq\,C\left(n\right)
\]
\[
\left|\partial_{z}\hat{w}\left(z,\,\tau\right)\right|\,+\,\left|\partial_{zz}^{2}\hat{w}\left(z,\,\tau\right)\right|\,\leq\,C\left(n\right)
\]
for $0\leq z\leq\beta$, $\frac{1}{2\sigma}\left(-\frac{1}{2}t_{0}\right)^{-2\sigma}\leq\tau<\infty$.
Thus, we get 
\[
\frac{1}{C\left(n\right)}\,\leq\,\left|A_{\Gamma_{\tau}}\right|\leq\,C\left(n\right)
\]

As for the mean curvature, note, from (\ref{asymptotic psi'}), that
\[
\left|\left(\partial_{\tau}+\frac{\frac{1}{2}+\sigma}{2\sigma\tau}z\,\partial_{z}-\frac{\frac{1}{2}+\sigma}{2\sigma\tau}\right)\hat{\psi}_{k}\left(z\right)\right|\,=\,\left|-\frac{\frac{1}{2}+\sigma}{2\sigma\tau}\left(\hat{\psi}_{k}\left(z\right)-z\,\partial_{z}\hat{\psi}_{k}\left(z\right)\right)\right|\,\leq\,\frac{C\left(n\right)}{2\sigma\tau}
\]
By (\ref{C^infty w' bound}), we get 
\[
\left|\left(\partial_{\tau}+\frac{\frac{1}{2}+\sigma}{2\sigma\tau}z\,\partial_{z}-\frac{\frac{1}{2}+\sigma}{2\sigma\tau}\right)\hat{w}\left(z,\,\tau\right)\right|
\]
\[
\leq\left|\left(\partial_{\tau}+\frac{\frac{1}{2}+\sigma}{2\sigma\tau}z\,\partial_{z}-\frac{\frac{1}{2}+\sigma}{2\sigma\tau}\right)\hat{\psi}_{k}\left(z\right)\right|\,+\,C\left(n,\,t_{0}\right)\left(2\sigma\tau\right)^{-\varrho}
\]
\[
\leq C\left(n,\,t_{0}\right)\left(2\sigma\tau\right)^{-\varrho}
\]
Thus, 
\[
\left|H_{\Gamma_{\tau}}\right|=\frac{\left|\partial_{\tau}\hat{w}-\frac{\frac{1}{2}+\sigma}{2\sigma\tau}\left(-z\,\partial_{z}\hat{w}+\hat{w}\right)\right|}{\sqrt{1+\left(\partial_{z}\hat{w}\right)^{2}}}\,\leq\,C\left(n,\,t_{0}\right)\left(2\sigma\tau\right)^{-\varrho}
\]
The conclusion follows by noting that 
\[
A_{\Gamma_{\tau}}\left(z\right)=\left.\left(-t\right)^{\frac{1}{2}+\sigma}\,A_{\Sigma_{t}}\left(\left(-t\right)^{\frac{1}{2}+\sigma}z\right)\right|_{t=-\left(2\sigma\tau\right)^{\frac{-1}{2\sigma}}}
\]
 
\[
H_{\Gamma_{\tau}}\left(z\right)=\left.\left(-t\right)^{\frac{1}{2}+\sigma}\,H_{\Sigma_{t}}\left(\left(-t\right)^{\frac{1}{2}+\sigma}z\right)\right|_{t=-\left(2\sigma\tau\right)^{\frac{-1}{2\sigma}}}
\]
\end{proof}
Lastly, we would like to show that the mean curvature blows up in
the tip region at a rate at least $\frac{1}{\left(-t\right)^{\frac{1}{2}-\sigma}}$
as $t\nearrow0$.
\begin{prop}
\label{blow-up mean curvature}Let $\left\{ \Sigma_{t}\right\} _{t_{0}\leq t<0}$
be Vel$\acute{a}$zquez's solution in Theorem \ref{Velazquez}. Let
$H_{\Sigma_{t}}\left(x\right)$ be the mean curvature of $\Sigma_{t}$
at 
\[
X_{t}\left(x,\,\nu,\,\omega\right)=\left(x\,\nu,\,\hat{u}\left(x,\,t\right)\omega\right)
\]
(see (\ref{u'})). Then for any $z\geq0$, there holds 
\[
\limsup_{t\nearrow0}\,\,\left(-t\right)^{\frac{1}{2}-\sigma}\left|H_{\Sigma_{t}}\left(\left(-t\right)^{\frac{1}{2}+\sigma}z\right)\right|>0
\]
\end{prop}
\begin{proof}
Note that
\[
H_{\Sigma_{t}}=\frac{1}{\sqrt{1+\left(\partial_{x}\hat{u}\right)^{2}}}\left(\frac{\partial_{xx}^{2}\hat{u}}{1+\left(\partial_{x}\hat{u}\right)^{2}}+\left(n-1\right)\left(\frac{\partial_{x}\hat{u}}{x}-\frac{1}{\hat{u}}\right)\right)
\]
\begin{equation}
=\frac{\partial_{t}\hat{u}}{\sqrt{1+\left(\partial_{x}\hat{u}\right)^{2}}}\label{mean curvature}
\end{equation}
We claim that for any $z\geq0$, there holds
\begin{equation}
\limsup_{t\nearrow0}\frac{\left|\partial_{t}\hat{u}\left(\left(-t\right)^{\frac{1}{2}+\sigma}z,\,t\right)\right|}{\left(-t\right)^{-\frac{1}{2}+\sigma}}>0\label{claim}
\end{equation}
The conclusion follows immediately from (\ref{C^2 w' bound}), (\ref{mean curvature})
and (\ref{claim}).

To prove (\ref{claim}), we use a contradiction argument. Suppose
that there is $z\geq0$ so that 
\[
\limsup_{t\nearrow0}\frac{\left|\partial_{t}\hat{u}\left(\left(-t\right)^{\frac{1}{2}+\sigma}z,\,t\right)\right|}{\left(-t\right)^{-\frac{1}{2}+\sigma}}=0
\]
then obviously,
\begin{equation}
\lim_{t\nearrow0}\frac{\left|\partial_{t}\hat{u}\left(\left(-t\right)^{\frac{1}{2}+\sigma}z,\,t\right)\right|}{\left(-t\right)^{-\frac{1}{2}+\sigma}}=0\label{hypothesis}
\end{equation}
Recall that by (\ref{C^infty w' bound}), we have 
\[
\frac{1}{\left(-t\right)^{\frac{1}{2}+\sigma}}\hat{u}\left(\left(-t\right)^{\frac{1}{2}+\sigma}z,\,t\right)\,=\,\hat{w}\left(z,\,\frac{1}{2\sigma\left(-t\right)^{2\sigma}}\right)\,\rightarrow\,\hat{\psi}_{k}\left(z\right)\qquad\textrm{as}\quad t\nearrow0
\]
It follows, by L'H$\hat{o}$pital's rule, that 
\[
\hat{\psi}_{k}\left(z\right)=\lim_{t\nearrow0}\frac{\hat{u}\left(\left(-t\right)^{\frac{1}{2}+\sigma}z,\,t\right)}{\left(-t\right)^{\frac{1}{2}+\sigma}}=\lim_{t\nearrow0}\left(\frac{\partial_{t}\hat{u}\left(\left(-t\right)^{\frac{1}{2}+\sigma}z,\,t\right)}{-\left(\frac{1}{2}+\sigma\right)\left(-t\right)^{-\frac{1}{2}+\sigma}}+z\,\partial_{z}\hat{w}\left(z,\,\frac{1}{2\sigma\left(-t\right)^{2\sigma}}\right)\right)
\]
Notice that the limit on the RHS exists because of (\ref{C^infty w' bound})
and (\ref{hypothesis}), so L'H$\hat{o}$pital's rule is applicable
here. Thus, we get
\[
\lim_{t\nearrow0}\frac{\partial_{t}\hat{u}\left(\left(-t\right)^{\frac{1}{2}+\sigma}z,\,t\right)}{-\left(\frac{1}{2}+\sigma\right)\left(-t\right)^{-\frac{1}{2}+\sigma}}\,=\,\hat{\psi}_{k}\left(z\right)-z\,\partial_{z}\hat{\psi}_{k}\left(z\right)\,>0
\]
by (\ref{positivity}), which contradicts with (\ref{hypothesis}). 
\end{proof}

\section{\label{degree C^0}$C^{0}$ estimates in Proposition \ref{degree}
and Proposition \ref{uniform estimates}}

Starting from this section, we are devoted to prove Proposition \ref{degree}
and Proposition \ref{uniform estimates}. From now on, we focus on
the estimate of the $\mathbf{admissible}$ MCF $\left\{ \Sigma_{t}^{\left(a_{0},\,a_{1}\right)}\right\} _{t_{0}\leq t\leq\mathring{t}}$
for which 
\begin{equation}
\Phi_{t_{1}}\left(a_{0},\,a_{1}\right)=\left(0,\,0\right)\label{backward condition}
\end{equation}
where $t_{0}\leq t_{1}\leq\mathring{t}<0$ are constants and $\mathring{t}\leq$$e^{-1}t_{1}$.
In this section, we would show that if $0<\rho\ll1\ll\beta$ (depending
on $n$, $\Lambda$) and $\left|t_{0}\right|\ll1$ (depending on $n$,
$\Lambda$, $\rho$, $\beta$) , there holds 
\begin{equation}
\sqrt{a_{0}^{2}+a_{1}^{2}}\,\leq\,C\left(n,\,\Lambda,\,\rho,\,\beta\right)\left(-t_{0}\right)^{\varsigma\lambda_{2}}\label{coefficients}
\end{equation}
where $\varsigma>0$ is a constant defined in (\ref{varsigma}). Moreover,
there is 
\begin{equation}
k\in\left(1-C\left(n,\,\Lambda,\,\rho,\,\beta\right)\left(-t_{0}\right)^{\varsigma\lambda_{2}},\,\,1+C\left(n,\,\Lambda,\,\rho,\,\beta\right)\left(-t_{0}\right)^{\varsigma\lambda_{2}}\right)\label{k}
\end{equation}
so that the following hold.
\begin{enumerate}
\item In the $\mathbf{outer}$ $\mathbf{region}$, the function $u\left(x,\,t\right)$
of $\Sigma_{t}^{\left(a_{0},\,a_{1}\right)}$ defined in (\ref{u})
satisfies 
\begin{equation}
\left|u(x,\,t)-u(x,\,t_{0})\right|\,\leq\,C\left(n\right)\sqrt{t-t_{0}}\label{C^0 outside u bound}
\end{equation}
for $x\geq\frac{1}{5}\rho$, $t_{0}\leq t\leq\mathring{t}$, and 
\begin{equation}
\left|u\left(x,\,t\right)-\frac{k}{c_{2}}\left(-t\right)^{\lambda_{2}+\frac{1}{2}}\varphi_{2}\left(\frac{x}{\sqrt{-t}}\right)\right|\,\leq\,C\left(n,\,\Lambda,\,\rho,\,\beta\right)\left(-t_{0}\right)^{\varkappa}x^{2\lambda_{2}+1}\label{C^0 u bound}
\end{equation}
for $\frac{1}{3}\sqrt{-t}\leq x\leq\rho$, $t_{0}\leq t\leq\mathring{t}$,
where $\varkappa>0$ is a constant defined in (\ref{varkappa}). Note
that
\[
\frac{k}{c_{2}}\left(-t\right)^{\lambda_{2}+\frac{1}{2}}\varphi_{2}\left(\frac{x}{\sqrt{-t}}\right)=kx^{2\lambda_{2}+1}\left(\varUpsilon_{2}+2\varUpsilon_{1}\left(\frac{-t}{x^{2}}\right)+\left(\frac{-t}{x^{2}}\right)^{2}\right)
\]
\item In the $\mathbf{intermediate}$$\mathbf{region}$, if we do the type
$\mathrm{I}$ rescaling, the function $v\left(y,\,s\right)$ of the
rescaled hypersurface $\Pi_{s}^{\left(a_{0},\,a_{1}\right)}$ defined
in (\ref{v}) satisfies
\begin{equation}
\left|v\left(y,\,s\right)-\frac{k}{c_{2}}e^{-\lambda_{2}s}\varphi_{2}\left(y\right)\right|\leq C\left(n,\,\Lambda,\,\rho,\,\beta\right)e^{-\varkappa s}\left(e^{-\lambda_{2}s}y^{\alpha+2}\right)\label{C^0 v bound intermediate}
\end{equation}
for $\frac{1}{2}e^{-\vartheta\sigma s}\leq y\leq\sqrt{\varsigma\lambda_{2}s}$,
$s_{0}\leq s\leq\mathring{s}$, and 
\begin{equation}
\left|\left(v\left(y,\,s\right)\,-\,e^{-\sigma s}\,\psi_{k}\left(e^{\sigma s}y\right)\right)\right|\,\leq\,C\left(n\right)\beta^{\alpha-3}e^{-2\varrho\sigma\left(s-s_{0}\right)}\left(e^{-\lambda_{2}s}y^{\alpha}\right)\label{C^0 v bound tip}
\end{equation}
for $\frac{4}{3}\beta e^{-\sigma s}\leq y\leq\frac{1}{2}e^{-\vartheta\sigma s}$,
$s_{0}\leq s\leq\mathring{s}$, where $\mathring{s}=-\ln\left(-\mathring{t}\right)$
and $0<\varrho<\vartheta<1$ are constants (see (\ref{vartheta})
and (\ref{varrho}) for definition). Note that 
\[
\frac{k}{c_{2}}e^{-\lambda_{2}s}\varphi_{2}\left(y\right)=ke^{-\lambda_{2}s}y^{\alpha}\left(1+2\varUpsilon_{1}y^{2}+\varUpsilon_{2}y^{4}\right)
\]
 
\[
e^{-\sigma s}\,\psi_{k}\left(e^{\sigma s}y\right)=ke^{-\lambda_{2}s}y^{\alpha}\left(1+O\left(\left(e^{\sigma s}y\right)^{-2\left(1-\alpha\right)}\right)\right)
\]
\item In the $\mathbf{tip}$ $\mathbf{region}$, if we do the type $\mathrm{II}$
rescaling, the function $\hat{w}\left(z,\,\tau\right)$ of the rescaled
hypersurface $\Gamma_{\tau}^{\left(a_{0},\,a_{1}\right)}$ defined
in (\ref{w'}) satisfies
\begin{equation}
\hat{\psi}_{\left(1-\beta^{\alpha-3}\left(\frac{\tau}{\tau_{0}}\right)^{-\varrho}\right)k}\left(z\right)\,\leq\,\hat{w}\left(z,\,\tau\right)\,\leq\,\hat{\psi}_{\left(1+\beta^{\alpha-3}\left(\frac{\tau}{\tau_{0}}\right)^{-\varrho}\right)k}\left(z\right)\label{C^0 w' bound}
\end{equation}
for $0\leq z\leq\left(2\sigma\tau\right)^{\frac{1}{2}\left(1-\vartheta\right)}$,
$\tau_{0}\leq\tau\leq\mathring{\tau}$, where $\mathring{\tau}=\frac{1}{2\sigma\left(-\mathring{t}\right)^{2\sigma}}$
.
\end{enumerate}
To achieve that, we first establish (\ref{C^0 v bound intermediate})
(see Proposition \ref{C^0 v}) by using the energy estimate and Sobolev
inequality. Next, we use the comparison principle and the boundary
values of (\ref{C^0 v bound intermediate}) to show (\ref{C^0 u bound})
(see Proposition \ref{C^0 u}) and (\ref{C^0 w' bound}) (see Proposition
\ref{C^0 w'}). Then we use (\ref{C^0 w' bound}) to deduce (\ref{C^0 v bound tip})
by rescaling and analyzing the projected curves. Lastly, we use the
gradient and curvature estimates in \cite{EH} to prove (\ref{C^0 outside u bound})
(see Proposition \ref{C^0 outside u}). The ideas of proving (\ref{C^0 u bound}),
(\ref{C^0 v bound intermediate}) and (\ref{C^0 w' bound}) are due
to Vel$\acute{a}$zquez (see \cite{V}). Here we improve his estimates
to get better results. 
\begin{rem}
\label{Lambda C^0 }By the above $C^{0}$ estimates, we deduce that
\[
-2\left(\varUpsilon_{1}^{2}-\varUpsilon_{2}\right)x^{2\lambda_{2}+1}\,\leq\,u\left(x,\,t\right)\,\leq\,2\left(1+2\varUpsilon_{1}+\varUpsilon_{2}\right)x^{2\lambda_{2}+1}
\]
for $\sqrt{-t}\leq x\leq\rho$, $t_{0}\leq t\leq\mathring{t}$, and
\[
2\left(1+2\varUpsilon_{1}+\varUpsilon_{2}\right)e^{-\lambda_{2}s}y^{\alpha}\,\leq\,v\left(y,\,s\right)\,\leq\,2e^{-\lambda_{2}s}y^{\alpha}
\]
for $\frac{4}{3}\beta e^{-\sigma s}\leq y\leq1$, $s_{0}\leq s\leq\mathring{s}$,
provided that $\beta\gg1$ (depending on $n$) and $\left|t_{0}\right|\ll1$
(depending on $n$, $\Lambda$, $\rho$, $\beta$). In Section \ref{degree Lambda},
we would use these etstimates to choose the constant $\Lambda=\Lambda\left(n\right)$. 
\end{rem}
In order to prove (\ref{C^0 v bound intermediate}), we need the following
two lemmas. The first lemma is the energy estimates for solutions
to a parabolic equation associated with the linear operator $\mathcal{L}$
(see (\ref{L})). Recall that in Proposition \ref{linear operator},
the eigenvalues of $\mathcal{L}$ satisfy $\lambda_{i}\geq\lambda_{3}>1$
for $i\geq3$.
\begin{lem}
\label{energy}Let $\boldsymbol{\mathrm{H}}_{*}$ be the closed subspace
of $\boldsymbol{\mathrm{H}}$ (see Proposition \ref{linear operator})
spanned by eigenfunctions $\left\{ \varphi_{i}\right\} _{i\geq3}$
of $\mathcal{L}$. Given 
\[
\mathsf{f}\left(\cdot,\,s\right)\in L^{2}\left(\left[s_{0},\,\mathring{s}\right];\,L^{2}\left(\mathbb{R}_{+},\,y^{2\left(n-1\right)}e^{-\frac{y^{2}}{4}}dy\right)\right)
\]
 and $\mathsf{h}\in\boldsymbol{\mathrm{H}}_{*}$, let $\mathsf{v}\left(\cdot,\,s\right)\in C\left(\left[s_{0},\,\mathring{s}\right];\,\boldsymbol{\mathrm{H}}_{*}\right)$
be the weak solution of 
\begin{equation}
\left\{ \begin{array}{c}
\left(\partial_{s}+\mathcal{L}\right)\mathsf{v}\left(\cdot,\,s\right)=\mathsf{f}\left(\cdot,\,s\right)\quad\textrm{for}\;\,s_{0}\leq s\leq\mathring{s}\\
\\
\mathsf{v}\left(\cdot,\,s_{0}\right)=\mathsf{h}
\end{array}\right.\label{Cauchy problem}
\end{equation}
Then for any $0<\delta<1$, there hold
\[
\left\Vert \mathsf{v}\left(\cdot,\,s\right)\right\Vert ^{2}
\]
\[
\leq\,e^{-2\left(1-\delta\right)\lambda_{3}\left(s-s_{0}\right)}\left\Vert \mathsf{v}\left(\cdot,\,s_{0}\right)\right\Vert ^{2}\,+\,\frac{1}{2\delta\lambda_{3}}\int_{s_{0}}^{s}e^{-2\left(1-\delta\right)\lambda_{3}\left(s-\xi\right)}\left\Vert \mathsf{f}\left(\cdot,\,\xi\right)\right\Vert ^{2}d\xi
\]
and 
\[
\left\langle \mathcal{L}\mathsf{v}\left(\cdot,\,s\right),\,\mathsf{v}\left(\cdot,\,s\right)\right\rangle 
\]
\[
\leq\,e^{-2\left(1-\delta\right)\lambda_{3}\left(s-s_{0}\right)}\left\langle \mathcal{L}\mathsf{h},\,\mathsf{h}\right\rangle \,+\,\frac{1}{2\delta}\int_{s_{0}}^{s}e^{-2\left(1-\delta\right)\lambda_{3}\left(s-\xi\right)}\left\Vert \mathsf{f}\left(\cdot,\,\xi\right)\right\Vert ^{2}d\xi
\]
for $s_{0}\leq s\leq\mathring{s}$, where the inner product $\left\langle \cdot,\cdot\right\rangle $
and the corresponding norm $\left\Vert \cdot\right\Vert $ are defined
in Proposition \ref{linear operator}.
\end{lem}
\begin{proof}
Let $\left\{ \mathsf{v}_{m}\right\} _{m\geq3}$ be the Galerkin's
approximation of $\mathsf{v}$. Namely,
\[
\mathsf{v}_{m}\left(y,\,s\right)=\sum_{i=3}^{m}\left(e^{-\lambda_{i}\left(s-s_{0}\right)}\left\langle \mathsf{h},\,\varphi_{i}\right\rangle +\int_{s_{0}}^{s}e^{-\lambda_{i}\left(s-\xi\right)}\left\langle \mathsf{f}\left(\cdot,\,\xi\right),\,\varphi_{i}\right\rangle d\xi\right)\varphi_{i}\left(y\right)
\]
Then we have 
\[
\left\{ \begin{array}{c}
\partial_{s}\mathsf{v}_{m}\left(\cdot,\,s\right)+\mathcal{L}\mathsf{v}_{m}\left(\cdot,\,s\right)=\mathsf{f}_{m}\left(\cdot,\,s\right)\quad\textrm{for}\;\,s_{0}\leq s\leq\mathring{s}\\
\\
\mathsf{v}_{m}\left(\cdot,\,s_{0}\right)=\sum_{i=3}^{m}\left\langle \mathsf{h},\,\varphi_{i}\right\rangle \varphi_{i}\,\rightarrow\,\mathsf{h}\quad\textrm{in}\;\,\boldsymbol{\mathrm{H}}_{*}
\end{array}\right.
\]
 where 
\[
\mathsf{f}_{m}\left(\cdot,\,s\right)=\sum_{i=3}^{m}\left\langle \mathsf{f}\left(\cdot,\,s\right),\,\varphi_{i}\right\rangle \varphi_{i}\,\rightarrow\,\mathsf{f}\left(\cdot,\,s\right)\quad\textrm{in}\;\,L^{2}\left(\left[s_{0},\,\mathring{s}\right];\,L^{2}\left(\mathbb{R}_{+},\,y^{2\left(n-1\right)}e^{-\frac{y^{2}}{4}}dy\right)\right)
\]
It follows that 
\[
\left\langle \partial_{s}\mathsf{v}_{m}\left(\cdot,\,s\right),\,\mathsf{v}_{m}\left(\cdot,\,s\right)\right\rangle +\left\langle \mathcal{L}\mathsf{v}_{m}\left(\cdot,\,s\right),\,\mathsf{v}_{m}\left(\cdot,\,s\right)\right\rangle =\left\langle \mathsf{f}_{m}\left(\cdot,\,s\right),\,\mathsf{v}_{m}\left(\cdot,\,s\right)\right\rangle 
\]
which, by Cauchy-Schwarz inequality, yields 
\[
\frac{1}{2}\,\partial_{s}\left\Vert \mathsf{v}_{m}\left(\cdot,\,s\right)\right\Vert ^{2}+\lambda_{3}\left\Vert \mathsf{v}_{m}\left(\cdot,\,s\right)\right\Vert ^{2}\,\leq\,\delta\lambda_{3}\left\Vert \mathsf{v}_{m}\left(\cdot,\,s\right)\right\Vert ^{2}+\frac{1}{4\delta\lambda_{3}}\left\Vert \mathsf{f}_{m}\left(\cdot,\,s\right)\right\Vert ^{2}
\]
\[
\Leftrightarrow\quad\partial_{s}\left\Vert \mathsf{v}_{m}\left(\cdot,\,s\right)\right\Vert ^{2}\,\leq\,-2\left(1-\delta\right)\lambda_{3}\left\Vert \mathsf{v}_{m}\left(\cdot,\,s\right)\right\Vert ^{2}+\frac{1}{2\delta\lambda_{3}}\left\Vert \mathsf{f}_{m}\left(\cdot,\,s\right)\right\Vert ^{2}
\]
for any $0<\delta<1$. Thus, by integrating the inquality with repect
to $s$, we get 
\begin{equation}
\left\Vert \mathsf{v}_{m}\left(\cdot,\,s\right)\right\Vert ^{2}\label{L^2 v}
\end{equation}
\[
\leq\,e^{-2\left(1-\delta\right)\lambda_{3}\left(s-s_{0}\right)}\left\Vert \mathsf{v}_{m}\left(\cdot,\,s_{0}\right)\right\Vert ^{2}+\frac{1}{2\delta\lambda_{3}}\int_{s_{0}}^{s}e^{-2\left(1-\delta\right)\lambda_{3}\left(s-\xi\right)}\left\Vert \mathsf{f}_{m}\left(\cdot,\,\xi\right)\right\Vert ^{2}d\xi
\]
for $s_{0}\leq s\leq\mathring{s}$. 

Similarly, we have 
\[
\left\langle \partial_{s}\mathsf{v}_{m}\left(\cdot,\,s\right),\,\partial_{s}\mathsf{v}_{m}\left(\cdot,\,s\right)\right\rangle +\left\langle \mathcal{L}\mathsf{v}_{m}\left(\cdot,\,s\right),\,\partial_{s}\mathsf{v}_{m}\left(\cdot,\,s\right)\right\rangle =\left\langle \mathsf{f}_{m}\left(\cdot,\,s\right),\,\partial_{s}\mathsf{v}_{m}\left(\cdot,\,s\right)\right\rangle 
\]
Substitute $\partial_{s}\mathsf{v}_{m}\left(\cdot,\,s\right)=-\mathcal{L}\mathsf{v}_{m}\left(\cdot,\,s\right)+\mathsf{f}_{m}\left(\cdot,\,s\right)$
into the above equation to get 
\[
\frac{1}{2}\,\partial_{s}\left\langle \mathcal{L}\mathsf{v}_{m}\left(\cdot,\,s\right),\,\mathsf{v}_{m}\left(\cdot,\,s\right)\right\rangle =-\left\langle \mathcal{L}\mathsf{v}_{m}\left(\cdot,\,s\right),\,\mathcal{L}\mathsf{v}_{m}\left(\cdot,\,s\right)\right\rangle +\left\langle \mathcal{L}\mathsf{v}_{m}\left(\cdot,\,s\right),\,\mathsf{f}_{m}\left(\cdot,\,s\right)\right\rangle 
\]
By Cauchy-Schwarz inequality, we get 
\[
\partial_{s}\left\langle \mathcal{L}\mathsf{v}_{m}\left(\cdot,\,s\right),\,\mathsf{v}_{m}\left(\cdot,\,s\right)\right\rangle 
\]
\[
\leq-2\left(1-\delta\right)\left\langle \mathcal{L}\mathsf{v}_{m}\left(\cdot,\,s\right),\,\mathcal{L}\mathsf{v}_{m}\left(\cdot,\,s\right)\right\rangle +\frac{1}{2\delta}\left\Vert \mathsf{f}_{m}\left(\cdot,\,s\right)\right\Vert ^{2}
\]
\[
\leq-2\left(1-\delta\right)\lambda_{3}\left\langle \mathcal{L}\mathsf{v}_{m}\left(\cdot,\,s\right),\,\mathsf{v}_{m}\left(\cdot,\,s\right)\right\rangle +\frac{1}{2\delta}\left\Vert \mathsf{f}_{m}\left(\cdot,\,s\right)\right\Vert ^{2}
\]
for any $0<\delta<1$. Thus, we have
\begin{equation}
\left\langle \mathcal{L}\mathsf{v}_{m}\left(\cdot,\,s\right),\,\mathsf{v}_{m}\left(\cdot,\,s\right)\right\rangle \label{H^1 v}
\end{equation}
\[
\leq e^{-2\left(1-\delta\right)\lambda_{3}\left(s-s_{0}\right)}\left\langle \mathcal{L}\mathsf{v}_{m}\left(\cdot,\,s_{0}\right),\,\mathsf{v}_{m}\left(\cdot,\,s_{0}\right)\right\rangle \,+\,\frac{1}{2\delta}\int_{s_{0}}^{s}e^{-2\left(1-\delta\right)\lambda_{3}\left(s-\xi\right)}\left\Vert \mathsf{f}_{m}\left(\cdot,\,\xi\right)\right\Vert ^{2}d\xi
\]
for $s_{0}\leq s\leq\mathring{s}$. 

On the other hand, for any $m,\,l\geq3$, there holds 
\[
\partial_{s}\left(\mathsf{v}_{m}\left(\cdot,\,s\right)-\mathsf{v}_{l}\left(\cdot,\,s\right)\right)+\mathcal{L}\left(\mathsf{v}_{m}\left(\cdot,\,s\right)-\mathsf{v}_{l}\left(\cdot,\,s\right)\right)=\mathsf{f}_{m}\left(\cdot,\,s\right)-\mathsf{f}_{l}\left(\cdot,\,s\right)
\]
By the same arguments as above, for any $0<\delta<1$, we can deduce
that 
\begin{equation}
\left\Vert \mathsf{v}_{m}\left(\cdot,\,s\right)-\mathsf{v}_{l}\left(\cdot,\,s\right)\right\Vert ^{2}\label{L^2 Cauchy}
\end{equation}
 
\[
\leq\,e^{-2\left(1-\delta\right)\lambda_{3}\left(s-s_{0}\right)}\left\Vert \mathsf{v}_{m}\left(\cdot,\,s_{0}\right)-\mathsf{v}_{l}\left(\cdot,\,s_{0}\right)\right\Vert ^{2}
\]
\[
+\,\frac{1}{2\delta\lambda_{3}}\int_{s_{0}}^{s}e^{-2\left(1-\delta\right)\lambda_{3}\left(s-\xi\right)}\left\Vert \mathsf{f}_{m}\left(\cdot,\,\xi\right)-\mathsf{f}_{l}\left(\cdot,\,\xi\right)\right\Vert ^{2}d\xi
\]
 and 
\begin{equation}
\left\langle \mathcal{L}\left(\mathsf{v}_{m}\left(\cdot,\,s\right)-\mathsf{v}_{l}\left(\cdot,\,s\right)\right),\,\left(\mathsf{v}_{m}\left(\cdot,\,s\right)-\mathsf{v}_{l}\left(\cdot,\,s\right)\right)\right\rangle \label{H^1 Cauchy}
\end{equation}
 
\[
\leq\,e^{-2\left(1-\delta\right)\lambda_{3}\left(s-s_{0}\right)}\left\langle \mathcal{L}\left(\mathsf{v}_{m}\left(\cdot,\,s_{0}\right)-\mathsf{v}_{l}\left(\cdot,\,s_{0}\right)\right),\,\mathsf{v}_{m}\left(\cdot,\,s_{0}\right)-\mathsf{v}_{l}\left(\cdot,\,s_{0}\right)\right\rangle 
\]
\[
+\,\frac{1}{2\delta}\,\int_{s_{0}}^{s}e^{-2\left(1-\delta\right)\lambda_{3}\left(s-\xi\right)}\left\Vert \mathsf{f}_{m}\left(\cdot,\,\xi\right)-\mathsf{f}_{l}\left(\cdot,\,\xi\right)\right\Vert ^{2}d\xi
\]
for $s_{0}\leq s\leq\mathring{s}$. Therefore, by (\ref{coercivity}),
(\ref{L^2 Cauchy}), (\ref{H^1 Cauchy}) and the uniqueness of weak
solutions, we get 
\[
\mathsf{v}_{m}\,\rightarrow\,\mathsf{v}\quad\textrm{in}\;\,C\left(\left[s_{0},\,\mathring{s}\right];\,\boldsymbol{\mathrm{H}}_{*}\right)
\]
The conclusion follows by passing (\ref{L^2 v}) and (\ref{H^1 v})
to limit.
\end{proof}
The second lemma is a Sobolev type inequality for functions in $\boldsymbol{\mathrm{H}}$,
which is the Hilbert space defined in Proposition \ref{linear operator}. 
\begin{lem}
\label{Sobolev} Functions in $\boldsymbol{\mathrm{H}}$ are actually
continuous, i.e., $\boldsymbol{\mathrm{H}}\subset C\left(\mathbb{R}_{+}\right)$.
Moreover, for any $\mathsf{v}\in\mathrm{\boldsymbol{\mathrm{H}}}$,
there holds 
\[
\left|\mathsf{v}\left(y\right)\right|\,\leq\,C\left(n\right)\left(\frac{1}{y^{n-\frac{1}{2}}}+e^{\frac{\left(y+1\right)^{2}}{4}}\right)\left(\left\Vert \partial_{y}\mathsf{v}\right\Vert +\left\Vert \mathsf{v}\right\Vert \right)
\]
for $y>0$.
\end{lem}
\begin{proof}
Let's first assume that $\mathsf{v}\in C^{1}\left(\mathbb{R}_{+}\right)\cap\boldsymbol{\mathrm{H}}$. 

For each $0<y\leq1$, by the fundamental theorem of calculus, we have
\[
\mathsf{v}\left(y\right)=\mathsf{v}\left(z\right)+\int_{z}^{y}\partial_{y}\mathsf{v}\left(\xi\right)d\xi\quad\quad\forall\;\;\frac{y}{2}\leq z\leq y
\]
which, by H$\ddot{o}$lder's inequality, implies 
\[
\left|\mathsf{v}\left(y\right)\right|^{2}\,\leq C\left(\left|\mathsf{v}\left(z\right)\right|^{2}\,+\,y\int_{\frac{y}{2}}^{y}\left|\partial_{y}\mathsf{v}\left(\xi\right)\right|^{2}d\xi\right)
\]
\[
\leq\,C\left|\mathsf{v}\left(z\right)\right|^{2}\,+\,C\left(n\right)\frac{y}{y^{2\left(n-1\right)}}\left(\int_{\frac{y}{2}}^{y}\left|\partial_{y}\mathsf{v}\left(\xi\right)\right|^{2}\xi^{2\left(n-1\right)}e^{-\frac{\xi^{2}}{4}}d\xi\right)
\]
for $\frac{y}{2}\leq z\leq y$. Integrate the above inequality against
$z^{2\left(n-1\right)}e^{-\frac{z^{2}}{4}}\,dz$ from $\frac{y}{2}$
to $y$ to get 
\[
\left|\mathsf{v}\left(y\right)\right|^{2}\left(\int_{\frac{y}{2}}^{y}z^{2\left(n-1\right)}e^{-\frac{z^{2}}{4}}\,dz\right)\,\leq\,C\int_{\frac{y}{2}}^{y}\left|\mathsf{v}\left(z\right)\right|^{2}z^{2\left(n-1\right)}e^{-\frac{z^{2}}{4}}dz
\]
\[
+\,C\left(n\right)\frac{1}{y^{2n-3}}\left(\int_{\frac{y}{2}}^{y}\left|\partial_{y}\mathsf{v}\left(\xi\right)\right|^{2}\xi^{2\left(n-1\right)}e^{-\frac{\xi^{2}}{4}}d\xi\right)\left(\int_{\frac{y}{2}}^{y}z^{2\left(n-1\right)}e^{-\frac{z^{2}}{4}}\,dz\right)
\]
which implies 
\[
\left|\mathsf{v}\left(y\right)\right|^{2}\,\leq\,C\left(n\right)\frac{1}{y^{2n-1}}\left(\int_{\frac{y}{2}}^{y}\left|\mathsf{v}\left(z\right)\right|^{2}z^{2\left(n-1\right)}e^{-\frac{z^{2}}{4}}dz\right)
\]
\[
+\,C\left(n\right)\frac{1}{y^{2n-3}}\left(\int_{\frac{y}{2}}^{y}\left|\partial_{y}\mathsf{v}\left(\xi\right)\right|^{2}\xi^{2\left(n-1\right)}e^{-\frac{\xi^{2}}{4}}d\xi\right)
\]
That is, 
\[
\left|\mathsf{v}\left(y\right)\right|\,\leq\,C\left(n\right)\left(\frac{1}{y^{n-\frac{1}{2}}}\left\Vert \mathsf{v}\right\Vert +\frac{1}{y^{n-\frac{3}{2}}}\left\Vert \partial_{y}\mathsf{v}\right\Vert \right)
\]
\[
\leq\,C\left(n\right)\frac{1}{y^{n-\frac{1}{2}}}\left(\left\Vert \partial_{y}\mathsf{v}\right\Vert +\left\Vert \mathsf{v}\right\Vert \right)
\]
for $0<y\leq1$.

Likewise, for each $y\geq1$, by the fundamental theorem of calculus,
we have 
\[
\mathsf{v}\left(y\right)=\mathsf{v}\left(z\right)-\int_{y}^{z}\partial_{y}\mathsf{v}\left(\xi\right)d\xi\quad\quad\forall\;\;y\leq z\leq y+1
\]
which implies 
\[
\left|\mathsf{v}\left(y\right)\right|^{2}\,\leq C\left(\left|\mathsf{v}\left(z\right)\right|^{2}+\int_{y}^{y+1}\left|\partial_{y}\mathsf{v}\left(\xi\right)\right|^{2}d\xi\right)
\]
\[
\leq C\left|\mathsf{v}\left(z\right)\right|^{2}\,+\,C\,y^{-2\left(n-1\right)}e^{\frac{\left(y+1\right)^{2}}{4}}\left(\int_{y}^{y+1}\left|\partial_{y}\mathsf{v}\left(\xi\right)\right|^{2}\xi^{2\left(n-1\right)}e^{-\frac{\xi^{2}}{4}}d\xi\right)
\]
for $y\leq z\leq y+1$. Integrate both sides againt $z^{2\left(n-1\right)}e^{-\frac{z^{2}}{4}}\,dz$
from $y$ to $y+1$ to get 
\[
\left|\mathsf{v}\left(y\right)\right|^{2}\left(\int_{y}^{y+1}z^{2\left(n-1\right)}e^{-\frac{z^{2}}{4}}dz\right)\,\leq C\,\int_{y}^{y+1}\left|\mathsf{v}\left(z\right)\right|^{2}z^{2\left(n-1\right)}e^{-\frac{z^{2}}{4}}dz
\]
\[
+\,C\,y^{-2\left(n-1\right)}e^{\frac{\left(y+1\right)^{2}}{4}}\left(\int_{y}^{y+1}\left|\partial_{y}\mathsf{v}\left(\xi\right)\right|^{2}\xi^{2\left(n-1\right)}e^{-\frac{\xi^{2}}{4}}d\xi\right)\left(\int_{y}^{y+1}z^{2\left(n-1\right)}e^{-\frac{z^{2}}{4}}dz\right)
\]
which yields 
\[
\left|\mathsf{v}\left(y\right)\right|^{2}\,\leq C\left(n\right)y^{-2\left(n-1\right)}e^{\frac{\left(y+1\right)^{2}}{4}}\left(\left\Vert \mathsf{v}\right\Vert ^{2}+\left\Vert \partial_{y}\mathsf{v}\right\Vert ^{2}\right)
\]
\[
\leq C\left(n\right)e^{\frac{\left(y+1\right)^{2}}{4}}\left(\left\Vert \partial_{y}\mathsf{v}\right\Vert +\left\Vert \mathsf{v}\right\Vert \right)
\]
for $y\geq1$.

More generally, given a function $\mathsf{v}\in\boldsymbol{\mathrm{H}}$,
then choose a sequence $\left\{ \mathsf{v}_{i}\right\} \subset C_{c}^{1}\left(\mathbb{R}_{+}\right)\cap\mathrm{\boldsymbol{\mathrm{H}}}$
so that 
\[
\mathsf{v}_{i}\,\overset{\mathrm{\boldsymbol{\mathrm{H}}}}{\longrightarrow}\,\mathsf{v}
\]
By the above arguments, we have
\[
\left|\mathsf{v}_{i}\left(y\right)\right|\,\leq\,C\left(n\right)\left(\frac{1}{y^{n-\frac{1}{2}}}+e^{\frac{\left(y+1\right)^{2}}{4}}\right)\left(\left\Vert \partial_{y}\mathsf{v}_{i}\right\Vert +\left\Vert \mathsf{v}_{i}\right\Vert \right)
\]
\[
\left|\mathsf{v}_{i}\left(y\right)-\mathsf{v}_{j}\left(y\right)\right|\,\leq\,C\left(n\right)\left(\frac{1}{y^{n-\frac{1}{2}}}+e^{\frac{\left(y+1\right)^{2}}{4}}\right)\left(\left\Vert \partial_{y}\mathsf{v}_{i}-\partial_{y}\mathsf{v}_{j}\right\Vert +\left\Vert \mathsf{v}_{i}-\mathsf{v}_{j}\right\Vert \right)
\]
for $y>0$. It follows, by the second inequality, that 
\[
\mathsf{v}_{i}\,\overset{C_{loc}}{\longrightarrow}\,\mathsf{v}
\]
Hence $\mathsf{v}\in C\left(\mathbb{R}_{+}\right)$. In addition,
by passing the first inequality to limit, we get 
\[
\left|\mathsf{v}\left(y\right)\right|\,\leq\,C\left(n\right)\left(\frac{1}{y^{n-\frac{1}{2}}}+e^{\frac{\left(y+1\right)^{2}}{4}}\right)\left(\left\Vert \partial_{y}\mathsf{v}\right\Vert +\left\Vert \mathsf{v}\right\Vert \right)
\]
for $y>0$.
\end{proof}
Now we are ready to prove (\ref{C^0 v bound intermediate}). The idea
is to linearize (\ref{eq v}) and do Fourier expansion. The condition
(\ref{backward condition}) allow us to control the evolution of components
in negative eigenvalue functions. For the remainder terms, we can
use the energy estimate and Sobolev inequality to get a $L^{\infty}$
estimate. 
\begin{prop}
\label{C^0 v}If $0<\rho\ll1\ll\beta$ (depending on $n$, $\Lambda$)
and $s_{0}\gg1$ (depending on $n$, $\Lambda$, $\rho$, $\beta$),
then (\ref{coefficients}) holds. Moreover, there is a constant $k$
satisfying (\ref{k}), for which the function $v\left(y,\,s\right)$
of the type $\mathrm{I}$ rescaled hypersurface $\Pi_{s}^{\left(a_{0},\,a_{1}\right)}$
(see (\ref{eq v})) satisfies (\ref{C^0 v bound intermediate}). 
\end{prop}
\begin{proof}
Let 
\[
\widetilde{v}\left(y,\,s\right)\,=\,\zeta\left(e^{\sigma s}y-\beta\right)\,\zeta\left(\rho e^{\frac{s}{2}}-y\right)\,v\left(y,\,s\right)
\]
then $\widetilde{v}\left(\cdot,\,s\right)\in C\left(\left[s_{0},\,\mathring{s}\right];\,\boldsymbol{\mathrm{H}}\right)$.
From (\ref{linearize eq v}), we have 
\[
\left(\partial_{s}+\mathcal{L}\right)v\left(\cdot,\,s\right)=\mathcal{Q}v\left(\cdot,\,s\right)
\]
which implies 
\begin{equation}
\left(\partial_{s}+\mathcal{L}\right)\widetilde{v}\left(\cdot,\,s\right)\,=\,f\left(\cdot,\,s\right)\equiv f_{\mathrm{I}}\left(\cdot,\,s\right)+f_{\mathrm{II}}\left(\cdot,\,s\right)+f_{\mathrm{III}}\left(\cdot,\,s\right)\label{eq localized v}
\end{equation}
where 
\[
f_{\mathrm{I}}\left(y,\,s\right)=\zeta\left(e^{\sigma s}y-\beta\right)\,\zeta\left(\rho e^{\frac{s}{2}}-y\right)\,\mathcal{Q}v\left(y,\,s\right)
\]
 
\[
f_{\mathrm{II}}\left(y,\,s\right)=\zeta'\left(e^{\sigma s}y-\beta\right)\,e^{\sigma s}\left(-2\,\partial_{y}v\left(y,\,s\right)+\left(-\frac{2\left(n-1\right)}{y}+\left(\sigma+\frac{1}{2}\right)y\right)v\left(y,\,s\right)\right)
\]
\[
-\zeta''\left(e^{\sigma s}y-\beta\right)\,e^{2\sigma s}\,v\left(y,\,s\right)
\]
 
\[
f_{\mathrm{III}}\left(y,\,s\right)=\zeta'\left(\rho e^{\frac{s}{2}}-y\right)\left(\left(\frac{\rho}{2}e^{\frac{s}{2}}-\frac{y}{2}+\frac{2\left(n-1\right)}{y}\right)v\left(y,\,s\right)+2\,\partial_{y}v\left(y,\,s\right)\right)
\]
\[
-\zeta''\left(\rho e^{\frac{s}{2}}-y\right)v\left(y,\,s\right)
\]
We claim that
\begin{equation}
\left\Vert f\left(\cdot,\,s\right)\right\Vert \leq C\left(n,\,\Lambda,\,\rho,\,\beta\right)e^{-\left(1+\varsigma\right)\lambda_{2}s}\label{L^2 f}
\end{equation}
for $s_{0}\leq s\leq\mathring{s}$, provided that $0<\rho\ll1\ll\beta$
(depending on $n$, $\Lambda$) and $s_{0}\gg1$ (depending on $n$,
$\Lambda$, $\rho$, $\beta$), where the norm $\left\Vert \cdot\right\Vert $
is defined in Proposition \ref{linear operator}. Notice that by (\ref{a priori bound v}),
we have 
\[
\max\left\{ \left|\frac{v\left(y,\,s\right)}{y}\right|,\,\left|\partial_{y}v\left(y,\,s\right)\right|\right\} \,\leq\,\Lambda e^{-\lambda_{2}s}\left(y^{\alpha-1}+y^{2\lambda_{2}}\right)\,\lesssim\,\Lambda\left(\beta^{\alpha-1}+\rho^{2\lambda_{2}}\right)
\]
for $\beta e^{-\sigma s}\leq y\leq\rho e^{\frac{s}{2}}$, so we have
\[
\max\left\{ \left|\frac{v\left(y,\,s\right)}{y}\right|,\,\left|\partial_{y}v\left(y,\,s\right)\right|\right\} \,\leq\,\frac{1}{3}
\]
for $\beta e^{-\sigma s}\leq y\leq\rho e^{\frac{s}{2}}$ provided
that $0<\rho\ll1\ll\beta$ (depending on $n$, $\Lambda$). To prove
(\ref{L^2 f}), we use (\ref{a priori bound v}) to get
\[
\left\Vert f_{\mathrm{I}}\right\Vert =\left\Vert \zeta\left(e^{\sigma s}y-\beta\right)\,\zeta\left(\rho e^{\frac{s}{2}}-y\right)\,\mathcal{Q}v\left(y,\,s\right)\right\Vert 
\]
\[
\leq C\left(n\right)\Lambda^{3}\left\Vert \left(e^{-\lambda_{2}s}\left(y^{\alpha-1}+y^{2\lambda_{2}}\right)\right)^{2}\,e^{-\lambda_{2}s}\left(y^{\alpha-2}+y^{2\lambda_{2}-1}\right)\,\chi_{\left(\beta e^{-\sigma s},\,\rho e^{\frac{s}{2}}\right)}\right\Vert 
\]
\[
\leq C\left(n\right)\Lambda^{3}\,e^{-\left(1+\varsigma\right)\lambda_{2}s}\left\Vert \left(e^{-\lambda_{2}s}\left(y^{\alpha-1}+y^{2\lambda_{2}}\right)\right)^{2-\varsigma}\,\left(y^{\alpha-2+\varsigma\left(\alpha-1\right)}+y^{2\lambda_{2}-1+2\varsigma\lambda_{2}}\right)\,\chi_{\left(\beta e^{-\sigma s},\,\rho e^{\frac{s}{2}}\right)}\right\Vert 
\]
\[
\leq C\left(n\right)\Lambda^{3}\,e^{-\left(1+\varsigma\right)\lambda_{2}s}\left\Vert \left(\beta^{\alpha-1}+\rho^{2\lambda_{2}}\right)^{2-\varsigma}\,\left(y^{\alpha-2+\varsigma\left(\alpha-1\right)}+y^{2\lambda_{2}-1+2\varsigma\lambda_{2}}\right)\,\chi_{\left(\beta e^{-\sigma s},\,\rho e^{\frac{s}{2}}\right)}\right\Vert 
\]
\[
\leq C\left(n\right)\Lambda^{3}\,e^{-\left(1+\varsigma\right)\lambda_{2}s}\left(\int_{0}^{\infty}\left(y^{2\left(\alpha-2+\varsigma\left(\alpha-1\right)\right)}+y^{2\left(2\lambda_{2}-1+2\varsigma\lambda_{2}\right)}\right)y^{2\left(n-1\right)}e^{-\frac{y^{2}}{4}}dy\right)^{\frac{1}{2}}
\]
\[
\leq C\left(n\right)\Lambda^{3}\,e^{-\left(1+\varsigma\right)\lambda_{2}s}
\]
since $\varsigma\leq\lambda_{2}^{-1}\leq1$ and $2\left(\alpha-2+\varsigma\left(\alpha-1\right)\right)+2\left(n-1\right)>-1$;
\[
\left\Vert f_{\mathrm{II}}\right\Vert \,\leq\,C\left(n\right)\Lambda\left\Vert e^{-\lambda_{2}s}y^{\alpha-2}\,\chi_{\left(\beta e^{-\sigma s},\,\left(\beta+1\right)e^{-\sigma s}\right)}\right\Vert 
\]
\[
\leq C\left(n\right)\Lambda\,e^{-\lambda_{2}s}\left(\int_{\beta e^{-\sigma s}}^{\left(\beta+1\right)e^{-\sigma s}}y^{2\left(\alpha-2\right)}y^{2\left(n-1\right)}dy\right)^{\frac{1}{2}}
\]
\[
\leq C\left(n\right)\Lambda\,e^{-\lambda_{2}s}\left(\beta e^{-\sigma s}\right)^{n+\alpha-\frac{5}{2}}\,\,\leq C\left(n\right)\Lambda\,\beta^{n+\alpha-\frac{5}{2}}e^{-\left(1+\varsigma\right)\lambda_{2}s}
\]
and
\[
\left\Vert f_{\mathrm{III}}\right\Vert \,\leq\,C\left(n\right)\Lambda\left\Vert e^{-\lambda_{2}s}y^{2\lambda_{2}+2}\chi_{\left(\rho e^{\frac{s}{2}}-1,\,\rho e^{\frac{s}{2}}\right)}\right\Vert 
\]
\[
=C\left(n\right)\Lambda\,e^{-\lambda_{2}s}\left(\int_{\rho e^{\frac{s}{2}}-1}^{\rho e^{\frac{s}{2}}}y^{2\left(2\lambda_{2}+2\right)}y^{2\left(n-1\right)}e^{-\frac{y^{2}}{4}}dy\right)^{\frac{1}{2}}
\]
\[
\leq\,C\left(n\right)\Lambda\,e^{-\lambda_{2}s}e^{-s}\,\leq\,C\left(n\right)\Lambda\,e^{-\left(1+\varsigma\right)\lambda_{2}s}
\]
provided that $s_{0}\gg1$ (depending on $n$, $\rho$).

Next, we would like to estimate the components of negative eigenvalue
functions in the Fourier expansion of $\widetilde{v}\left(\cdot,\,s\right)$.
For each $i\in\left\{ 0,\,1\right\} $, by Proposition \ref{linear operator},
(\ref{backward condition}) and (\ref{eq localized v}), we have 
\[
\left\{ \begin{array}{c}
\partial_{s}\left\langle \widetilde{v}\left(\cdot,\,s\right),\,\varphi_{i}\right\rangle +\lambda_{i}\left\langle \widetilde{v}\left(\cdot,\,s\right),\,\varphi_{i}\right\rangle =\left\langle f\left(\cdot,\,s\right),\,\varphi_{i}\right\rangle \\
\\
\left\langle \widetilde{v}\left(\cdot,\,s_{1}\right),\,\varphi_{i}\right\rangle =0
\end{array}\right.
\]
Note that $\lambda_{i}=\lambda_{2}-\left(2-i\right)<0$ and 
\[
\mathring{s}=-\ln\left(-\mathring{t}\right)\,\leq\,-\ln\left(-e^{-1}t_{1}\right)=s_{1}+1
\]
Therefore, for $s_{1}\leq s\leq\mathring{s}$, we have
\[
\left|\left\langle \widetilde{v}\left(\cdot,\,s\right),\,\varphi_{i}\right\rangle \right|=\left|\int_{s_{1}}^{s}e^{-\lambda_{i}\left(s-\xi\right)}\left\langle f\left(\cdot,\,\xi\right),\,\varphi_{i}\right\rangle d\xi\right|\leq\,\int_{s_{1}}^{s}e^{-\left(\lambda_{2}-2\right)\left(s-\xi\right)}\left\Vert f\left(\cdot,\,\xi\right)\right\Vert d\xi
\]
\[
\leq C\left(n,\,\Lambda,\,\rho,\,\beta\right)e^{-\left(\lambda_{2}-2\right)\left(s-s_{1}\right)}e^{-\left(1+\varsigma\right)\lambda_{2}s_{1}}
\]
\[
\leq C\left(n,\,\Lambda,\,\rho,\,\beta\right)e^{-\left(1+\varsigma\right)\lambda_{2}s}
\]
and for $s_{0}\leq s\leq s_{1}$, we have 
\[
\left|\left\langle \widetilde{v}\left(\cdot,\,s\right),\,\varphi_{i}\right\rangle \right|=\left|\int_{s}^{s_{1}}e^{\lambda_{i}\left(\xi-s\right)}\left\langle f\left(\cdot,\,\xi\right),\,\varphi_{i}\right\rangle d\xi\right|\leq\,\int_{s}^{s_{1}}e^{\left(\lambda_{2}-1\right)\left(\xi-s\right)}\left\Vert f\left(\cdot,\,\xi\right)\right\Vert d\xi
\]
\[
\leq C\left(n,\,\Lambda,\,\rho,\,\beta\right)e^{-\left(1+\varsigma\right)\lambda_{2}s}
\]
Thus, for $i\in\left\{ 0,\,1\right\} $, there holds 
\begin{equation}
\left|\left\langle \widetilde{v}\left(\cdot,\,s\right),\,\varphi_{i}\right\rangle \right|\,\leq\,C\left(n,\,\Lambda,\,\rho,\,\beta\right)e^{-\left(1+\varsigma\right)\lambda_{2}s}\label{estimate v1}
\end{equation}
for $s_{0}\leq s\leq\mathring{s}$. In addition, for $i\in\left\{ 0,\,1\right\} $,
by Lemma \ref{cut-off} we have 
\[
\left|\left\langle \widetilde{v}\left(\cdot,\,s_{0}\right),\,c_{i}\,\varphi_{i}\right\rangle -a_{i}e^{-\lambda_{2}s_{0}}\right|
\]
\[
=\left|\Bigl\langle\zeta\left(e^{\sigma s_{0}}y-\beta\right)\,\zeta\left(\rho e^{\frac{s_{0}}{2}}-y\right)v\left(\cdot,\,s_{0}\right),\,c_{i}\,\varphi_{i}\Bigr\rangle-a_{i}e^{-\lambda_{2}s_{0}}\right|
\]
\[
=e^{-\lambda_{2}s_{0}}\left|\Bigl\langle\zeta\left(e^{\sigma s_{0}}y-\beta\right)\,\zeta\left(\rho e^{\frac{s_{0}}{2}}-y\right)\left(\frac{1}{c_{2}}\varphi_{2}\left(y\right)+\frac{a_{0}}{c_{0}}\varphi_{0}\left(y\right)+\frac{a_{1}}{c_{1}}\varphi_{1}\left(y\right)\right),\,c_{i}\varphi_{i}\Bigr\rangle-a_{i}\right|
\]
\[
\leq C\left(n,\,\Lambda,\,\rho,\,\beta\right)e^{-\left(1+2\varsigma\right)\lambda_{2}s_{0}}
\]
which, together with (\ref{estimate v1}), implies 
\[
\left|a_{i}\right|\,\leq\,\left|e^{\lambda_{2}s_{0}}\left\langle \widetilde{v}\left(\cdot,\,s_{0}\right),\,c_{i}\,\varphi_{i}\right\rangle \right|\,+\,\left|e^{\lambda_{2}s_{0}}\left\langle \widetilde{v}\left(\cdot,\,s_{0}\right),\,c_{i}\,\varphi_{i}\right\rangle -a_{i}\right|
\]
\[
\leq C\left(n,\,\Lambda,\,\rho,\,\beta\right)e^{-\varsigma\lambda_{2}s_{0}}
\]

We continue to estimate the components of the first positive eigenvalue
functions in the Fourier expansion of $\widetilde{v}\left(\cdot,\,s\right)$.
By Proposition \ref{linear operator}, Lemma \ref{cut-off}, (\ref{initial v})
and (\ref{eq localized v}), we have 
\[
\left\{ \begin{array}{c}
\partial_{s}\left(e^{\lambda_{2}s}\left\langle \widetilde{v}\left(\cdot,\,s\right),\,\varphi_{2}\right\rangle \right)=e^{\lambda_{2}s}\left\langle f\left(\cdot,\,s\right),\,\varphi_{2}\right\rangle \\
\\
\left|e^{\lambda_{2}s_{0}}\left\langle \widetilde{v}\left(\cdot,\,s_{0}\right),\,c_{2}\varphi_{2}\right\rangle -1\right|\,\leq C\left(n\right)e^{-2\varsigma\lambda_{2}s_{0}}
\end{array}\right.
\]
Now let 
\[
k=e^{\lambda_{2}s_{1}}\left\langle \tilde{v}\left(\cdot,\,s_{1}\right),\,c_{2}\,\varphi_{2}\right\rangle 
\]
then for $s_{1}\leq s\leq\mathring{s}$, we have 
\[
\left|e^{\lambda_{2}s}\left\langle \widetilde{v}\left(\cdot,\,s\right),\,c_{2}\varphi_{2}\right\rangle -k\right|\,=\,\left|e^{\lambda_{2}s}\left\langle \widetilde{v}\left(\cdot,\,s\right),\,\varphi_{2}\right\rangle -e^{\lambda_{2}s_{1}}\left\langle \widetilde{v}\left(\cdot,\,s_{1}\right),\,c_{2}\,\varphi_{2}\right\rangle \right|
\]
\[
=\left|\int_{s_{1}}^{s}e^{\lambda_{2}\xi}\left\langle f\left(\cdot,\,\xi\right),\,\varphi_{2}\right\rangle d\xi\right|\,\leq\,\int_{s_{1}}^{s_{1}+1}e^{\lambda_{2}\xi}\left\Vert f\left(\cdot,\,\xi\right)\right\Vert d\xi
\]
\[
\leq\,C\left(n,\,\Lambda,\,\rho,\,\beta\right)e^{-\varsigma\lambda_{2}s}
\]
(since $\mathring{s}\leq s_{1}+1$), and for $s_{0}\leq s\leq s_{1}$
we have 
\[
\left|e^{\lambda_{2}s}\left\langle \widetilde{v}\left(\cdot,\,s\right),\,c_{2}\,\varphi_{2}\right\rangle -k\right|\,=\,\left|e^{\lambda_{2}s}\left\langle \widetilde{v}\left(\cdot,\,s\right),\,c_{2}\,\varphi_{2}\right\rangle -e^{\lambda_{2}s_{1}}\left\langle \widetilde{v}\left(\cdot,\,s_{1}\right),\,c_{2}\,\varphi_{2}\right\rangle \right|
\]
\[
=\left|\int_{s}^{s_{1}}e^{\lambda_{2}\xi}\left\langle f\left(\cdot,\,\xi\right),\,\varphi_{2}\right\rangle d\xi\right|\,\leq\,\int_{s}^{s_{1}}e^{\lambda_{2}\xi}\left\Vert f\left(\cdot,\,\xi\right)\right\Vert d\xi
\]
\[
\leq\,C\left(n,\,\Lambda,\,\rho,\,\beta\right)e^{-\varsigma\lambda_{2}s}
\]
Thus, we get 
\[
\left|k-1\right|\,\leq\,\left|k-e^{\lambda_{2}s_{0}}\left\langle \widetilde{v}\left(\cdot,\,s_{0}\right),\,c_{2}\,\varphi_{2}\right\rangle \right|\,+\,\left|e^{\lambda_{2}s_{0}}\left\langle \widetilde{v}\left(\cdot,\,s_{0}\right),\,c_{2}\,\varphi_{2}\right\rangle -1\right|
\]
\[
\leq\,C\left(n,\,\Lambda,\,\rho,\,\beta\right)e^{-\varsigma\lambda_{2}s}
\]
and 
\begin{equation}
\left|\left\langle \widetilde{v}\left(\cdot,\,s\right),\,\varphi_{2}\right\rangle -\frac{k}{c_{2}}e^{-\lambda_{2}s}\right|\,\leq\,C\left(n,\,\Lambda,\,\rho,\,\beta\right)e^{-\left(1+\varsigma\right)\lambda_{2}s}\label{estimate v2}
\end{equation}
for $s_{0}\leq s\leq\mathring{s}$. 

Now we would like to estimate the remaining parts in the Fourier expansion
of $\widetilde{v}\left(\cdot,\,s\right)$. Let
\[
\widetilde{v}_{*}\left(\cdot,\,s\right)=\widetilde{v}\left(\cdot,\,s\right)-\sum_{i=0}^{2}\left\langle \widetilde{v}\left(\cdot,\,s\right),\,\varphi_{i}\right\rangle \varphi_{i}
\]
then $\widetilde{v}_{*}\left(\cdot,\,s\right)\in C\left(\left[s_{0},\,s_{1}\right];\,\boldsymbol{\mathrm{H}}_{*}\right)$,
where $\boldsymbol{\mathrm{H}}_{*}$ is defined in Lemma \ref{energy}.
By Proposition \ref{linear operator} and (\ref{eq localized v}),
we have 
\[
\left(\partial_{s}+\mathcal{L}\right)\widetilde{v}_{*}\left(\cdot,\,s\right)\,=\,f\left(\cdot,\,s\right)-\sum_{i=0}^{2}\left\langle f\left(\cdot,\,s\right),\,\varphi_{i}\right\rangle \varphi_{i}\,\equiv\,f_{*}\left(\cdot,\,s\right)
\]
Note that $\left\Vert f_{*}\left(\cdot,\,s\right)\right\Vert \leq\left\Vert f\left(\cdot,\,s\right)\right\Vert $
and that $\lambda_{3}=\lambda_{2}+1$. By Lemma \ref{energy}, for
any $0<\delta<1$, we have 
\[
\left\Vert \widetilde{v}_{*}\left(\cdot,\,s\right)\right\Vert ^{2}
\]
\[
\leq\,e^{-2\left(1-\delta\right)\left(\lambda_{2}+1\right)\left(s-s_{0}\right)}\left\Vert \widetilde{v}_{*}\left(\cdot,\,s_{0}\right)\right\Vert ^{2}\,+\,\frac{1}{2\delta\lambda_{3}}\int_{s_{0}}^{s}e^{-2\left(1-\delta\right)\left(\lambda_{2}+1\right)\left(s-\xi\right)}\left\Vert f\left(\cdot,\,\xi\right)\right\Vert ^{2}d\xi
\]
 
\[
\left\langle \mathcal{L}\widetilde{v}_{*}\left(\cdot,\,s\right),\,\widetilde{v}_{*}\left(\cdot,\,s\right)\right\rangle 
\]
\[
=e^{-2\left(1-\delta\right)\left(\lambda_{2}+1\right)\left(s-s_{0}\right)}\left\langle \mathcal{L}\widetilde{v}_{*}\left(\cdot,\,s_{0}\right),\,\widetilde{v}_{*}\left(\cdot,\,s_{0}\right)\right\rangle \,+\,\frac{1}{2\delta}\int_{s_{0}}^{s}e^{-2\left(1-\delta\right)\left(\lambda_{2}+1\right)\left(s-\xi\right)}\left\Vert f\left(\cdot,\,\xi\right)\right\Vert ^{2}d\xi
\]
for $s_{0}\leq s\leq\mathring{s}$. We claim that 
\begin{equation}
\left\Vert \widetilde{v}_{*}\left(\cdot,\,s_{0}\right)\right\Vert \,+\,\left\Vert \mathcal{L}\widetilde{v}_{*}\left(\cdot,\,s_{0}\right)\right\Vert \leq\,C\left(n,\,\Lambda,\,\rho,\,\beta\right)e^{-\left(1+\varsigma\right)\lambda_{2}s_{0}}\label{H^1 localized v}
\end{equation}
Note that since $\varsigma<\lambda_{2}^{-1}$, there is $\delta\in\left(0,\,1\right)$
so that $\left(1-\delta\right)\left(\lambda_{2}+1\right)>\left(1+\varsigma\right)\lambda_{2}$.
Thus, we get 
\[
\left\Vert \widetilde{v}_{*}\left(\cdot,\,s\right)\right\Vert ^{2}+\left\langle \mathcal{L}\widetilde{v}_{*}\left(\cdot,\,s\right),\,\widetilde{v}_{*}\left(\cdot,\,s\right)\right\rangle \,\leq\,C\left(n,\,\Lambda,\,\rho,\,\beta\right)e^{-2\left(1+\varsigma\right)\lambda_{2}s}
\]
which, by (\ref{coercivity}), yields 
\[
\left\Vert \widetilde{v}_{*}\left(\cdot,\,s\right)\right\Vert ^{2}+\left\Vert \partial_{y}\widetilde{v}_{*}\left(\cdot,\,s\right)\right\Vert ^{2}\,\leq C\left(n,\,\Lambda,\,\rho,\,\beta\right)e^{-2\left(1+\varsigma\right)\lambda_{2}s}
\]
By Lemma \ref{Sobolev}, we then get
\[
\left|\widetilde{v}_{*}\left(y,\,s\right)\right|\,\leq\,C\left(n\right)\left(\left\Vert \partial_{y}\widetilde{v}_{*}\left(\cdot,\,s\right)\right\Vert \,+\,\left\Vert \widetilde{v}_{*}\left(\cdot,\,s\right)\right\Vert \right)\left(\frac{1}{y^{n-\frac{1}{2}}}+e^{\frac{\left(y+1\right)^{2}}{4}}\right)
\]
\begin{equation}
\leq C\left(n,\,\Lambda,\,\rho,\,\beta\right)e^{-\left(1+\varsigma\right)\lambda_{2}s}\left(\frac{1}{y^{n-\frac{1}{2}}}+e^{\frac{\left(y+1\right)^{2}}{4}}\right)\label{estimate v3}
\end{equation}
for $s_{0}\leq s\leq\mathring{s}$. To prove (\ref{H^1 localized v}),
we use Proposition \ref{linear operator}, Lemma \ref{cut-off}, (\ref{initial v})
and previous computation for derving (\ref{estimate v1}) and (\ref{estimate v2})
to get 
\[
\left\Vert \widetilde{v}_{*}\left(\cdot,\,s_{0}\right)\right\Vert =\left\Vert \widetilde{v}\left(\cdot,\,s_{0}\right)-\sum_{i=0}^{2}\left\langle \widetilde{v}\left(\cdot,\,s_{0}\right),\,\varphi_{i}\right\rangle \varphi_{i}\right\Vert 
\]
\[
\leq\left\Vert \widetilde{v}\left(\cdot,\,s_{0}\right)-e^{-\lambda_{2}s_{0}}\sum_{i=0}^{2}\,\frac{a_{i}}{c_{i}}\varphi_{i}\right\Vert \,+\,\left\Vert e^{-\lambda_{2}s_{0}}\sum_{i=0}^{2}\,\frac{a_{i}}{c_{i}}\varphi_{i}-\sum_{i=0}^{2}\left\langle \widetilde{v}\left(\cdot,\,s_{0}\right),\,\varphi_{i}\right\rangle \varphi_{i}\right\Vert 
\]
\[
\leq e^{-\lambda_{2}s_{0}}\left\Vert \left(1-\zeta\left(e^{\sigma s_{0}}y-\beta\right)\,\zeta\left(\rho e^{\frac{s_{0}}{2}}-y\right)\right)\sum_{i=0}^{2}\,\frac{a_{i}}{c_{i}}\varphi_{i}\right\Vert \,+\,\sum_{i=0}^{2}\,\frac{1}{c_{i}}\left|\left\langle \widetilde{v}\left(\cdot,\,s_{0}\right),\,c_{i}\varphi_{i}\right\rangle -a_{i}e^{-\lambda_{2}s_{0}}\right|
\]
\[
\leq\,C\left(n,\,\Lambda,\,\rho,\,\beta\right)e^{-\left(1+\varsigma\right)\lambda_{2}s_{0}}
\]
where $a_{2}=1$, and 
\[
\left\Vert \mathcal{L}\widetilde{v}_{*}\left(\cdot,\,s_{0}\right)\right\Vert =\left\Vert \mathcal{L}\left(\zeta\left(e^{\sigma s_{0}}y-\beta\right)\,\zeta\left(\rho e^{\frac{s_{0}}{2}}-y\right)v\left(\cdot,\,s_{0}\right)\right)-\sum_{i=0}^{2}\left\langle \widetilde{v}\left(\cdot,\,s\right),\,\varphi_{i}\right\rangle \lambda_{i}\varphi_{i}\right\Vert 
\]
\[
=\left\Vert \mathcal{L}\left(\zeta\left(e^{\sigma s_{0}}y-\beta\right)\,\zeta\left(\rho e^{\frac{s_{0}}{2}}-y\right)e^{-\lambda_{2}s_{0}}\sum_{i=0}^{2}\,\frac{a_{i}}{c_{i}}\varphi_{i}\right)-\sum_{i=0}^{2}\left\langle \widetilde{v}\left(\cdot,\,s\right),\,\varphi_{i}\right\rangle \lambda_{i}\varphi_{i}\right\Vert 
\]
\[
\leq e^{-\lambda_{2}s_{0}}\left\Vert \mathcal{L}\left(\zeta\left(e^{\sigma s_{0}}y-\beta\right)\,\zeta\left(\rho e^{\frac{s_{0}}{2}}-y\right)\sum_{i=0}^{2}\,\frac{a_{i}}{c_{i}}\varphi_{i}\right)-\sum_{i=0}^{2}\,\frac{a_{i}}{c_{i}}\lambda_{i}\varphi_{i}\right\Vert 
\]
\[
+\left\Vert \sum_{i=0}^{2}\left\langle \widetilde{v}\left(\cdot,\,s\right),\,\varphi_{i}\right\rangle \lambda_{i}\varphi_{i}-e^{-\lambda_{2}s_{0}}\sum_{i=0}^{2}\,\frac{a_{i}}{c_{i}}\lambda_{i}\varphi_{i}\right\Vert 
\]
\[
\leq\left\Vert h\right\Vert \,+\,\sum_{i=0}^{2}\,\frac{\lambda_{i}}{c_{i}}\left\Vert \left\langle \tilde{v}\left(\cdot,\,s_{0}\right),\,c_{i}\varphi_{i}\right\rangle -a_{i}e^{-\lambda_{2}s_{0}}\right\Vert 
\]
where 
\[
h\left(y\right)=\zeta'\left(e^{\sigma s_{0}}y-\beta\right)\,e^{\sigma s_{0}}\left(-2\,\partial_{y}v\left(y,\,s_{0}\right)+\left(-\frac{2\left(n-1\right)}{y}+\frac{y}{2}\right)v\left(y,\,s_{0}\right)\right)
\]
\[
+\zeta'\left(\rho e^{\frac{s_{0}}{2}}-y\right)\left(\left(-\frac{y}{2}+\frac{2\left(n-1\right)}{y}\right)v\left(y,\,s_{0}\right)+2\,\partial_{y}v\left(y,\,s_{0}\right)\right)
\]
\[
-\zeta''\left(e^{\sigma s_{0}}y-\beta\right)\,e^{2\sigma s_{0}}\,v\left(y,\,s_{0}\right)-\zeta''\left(\rho e^{\frac{s_{0}}{2}}-y\right)v\left(y,\,s_{0}\right)
\]
Note that by similar computation as for $f_{\mathrm{II}}\left(\cdot,\,s\right)$
and $f_{\mathrm{III}}\left(\cdot,\,s\right)$, we have 
\[
\left\Vert h\right\Vert \leq C\left(n,\,\Lambda,\,\rho,\,\beta\right)e^{-\left(1+\varsigma\right)\lambda_{2}s_{0}}
\]
Hence, 
\[
\left\Vert \mathcal{L}\widetilde{v}_{*}\left(\cdot,\,s_{0}\right)\right\Vert \leq\,C\left(n,\,\Lambda,\,\rho,\,\beta\right)e^{-\left(1+\varsigma\right)\lambda_{2}s_{0}}
\]

Lastly, combining (\ref{estimate v1}), (\ref{estimate v2}), and
(\ref{estimate v3}), we conclude 
\[
\left|\widetilde{v}\left(y,\,s\right)-\frac{k}{c_{2}}e^{-\lambda_{2}s}\varphi_{2}\left(y\right)\right|\,=\,\left|\sum_{i=0}^{2}\left\langle \widetilde{v}\left(\cdot,\,s\right),\,\varphi_{i}\right\rangle \varphi_{i}\left(y\right)\,+\,\widetilde{v}_{*}\left(y,\,s\right)\,-\,\frac{k}{c_{2}}e^{-\lambda_{2}s}\varphi_{2}\left(y\right)\right|
\]
\[
\leq\,\sum_{i=0}^{1}\left|\left\langle \widetilde{v}\left(\cdot,\,s\right),\,\varphi_{i}\right\rangle \varphi_{i}\left(y\right)\right|\,+\,\left|\left\langle \widetilde{v}\left(\cdot,\,s\right),\,\varphi_{2}\right\rangle \varphi_{2}\left(y\right)-\frac{k}{c_{2}}e^{-\lambda_{2}s}\varphi_{2}\left(y\right)\right|\,+\,\left|\widetilde{v}_{*}\left(y,\,s\right)\right|
\]
\[
\leq C\left(n,\,\Lambda,\,\rho,\,\beta\right)e^{-\left(1+\varsigma\right)\lambda_{2}s}\left(\frac{1}{y^{n-\frac{1}{2}}}+e^{\frac{\left(y+1\right)^{2}}{4}}\right)
\]
for $s_{0}\leq s\leq\mathring{s}$. As a result, for $\frac{1}{2}e^{-\vartheta\sigma s}\leq y\leq1$,
we have 
\[
\left|v\left(y,\,s\right)-\frac{k}{c_{2}}e^{-\lambda_{2}s}\varphi_{2}\left(y\right)\right|\,\leq\,C\left(n,\,\Lambda,\,\rho,\,\beta\right)\left(\frac{e^{-\varsigma\lambda_{2}s}}{y^{n+\alpha+\frac{3}{2}}}\right)e^{-\lambda_{2}s}y^{\alpha+2}
\]
\[
\leq C\left(n,\,\Lambda,\,\rho,\,\beta\right)e^{-\left(\varsigma\lambda_{2}-\vartheta\sigma\left(n+\alpha+\frac{3}{2}\right)\right)s}e^{-\lambda_{2}s}y^{\alpha+2}
\]
 and for $1\leq y\leq\sqrt{\varsigma\lambda_{2}s}$, we have
\[
\left|v\left(y,\,s\right)-\frac{k}{c_{2}}e^{-\lambda_{2}s}\varphi_{2}\left(y\right)\right|\,\leq\,C\left(n,\,\Lambda,\,\rho,\,\beta\right)\left(e^{-\varsigma\lambda_{2}s}e^{\frac{\left(y+1\right)^{2}}{4}}\right)e^{-\lambda_{2}s}y^{\alpha+2}
\]
\[
\leq C\left(n,\,\Lambda,\,\rho,\,\beta\right)e^{-\frac{\varsigma\lambda_{2}}{2}s}e^{-\lambda_{2}s}y^{\alpha+2}
\]
\end{proof}
As a corollary, by (\ref{uv}), Proposition \ref{C^0 v} and Remark
\ref{parameters}, we get 
\[
\left|u\left(x,\,t\right)-\frac{k}{c_{2}}\left(-t\right)^{\lambda_{2}+\frac{1}{2}}\varphi_{2}\left(\frac{x}{\sqrt{-t}}\right)\right|\,\leq\,C\left(n,\,\Lambda,\,\rho,\,\beta\right)\left(-t\right)^{\varkappa}\left(-t\right)x^{\alpha+2}
\]
\begin{equation}
\leq C\left(n,\,\Lambda,\,\rho,\,\beta\right)\left(-t\right)^{\varkappa}x^{2\lambda_{2}+1}\label{C^0 u bound intermediate}
\end{equation}
for $\frac{1}{3}\sqrt{-t}\leq x\leq\sqrt{\varsigma\lambda_{2}\,t\,\ln\left(-t\right)}$,
$t_{0}\leq t\leq\mathring{t}$. Below we use (\ref{eq u}), (\ref{initial u intermediate}),
(\ref{C^0 u bound intermediate}) and the comparison principle to
prove (\ref{C^0 u bound}).
\begin{prop}
\label{C^0 u}If $0<\rho\ll1$ (depending on $n$, $\Lambda$) and
$\left|t_{0}\right|\ll1$ (depending on $n$, $\Lambda$, $\rho$),
there holds (\ref{C^0 u bound}).
\end{prop}
\begin{proof}
First, by (\ref{a priori bound u}) we have
\[
\max\left\{ \left|\frac{u\left(x,\,t\right)}{x}\right|,\,\left|\partial_{x}u\left(x,\,t\right)\right|\right\} \,\leq\,\Lambda\left(\left(-t\right)^{2}x^{\alpha-1}+x^{2\lambda_{2}}\right)\,\leq\,\frac{1}{3}
\]
for $\sqrt{\varsigma\lambda_{2}\,t\,\ln\left(-t\right)}\leq x\leq\rho$,
$t_{0}\leq t\leq\mathring{t}$, provided that $0<\rho\ll1$ (depending
on $n$, $\Lambda$) and $\left|t_{0}\right|\ll1$ (depending on $n$,
$\Lambda$, $\rho$). 

By (\ref{eq u}), (\ref{a priori bound u}) and Remark \ref{parameters},
there holds
\[
\left|\partial_{t}u\left(x,\,t\right)\right|\,\leq\,C\left(n\right)\left(\left|\partial_{xx}^{2}u\left(x,\,t\right)\right|\,+\,\left|\frac{\partial_{x}u\left(x,\,t\right)}{x}\right|\,+\,\left|\frac{u\left(x,\,t\right)}{x^{2}}\right|\right)
\]
\[
\leq\,C\left(n\right)\Lambda\left(x^{\alpha+2}+\left(-t\right)^{2}x^{\alpha-2}\right)\,\leq\,C\left(n,\,\Lambda\right)x^{\alpha+2}
\]
for $\sqrt{\varsigma\lambda_{2}\,t\,\ln\left(-t\right)}\leq x\leq\rho$,
$t_{0}\leq t\leq\mathring{t}$. In addition, we have
\[
\partial_{t}\left(k\left(-t\right)^{\lambda_{2}+\frac{1}{2}}\varphi_{2}\left(\frac{x}{\sqrt{-t}}\right)\right)=k\,\,\partial_{t}\left(\varUpsilon_{2}x^{2\lambda_{2}+1}+2\varUpsilon_{1}\left(-t\right)x^{\alpha+2}+\left(-t\right)^{2}x^{\alpha}\right)
\]
\[
=-2k\left(\varUpsilon_{1}x^{\alpha+2}+\left(-t\right)x^{\alpha}\right)
\]
Thus, we get
\begin{equation}
\left|\partial_{t}\left(u\left(x,\,t\right)-k\left(-t\right)^{\lambda_{2}+\frac{1}{2}}\varphi_{2}\left(\frac{x}{\sqrt{-t}}\right)\right)\right|\,\leq\,C\left(n,\,\Lambda\right)x^{\alpha+2}\label{C^0 u1}
\end{equation}
for $\sqrt{\varsigma\lambda_{2}\,t\,\ln\left(-t\right)}\leq x\leq\rho$,
$t_{0}\leq t\leq\mathring{t}$.

On the other hand, at time $t_{0}$, by (\ref{varkappa}), (\ref{coefficients})
and (\ref{k}), there holds
\[
\left|u\left(x,\,t_{0}\right)-k\left(-t_{0}\right)^{\lambda_{2}+\frac{1}{2}}\varphi_{2}\left(\frac{x}{\sqrt{-t}}\right)\right|
\]
\[
\leq\left(-t_{0}\right)^{\lambda_{2}+\frac{1}{2}}\left(\frac{\left|k-1\right|}{c_{0}}\varphi_{2}\left(\frac{x}{\sqrt{-t}}\right)\,+\,\sum_{i=0}^{1}\,\frac{\left|a_{i}\right|}{c_{i}}\varphi_{i}\left(\frac{x}{\sqrt{-t}}\right)\right)
\]
\begin{equation}
\leq C\left(n,\,\Lambda,\,\rho,\,\beta\right)\left(-t_{0}\right)^{\varkappa}x^{2\lambda_{2}+1}\label{C^0 u2}
\end{equation}
for $\sqrt{\varsigma\lambda_{2}\,t\,\ln\left(-t\right)}\leq x\leq\rho$.
Moreover, by (\ref{C^0 u bound intermediate}) we have 
\begin{equation}
\left|u\left(x,\,t\right)-\frac{k}{c_{2}}\left(-t\right)^{\lambda_{2}+\frac{1}{2}}\varphi_{2}\left(\frac{x}{\sqrt{-t}}\right)\right|\,\leq\,C\left(n,\,\Lambda,\,\rho,\,\beta\right)\left(-t_{0}\right)^{\varkappa}x^{2\lambda_{2}+1}\label{C^0 u3}
\end{equation}
for $x=\sqrt{\varsigma\lambda_{2}\,t\,\ln\left(-t\right)}$, $t_{0}\leq t\leq\mathring{t}$.

Combining (\ref{C^0 u1}), (\ref{C^0 u2}) and (\ref{C^0 u3}), we
get 
\[
\left|u\left(x,\,t\right)-k\left(-t\right)^{\lambda_{2}+\frac{1}{2}}\varphi_{2}\left(\frac{x}{\sqrt{-t}}\right)\right|
\]
\[
\leq C\left(n,\,\Lambda,\,\rho,\,\beta\right)\left(-t_{0}\right)^{\varkappa}x^{2\lambda_{2}+1}\,+\,C\left(n,\,\Lambda\right)x^{\alpha+2}\left(t-t_{0}\right)
\]
\[
\leq C\left(n,\,\Lambda,\,\rho,\,\beta\right)\left(-t_{0}\right)^{\varkappa}x^{2\lambda_{2}+1}
\]
for $\sqrt{\varsigma\lambda_{2}\,t\,\ln\left(-t\right)}\leq x\leq\rho$,
$t_{0}\leq t\leq\mathring{t}$. The conclusion follows by (\ref{C^0 u bound intermediate})
and the above.
\end{proof}
Next, by (\ref{uv}) and Proposition \ref{C^0 v}, we have
\[
\left|w\left(z,\,\tau\right)-\frac{k}{c_{2}}\left(2\sigma\tau\right)^{\frac{\alpha}{2}}\varphi_{2}\left(\frac{z}{\sqrt{2\sigma\tau}}\right)\right|\,\leq\,C\left(n,\,\Lambda,\,\rho,\,\beta\right)\left(2\sigma\tau\right)^{-\frac{\varkappa}{2\sigma}}\,\frac{z^{2}}{2\sigma\tau}z^{\alpha}
\]
for $\frac{1}{2}\left(2\sigma\tau\right)^{\frac{1}{2}\left(1-\vartheta\right)}\leq z\leq\sqrt{2\sigma\tau},$
$\tau_{0}\leq\tau\leq\mathring{\tau}$. Notice that 
\[
\frac{k}{c_{2}}\left(2\sigma\tau\right)^{\frac{\alpha}{2}}\varphi_{2}\left(\frac{z}{\sqrt{2\sigma\tau}}\right)=kz^{\alpha}\left(1+2\varUpsilon_{1}\frac{z^{2}}{2\sigma\tau}+\varUpsilon_{2}\left(\frac{z^{2}}{2\sigma\tau}\right)^{2}\right)
\]
Hence we get
\[
\left|w\left(z,\,\tau\right)-kz^{\alpha}\right|\,\leq\,C\left(n\right)\frac{z^{2}}{2\sigma\tau}z^{\alpha}
\]
for $\frac{1}{2}\left(2\sigma\tau\right)^{\frac{1}{2}\left(1-\vartheta\right)}\leq z\leq\sqrt{2\sigma\tau}$,
$\tau_{0}\leq\tau\leq\mathring{\tau}$, provided that $\tau_{0}\gg1$
(depending on $n$, $\Lambda$, $\rho$, $\beta$). On the other hand,
by Lemma \ref{asymptotic psi} and (\ref{k}), we have
\[
\left|\psi_{k}\left(z\right)-kz^{\alpha}\right|\,\leq\,C\left(n\right)k^{3}z^{3\alpha-2}\,\leq\,C\left(n\right)z^{3\alpha-2}
\]
for $z\geq\frac{\hat{\psi}_{2}\left(0\right)}{\sqrt{2}}$, provided
that $\tau_{0}\gg1$ (depending on $n$, $\Lambda$, $\rho$, $\beta$).
Therefore, we get
\[
\left|w\left(z,\,\tau\right)-\psi_{k}\left(z\right)\right|\,\leq\,\left|w\left(z,\,\tau\right)-kz^{\alpha}\right|\,+\,\left|kz^{\alpha}-\psi_{k}\left(z\right)\right|
\]
\begin{equation}
\leq C\left(n\right)\left(\frac{z^{2}}{2\sigma\tau}+z^{2\left(\alpha-1\right)}\right)z^{\alpha}\label{boundary w}
\end{equation}
for $\frac{1}{2}\left(2\sigma\tau\right)^{\frac{1}{2}\left(1-\vartheta\right)}\leq z\leq\sqrt{2\sigma\tau}$,
$\tau_{0}\leq\tau\leq\mathring{\tau}$. Now consider the projected
curves $\bar{\mathcal{M}}_{k}$ and $\bar{\Gamma}_{\tau}^{\left(a_{0},\,a_{1}\right)}$
(see (\ref{projected minimal hypersurface}) and (\ref{projected Gamma})),
which can be viewed as graphes of $w\left(z,\,\tau\right)$ and $\psi_{k}\left(z\right)$
over $\mathcal{\bar{C}}$ (see (\ref{projected cone})), respectively.
Thus, (\ref{boundary w}) implies that
\[
\left|\hat{w}\left(z,\,\tau\right)-\hat{\psi}_{k}\left(z\right)\right|\,\leq\,C\left(n\right)\left(\frac{z^{2}}{2\sigma\tau}+z^{2\left(\alpha-1\right)}\right)z^{\alpha}
\]
for $\left(2\sigma\tau\right)^{\frac{1}{2}\left(1-\vartheta\right)}\leq z\leq\frac{1}{2}\sqrt{2\sigma\tau}$,
$\tau_{0}\leq\tau\leq\mathring{\tau}$, provided that $\tau_{0}\gg1$
(depending on $n$, $\Lambda$, $\rho$, $\beta$). In particular,
there holds 
\begin{equation}
\left|\hat{w}\left(z,\,\tau\right)-\hat{\psi}_{k}\left(z\right)\right|\,\leq\,C\left(n\right)\left(2\sigma\tau\right)^{-\vartheta}z^{\alpha}\label{boundary w'}
\end{equation}
for $z=\left(2\sigma\tau\right)^{\frac{1}{2}\left(1-\vartheta\right)},$
$\tau_{0}\leq\tau\leq\mathring{\tau}$, since $0<\vartheta<\frac{1-\alpha}{2-\alpha}$
(see \ref{vartheta}).

In addition, when $\tau=\tau_{0}$, by (\ref{initial w intermediate}),
(\ref{coefficients}) and (\ref{k}), we have
\[
\left|w\left(z,\,\tau_{0}\right)-\psi_{k}\left(z\right)\right|\,\leq\,\left|w\left(z,\,\tau_{0}\right)-kz^{\alpha}\right|\,+\,\left|kz^{\alpha}-\psi_{k}\left(z\right)\right|
\]
\[
\leq\,\left(\left|k-1\right|\,+\,\left|a_{0}\right|\,+\,\left|a_{1}\right|\,+\,C\left(n\right)\left(\frac{z^{2}}{2\sigma\tau_{0}}+z^{2\left(\alpha-1\right)}\right)\right)z^{\alpha}
\]
\[
\leq\left(C\left(n,\,\Lambda,\,\rho,\,\beta\right)\left(2\sigma\tau_{0}\right)^{-\frac{1-\alpha}{2}\varsigma}\,+\,C\left(n\right)\left(\left(2\sigma\tau_{0}\right)^{-\vartheta}+\beta^{2\left(\alpha-1\right)}\right)\right)z^{\alpha}
\]
\[
\leq C\left(n\right)\beta^{2\left(\alpha-1\right)}z^{\alpha}
\]
for $\beta\leq z\leq2\left(2\sigma\tau_{0}\right)^{\frac{1}{2}\left(1-\vartheta\right)}$,
provided that $\tau_{0}\gg1$ (depending on $n$, $\Lambda$, $\rho$,
$\beta$). By reparametrizing $\bar{\Gamma}_{\tau_{0}}^{\left(a_{0},\,a_{1}\right)}$
and $\mathcal{\bar{M}}_{k}$, we deduce that 
\begin{equation}
\left|\hat{w}\left(z,\,\tau_{0}\right)-\hat{\psi}_{k}\left(z\right)\right|\,\leq\,C\left(n\right)\beta^{2\left(\alpha-1\right)}z^{\alpha}\label{initial w' intermediate}
\end{equation}
for $\frac{3}{2}\beta\leq z\leq\left(2\sigma\tau_{0}\right)^{\frac{1}{2}\left(1-\vartheta\right)}$,
provided that $\tau_{0}\gg1$ (depending on $n$, $\Lambda$, $\rho$,
$\beta$). 

Below we use (\ref{eq w'}), (\ref{boundary w'}), (\ref{initial w' intermediate})
and the comparison principle to prove (\ref{C^0 w' bound}). We follow
Vel$\acute{a}$zquez's idea of using the perturbation of $\hat{\psi}_{k}$
to construct barriers; moreover, we allow the perturbation to be time-dependent. 
\begin{prop}
\label{C^0 w'}If $\beta\gg1$ (depending on $n$) and $\tau_{0}\gg1$
(depending on $n$, $\Lambda$, $\rho$, $\beta$), there holds (\ref{C^0 w' bound}).
In particular, we have 
\begin{equation}
\left|\hat{w}\left(z,\,\tau\right)-\hat{\psi}_{k}\left(z\right)\right|\,\leq\,C\left(n\right)\beta^{\alpha-3}\left(\frac{\tau}{\tau_{0}}\right)^{-\varrho}z^{\alpha}\label{C^0 w' bound intermediate}
\end{equation}
for $\beta\leq z\leq\left(2\sigma\tau\right)^{\frac{1}{2}\left(1-\vartheta\right)}$,
$\tau_{0}\leq\tau\leq\mathring{\tau}$, and 
\begin{equation}
\left|\hat{w}\left(z,\,\tau\right)-\hat{\psi}_{k}\left(z\right)\right|\,\leq\,C\left(n\right)\beta^{\alpha-3}\left(\frac{\tau}{\tau_{0}}\right)^{-\varrho}\label{C^0 w' bound tip}
\end{equation}
for $0\leq z\leq5\beta$, $\tau_{0}\leq\tau\leq\mathring{\tau}$.
\end{prop}
\begin{proof}
Given functions $\lambda\left(\tau\right)$ and $\mu\left(\tau\right)$,
we define the perturbation of $\hat{\psi}_{k}$ by
\[
\hat{\psi}_{k}^{\lambda,\,\mu}\left(z,\,\tau\right)\,\equiv\,\hat{\psi}_{\lambda\left(\tau\right)\,k}\left(\frac{z}{\mu\left(\tau\right)}\right)\,=\,\lambda^{\frac{1}{1-\alpha}}\left(\tau\right)\,\hat{\psi}_{k}\left(\frac{z}{\lambda^{\frac{1}{1-\alpha}}\left(\tau\right)\,\mu\left(\tau\right)}\right)
\]
(see also (\ref{psi'})). By (\ref{eq psi'}), there holds 
\[
\partial_{\tau}\hat{\psi}_{k}^{\lambda,\,\mu}-\left(\frac{\partial_{zz}^{2}\hat{\psi}_{k}^{\lambda,\,\mu}}{1+\left(\partial_{z}\hat{\psi}_{k}^{\lambda,\,\mu}\right)^{2}}+\left(n-1\right)\left(\frac{\partial_{z}\hat{\psi}_{k}^{\lambda,\,\mu}}{z}-\frac{1}{\hat{\psi}_{k}^{\lambda,\,\mu}}\right)+\frac{\frac{1}{2}+\sigma}{2\sigma\tau}\left(-z\,\partial_{z}\hat{\psi}_{k}^{\lambda,\,\mu}+\hat{\psi}_{k}^{\lambda,\,\mu}\right)\right)
\]
\[
=\left.\left(-\frac{\frac{1}{2}+\sigma}{2\sigma\tau}\lambda^{\frac{1}{1-\alpha}}\,+\,\frac{\lambda^{\frac{\alpha}{1-\alpha}}}{1-\alpha}\left(\partial_{\tau}\lambda\right)\right)\left(\hat{\psi}_{k}\left(r\right)-r\,\partial_{r}\hat{\psi}_{k}\left(r\right)\right)-\,\frac{\lambda^{\frac{1}{1-\alpha}}}{\mu}\left(\partial_{\tau}\mu\right)\left(r\,\partial_{r}\hat{\psi}_{k}\left(r\right)\right)\right|_{r=\frac{z}{\lambda^{\frac{1}{1-\alpha}}\mu}}
\]
\begin{equation}
+\left.\frac{\mu^{2}-1}{\lambda^{\frac{1}{1-\alpha}}\mu^{2}}\left(\frac{\partial_{rr}^{2}\hat{\psi}_{k}\left(r\right)}{\left(1+\left(\partial_{r}\hat{\psi}_{k}\left(r\right)\right)^{2}\right)\left(1+\left(\frac{\partial_{r}\hat{\psi}_{k}\left(r\right)}{\mu}\right)^{2}\right)}+\left(n-1\right)\frac{\partial_{r}\hat{\psi}_{k}\left(r\right)}{r}\right)\right|_{r=\frac{z}{\lambda^{\frac{1}{1-\alpha}}\mu}}\label{eq perturbation psi}
\end{equation}
Notice that 
\begin{equation}
\left\{ \begin{array}{c}
\partial_{\lambda}\left(\hat{\psi}_{k}^{\lambda,\,\mu}\left(z\right)\right)=\left.\frac{\lambda^{\frac{\alpha}{1-\alpha}}}{1-\alpha}\left(\hat{\psi}_{k}\left(r\right)-r\,\partial_{r}\hat{\psi}_{k}\left(r\right)\right)\right|_{r=\frac{z}{\lambda^{\frac{1}{1-\alpha}}\mu}}\\
\\
\partial_{\mu}\left(\hat{\psi}_{k}^{\lambda,\,\mu}\left(z\right)\right)=\left.-\frac{\lambda^{\frac{1}{1-\alpha}}}{\mu}\left(r\,\partial_{r}\hat{\psi}_{k}\left(r\right)\right)\right|_{r=\frac{z}{\lambda^{\frac{1}{1-\alpha}}\mu}}
\end{array}\right.\label{expansion perturbation psi}
\end{equation}
Moreover, by (\ref{asymptotic psi'}), there holds 
\[
\lim_{r\nearrow\infty}\frac{\hat{\psi}_{k}\left(r\right)-r\,\partial_{r}\hat{\psi}_{k}}{r^{\alpha}}\,=\,k\,\lim_{r\nearrow\infty}\frac{\hat{\psi}\left(r\right)-r\,\partial_{r}\hat{\psi}}{r^{\alpha}}=k\left(1-\alpha\right)\,2^{\frac{\alpha+1}{2}}
\]
which implies 
\begin{equation}
\hat{\psi}_{k}\left(r\right)-r\,\partial_{r}\hat{\psi}_{k}=\left(1+o\left(1\right)\right)\left(1-\alpha\right)\,2^{\frac{\alpha+1}{2}}r^{\alpha}\label{asymptotic psi'0}
\end{equation}
for $r\geq\beta$, if $\beta\gg1$ (depending on $n$) and $\tau_{0}\gg1$
(depending on $n$, $\Lambda$, $\rho$, $\beta$).

To get a lower barrier, we set 
\[
\hat{w}_{-}\left(z,\,\tau\right)=\hat{\psi}_{k}^{\lambda_{-},\,\mu_{-}}\left(z,\,\tau\right)
\]
with 
\[
\lambda_{-}\left(\tau\right)=1-\beta^{\alpha-3}\left(\frac{\tau}{\tau_{0}}\right)^{-\varrho},\qquad\mu_{-}\left(\tau\right)=1
\]
where $\beta\gg1$ (depending on $n$). Firstly, for the initial value,
by Lemma \ref{monotonicity} and (\ref{initial w'}), we have
\begin{equation}
\hat{w}_{-}\left(z,\,\tau_{0}\right)\,=\,\hat{\psi}_{\lambda_{-}\left(\tau_{0}\right)\,k}\left(z\right)\,=\,\hat{\psi}_{\left(1-\beta^{\alpha-3}\right)\left(1+o\left(1\right)\right)}\left(z\right)<\,\hat{w}\left(z,\,\tau_{0}\right)\label{initial w'- tip}
\end{equation}
for $0\leq z\leq\frac{3}{2}\beta$, provided that $\beta\gg1$ (depending
on $n$). Also, for each $\frac{3}{2}\beta\leq z\leq\left(2\sigma\tau_{0}\right)^{\frac{1}{2}\left(1-\vartheta\right)}$,
by (\ref{expansion perturbation psi}), (\ref{asymptotic psi'0}),
(\ref{initial w' intermediate}) and the mean value theorem, there
is $\lambda_{-}\left(\tau_{0}\right)\leq\lambda_{*}\leq1$ so that
\[
\hat{w}_{-}\left(z,\,\tau_{0}\right)\,=\,\hat{\psi}_{k}\left(z\right)\,+\,\left(\lambda_{-}\left(\tau_{0}\right)-1\right)\,\left.\partial_{\lambda}\left(\hat{\psi}_{k}^{\lambda,\,\mu}\left(z\right)\right)\right|_{\lambda=\lambda_{*},\;z=z_{*}\equiv\frac{z}{\lambda_{*}^{\frac{1}{1-\alpha}}}}
\]
\[
=\,\hat{\psi}_{k}\left(z\right)\,-\,\beta^{\alpha-3}\frac{\lambda_{*}^{\frac{\alpha}{1-\alpha}}}{1-\alpha}\left(\hat{\psi}_{k}\left(z_{*}\right)-z_{*}\,\partial_{z}\hat{\psi}_{k}\left(z_{*}\right)\right)
\]
\begin{equation}
\leq\,\hat{\psi}_{k}\left(z\right)\,-\,\left(1-o\left(1\right)\right)\beta^{\alpha-3}2^{\frac{\alpha+1}{2}}z^{\alpha}\,<\,\hat{w}\left(z,\,\tau_{0}\right)\label{initial w'- intermediate}
\end{equation}
provided that $\beta\gg1$ (depending on $n$). Secondly, for the
boundary value, fix $\tau_{0}\leq\tau\leq\mathring{\tau}$ and let
$z=\left(2\sigma\tau\right)^{\frac{1}{2}\left(1-\vartheta\right)}$.
By (\ref{boundary w'}), (\ref{expansion perturbation psi}), (\ref{asymptotic psi'0})
and the mean value theorem, there is $\lambda_{-}\left(\tau_{0}\right)\leq\lambda_{*}\leq1$
so that 
\[
\hat{w}_{-}\left(z,\,\tau_{0}\right)\,=\,\hat{\psi}_{k}\left(z\right)\,+\,\left(\lambda_{-}\left(\tau_{0}\right)-1\right)\,\left.\partial_{\lambda}\left(\hat{\psi}_{k}^{\lambda,\,\mu}\left(z\right)\right)\right|_{\lambda=\lambda_{*},\;z=z_{*}\equiv\frac{z}{\lambda_{*}^{\frac{1}{1-\alpha}}}}
\]
\[
=\,\hat{\psi}_{k}\left(z\right)\,-\,\beta^{\alpha-3}\left(\frac{\tau}{\tau_{0}}\right)^{-\varrho}\frac{\lambda_{*}^{\frac{\alpha}{1-\alpha}}}{1-\alpha}\left(\hat{\psi}_{k}\left(z_{*}\right)-z_{*}\,\partial_{z}\hat{\psi}_{k}\left(z_{*}\right)\right)
\]
\[
\leq\,\hat{\psi}_{k}\left(z\right)\,-\left(1-o\left(1\right)\right)\beta^{\alpha-3}2^{\frac{\alpha+1}{2}}\left(\frac{\tau}{\tau_{0}}\right)^{-\varrho}z^{\alpha}
\]
\begin{equation}
<\,\hat{\psi}_{k}\left(z\right)\,-C\left(n\right)\left(2\sigma\tau\right)^{-\vartheta}z^{\alpha}\,\leq\,\hat{w}\left(z,\,\tau\right)\label{boundary w'-}
\end{equation}
provided that $\tau_{0}\gg1$ (depending on $n$, $\beta$), since
$0<\varrho<\vartheta$. Thirdly, for the equation, by (\ref{eq perturbation psi}),
there holds
\[
\partial_{\tau}\hat{w}_{-}-\left(\frac{\partial_{zz}^{2}\hat{w}_{-}}{1+\left(\partial_{z}\hat{w}_{-}\right)^{2}}+\left(n-1\right)\left(\frac{\partial_{z}\hat{w}_{-}}{z}-\frac{1}{\hat{w}_{-}}\right)+\frac{\frac{1}{2}+\sigma}{2\sigma\tau}\left(-z\,\partial_{z}\hat{w}_{-}+\hat{w}_{-}\right)\right)
\]
\[
=\left.\left(-\frac{\frac{1}{2}+\sigma}{2\sigma\tau}\lambda_{-}^{\frac{1}{1-\alpha}}\left(\tau\right)\,+\,\frac{\lambda_{-}^{\frac{\alpha}{1-\alpha}}\left(\tau\right)}{1-\alpha}\left(\partial_{\tau}\lambda_{-}\left(\tau\right)\right)\right)\left(\hat{\psi}_{k}-r\,\partial_{r}\hat{\psi}_{k}\right)\right|_{r=\frac{z}{\lambda_{-}^{\frac{1}{1-\alpha}}\left(\tau\right)}}
\]
\[
=\left.\frac{\lambda_{-}^{\frac{1}{1-\alpha}}\left(\tau\right)}{2\sigma\tau}\left(-\left(\frac{1}{2}+\sigma\right)+\frac{2\sigma\varrho\,\beta^{\alpha-3}\left(\frac{\tau}{\tau_{0}}\right)^{-\varrho}}{\left(1-\alpha\right)\lambda_{-}\left(\tau\right)}\right)\right|_{r=\frac{z}{\lambda_{-}^{\frac{1}{1-\alpha}}\left(\tau\right)}}\leq0
\]
for $0\leq z\leq\left(2\sigma\tau\right)^{\frac{1}{2}\left(1-\vartheta\right)}$,
$\tau_{0}\leq\tau\leq\mathring{\tau}$, provided that $\beta\gg1$
(depending on $n$). Then we subtract the above equation from (\ref{eq w'})
to get 
\begin{equation}
\partial_{\tau}\left(\hat{w}-\hat{w}_{-}\right)-\left(\frac{1}{1+\left(\partial_{z}\hat{w}\right)^{2}}\,\partial_{zz}^{2}\left(\hat{w}-\hat{w}_{-}\right)\,+\,\frac{n-1}{z}\,\partial_{z}\left(\hat{w}-\hat{w}_{-}\right)\right)\label{eq Cauchy w'-}
\end{equation}
\[
+\left(\frac{\partial_{zz}^{2}\hat{w}_{-}\left(\partial_{z}\hat{w}+\partial_{z}\hat{w}_{-}\right)}{\left(1+\left(\partial_{z}\hat{w}\right)^{2}\right)\left(1+\left(\partial_{z}\hat{w}_{-}\right)^{2}\right)}\,+\,\frac{\frac{1}{2}+\sigma}{2\sigma\tau}z\right)\partial_{z}\left(\hat{w}-\hat{w}_{-}\right)\,-\left(\frac{n-1}{\hat{w}\,\hat{w}_{-}}\,+\,\frac{\frac{1}{2}+\sigma}{2\sigma\tau}\right)\left(\hat{w}-\hat{w}_{-}\right)
\]
\[
\geq\,0
\]
Now we are ready to show that $\hat{w}_{-}$ is a lower barrier. Let
\[
\left(\hat{w}-\hat{w}_{-}\right)_{\min}\left(\tau\right)=\min_{0\leq z\leq\left(2\sigma\tau\right)^{\frac{1}{2}\left(1-\vartheta\right)}}\left(\hat{w}-\hat{w}_{-}\right)\left(z,\,\tau\right)
\]
then by (\ref{initial w'- tip}) and (\ref{initial w'- intermediate}),
we have 
\[
\left(\hat{w}-\hat{w}_{-}\right)_{\min}\left(\tau_{0}\right)>0
\]
We claim that 
\[
\left(\hat{w}-\hat{w}_{-}\right)_{\min}\left(\tau\right)\geq0\qquad\forall\;\,\tau_{0}\leq\tau\leq\mathring{\tau}
\]
Suppose the contrary, then there is $\tau_{0}<\tau_{1}^{*}\leq\mathring{\tau}$
so that 
\begin{equation}
\left(\hat{w}-\hat{w}_{-}\right)_{\min}\left(\tau_{1}^{*}\right)<0\label{w'- negative}
\end{equation}
Let $\tau_{0}^{*}\in\left[\tau_{0},\,\tau_{1}^{*}\right)$ be the
first time after which $\left(\hat{w}-\hat{w}_{-}\right)_{\min}$
stays negative all the way up to $\tau_{1}^{*}$. By continuity, there
holds 
\begin{equation}
\left(\hat{w}-\hat{w}_{-}\right)_{\min}\left(\tau_{0}^{*}\right)=0\label{w'- vanish}
\end{equation}
On the other hand, by (\ref{boundary w'-}), the negative minimum
of $\hat{w}-\hat{w}_{-}$ for each time-slice is achieved in $\left[0,\,\left(2\sigma\tau\right)^{\frac{1}{2}\left(1-\vartheta\right)}\right)$.
Hence, applying the maximum principle to (\ref{eq Cauchy w'-}), we
get 
\[
\partial_{\tau}\left(\hat{w}-\hat{w}_{-}\right)_{\min}-\left(\frac{n-1}{\hat{w}\,\hat{w}_{-}}\,+\,\frac{\frac{1}{2}+\sigma}{2\sigma\tau}\right)\left(\hat{w}-\hat{w}_{-}\right)_{\min}\geq0
\]
Notice that 
\[
\partial_{z}\hat{w}\left(0,\,\tau\right)\,=0=\,\partial_{z}\hat{w}_{-}\left(0,\,\tau\right)\qquad\forall\;\,\tau_{0}\leq\tau\leq\mathring{\tau}
\]
So at $z=0$, by L'H$\hat{o}$pital's rule, the third term in (\ref{eq Cauchy w'-})
is interpreted as 
\[
\lim_{z\rightarrow0}\frac{n-1}{z}\,\partial_{z}\left(\hat{w}-\hat{w}_{-}\right)\left(z,\,\tau\right)=\left(n-1\right)\,\partial_{zz}^{2}\left(\hat{w}-\hat{w}_{-}\right)\left(0,\,\tau\right)
\]
It follows that 
\[
\partial_{\tau}\left(e^{-\int\frac{n-1}{\hat{w}\,\hat{w}_{-}}d\tau}\tau^{-\left(\frac{1}{2}+\frac{1}{4\sigma}\right)}\,\left(\hat{w}-\hat{w}_{-}\right)_{\min}\left(\tau\right)\right)\geq0
\]
which, together with (\ref{w'- negative}), contradicts with (\ref{w'- vanish}). 

Next, for the upper barrier, we set 
\[
\hat{w}_{+}\left(z,\,\tau\right)=\hat{\psi}_{k}^{\lambda_{+},\,\mu_{+}}\left(z,\,\tau\right)
\]
with 
\[
\lambda_{+}\left(\tau\right)=1+\beta^{\alpha-3}\left(\frac{\tau}{\tau_{0}}\right)^{-\varrho},\qquad\mu_{+}\left(\tau\right)=1+\delta\beta^{\alpha-3}\left(2\sigma\tau\right)^{-1+\varrho}\left(\frac{\tau}{\tau_{0}}\right)^{-\varrho}
\]
where 
\begin{equation}
\delta=\delta\left(n,\,\beta\right)=\frac{1}{4\left(1-\alpha\right)}\,\inf_{0\leq r\leq\frac{3}{2}\beta}\,\frac{\hat{\psi}_{k}\left(r\right)-r\,\partial_{r}\hat{\psi}_{k}\left(r\right)}{r\,\partial_{r}\hat{\psi}_{k}\left(r\right)}\,>0\label{delta}
\end{equation}
by (\ref{positivity}). Note that by (see (\ref{eq psi'})),
\begin{equation}
0\leq\partial_{r}\hat{\psi}_{k}\left(r\right)\leq1,\qquad\partial_{rr}^{2}\hat{\psi}_{k}\left(r\right)>0\label{derivative psi'}
\end{equation}
for all $r\geq0$. Firstly, for the initial value, given $0\leq z\leq\frac{3}{2}\beta$,
by Lemma \ref{monotonicity}, (\ref{initial w'}), (\ref{expansion perturbation psi}),
(\ref{asymptotic psi'0}) and the mean value theorem, there are 
\[
1+\frac{1}{2}\beta^{\alpha-3}\leq\,\lambda_{*}\,\leq\lambda_{+}\left(\tau_{0}\right),\qquad1\leq\,\mu_{*}\,\leq\mu_{+}\left(\tau_{0}\right)
\]
so that
\[
\hat{w}_{+}\left(z,\,\tau_{0}\right)\,=\,\hat{\psi}_{k}^{1+\frac{1}{2}\beta^{\alpha-3},\,1}\left(z,\,\tau_{0}\right)
\]
\[
+\left.\left(\lambda_{+}\left(\tau_{0}\right)-\left(1+\frac{1}{2}\beta^{\alpha-3}\right)\right)\,\partial_{\lambda}\left(\hat{\psi}_{k}^{\lambda,\,\mu}\left(z\right)\right)\right|_{\lambda=\lambda_{*},\;\mu=\mu_{*},\;z=z_{*}\equiv\frac{z}{\lambda_{*}^{\frac{1}{1-\alpha}}\mu_{*}}}
\]
\[
+\left.\left(\mu_{+}\left(\tau_{0}\right)-1\right)\,\partial_{\mu}\left(\hat{\psi}_{k}^{\lambda,\,\mu}\left(z\right)\right)\right|_{\lambda=\lambda_{*},\;\mu=\mu_{*},\;z=z_{*}\equiv\frac{z}{\lambda_{*}^{\frac{1}{1-\alpha}}\mu_{*}}}
\]
\[
=\,\hat{\psi}_{k}^{1+\frac{1}{2}\beta^{\alpha-3},\,1}\left(z,\,\tau_{0}\right)\,+\,\frac{\beta^{\alpha-3}\lambda_{*}^{\frac{\alpha}{1-\alpha}}}{2\left(1-\alpha\right)}\left(\hat{\psi}_{k}\left(z_{*}\right)-z_{*}\,\partial_{z}\hat{\psi}_{k}\left(z_{*}\right)\right)
\]
\[
-\delta\beta^{\alpha-3}\left(2\sigma\tau_{0}\right)^{-1+\varrho}\frac{\lambda_{*}^{\frac{1}{1-\alpha}}}{\mu_{*}}\left(z_{*}\,\partial_{z}\hat{\psi}_{k}\left(z_{*}\right)\right)
\]
\[
\geq\,\hat{\psi}_{k}^{1+\frac{1}{2}\beta^{\alpha-3},\,1}\left(z,\,\tau_{0}\right)\,+\,\frac{\beta^{\alpha-3}\lambda_{*}^{\frac{\alpha}{1-\alpha}}}{2\left(1-\alpha\right)}\left(1-\frac{\lambda_{*}}{2\mu_{*}}\left(2\sigma\tau_{0}\right)^{-1+\varrho}\right)\left(\hat{\psi}_{k}\left(z_{*}\right)-z_{*}\,\partial_{z}\hat{\psi}_{k}\left(z_{*}\right)\right)
\]
\[
\geq\,\hat{\psi}_{k}^{1+\frac{1}{2}\beta^{\alpha-3},\,1}\left(z,\,\tau_{0}\right)\,=\,\hat{\psi}_{\left(1+\frac{1}{2}\beta^{\alpha-3}\right)k}\left(z,\,\tau_{0}\right)\,=\,\hat{\psi}_{\left(1+\frac{1}{2}\beta^{\alpha-3}\right)\left(1+o\left(1\right)\right)}\left(z,\,\tau_{0}\right)
\]
\begin{equation}
>\,w\left(z,\,\tau_{0}\right)\label{initial w'+ tip}
\end{equation}
provided that $\beta\gg1$ (depending on $n$). Also, for each $\frac{3}{2}\beta\leq z\leq\left(2\sigma\tau_{0}\right)^{\frac{1}{2}\left(1-\vartheta\right)}$,
by (\ref{varrho}), (\ref{initial w' intermediate}), (\ref{expansion perturbation psi}),
(\ref{asymptotic psi'0}), (\ref{delta}), (\ref{derivative psi'})
and the mean value theorem, there are 
\[
1\leq\,\lambda_{*}\,\leq\lambda_{+}\left(\tau_{0}\right),\qquad1\leq\,\mu_{*}\,\leq\mu_{+}\left(\tau_{0}\right)
\]
so that 
\[
\hat{w}_{+}\left(z,\,\tau_{0}\right)\,=\,\hat{\psi}_{k}\left(z,\,\tau_{0}\right)\,+\,\left.\left(\lambda_{+}\left(\tau_{0}\right)-1\right)\,\partial_{\lambda}\left(\hat{\psi}_{k}^{\lambda,\,\mu}\left(z\right)\right)\right|_{\lambda=\lambda_{*},\;\mu=\mu_{*},\;z=z_{*}\equiv\frac{z}{\lambda_{*}^{\frac{1}{1-\alpha}}\mu_{*}}}
\]
\[
+\left.\left(\mu_{+}\left(\tau_{0}\right)-1\right)\,\partial_{\mu}\left(\hat{\psi}_{k}^{\lambda,\,\mu}\left(z\right)\right)\right|_{\lambda=\lambda_{*},\;\mu=\mu_{*},\;z=z_{*}\equiv\frac{z}{\lambda_{*}^{\frac{1}{1-\alpha}}\mu_{*}}}
\]
\[
=\,\hat{\psi}_{k}\left(z,\,\tau_{0}\right)\,+\,\frac{\beta^{\alpha-3}\lambda_{*}^{\frac{\alpha}{1-\alpha}}}{1-\alpha}\left(\hat{\psi}_{k}\left(z_{*}\right)-z_{*}\,\partial_{z}\hat{\psi}_{k}\left(z_{*}\right)\right)\,-\,\delta\beta^{\alpha-3}\left(2\sigma\tau_{0}\right)^{-1+\varrho}\frac{\lambda_{*}^{\frac{1}{1-\alpha}}}{\mu_{*}}\left(z_{*}\,\partial_{z}\hat{\psi}_{k}\left(z_{*}\right)\right)
\]
\[
\geq\,\hat{\psi}_{k}\left(z,\,\tau_{0}\right)\,+\,\left(1+o\left(1\right)\right)\beta^{\alpha-3}\mu_{*}^{-\alpha}2^{\frac{\alpha+1}{2}}z^{\alpha}\,-\,\frac{\delta\beta^{\alpha-3}}{\mu_{*}^{2}}\left(2\sigma\tau_{0}\right)^{-\frac{1}{2}\left(1-\vartheta\right)\left(1-\alpha\right)}z
\]
\[
=\,\hat{\psi}_{k}\left(z,\,\tau_{0}\right)\,+\,\frac{1}{2}\left(1+o\left(1\right)\right)\beta^{\alpha-3}\mu_{*}^{-\alpha}2^{\frac{\alpha+1}{2}}z^{\alpha}
\]
\[
+\,\frac{1}{2}\beta^{\alpha-3}z^{\alpha}\left(\left(1+o\left(1\right)\right)2^{\frac{\alpha+1}{2}}\mu_{*}^{-\alpha}\,-\,\frac{2\delta}{\mu_{*}^{2}}\left(\frac{z}{\left(2\sigma\tau_{0}\right)^{\frac{1}{2}\left(1-\vartheta\right)}}\right)^{1-\alpha}\right)
\]
\begin{equation}
\geq\,\hat{\psi}_{k}\left(z,\,\tau_{0}\right)\,+\,\frac{1}{2}\left(1+o\left(1\right)\right)\beta^{\alpha-3}\mu_{*}^{-\alpha}2^{\frac{\alpha+1}{2}}z^{\alpha}\,>\,w\left(z,\,\tau_{0}\right)\label{initial w'+ intermediate}
\end{equation}
provided that $\beta\gg1$ (depending on $n$), since $z\leq\left(2\sigma\tau_{0}\right)^{\frac{1}{2}\left(1-\vartheta\right)}$.
Secondly, for the boundary value, fix $\tau_{0}\leq\tau\leq\mathring{\tau}$
and let $z=\left(2\sigma\tau\right)^{\frac{1}{2}\left(1-\vartheta\right)}$,
by (\ref{varrho}), (\ref{boundary w'}), (\ref{expansion perturbation psi}),
(\ref{asymptotic psi'0}), (\ref{delta}), (\ref{derivative psi'})
and the mean value theorem, there are 
\[
1\leq\lambda_{*}\leq\lambda_{+}\left(\tau\right),\qquad1\leq\mu_{*}\leq\mu_{+}\left(\tau\right)
\]
so that 
\[
\hat{w}_{+}\left(z,\,\tau\right)\,=\,\hat{\psi}_{k}\left(z,\,\tau\right)\,+\,\left.\left(\lambda_{+}\left(\tau\right)-1\right)\,\partial_{\lambda}\left(\hat{\psi}_{k}^{\lambda,\,\mu}\left(z\right)\right)\right|_{\lambda=\lambda_{*},\;\mu=\mu_{*},\;z=z_{*}\equiv\frac{z}{\lambda_{*}^{\frac{1}{1-\alpha}}\mu_{*}}}
\]
\[
+\left.\left(\mu_{+}\left(\tau\right)-1\right)\,\partial_{\mu}\left(\hat{\psi}_{k}^{\lambda,\,\mu}\left(z\right)\right)\right|_{\lambda=\lambda_{*},\;\mu=\mu_{*},\;z=z_{*}\equiv\frac{z}{\lambda_{*}^{\frac{1}{1-\alpha}}\mu_{*}}}
\]
\[
=\,\hat{\psi}_{k}\left(z\right)\,+\,\beta^{\alpha-3}\left(\frac{\tau}{\tau_{0}}\right)^{-\varrho}\frac{\lambda_{*}^{\frac{\alpha}{1-\alpha}}}{1-\alpha}\left(\hat{\psi}_{k}\left(z_{*}\right)-z_{*}\,\partial_{z}\hat{\psi}_{k}\left(z_{*}\right)\right)
\]
\[
-\delta\beta^{\alpha-3}\left(2\sigma\tau\right)^{-1+\varrho}\left(\frac{\tau}{\tau_{0}}\right)^{-\varrho}\frac{\lambda_{*}^{\frac{1}{1-\alpha}}}{\mu_{*}}\left(z_{*}\,\partial_{z}\hat{\psi}_{k}\left(z_{*}\right)\right)
\]
\[
\geq\,\hat{\psi}_{k}\left(z\right)\,+\,\left(1+o\left(1\right)\right)\beta^{\alpha-3}\mu_{*}^{-\alpha}2^{\frac{\alpha+1}{2}}\left(\frac{\tau}{\tau_{0}}\right)^{-\varrho}z^{\alpha}\,-\,\frac{\delta\beta^{\alpha-3}}{\mu_{*}^{2}}\left(\frac{\tau}{\tau_{0}}\right)^{-\varrho}\left(2\sigma\tau\right)^{-\frac{1}{2}\left(1-\vartheta\right)\left(1-\alpha\right)}z
\]
\[
\geq\,\hat{\psi}_{k}\left(z\right)\,+\,\frac{1}{2}\left(1+o\left(1\right)\right)\beta^{\alpha-3}\mu_{*}^{-\alpha}2^{\frac{\alpha+1}{2}}\left(\frac{\tau}{\tau_{0}}\right)^{-\varrho}z^{\alpha}\,
\]
\[
+\,\frac{1}{2}\beta^{\alpha-3}\left(\frac{\tau}{\tau_{0}}\right)^{-\varrho}z^{\alpha}\left(\left(1+o\left(1\right)\right)2^{\frac{\alpha+1}{2}}\mu_{*}^{-\alpha}\,-\,\frac{2\delta}{\mu_{*}^{2}}\left(\frac{z}{\left(2\sigma\tau\right)^{\frac{1}{2}\left(1-\vartheta\right)}}\right)^{1-\alpha}\right)
\]
\begin{equation}
>\,\hat{\psi}_{k}\left(z\right)\,-C\left(n\right)\left(2\sigma\tau\right)^{-\vartheta}z^{\alpha}\,\geq\,\hat{w}\left(z,\,\tau\right)\label{boundary w'+}
\end{equation}
provided that $\tau_{0}\gg1$ (depending on $n$, $\beta$), since
$z=\left(2\sigma\tau_{0}\right)^{\frac{1}{2}\left(1-\vartheta\right)}$
and $0<\varrho<\vartheta$. Thirdly, by (\ref{eq perturbation psi})
and (\ref{derivative psi'}), there holds 
\[
\partial_{\tau}\hat{w}_{+}-\left(\frac{\partial_{zz}^{2}\hat{w}_{+}}{1+\left(\partial_{z}\hat{w}_{+}\right)^{2}}+\left(n-1\right)\left(\frac{\partial_{z}\hat{w}_{+}}{z}-\frac{1}{\hat{w}_{+}}\right)+\frac{\frac{1}{2}+\sigma}{2\sigma\tau}\left(-z\,\partial_{z}\hat{w}_{+}+\hat{w}_{+}\right)\right)
\]
\[
=\left.\frac{\mu_{+}^{2}-1}{\lambda_{+}^{\frac{1}{1-\alpha}}\mu_{+}^{2}}\left(\frac{\partial_{rr}^{2}\hat{\psi}_{k}\left(r\right)}{\left(1+\left(\partial_{r}\hat{\psi}_{k}\left(r\right)\right)^{2}\right)\left(1+\left(\frac{\partial_{r}\hat{\psi}_{k}\left(r\right)}{\mu_{+}}\right)^{2}\right)}+\left(n-1\right)\frac{\partial_{r}\hat{\psi}_{k}\left(r\right)}{r}\right)\right|_{r=\frac{z}{\lambda_{+}^{\frac{1}{1-\alpha}}\mu_{+}}}
\]
\[
+\left.\left(-\frac{\frac{1}{2}+\sigma}{2\sigma\tau}\lambda_{+}^{\frac{1}{1-\alpha}}\,+\,\frac{\lambda_{+}^{\frac{\alpha}{1-\alpha}}}{1-\alpha}\left(\partial_{\tau}\lambda_{+}\right)\right)\left(\hat{\psi}_{k}\left(r\right)-r\,\partial_{r}\hat{\psi}_{k}\left(r\right)\right)\right|_{r=\frac{z}{\lambda_{+}^{\frac{1}{1-\alpha}}\mu_{+}}}
\]
\[
\,-\,\left.\frac{\lambda_{+}^{\frac{1}{1-\alpha}}}{\mu_{+}}\left(\partial_{\tau}\mu_{+}\right)\left(r\,\partial_{r}\hat{\psi}_{k}\left(r\right)\right)\right|_{r=\frac{z}{\lambda_{+}^{\frac{1}{1-\alpha}}\mu_{+}}}
\]
\[
\geq\left.2\left(1-O\left(\beta^{\alpha-3}\right)\right)\delta\beta^{\alpha-3}\left(2\sigma\tau_{0}\right)^{\varrho}\left(2\sigma\tau\right)^{-1}\left(\frac{1}{4}\partial_{rr}^{2}\hat{\psi}_{k}\left(r\right)+\left(n-1\right)\frac{\partial_{r}\hat{\psi}_{k}\left(r\right)}{r}\right)\right|_{r=\frac{z}{\lambda_{+}^{\frac{1}{1-\alpha}}\mu_{+}}}
\]
\[
-\left.\left(1+O\left(\beta^{\alpha-3}\right)\right)\left(\frac{1}{2}+\sigma+\frac{2\sigma\varrho\beta^{\alpha-3}}{1-\alpha}\left(\frac{\tau}{\tau_{0}}\right)^{-\varrho}\right)\left(2\sigma\tau\right)^{-1}\left(\hat{\psi}_{k}\left(r\right)-r\,\partial_{r}\hat{\psi}_{k}\left(r\right)\right)\right|_{r=\frac{z}{\lambda_{+}^{\frac{1}{1-\alpha}}\mu_{+}}}
\]
\[
\geq0
\]
provided that $\tau_{0}\gg1$ (depending on $n$, $\Lambda$, $\beta$),
since 
\[
\frac{\partial_{r}\hat{\psi}_{k}\left(r\right)}{r}=\left(1+o\left(1\right)\right)r^{-1}\,>\,\left(1+o\left(1\right)\right)k\left(1-\alpha\right)\,2^{\frac{\alpha+1}{2}}r^{\alpha}=\hat{\psi}_{k}\left(r\right)-r\,\partial_{r}\hat{\psi}_{k}\left(r\right)
\]
for $r\gg1$ (noting that $\alpha<-1$). Then we subtract the equation
of $\hat{w}_{+}\left(z,\,\tau\right)$ by (\ref{eq w'}) to get 
\begin{equation}
\partial_{\tau}\left(\hat{w}_{+}-\hat{w}\right)-\left(\frac{1}{1+\left(\partial_{z}\hat{w}\right)^{2}}\,\partial_{zz}^{2}\left(\hat{w}_{+}-\hat{w}\right)+\frac{n-1}{z}\,\partial_{z}\left(\hat{w}_{+}-\hat{w}\right)\right)\label{eq Cauchy w'+}
\end{equation}
\[
+\left(\frac{\partial_{zz}^{2}\hat{w}_{+}\left(\partial_{z}\hat{w}_{+}+\partial_{z}\hat{w}\right)}{\left(1+\left(\partial_{z}\hat{w}_{+}\right)^{2}\right)\left(1+\left(\partial_{z}\hat{w}\right)^{2}\right)}\,+\,\frac{\frac{1}{2}+\sigma}{2\sigma\tau}z\right)\partial_{z}\left(\hat{w}_{+}-\hat{w}\right)\,-\left(\frac{n-1}{\hat{w}_{+}\,\hat{w}}\,+\,\frac{\frac{1}{2}+\sigma}{2\sigma\tau}\right)\left(\hat{w}_{+}-\hat{w}\right)
\]
\[
\geq0
\]
To show that $\hat{w}_{+}$ is an upper barrier, let 
\[
\left(\hat{w}_{+}-\hat{w}\right)_{\min}\left(\tau\right)=\min_{0\leq z\leq\left(2\sigma\tau\right)^{\frac{1}{2}\left(1-\vartheta\right)}}\left(\hat{w}_{+}-\hat{w}\right)\left(z,\,\tau\right)
\]
Note that by (\ref{initial w'+ tip}) and (\ref{initial w'+ intermediate}),
we have 
\[
\left(\hat{w}_{+}-\hat{w}\right)_{\min}\left(\tau_{0}\right)>0
\]
We claim that 
\[
\left(\hat{w}_{+}-\hat{w}\right)_{\min}\left(\tau\right)\geq0\qquad\textrm{for}\;\,\tau_{0}\leq\tau\leq\mathring{\tau}
\]
Suppose the contrary, then there is $\tau_{0}<\tau_{1}^{*}\leq\mathring{\tau}$
so that 
\begin{equation}
\left(\hat{w}_{+}-\hat{w}\right)_{\min}\left(\tau_{1}^{*}\right)<0\label{w'+ negative}
\end{equation}
Let $\tau_{0}^{*}\in\left[\tau_{0},\,\tau_{1}^{*}\right)$ be the
first time after which $\left(\hat{w}^{+}-\hat{w}\right)_{\min}$
is negative all the way up to $\tau_{1}^{*}$, then by the continuity,
we must have 
\begin{equation}
\left(\hat{w}_{+}-\hat{w}\right)_{\min}\left(\tau_{0}^{*}\right)=0\label{w'+ vanish}
\end{equation}
 On the other hand, by (\ref{boundary w'+}), the minimum of $\hat{w}^{+}-\hat{w}$
for each time-slice is achieved in $\left[0,\,\left(2\sigma\tau\right)^{\frac{1}{2}\left(1-\vartheta\right)}\right)$.
Applying the maximum principle to (\ref{eq Cauchy w'+}), we get 
\[
\partial_{\tau}\left(\hat{w}_{+}-\hat{w}\right)_{\min}-\left(\frac{n-1}{\hat{w}_{+}\,\hat{w}}\,+\,\frac{\frac{1}{2}+\sigma}{2\sigma\tau}\right)\left(\hat{w}_{+}-\hat{w}\right)_{\min}\geq0
\]
Note that at $z=0$, we always have 
\[
\partial_{z}\hat{w}\left(0,\,\tau\right)\,=0=\,\partial_{z}\hat{w}_{+}\left(0,\,\tau\right)\qquad\forall\;\,\tau_{0}\leq\tau\leq\mathring{\tau}
\]
so L'H$\hat{o}$pital's rule implies 
\[
\lim_{z\rightarrow0}\frac{n-1}{z}\,\partial_{z}\left(\hat{w}_{+}-\hat{w}\right)\left(z,\,\tau\right)=\frac{n-1}{z}\,\partial_{z}^{2}\left(\hat{w}_{+}-\hat{w}\right)\left(0,\,\tau\right)
\]
It follows that 
\[
\partial_{\tau}\left(e^{-\int\frac{n-1}{\hat{w}_{+}\,\hat{w}}d\tau}\tau^{-\frac{\frac{1}{2}+\sigma}{2\sigma}}\left(\hat{w}_{+}-\hat{w}\right)_{\min}\right)\geq0
\]
which, together with (\ref{w'+ negative}), contraditcts with (\ref{w'+ vanish}).

Lastly, by (\ref{expansion perturbation psi}) and $\mu_{+}\left(\tau\right)\geq1$,
we have 
\[
\hat{w}_{+}\left(z,\,\tau\right)=\hat{\psi}_{k}^{\lambda_{+},\,\mu_{+}}\left(z,\,\tau\right)\,\leq\,\hat{\psi}_{k}^{\lambda_{+},\,1}\left(z,\,\tau\right)=\hat{\psi}_{\lambda_{+}\left(\tau\right)k}\left(z\right)
\]
Thus, we get 
\[
\hat{\psi}_{\lambda_{-}\left(\tau\right)k}\left(z\right)=\hat{w}_{-}\left(z,\,\tau\right)\,\leq\,\hat{w}\left(z,\,\tau\right)\,\leq\,\hat{w}_{+}\left(z,\,\tau\right)\leq\hat{\psi}_{\lambda_{+}\left(\tau\right)k}\left(z\right)
\]
For (\ref{C^0 w' bound intermediate}), given $\tau_{0}\leq\tau\leq\mathring{\tau}$,
$\beta\leq z\leq\left(2\sigma\tau\right)^{\frac{1}{2}\left(1-\vartheta\right)}$,
by (\ref{expansion perturbation psi}), (\ref{asymptotic psi'0})
and the mean value theorem, there is $1\leq\,\lambda_{*}\,\leq\lambda_{+}\left(\tau\right)$
so that 
\[
\hat{\psi}_{k}^{\lambda_{+},\,1}\left(z,\,\tau\right)\,=\,\hat{\psi}_{k}\left(z,\,\tau\right)\,+\,\left.\left(\lambda_{+}\left(\tau\right)-1\right)\,\partial_{\lambda}\left(\hat{\psi}_{k}^{\lambda,\,\mu}\left(z\right)\right)\right|_{\lambda=\lambda_{*},\;z=z_{*}\equiv\frac{z}{\lambda_{*}^{\frac{1}{1-\alpha}}\mu_{*}}}
\]
\[
=\,\hat{\psi}_{k}\left(z,\,\tau\right)\,+\,\beta^{\alpha-3}\left(\frac{\tau}{\tau_{0}}\right)^{-\varrho}\frac{\lambda_{*}^{\frac{\alpha}{1-\alpha}}}{1-\alpha}\left(\hat{\psi}_{k}\left(z_{*}\right)-z_{*}\,\partial_{z}\hat{\psi}_{k}\left(z_{*}\right)\right)
\]
\[
\leq\,\hat{\psi}_{k}\left(z,\,\tau\right)\,+\,\left(1+o\left(1\right)\right)2^{\frac{\alpha+1}{2}}\beta^{\alpha-3}\left(\frac{\tau}{\tau_{0}}\right)^{-\varrho}z^{\alpha}
\]
Similarly, 
\[
\hat{\psi}_{k}^{\lambda_{-},\,1}\left(z,\,\tau\right)\,\geq\,\hat{\psi}_{k}\left(z,\,\tau\right)\,-\,\left(1+o\left(1\right)\right)2^{\frac{\alpha+1}{2}}\beta^{\alpha-3}\left(\frac{\tau}{\tau_{0}}\right)^{-\varrho}z^{\alpha}
\]
As for (\ref{C^0 w' bound tip}), given $\tau_{0}\leq\tau\leq\mathring{\tau}$,
$0\leq z\leq5\beta$, by (\ref{expansion perturbation psi}), (\ref{asymptotic psi'0})
and the mean value theorem, there is $1\leq\,\lambda_{*}\,\leq\lambda_{+}\left(\tau\right)$
so that 
\[
\hat{\psi}_{k}^{\lambda_{+},\,1}\left(z,\,\tau\right)\,=\,\hat{\psi}_{k}\left(z,\,\tau\right)\,+\,\left.\left(\lambda_{+}\left(\tau\right)-1\right)\,\partial_{\lambda}\left(\hat{\psi}_{k}^{\lambda,\,\mu}\left(z\right)\right)\right|_{\lambda=\lambda_{*},\;z=z_{*}\equiv\frac{z}{\lambda_{*}^{\frac{1}{1-\alpha}}\mu_{*}}}
\]
\[
=\,\hat{\psi}_{k}\left(z,\,\tau\right)\,+\,\beta^{\alpha-3}\left(\frac{\tau}{\tau_{0}}\right)^{-\varrho}\frac{\lambda_{*}^{\frac{\alpha}{1-\alpha}}}{1-\alpha}\left(\hat{\psi}_{k}\left(z_{*}\right)-z_{*}\,\partial_{z}\hat{\psi}_{k}\left(z_{*}\right)\right)
\]
\[
\leq\,\hat{\psi}_{k}\left(z,\,\tau\right)\,+\,\frac{\beta^{\alpha-3}\mathfrak{C}}{1-\alpha}\left(\frac{\tau}{\tau_{0}}\right)^{-\varrho}
\]
where 
\[
\mathfrak{C}=\sup_{r\geq0}\left(\hat{\psi}_{k}\left(r\right)-r\,\partial_{r}\hat{\psi}_{k}\left(r\right)\right)\,\leq\,C\left(n\right)
\]
(by (\ref{asymptotic psi'0})). Similarly, 
\[
\hat{\psi}_{k}^{\lambda_{-},\,1}\left(z,\,\tau\right)\,\geq\,\hat{\psi}_{k}\left(z,\,\tau\right)\,-\,\frac{\beta^{\alpha-3}\mathfrak{C}}{1-\alpha}\left(\frac{\tau}{\tau_{0}}\right)^{-\varrho}
\]
\end{proof}
As a corollary, if we regard the projected curves $\bar{\Gamma}_{\tau}^{\left(a_{0},\,a_{1}\right)}$
and $\mathcal{\bar{M}}_{k}$ as graphes over $\mathcal{\bar{C}}$,
(\ref{C^0 w' bound intermediate}) implies 
\begin{equation}
\left|w\left(z,\,\tau\right)-\psi_{k}\left(z\right)\right|\,\leq\,C\left(n\right)\beta^{\alpha-3}\left(\frac{\tau}{\tau_{0}}\right)^{-\varrho}z^{\alpha}\label{C^0 w bound}
\end{equation}
for $\frac{4}{3}\beta\leq z\leq\frac{1}{2}\left(2\sigma\tau\right)^{\frac{1}{2}\left(1-\vartheta\right)}$,
$\tau_{0}\leq\tau\leq\mathring{\tau}$. Then (\ref{C^0 v bound tip})
follows immediately by (\ref{uw}).

Lastly, we prove (\ref{C^0 outside u bound}) by using the gradient
and curvature estimates in \cite{EH}.
\begin{prop}
\label{C^0 outside u}If $0<\rho\ll1$ (depending on $n$, $\Lambda$)
and $\left|t_{0}\right|\ll\rho^{2}$ (depending on $n$), there holds
(\ref{C^0 outside u bound}). Moreover, we have
\begin{equation}
\left\{ \begin{array}{c}
\left|\partial_{x}u(x,\,t)\right|\,\lesssim\,1\\
\\
\left|\partial_{xx}^{2}u(x,\,t)\right|\,\leq\,\frac{C\left(n\right)}{\sqrt{t-t_{0}}}
\end{array}\right.\label{preliminary bound C^2 outside u}
\end{equation}
 for $x\geq\frac{1}{5}\rho$, $t_{0}\leq t\leq\mathring{t}$.
\end{prop}
\begin{proof}
For ease of notation, we denote $\Sigma_{t}^{\left(a_{0},\,a_{1}\right)}$
by $\Sigma_{t}$. Let's first parametrize $\Sigma_{t_{0}}$ by (\ref{u}),
i.e. 
\[
X_{t_{0}}\left(x,\,\nu,\,\omega\right)=\left(\left(x-u\left(x,\,t_{0}\right)\right)\frac{\nu}{\sqrt{2}},\,\left(x+u\left(x,\,t_{0}\right)\right)\frac{\omega}{\sqrt{2}}\right)
\]
for $x\geq\frac{1}{6}\rho$, $\nu,\,\omega\in\mathbb{S}^{n-1}$. Then
the (upward) unit normal vector of $\Sigma_{t_{0}}$ at $X_{t_{0}}$
is given by 
\[
N_{\Sigma_{t_{0}}}\left(X_{t_{0}}\right)=\left(\left(\frac{1+\partial_{x}u\left(x,\,t_{0}\right)}{\sqrt{1+\left(\partial_{x}u\left(x,\,t_{0}\right)\right)^{2}}}\right)\frac{-\nu}{\sqrt{2}},\,\left(\frac{1-\partial_{x}u\left(x,\,t_{0}\right)}{\sqrt{1+\left(\partial_{x}u\left(x,\,t_{0}\right)\right)^{2}}}\right)\frac{\omega}{\sqrt{2}}\right)
\]
Note that by (\ref{initial u}) we have
\[
\max\left\{ \left|\frac{u\left(x,\,t_{0}\right)}{x}\right|,\,\left|\partial_{x}u\left(x,\,t_{0}\right)\right|\right\} \leq\frac{1}{3}
\]
for $x\geq\frac{1}{6}\rho$. 

Now fix $x_{*}\geq\frac{1}{5}\rho$ and let 
\[
\nu_{*}=\omega_{*}=\left(\overset{\left(\textrm{n-1}\right)\,\textrm{copies}}{\overbrace{0,\cdots,\,0}},\,1\right)
\]
\[
\mathbf{e}=\left(\frac{-1}{\sqrt{2}}\nu_{*},\,\frac{1}{\sqrt{2}}\omega_{*}\right)
\]
 
\[
X_{*}=X_{t_{0}}\left(x_{*},\,\nu_{*},\,\omega_{*}\right)=\left(\left(x_{*}-u\left(x_{*},\,t_{0}\right)\right)\frac{\nu_{*}}{\sqrt{2}},\,\left(x_{*}+u\left(x_{*},\,t_{0}\right)\right)\frac{\omega_{*}}{\sqrt{2}}\right)
\]
Notice that 
\[
\left|X_{t_{0}}-X_{*}\right|^{2}\,\geq\,\frac{1}{2}\left(x-u\left(x,\,t_{0}\right)\right)^{2}\left(1-\left(\nu\cdot\nu_{*}\right)^{2}\right)\,+\,\frac{1}{2}\left(x+u\left(x,\,t_{0}\right)\right)^{2}\left(1-\left(\omega\cdot\omega_{*}\right)^{2}\right)
\]
\[
\geq\frac{x^{2}}{2}\left(1-\left|\frac{u\left(x,\,t_{0}\right)}{x}\right|\right)^{2}\max\left\{ 1-\left(\nu\cdot\nu_{*}\right)^{2},\,1-\left(\omega\cdot\omega_{*}\right)^{2}\right\} 
\]
\[
\geq\frac{\rho^{2}}{9}\,\max\left\{ 1-\left(\nu\cdot\nu_{*}\right)^{2},\,1-\left(\omega\cdot\omega_{*}\right)^{2}\right\} 
\]
Thus, for $X_{t_{0}}\in\Sigma_{t_{0}}\cap B\left(X_{*};\,\frac{1}{30}\rho\right)$,
there holds 
\[
\min\left\{ \nu\cdot\nu_{*},\,\omega\cdot\omega_{*}\right\} \,\geq\,\frac{\sqrt{91}}{10}
\]
which implies 
\[
\left(N_{\Sigma_{t_{0}}}\left(X_{t_{0}}\right)\cdot\mathbf{e}\right)^{-1}=\frac{2\,\sqrt{1+\left(\partial_{x}u\left(x,\,t_{0}\right)\right)^{2}}}{\left(1+\partial_{x}u\left(x,\,t_{0}\right)\right)\left(\nu\cdot\nu_{*}\right)+\left(1-\partial_{x}u\left(x,\,t_{0}\right)\right)\left(\omega\cdot\omega_{*}\right)}
\]
\begin{equation}
\leq\,\frac{\sqrt{10}}{\nu\cdot\nu_{*}+\omega\cdot\omega_{*}}\leq\,\frac{10\sqrt{10}}{2\sqrt{91}}\label{gradient formula}
\end{equation}
By the gradient estimates in \cite{EH}, we then get
\[
\left(N_{\Sigma_{t}}\left(X_{t}\right)\cdot\mathbf{e}\right)^{-1}\leq\left(1-\frac{\left|X_{t}-X_{*}\right|^{2}+2n\left(t-t_{0}\right)}{\left(\frac{1}{30}\rho\right)^{2}}\right)^{-1}\sup_{\Sigma_{t_{0}}\cap B\left(X_{*};\,\frac{1}{30}\rho\right)}\left(N_{\Sigma_{t_{0}}}\cdot\mathbf{e}\right)^{-1}
\]
for $X_{t}\in\Sigma_{t}\cap B\left(X_{*};\,\sqrt{\left(\frac{1}{30}\rho\right)^{2}-2n\left(t-t_{0}\right)}\right)$,
where $N_{\Sigma_{t}}\left(X_{t}\right)$ is the unit normal vector
of $\Sigma_{t}$ at $X_{t}$. Consequently, 
\begin{equation}
\left(N\left(X_{t}\right)\cdot\mathbf{e}\right)^{-1}\leq\,\left(1-\left(\frac{30}{31}\right)^{2}\right)\frac{10\sqrt{10}}{2\sqrt{91}}\label{bound outside Du}
\end{equation}
for $X_{t}\in\Sigma_{t}\cap B\left(X_{*};\,\sqrt{\left(\frac{1}{31}\rho\right)^{2}-2n\left(t-t_{0}\right)}\right)$.
It follows, by the curvature estimates in \cite{EH}, that 
\[
\left|A_{\Sigma_{t}}\left(X_{t}\right)\right|\leq C\left(n\right)\left(\frac{1}{\sqrt{t-t_{0}}}+\frac{1}{\rho}\right)
\]
for $X_{t}\in\Sigma_{t}\cap B\left(X_{*};\,\sqrt{\left(\frac{1}{32}\rho\right)^{2}-2n\left(t-t_{0}\right)}\right)$,
where $A_{\Sigma_{t}}\left(X_{t}\right)$ is the second fundamental
form of $\Sigma_{t}$ at $X_{t}$. Thus, by choosing $\left|t_{0}\right|\ll\rho^{2}$
(depending on $n$), we may assume that 
\[
\sqrt{\left(\frac{1}{32}\rho\right)^{2}-2n\left(t-t_{0}\right)}\,\geq\,\frac{1}{33}\rho
\]
for all $t_{0}\leq t\leq\mathring{t}$, and 
\begin{equation}
\left|A_{\Sigma_{t}}\left(X_{t}\right)\right|\,\leq\,\frac{C\left(n\right)}{\sqrt{t-t_{0}}}\label{bound outside DDu}
\end{equation}
for $X_{t}\in\Sigma_{t}\cap B\left(X_{*};\,\frac{\rho}{33}\right)$,
$t_{0}\leq t\leq\mathring{t}$.

Next, consider the ``normal parametrization'' for the MCF $\left\{ \Sigma_{t}\right\} _{t_{0}\leq t\leq\mathring{t}}$,
i.e. let $X_{t}\left(x,\,\nu,\,\omega\right)=X\left(x,\,\nu,\,\omega;\,t\right)$
so that
\[
\left\{ \begin{array}{c}
\partial_{t}\,X\left(x,\,\nu,\,\omega;\,t\right)=H_{\Sigma_{t}}\left(X\left(x,\,\nu,\,\omega;\,t\right)\right)\,N_{\Sigma_{t}}\left(X\left(x,\,\nu,\,\omega;\,t\right)\right)\\
\\
X\left(x,\,\nu,\,\omega;\,t_{0}\right)=X_{t_{0}}\left(x,\,\nu,\,\omega\right)
\end{array}\right.
\]
For each $x\geq\rho$, $\nu,\,\omega\in\mathbb{S}^{n-1}$, let $t_{\left(x,\,\nu,\,\omega\right)}\in\left(t_{0},\,\mathring{t}\right]$
be the maximal time so that 
\[
X_{t}\left(x,\,\nu,\,\omega\right)\,\in\,\Sigma_{t}\cap B\left(X_{t_{0}}\left(x,\,\nu,\,\omega\right);\,\frac{1}{33}\rho\right)
\]
for all $t_{0}\leq t\leq t_{\left(x,\,\nu,\,\omega\right)}$. Then
we have 
\[
\left|\partial_{t}X_{t}\left(x,\,\nu,\,\omega\right)\right|\,=\,\left|H_{\Sigma_{t}}\left(X_{t}\left(x,\,\nu,\,\omega\right)\right)\right|\,\leq\,\frac{C\left(n\right)}{\sqrt{t-t_{0}}}
\]
and hence 
\begin{equation}
\left|X_{t}\left(x,\,\nu,\,\omega\right)-X_{t_{0}}\left(x,\,\nu,\,\omega\right)\right|\,\leq\,C\left(n\right)\sqrt{t-t_{0}}\label{distance}
\end{equation}
for all $t_{0}\leq t\leq t_{\left(x,\,\nu,\,\omega\right)}$. Thus,
if $\left|t_{0}\right|\ll1$ (depending on $n$), we may assume that
$t_{\left(x,\,\nu,\,\omega\right)}=\mathring{t}$ and 
\begin{equation}
d_{H}\left(\Sigma_{t}\setminus B\left(O;\,\frac{1}{5}\rho\right),\,\Sigma_{t_{0}}\setminus B\left(O;\,\frac{1}{5}\rho\right)\right)\,\leq\,C\left(n\right)\sqrt{t-t_{0}}\label{bound outside u}
\end{equation}
for all $t_{0}\leq t\leq\mathring{t}$, where $d_{H}$ is the Hausdorff
distance. It follows that
\[
\left|u\left(x,\,t\right)-u\left(x,\,t_{0}\right)\right|\,\leq\,C\left(n\right)\sqrt{t-t_{0}}
\]
for $x\geq\frac{1}{5}\rho$, $t_{0}\leq t\leq\mathring{t}$.

Furthermore, by taking $x=x_{*}$, $\nu=\nu_{*}$, $\omega=\omega_{*}$
in (\ref{gradient formula}) and replace $t_{0}$ by $t$, one could
get 
\[
\left(N_{\Sigma_{t}}\left(X_{t}\left(x_{*},\,\nu_{*},\,\omega_{*}\right)\right)\cdot\mathbf{e}\right)^{-1}=\sqrt{1+\left(\partial_{x}u\left(x_{*},\,t\right)\right)^{2}}
\]
So by (\ref{bound outside Du}) and (\ref{distance}) , we have 
\begin{equation}
\left|\partial_{x}u(x_{*},\,t)\right|\,\lesssim\,1\label{gradient}
\end{equation}
for $t_{0}\leq t\leq\mathring{t}$ (and any $x_{*}\geq\frac{1}{5}\rho$).
For the second derivative, notice that
\[
\frac{\left|\partial_{xx}^{2}u\left(x_{*},\,t\right)\right|}{\left(1+\left(\partial_{x}u\left(x_{*},\,t\right)\right)^{2}\right)^{\frac{3}{2}}}\,\leq\,\left|A_{\Sigma_{t}}\left(X_{t}\left(x_{*},\,\nu_{*},\,\omega_{*}\right)\right)\right|
\]
By (\ref{bound outside DDu}), (\ref{distance}) and (\ref{gradient}),
we conclude
\[
\left|\partial_{xx}^{2}u(x_{*},\,t)\right|\,\leq\,\frac{C\left(n\right)}{\sqrt{t-t_{0}}}
\]
for $t_{0}\leq t\leq\mathring{t}$ (and any $x_{*}\geq\frac{1}{5}\rho$). 
\end{proof}

\section{\label{degree C^infty}Smooth estimates in Proposition \ref{degree}
and Proposition \ref{uniform estimates}}

This section is a continuation of Section \ref{degree C^0}. For ease
of notation, from now on, let's denote $\Sigma_{t}^{\left(a_{0},\,a_{1}\right)}$
by $\Sigma_{t}$, $\Gamma_{\tau}^{\left(a_{0},\,a_{1}\right)}$ by
$\Gamma_{\tau}$ and $\Pi_{s}^{\left(a_{0},\,a_{1}\right)}$ by $\Pi_{s}$.
Here we would like to show that if $0<\rho\ll1\ll\beta$ (depending
on $n$, $\Lambda$) and $\left|t_{0}\right|\ll1$ (depending on $n$,
$\Lambda$, $\rho$, $\beta$) , then
\begin{itemize}
\item In the $\mathbf{outer}$ $\mathbf{region}$, the function $u\left(x,\,t\right)$
of $\Sigma_{t}^{\left(a_{0},\,a_{1}\right)}$ defined in (\ref{u})
satisfies (\ref{C^2 outside u bound}). 
\item In the $\mathbf{tip}$ $\mathbf{region}$, if we do the type $\mathrm{II}$
rescaling, the function $\hat{w}\left(z,\,\tau\right)$ of the rescaled
hypersurface $\Gamma_{\tau}^{\left(a_{0},\,a_{1}\right)}$ defined
in (\ref{w'}) satisfies satisfies (\ref{C^2 w' bound}).
\end{itemize}
Moreover, for any $0<\delta\ll1$, $m,\,l\in\mathbb{Z}_{+}$, there
hold the following higher order derivatives estimates.
\begin{enumerate}
\item In the $\mathbf{outer}$ $\mathbf{region}$, the function $u\left(x,\,t\right)$
of $\Sigma_{t}^{\left(a_{0},\,a_{1}\right)}$ defined in (\ref{u})
satisfies (\ref{C^infty outside u bound}) and (\ref{C^infty u bound})
(see Proposition \ref{C^infty outside u} and Proposition \ref{C^infty u}).
\item In the $\mathbf{intermediate}$ $\mathbf{region}$, if we do the type
$\mathrm{I}$ rescaling, the function $v\left(y,\,s\right)$ of the
rescaled hypersurface $\Pi_{s}^{\left(a_{0},\,a_{1}\right)}$ defined
in (\ref{v}) satisfies (\ref{C^infty v bound intermediate}) and
(\ref{C^infty v bound tip}) (see Proposition \ref{C^infty v}).
\item In the $\mathbf{tip}$ $\mathbf{region}$, if we do the type $\mathrm{II}$
rescaling, the function $\hat{w}\left(z,\,\tau\right)$ of the rescaled
hypersurface $\Gamma_{\tau}^{\left(a_{0},\,a_{1}\right)}$ defined
in (\ref{w'}) satisfies (\ref{C^infty w' bound}) (see Proposition
\ref{C^infty w'}).
\end{enumerate}
We establish (\ref{C^2 outside u bound}) and (\ref{C^2 w' bound})
by using the maximum principle and curvature estimates in \cite{EH}.
Then we use Krylov-Safonov estimates and Schauder estimates, together
with (\ref{a priori bound u}) (which is equivalent to (\ref{a priori bound v})
and (\ref{a priori bound w})), (\ref{C^2 outside u bound}) and (\ref{C^2 w' bound}),
to derive (\ref{C^infty outside u bound}), (\ref{C^infty u bound}),
(\ref{C^infty v bound intermediate}), (\ref{C^infty v bound tip})
and (\ref{C^infty w' bound}).

Let's start with proving (\ref{C^2 outside u bound}). The $C^{0}$
estimats has already been shown in Proposition \ref{C^0 u} and Proposition
\ref{C^0 outside u}, in which we also get the first and second derivative
bounds for $u\left(x,\,t\right)$ (see (\ref{preliminary bound C^2 outside u})).
In the next lemma, we improve the first derivative bound in Proposition
\ref{C^0 outside u} by using the maximum principle, which turns out
to be useful when we derive an improved second derivative estimate
in Lemma \ref{outside DDu}. 
\begin{lem}
\label{outside Du}If $0<\rho\ll1$ (depending on $n$, $\Lambda$)
and $\left|t_{0}\right|\ll1$ (depending on $n$, $\rho$), there
holds
\[
\sup_{x\geq\frac{1}{4}\rho}\left|\partial_{x}u\left(x,\,t\right)\right|\,\leq\,\sup_{x\geq\frac{1}{5}\rho}\left|\partial_{x}u\left(x,\,t_{0}\right)\right|\,+\,C\left(n,\,\rho\right)\sqrt{t-t_{0}}
\]
for $t_{0}\leq t\leq\mathring{t}$. 
\end{lem}
\begin{proof}
First, differentiate (\ref{eq u}) with respect to $x$ to get 
\[
\partial_{t}\left(\partial_{x}u\right)-\frac{1}{1+\left(\partial_{x}u\right)^{2}}\,\partial_{xx}^{2}\left(\partial_{x}u\right)-\left(\,a\left(x,\,t\right)\partial_{xx}^{2}u+b\left(x,\,t\right)\,\right)\,\partial_{x}\left(\partial_{x}u\right)=f\left(x,\,t\right)
\]
where 
\[
a\left(x,\,t\right)=\frac{-2\,\partial_{x}u\left(x,\,t\right)}{\left(1+\left(\partial_{x}u\left(x,\,t\right)\right)^{2}\right)^{2}}
\]
\[
b\left(x,\,t\right)=\frac{2\left(n-1\right)}{x\left(1-\left(\frac{u\left(x,\,t\right)}{x}\right)^{2}\right)}
\]
\[
f\left(x,\,t\right)=\frac{-4\left(n-1\right)\left(\frac{u\left(x,\,t\right)}{x}\right)\left(1-\left(\partial_{x}u\left(x,\,t\right)\right)^{2}\right)}{x^{2}\left(1-\left(\frac{u\left(x,\,t\right)}{x}\right)^{2}\right)^{2}}
\]
For each $R\geq2$, let $\eta\left(x\right)$ be a smooth function
so that 
\[
\chi_{\left(\frac{1}{4}\rho,\,R-1\right)}\,\leq\,\eta\,\leq\,\chi_{\left(\frac{1}{5}\rho,\,R\right)}
\]
 
\begin{equation}
\left|\partial_{x}\eta\left(x\right)\right|\,+\,\left|\partial_{xx}^{2}\eta\left(x\right)\right|\leq C\left(\rho\right)\label{outside Du1}
\end{equation}
It follows that
\[
\partial_{t}\left(\eta\,\partial_{x}u\right)-\frac{1}{1+\left(\partial_{x}u\right)^{2}}\,\partial_{xx}^{2}\left(\eta\,\partial_{x}u\right)-\left(\,a\left(x,\,t\right)\partial_{xx}^{2}u+b\left(x,\,t\right)\,\right)\,\partial_{x}\left(\eta\,\partial_{x}u\right)
\]
\begin{equation}
=-\left(\frac{\partial_{xx}^{2}\eta}{1+\left(\partial_{x}u\right)^{2}}\,+\,\left(\,a\left(x,\,t\right)\partial_{xx}^{2}u+b\left(x,\,t\right)\,\right)\,\partial_{x}\eta\right)\left(\partial_{x}u\right)\label{eq outside Du}
\end{equation}
\[
+\,\eta\left(x\right)f\left(x,\,t\right)\,-\,\frac{2}{1+\left(\partial_{x}u\right)^{2}}\,\partial_{x}\eta\,\left(\partial_{xx}^{2}u\right)
\]
Now let 
\[
\left(\eta\,\partial_{x}u\right)_{\max}\left(t\right)=\max_{x}\,\left(\eta\left(x\right)\partial_{x}u\left(x,\,t\right)\right)
\]
By (\ref{a priori bound u}), (\ref{initial u}) and (\ref{preliminary bound C^2 outside u}),
if $0<\rho\ll1$ (depending on $n$, $\Lambda$), $\left|t_{0}\right|\ll1$
(depending on $n$, $\rho$), we may assume that
\begin{equation}
\left\{ \begin{array}{c}
\left|\frac{u\left(x,\,t\right)}{x}\right|\leq\frac{1}{3}\\
\\
\left|\partial_{x}u\left(x,\,t\right)\right|\lesssim1\\
\\
\left|\partial_{xx}^{2}u\left(x,\,t\right)\right|\leq\frac{C\left(n,\,\rho\right)}{\sqrt{t-t_{0}}}
\end{array}\right.\label{outside Du2}
\end{equation}
 for $x\geq\frac{1}{5}\rho$, $t_{0}\leq t\leq\mathring{t}$. Thus,
by (\ref{outside Du1}) and (\ref{outside Du2}), applying the maximum
principle to (\ref{eq outside Du}) yields 
\[
\partial_{t}\left(\eta\,\partial_{x}u\right)_{\max}\,\leq\,\frac{C\left(n,\,\rho\right)}{\sqrt{t-t_{0}}}
\]
which implies 
\[
\left(\eta\,\partial_{x}u\right)_{\max}\left(t\right)\,\leq\,\left(\eta\,\partial_{x}u\right)_{\max}\left(t_{0}\right)\,+\,C\left(n,\,\rho\right)\sqrt{t-t_{0}}
\]
Likewise, if we define 
\[
\left(\eta\,\partial_{x}u\right)_{\min}\left(t\right)=\min_{x}\,\left(\eta\left(x\right)\partial_{x}u\left(x,\,t\right)\right)
\]
by the same argument, we get 
\[
\left(\eta\,\partial_{x}u\right)_{\min}\left(t\right)\,\geq\,\left(\eta\,\partial_{x}u\right)_{\min}\left(t_{0}\right)\,-\,C\left(n,\,\rho\right)\sqrt{t-t_{0}}
\]
\end{proof}
Before moving on to the second derivative estimate, we derive the
following lemma, which is about some properties of the cut-off functions
to be used.
\begin{lem}
\label{localization}Let $\eta\left(r\right)$ be a smooth, non-increasing
function so that 
\[
\chi_{\left(-\infty,\,0\right)}\leq\eta\leq\chi_{\left(-\infty,\,1\right)}
\]
and $\eta\left(r\right)$ vanishes at $r=1$ to infinite order. Then
\[
\sup_{r}\,\frac{\left(\partial_{r}\eta\left(r\right)\right)^{2}}{\eta\left(r\right)}<\infty
\]
for $r\leq1$.
\end{lem}
\begin{proof}
By L'H$\hat{o}$pital's rule, we have
\[
\lim_{r\nearrow1}\frac{\left(\partial_{r}\eta\left(r\right)\right)^{2}}{\eta\left(r\right)}=2\,\lim_{r\nearrow1}\,\partial_{rr}^{2}\eta\left(r\right)=0
\]
Also, for $r\leq0$ or $r>1$, there holds
\[
\frac{\left(\partial_{r}\eta\left(r\right)\right)^{2}}{\eta\left(r\right)}=0
\]
Thus, the conclusion follows easily.
\end{proof}
Below is an improved estimate for the second derivative of $u\left(s,\,t\right)$
in the outer region. Note that the proof requres $\left|\partial_{x}u\left(x,\,t\right)\right|<\frac{1}{\sqrt{3}}$,
which is guqranteed by (\ref{initial u}) and Lemma \ref{outside Du}. 
\begin{lem}
\label{outside DDu}If $0<\rho\ll1$ (depending on $n$, $\Lambda$)
and $\left|t_{0}\right|\ll1$ (depending on $n$, $\rho$), there
holds
\[
\sup_{x\geq\frac{1}{3}\rho}\left|\partial_{xx}^{2}u\left(x,\,t\right)\right|\,\leq\,\sup_{x\geq\frac{1}{4}\rho}\left|\partial_{xx}^{2}u\left(x,\,t_{0}\right)\right|\,+\,C\left(n,\,\rho\right)
\]
for $t_{0}\leq t\leq\mathring{t}$.
\end{lem}
\begin{proof}
Differentiating (\ref{eq u}) with respect to $x$ twice yields
\[
\partial_{t}\left(\partial_{xx}^{2}u\right)-\,\frac{1}{1+\left(\partial_{x}u\right)^{2}}\,\partial_{xx}^{2}\left(\partial_{xx}^{2}u\right)-\left(\frac{-6\,\partial_{x}u}{\left(1+\left(\partial_{x}u\right)^{2}\right)^{2}}\left(\partial_{xx}^{2}u\right)+\frac{2\left(n-1\right)}{x\left(1-\left(\frac{u}{x}\right)^{2}\right)}\right)\partial_{x}\left(\partial_{xx}^{2}u\right)
\]
\[
=-\frac{2\left(1-3\left(\partial_{x}u\right)^{2}\right)}{\left(1+\left(\partial_{x}u\right)^{2}\right)^{3}}\left(\partial_{xx}^{2}u\right)^{3}-\frac{2\left(n-1\right)\left(1+\left(\frac{u}{x}\right)^{2}-6\left(\frac{u}{x}\right)\partial_{x}u\right)}{x^{2}\left(1-\left(\frac{u}{x}\right)^{2}\right)^{2}}\left(\partial_{xx}^{2}u\right)
\]
\[
-\frac{4\left(n-1\right)\left(1-\left(\partial_{x}u\right)^{2}\right)}{x^{3}\left(1-\left(\frac{u}{x}\right)^{2}\right)^{3}}\left(\left(1+3\left(\frac{u}{x}\right)^{2}\right)\left(\partial_{x}u\right)-\left(3+\left(\frac{u}{x}\right)^{2}\right)\left(\frac{u}{x}\right)\right)
\]
For each $R\geq2$, let $\eta\left(x\right)$ be a smooth function
so that 
\[
\chi_{\left(\frac{1}{3}\rho,\,R-1\right)}\,\leq\,\eta\,\leq\,\chi_{\left(\frac{1}{4}\rho,\,R\right)}
\]
and $\eta\left(x\right)$ is increasing in $\left[\frac{1}{4}\rho,\,\frac{1}{3}\rho\right]$
and decreasing on $\left[R-1,\,R\right]$; moreove, $\eta\left(x\right)$
vanishes at $x=\frac{1}{4}\rho$ and $x=R$ to infinite order. Notice
that by Lemma \ref{localization}, we may assume
\begin{equation}
\frac{\left(\partial_{x}\eta\left(x\right)\right)^{2}}{\eta\left(x\right)}\,+\,\left|\partial_{x}\eta\left(x\right)\right|\,+\,\left|\partial_{xx}^{2}\eta\left(x\right)\right|\leq C\left(\rho\right)\label{outside DDu1}
\end{equation}
It follows that
\[
\partial_{t}\left(\eta\,\partial_{xx}^{2}u\right)-\frac{1}{1+\left(\partial_{x}u\right)^{2}}\,\partial_{xx}^{2}\left(\eta\,\partial_{xx}^{2}u\right)-\left(\frac{-6\,\partial_{x}u}{\left(1+\left(\partial_{x}u\right)^{2}\right)^{2}}\left(\partial_{xx}^{2}u\right)+\frac{2\left(n-1\right)}{x\left(1-\left(\frac{u}{x}\right)^{2}\right)}\right)\partial_{x}\left(\eta\,\partial_{xx}^{2}u\right)
\]
\[
=-\frac{2\left(1-3\left(\partial_{x}u\right)^{2}\right)}{\left(1+\left(\partial_{x}u\right)^{2}\right)^{3}}\eta\left(\partial_{xx}^{2}u\right)^{3}\,-\,\frac{2\left(n-1\right)\eta\left(x\right)\left(1+\left(\frac{u}{x}\right)^{2}-6\left(\frac{u}{x}\right)\partial_{x}u\right)}{x^{2}\left(1-\left(\frac{u}{x}\right)^{2}\right)^{2}}\,\left(\partial_{xx}^{2}u\right)
\]
\[
-\eta\left(x\right)\frac{4\left(n-1\right)\left(1-\left(\partial_{x}u\right)^{2}\right)}{x^{3}\left(1-\left(\frac{u}{x}\right)^{2}\right)^{3}}\left(\left(1+3\left(\frac{u}{x}\right)^{2}\right)\left(\partial_{x}u\right)-\left(3+\left(\frac{u}{x}\right)^{2}\right)\left(\frac{u}{x}\right)\right)
\]
\[
+\left(-\frac{\partial_{xx}^{2}\eta}{1+\left(\partial_{x}u\right)^{2}}-\partial_{x}\eta\left(x\right)\left(\frac{-6\,\partial_{x}u}{\left(1+\left(\partial_{x}u\right)^{2}\right)^{2}}\left(\partial_{xx}^{2}u\right)+\frac{2\left(n-1\right)}{x\left(1-\left(\frac{u}{x}\right)^{2}\right)}\right)\right)\left(\partial_{xx}^{2}u\right)
\]
\[
-\frac{2}{1+\left(\partial_{x}u\right)^{2}}\,\partial_{x}\eta\,\partial_{x}\left(\partial_{xx}^{2}u\right)
\]
Note that we can rewrite the last term on the RHS of the above equation
as
\[
-\frac{2}{1+\left(\partial_{x}u\right)^{2}}\,\partial_{x}\eta\,\partial_{x}\left(\partial_{xx}^{2}u\right)=-\frac{2}{1+\left(\partial_{x}u\right)^{2}}\frac{\partial_{x}\eta}{\eta}\left(\partial_{x}\left(\eta\,\partial_{xx}^{2}u\right)-\left(\partial_{x}\eta\right)\left(\partial_{xx}^{2}u\right)\right)
\]
So the equation of $\eta\,\partial_{xx}^{2}u$ can be rewritten as
\begin{equation}
\partial_{t}\left(\eta\,\partial_{xx}^{2}u\right)-\frac{1}{1+\left(\partial_{x}u\right)^{2}}\,\partial_{xx}^{2}\left(\eta\,\partial_{xx}^{2}u\right)\label{eq DDu}
\end{equation}
\[
-\left(\frac{-6\,\partial_{x}u}{\left(1+\left(\partial_{x}u\right)^{2}\right)^{2}}\left(\partial_{xx}^{2}u\right)+\frac{2\left(n-1\right)}{x\left(1-\left(\frac{u}{x}\right)^{2}\right)}-\frac{2}{1+\left(\partial_{x}u\right)^{2}}\left(\frac{\partial_{x}\eta}{\eta}\right)\right)\partial_{x}\left(\eta\,\partial_{xx}^{2}u\right)
\]
 
\[
=-a\left(x,\,t\right)\eta\left(\partial_{xx}^{2}u\right)^{3}+b\left(x,\,t\right)\left(\partial_{xx}^{2}u\right)^{2}+c\left(x,\,t\right)\left(\partial_{xx}^{2}u\right)+\eta\left(x\right)f\left(x,\,t\right)
\]
where 
\[
a\left(x,\,t\right)=\frac{2\left(1-3\left(\partial_{x}u\right)^{2}\right)}{\left(1+\left(\partial_{x}u\right)^{2}\right)^{3}}
\]
 
\[
b\left(x,\,t\right)=\frac{6\,\partial_{x}\eta\,\partial_{x}u}{\left(1+\left(\partial_{x}u\right)^{2}\right)^{2}}
\]
 
\[
c\left(x,\,t\right)=-\frac{2\left(n-1\right)\eta\left(x\right)\left(1+\left(\frac{u}{x}\right)^{2}-6\left(\frac{u}{x}\right)\partial_{x}u\right)}{x^{2}\left(1-\left(\frac{u}{x}\right)^{2}\right)^{2}}
\]
\[
-\frac{\partial_{xx}^{2}\eta}{1+\left(\partial_{x}u\right)^{2}}-\frac{2\left(n-1\right)\partial_{x}\eta}{x\left(1-\left(\frac{u}{x}\right)^{2}\right)}+\frac{2}{1+\left(\partial_{x}u\right)^{2}}\frac{\left(\partial_{x}\eta\right)^{2}}{\eta}
\]
 
\[
f\left(x,\,t\right)=-\frac{4\left(n-1\right)\left(1-\left(\partial_{x}u\right)^{2}\right)}{x^{3}\left(1-\left(\frac{u}{x}\right)^{2}\right)^{3}}\left(\left(1+3\left(\frac{u}{x}\right)^{2}\right)\left(\partial_{x}u\right)-\left(3+\left(\frac{u}{x}\right)^{2}\right)\left(\frac{u}{x}\right)\right)
\]
By (\ref{a priori bound u}), (\ref{initial u}), (\ref{preliminary bound C^2 outside u})
and Lemma \ref{outside Du}, if $0<\rho\ll1$ (depending on $n$,
$\Lambda$) and $\left|t_{0}\right|\ll1$ (depending on $n$, $\rho$),
we have
\[
\max\left\{ \left|\frac{u\left(x,\,t\right)}{x}\right|,\,\left|\partial_{x}u\left(x,\,t\right)\right|\right\} \,\leq\,\frac{1}{3}
\]
for $x\geq\frac{1}{4}\rho$, $t_{0}\leq t\leq\mathring{t}$, which,
together with (\ref{outside DDu1}), implies
\begin{equation}
\left\{ \begin{array}{c}
\frac{972}{1000}\,\leq\,a\left(x,\,t\right)\,\leq\,2\\
\\
\left|b\left(x,\,t\right)\right|\,+\,\left|c\left(x,\,t\right)\right|\,+\,\left|f\left(x,\,t\right)\right|\,\leq\,C\left(n,\,\rho\right)
\end{array}\right.\label{outside DDu2}
\end{equation}
 for $x\geq\frac{1}{4}\rho$, $t_{0}\leq t\leq\mathring{t}$. Now
let 
\[
M\,=\,\max_{\frac{1}{4}\rho\leq x\leq R,\,\,t_{0}\leq t\leq\mathring{t}}\,\eta\left(x\right)\partial_{xx}^{2}u\left(x,\,t\right)
\]
If
\[
M\,\leq\,\max_{\frac{1}{4}\rho\leq x\leq R}\,\left(\eta\left(x\right)\partial_{xx}^{2}u\left(x,\,t_{0}\right)\right)_{+}
\]
then we are done; otherwise, we have
\[
M\,>\,\max_{\frac{1}{4}\rho\leq x\leq R}\,\left(\eta\left(x\right)\partial_{xx}^{2}u\left(x,\,t_{0}\right)\right)_{+}
\]
In the later case, let $\left(x_{*},\,t_{*}\right)$ be a maximum
point of $\eta\,\partial_{xx}^{2}u$ in the spacetime, i.e. 
\[
\eta\left(x_{*}\right)\partial_{xx}^{2}u\left(x_{*},\,t_{*}\right)=M
\]
then we have $\frac{1}{4}\rho<x_{*}<R$, $t_{0}<t\leq\mathring{t}$.
Applying the maximum pricinple to (\ref{eq DDu}) yields
\[
0\leq-a\left(x_{*},\,t_{*}\right)\eta\left(x_{*}\right)\left(\partial_{xx}^{2}u\left(x_{*},\,t_{*}\right)\right)^{3}+b\left(\left(x_{*},\,t_{*}\right)\right)\left(\partial_{xx}^{2}u\left(x_{*},\,t_{*}\right)\right)^{2}
\]
\[
+c\left(x_{*},\,t_{*}\right)\left(\partial_{xx}^{2}u\left(x_{*},\,t_{*}\right)\right)+\eta\left(x_{*}\right)f\left(x_{*},\,t_{*}\right)
\]
\[
=\frac{1}{\eta^{2}\left(x_{*}\right)}\left(-a\left(x_{*},\,t_{*}\right)M^{3}+b\left(x_{*},\,t_{*}\right)M^{2}+\eta\left(x_{*}\right)c\left(x_{*},\,t_{*}\right)M+\eta^{3}\left(x_{*}\right)f\left(x_{*},\,t_{*}\right)\right)
\]
It follows, by Young's inequality and (\ref{outside DDu2}), that
\[
M^{3}\,\leq\,\frac{8}{3}\left(\frac{\left|b\left(x_{*},\,t_{*}\right)\right|}{a\left(x_{*},\,t_{*}\right)}\right)^{3}+\frac{4\sqrt{2}}{3}\left(\frac{\left|c\left(x_{*},\,t_{*}\right)\right|}{a\left(x_{*},\,t_{*}\right)}\right)^{\frac{3}{2}}+\frac{\left|f\left(x_{*},\,t_{*}\right)\right|}{a\left(x_{*},\,t_{*}\right)}\,\leq\,C\left(n,\,\rho\right)
\]
Therefore, in either case, we have
\[
\max_{\frac{1}{4}\rho\leq x\leq R,\,t_{0}\leq t\leq\mathring{t}}\,\eta\left(x\right)\partial_{xx}^{2}u\left(x,\,t\right)\,\leq\,\max_{x\geq\frac{1}{4}\rho}\,\left(\eta\left(x\right)\partial_{xx}^{2}u\left(x,\,t_{0}\right)\right)_{+}+C\left(n,\,\rho\right)
\]
Likewise, by the same argument, one could show that 
\[
\min_{\frac{\rho}{4}\leq x\leq R,\,t_{0}\leq t\leq\mathring{t}}\,\eta\left(x\right)\partial_{xx}^{2}u\left(x,\,t\right)\,\geq\,-\min_{x\geq\frac{\rho}{4}}\,\left(\eta\left(x\right)\partial_{xx}^{2}u\left(x,\,t_{0}\right)\right)-C\left(n,\,\rho\right)
\]
\end{proof}
In the next proposition, we apply the standard regularity theory for
parabolic equations to (\ref{eq u}), together with (\ref{C^2 outside u bound}),
to derive (\ref{C^infty outside u bound}).
\begin{prop}
\label{C^infty outside u}There holds (\ref{C^2 outside u bound}).
\end{prop}
\begin{proof}
Given $0<\delta\ll1$, let's fix $x_{*}\geq\frac{1}{2}\rho$, $t_{0}+\delta^{2}\leq t_{*}\leq\mathring{t}$.
By (\ref{C^2 outside u bound}) and Krylov-Safonov H$\ddot{o}$lder
estimates (applying to (\ref{eq u})), there is 
\[
\gamma=\gamma\left(n,\,\rho\right)\in\left(0,\,1\right)
\]
so that
\begin{equation}
\left[u\right]_{\gamma;\,Q\left(x_{*},\,t_{*};\,\frac{\delta}{2}\right)}\leq C\left(n,\,\rho,\,\delta\right)\left\Vert u\right\Vert _{L^{\infty}\left(Q\left(x_{*},\,t_{*};\,\delta\right)\right)}\leq C\left(n,\,\rho,\,\delta\right)\label{C^infty outside u1}
\end{equation}
Next, differentiate (\ref{eq u}) with respect to $x$ to get 
\[
\partial_{t}\left(\partial_{x}u\right)-\frac{1}{1+\left(\partial_{x}u\right)^{2}}\,\partial_{xx}^{2}\left(\partial_{x}u\right)
\]
\[
-\left(\frac{-2\,\partial_{x}u\,\partial_{xx}^{2}u}{\left(1+\left(\partial_{x}u\right)^{2}\right)^{2}}+\frac{2\left(n-1\right)}{x\left(1-\left(\frac{u}{x}\right)^{2}\right)}\right)\partial_{x}\left(\partial_{x}u\right)-\left(\frac{4\left(n-1\right)\left(\frac{u}{x}\right)\,\partial_{x}u}{x^{2}\left(1-\left(\frac{u}{x}\right)^{2}\right)^{2}}\right)\left(\partial_{x}u\right)
\]
\[
=\frac{-4\left(n-1\right)\left(\frac{u}{x}\right)}{x^{2}\left(1-\left(\frac{u}{x}\right)^{2}\right)^{2}}
\]
Then by (\ref{C^2 outside u bound}) and Krylov-Safonov H$\ddot{o}$lder
estimates (applying to the above equation of $\partial_{x}u$), we
may assume that for the same exponent $\gamma$, there holds 
\[
\left[\partial_{x}u\right]_{\gamma;\,Q\left(x_{*},\,t_{*};\,\frac{\delta}{2}\right)}\leq C\left(n,\,\rho,\,\delta\right)\left(\left\Vert \partial_{x}u\right\Vert _{L^{\infty}\left(Q\left(x_{*},\,t_{*};\,\delta\right)\right)}+\left\Vert \frac{u}{x}\right\Vert _{L^{\infty}\left(Q\left(x_{*},\,t_{*};\,\delta\right)\right)}\right)
\]
\begin{equation}
\leq C\left(n,\,\rho,\,\delta\right)\label{C^infty outside u2}
\end{equation}
It follows, by (\ref{C^2 outside u bound}), (\ref{C^infty outside u1}),
(\ref{C^infty outside u2}) and Schauder $C^{2,\gamma}$ estimates
(applying to (\ref{eq u})), that 
\[
\left[\partial_{xx}^{2}u\right]_{\gamma;\,Q\left(x_{*},\,t_{*};\,\frac{\delta}{3}\right)}\leq C\left(n,\,\rho,\,\delta\right)\left\Vert u\right\Vert _{L^{\infty}\left(Q\left(x_{*},\,t_{*};\,\frac{\delta}{2}\right)\right)}\leq C\left(n,\,\rho,\,\delta\right)
\]
By the bootstrap argument, one could show that for any $m\in\mathbb{Z}_{+}$,
there holds 
\begin{equation}
\left\Vert \partial_{x}^{m}u\right\Vert _{L^{\infty}\left(Q\left(x_{*},\,t_{*};\,\frac{\delta}{m+1}\right)\right)}+\left[\partial_{x}^{m}u\right]_{\gamma;\,Q\left(x_{*},\,t_{*};\,\frac{\delta}{m+1}\right)}\leq C\left(n,\,\rho,\,\delta,\,m\right)\label{C^infty bound outside u}
\end{equation}
Moreover, by (\ref{eq u}) and (\ref{C^infty bound outside u}), we
immediately get 
\[
\left\Vert \partial_{x}^{m}\partial_{t}u\right\Vert _{L^{\infty}\left(Q\left(x_{*},\,t_{*};\,\frac{\delta}{m+3}\right)\right)}+\left[\partial_{x}^{m}\partial_{t}u\right]_{\gamma;\,Q\left(x_{*},\,t_{*};\,\frac{\delta}{m+3}\right)}\leq C\left(n,\,\rho,\,\delta,\,m\right)
\]
for any $m\in\mathbb{Z}_{+}$. Differentiating (\ref{eq u}) with
respect to $t$ and using the above estimates gives 
\[
\left\Vert \partial_{x}^{m}\partial_{t}^{2}u\right\Vert _{L^{\infty}\left(Q\left(x_{*},\,t_{*};\,\frac{\delta}{m+5}\right)\right)}+\left[\partial_{x}^{m}\partial_{t}^{2}u\right]_{\gamma;\,Q\left(x_{*},\,t_{*};\,\frac{\delta}{m+5}\right)}\leq C\left(n,\,\rho,\,\delta,\,m\right)
\]
for any $m\in\mathbb{Z}_{+}$. Continuing this process and using induction
yields 
\[
\left\Vert \partial_{x}^{m}\partial_{t}^{l}u\right\Vert _{L^{\infty}\left(Q\left(x_{*},\,t_{*};\,\frac{\delta}{m+2l+1}\right)\right)}+\left[\partial_{x}^{m}\partial_{t}^{l}u\right]_{\gamma;\,Q\left(x_{*},\,t_{*};\,\frac{\delta}{m+2l+1}\right)}\leq C\left(n,\,\rho,\,\delta,\,m,\,l\right)
\]
for any $m,\,l\in\mathbb{Z}_{+}$. 
\end{proof}
In the following proposition, we prove (\ref{C^infty u bound}) by
using (\ref{eq u}), (\ref{a priori bound u}), (\ref{C^0 u bound}),
(\ref{C^0 u bound intermediate}) and the regularity theory for parabolic
equations.
\begin{prop}
\label{C^infty u} If $0<\rho\ll1$ (depending on $n$, $\Lambda$)
and $\left|t_{0}\right|\ll1$ (depending on $n$, $\Lambda$, $\rho$),
there holds (\ref{C^infty u bound}).
\end{prop}
\begin{proof}
Notice that by (\ref{a priori bound u}), we have
\begin{equation}
\max\left\{ \left|\frac{u\left(x,\,t\right)}{x}\right|,\,\left|\partial_{x}u\left(x,\,t\right)\right|\right\} \,\leq\,\frac{1}{3}\label{C^infty u1}
\end{equation}
\begin{equation}
x^{i}\left|\partial_{x}^{i}u\left(x,\,t\right)\right|\leq\Lambda\left(\left(-t\right)^{2}x^{\alpha}+x^{2\lambda_{2}+1}\right)\leq C\left(n,\,\Lambda\right)x^{2\lambda_{2}+1}\qquad\forall\;\,i\in\left\{ 0,\,1,\,2\right\} \label{C^infty u2}
\end{equation}
 for $\frac{1}{3}\sqrt{-t}\leq x\leq\rho$, $t_{0}\leq t\leq\mathring{t}$,
provided that $0<\rho\ll1$ (depending on $n$, $\Lambda$) and $\left|t_{0}\right|\ll1$
(depending on $n$, $\Lambda$, $\rho$).

Given $0<\delta\ll1$, let's fix $\left(x_{*},\,t_{*}\right)$ so
that 
\[
\frac{1}{2}\sqrt{-t_{*}}\leq x_{*}\leq\frac{3}{4}\rho,\qquad t_{0}+\delta^{2}x_{*}^{2}\leq t_{*}\leq\mathring{t}
\]
Define 
\[
h\left(r,\,\iota\right)=u\left(rx_{*},\,t_{*}+\iota x_{*}^{2}\right)
\]
for $\frac{2}{3}\leq r\leq\frac{4}{3}$, $-\delta^{2}\leq\iota\leq0$.
From (\ref{eq u}), there holds 
\begin{equation}
\partial_{\iota}h-a\left(r,\,\iota\right)\partial_{rr}^{2}h-b\left(r,\,\iota\right)\partial_{r}h-c\left(r,\,\iota\right)h=0\label{rescaled eq u}
\end{equation}
where 
\[
a\left(r,\,\iota\right)=\left.\frac{1}{1+\left(\partial_{x}u\left(x,\,t\right)\right)^{2}}\right|_{x=rx_{*},\,t=t_{*}+\iota x_{*}^{2}}
\]
 
\[
b\left(r,\,\iota\right)=\left.\frac{1}{r}\left(\frac{2\left(n-1\right)}{1-\left(\frac{u\left(x,\,t\right)}{x}\right)^{2}}\right)\right|_{x=rx_{*},\,t=t_{*}+\iota x_{*}^{2}}
\]
 
\[
c\left(r,\,\iota\right)=\left.\frac{1}{r^{2}}\left(\frac{2\left(n-1\right)}{1-\left(\frac{u\left(x,\,t\right)}{x}\right)^{2}}\right)\right|_{x=rx_{*},\,t=t_{*}+\iota x_{*}^{2}}
\]
By (\ref{C^infty u1}), (\ref{C^infty u2}) and Krylov-Safonov H$\ddot{o}$lder
estimates, there is 
\[
\gamma=\gamma\left(n,\,\Lambda\right)\in\left(0,\,1\right)
\]
so that 
\[
\left[h\right]_{\gamma;Q\left(1,0;\,\frac{\delta}{2}\right)}\leq C\left(n,\,\delta\right)\left\Vert h\right\Vert _{L^{\infty}\left(Q\left(1,0;\,\delta\right)\right)}\leq C\left(n,\,\Lambda,\,\delta\right)x_{*}^{2\lambda_{2}+1}
\]
In other words, we get 
\begin{equation}
x_{*}^{\gamma}\left[u\right]_{\gamma;\,Q\left(x_{*},\,t_{*};\,\frac{\delta}{2}x_{*}\right)}\leq C\left(n,\,\Lambda,\,\delta\right)x_{*}^{2\lambda_{2}+1}\label{Holder u}
\end{equation}
Next, differentiate (\ref{eq u}) with respect to $x$ to get

\begin{equation}
\partial_{t}\left(\partial_{x}u\right)-\frac{1}{1+\left(\partial_{x}u\right)^{2}}\partial_{xx}^{2}\left(\partial_{x}u\right)\label{eq Du}
\end{equation}
\[
-\frac{1}{x}\left(\frac{-2\,\partial_{x}u\left(x\,\partial_{xx}^{2}u\right)}{\left(1+\left(\partial_{x}u\right)^{2}\right)^{2}}+\frac{2\left(n-1\right)}{1-\left(\frac{u}{x}\right)^{2}}\right)\partial_{x}\left(\partial_{x}u\right)-\frac{1}{x^{2}}\left(\frac{4\left(n-1\right)\left(\frac{u}{x}\right)\,\partial_{x}u}{\left(1-\left(\frac{u}{x}\right)^{2}\right)^{2}}\right)\left(\partial_{x}u\right)
\]
\[
=\frac{1}{x^{2}}\left(\frac{-4\left(n-1\right)}{\left(1-\left(\frac{u}{x}\right)^{2}\right)^{2}}\left(\frac{u}{x}\right)\right)
\]
Define
\[
\tilde{h}\left(r,\,\iota\right)=\partial_{x}u\left(rx_{*},\,t_{*}+\iota x_{*}^{2}\right)
\]
then we have 
\begin{equation}
\partial_{\iota}\tilde{h}-\tilde{a}\left(r,\,\iota\right)\partial_{rr}^{2}\tilde{h}-\tilde{b}\left(r,\,\iota\right)\partial_{r}\tilde{h}-\tilde{c}\left(r,\,\iota\right)\tilde{h}=\tilde{f}\left(r,\,\iota\right)\label{rescaled eq Du}
\end{equation}
where 
\[
\tilde{a}\left(r,\,\iota\right)=\left.\frac{1}{1+\left(\partial_{x}u\left(x,\,t\right)\right)^{2}}\right|_{x=rx_{*},\,t=t_{*}+\iota x_{*}^{2}}
\]
 
\[
\tilde{b}\left(r,\,\iota\right)=\left.\frac{1}{r}\left(\frac{-2\,\partial_{x}u\left(x,\,t\right)\left(x\,\partial_{xx}^{2}u\left(x,\,t\right)\right)}{\left(1+\left(\partial_{x}u\left(x,\,t\right)\right)^{2}\right)^{2}}+\frac{2\left(n-1\right)}{1-\left(\frac{u\left(x,\,t\right)}{x}\right)^{2}}\right)\right|_{x=rx_{*},\,t=t_{*}+\iota x_{*}^{2}}
\]
 
\[
\tilde{c}\left(r,\,\iota\right)=\left.\frac{1}{r^{2}}\left(\frac{4\left(n-1\right)\left(\frac{u\left(x,\,t\right)}{x}\right)\,\partial_{x}u\left(x,\,t\right)}{\left(1-\left(\frac{u\left(x,\,t\right)}{x}\right)^{2}\right)^{2}}\right)\right|_{x=rx_{*},\,t=t_{*}+\iota x_{*}^{2}}
\]
 
\[
\tilde{f}\left(r,\,\iota\right)=\left.\frac{1}{r^{2}}\left(\frac{-4\left(n-1\right)}{\left(1-\left(\frac{u\left(x,\,t\right)}{x}\right)^{2}\right)^{2}}\left(\frac{u\left(x,\,t\right)}{x}\right)\right)\right|_{x=rx_{*},\,t=t_{*}+\iota x_{*}^{2}}
\]
By (\ref{C^infty u1}), (\ref{C^infty u2}) and Krylov-Safonov H$\ddot{o}$lder
estimates, we may assume that for the same exponent $\gamma$, there
holds 
\[
\left[\tilde{h}\right]_{\gamma;Q\left(1,0;\,\frac{\delta}{2}\right)}\leq C\left(n,\,\Lambda,\,\delta\right)\left(\left\Vert \tilde{h}\right\Vert _{L^{\infty}\left(Q\left(1,0;\,\delta\right)\right)}+\left\Vert \tilde{f}\right\Vert _{L^{\infty}\left(Q\left(1,0;\,\delta\right)\right)}\right)
\]
\[
\leq C\left(n,\,\Lambda,\,\delta\right)x_{*}^{2\lambda_{2}}
\]
which implies 
\begin{equation}
x_{*}^{\gamma}\left[\partial_{x}u\right]_{\gamma;\,Q\left(x_{*},\,t_{*};\,\frac{\delta}{2}x_{*}\right)}\leq C\left(n,\,\Lambda,\,\delta\right)x_{*}^{2\lambda_{2}}\label{Holder Du}
\end{equation}
Thus, by (\ref{C^infty u1}), (\ref{C^infty u2}), (\ref{Holder u}),
(\ref{Holder Du}), applying Schauder $C^{2,\gamma}$ estimates to
(\ref{rescaled eq u}) yields
\[
\left[\partial_{rr}^{2}h\right]_{\gamma;Q\left(1,\,0;\,\frac{\delta}{3}\right)}\leq C\left(n,\,\Lambda,\,\delta\right)\left\Vert h\right\Vert _{L^{\infty}\left(Q\left(1,\,0;\,\frac{\delta}{2}\right)\right)}\leq C\left(n,\,\Lambda,\,\delta\right)x_{*}^{2\lambda_{2}+1}
\]
which implies
\begin{equation}
x_{*}^{2+\gamma}\left[\partial_{xx}^{2}u\right]_{\gamma;\,Q\left(x_{*},\,t_{*};\,\frac{\delta}{3}x_{*}\right)}\leq C\left(n,\,\Lambda,\,\delta\right)x_{*}^{2\lambda_{2}+1}\label{Holder DDu}
\end{equation}
By the bootstrap and rescaling argument, one could show that for any
$m\in\mathbb{Z}_{+}$, there holds 
\[
x_{*}^{m}\left\Vert \partial_{x}^{m}u\right\Vert _{L^{\infty}\left(Q\left(x_{*},\,t_{*};\,\frac{\delta}{m+1}x_{*}\right)\right)}\,+\,x_{*}^{m+\gamma}\left[\partial_{x}^{m}u\right]_{\gamma;\,Q\left(x_{*},\,t_{*};\,\frac{\delta}{m+1}x_{*}\right)}
\]
\begin{equation}
\leq C\left(n,\,\Lambda,\,\delta,\,m\right)x_{*}^{2\lambda_{2}+1}\label{C^infty bound u}
\end{equation}
It follows, by (\ref{eq u}) and (\ref{C^infty bound u}), that 
\[
x_{*}^{m+2}\left\Vert \partial_{x}^{m}\partial_{t}u\right\Vert _{L^{\infty}\left(Q\left(x_{*},\,t_{*};\,\frac{\delta}{m+3}x_{*}\right)\right)}\,+\,x_{*}^{m+2+\gamma}\left[\partial_{x}^{m}\partial_{t}u\right]_{\gamma;\,Q\left(x_{*},\,t_{*};\,\frac{\delta}{m+3}x_{*}\right)}
\]
\[
\leq C\left(n,\,\Lambda,\,\delta,\,m\right)x_{*}^{2\lambda_{2}+1}
\]
for any $m\in\mathbb{Z}_{+}$. Then differentiate (\ref{eq u}) with
respect to $t$ and use the above estimates to get 
\[
x_{*}^{m+4}\left\Vert \partial_{x}^{m}\partial_{t}^{2}u\right\Vert _{L^{\infty}\left(Q\left(x_{*},\,t_{*};\,\frac{\delta}{m+5}x_{*}\right)\right)}+x_{*}^{m+4+\gamma}\left[\partial_{x}^{m}\partial_{t}^{2}u\right]_{\gamma;\,Q\left(x_{*},\,t_{*};\,\frac{\delta}{m+5}x_{*}\right)}
\]
\[
\leq C\left(n,\,\Lambda,\,\delta,\,m\right)x_{*}^{2\lambda_{2}+1}
\]
Continuing this process and using induction yields 
\[
x_{*}^{m+2l}\left\Vert \partial_{x}^{m}\partial_{t}^{l}u\right\Vert _{L^{\infty}\left(Q\left(x_{*},\,t_{*};\,\frac{\delta}{m+2l+1}x_{*}\right)\right)}+x_{*}^{m+2l+\gamma}\left[\partial_{x}^{m}\partial_{t}^{l}u\right]_{\gamma;\,Q\left(x_{*},\,t_{*};\,\frac{\delta}{m+2l+1}x_{*}\right)}
\]
\begin{equation}
\leq C\left(n,\,\Lambda,\,\delta,\,m\right)x_{*}^{2\lambda_{2}+1}\label{smooth bound u}
\end{equation}
for any $m,\,l\in\mathbb{Z}_{+}$.

On the other hand, by Proposition \ref{linear operator}, there holds
\[
\left(\partial_{s}+\mathcal{L}\right)\left(ke^{-\lambda_{2}s}\varphi_{2}\left(y\right)\right)=0
\]
By a rescaling argument, we get 
\begin{equation}
\left(\partial_{t}-\partial_{xx}^{2}-\frac{2\left(n-1\right)}{x}\partial_{x}-\frac{2\left(n-1\right)}{x^{2}}\right)\left(k\left(-t\right)^{\lambda_{2}+\frac{1}{2}}\varphi_{2}\left(\frac{x}{\sqrt{-t}}\right)\right)=0\label{eigenfunction u}
\end{equation}
In addition, by (\ref{eq u}) we have 
\begin{equation}
\left(\partial_{t}-\partial_{xx}^{2}-\frac{2\left(n-1\right)}{x}\partial_{x}-\frac{2\left(n-1\right)}{x^{2}}\right)u\left(x,\,t\right)\,=\,\frac{f\left(x,\,t\right)}{x^{2}}\label{linearized eq u}
\end{equation}
where
\[
f\left(x,\,t\right)=-\frac{\left(\partial_{x}u\right)^{2}}{1+\left(\partial_{x}u\right)^{2}}\left(x^{2}\,\partial_{xx}^{2}u\right)+\frac{2\left(n-1\right)\left(\frac{u}{x}\right)^{2}}{1-\left(\frac{u}{x}\right)^{2}}\left(x\,\partial_{x}u\right)+\frac{2\left(n-1\right)\left(\frac{u}{x}\right)^{2}}{1-\left(\frac{u}{x}\right)^{2}}u
\]
Note that by (\ref{C^infty u1}) and (\ref{smooth bound u}) we have
\[
x_{*}^{m+2l}\left\Vert \partial_{x}^{m}\partial_{t}^{l}f\left(x,\,t\right)\right\Vert _{L^{\infty}\left(Q\left(x_{*},\,t_{*};\,\frac{\delta}{m+2l+1}x_{*}\right)\right)}\,+\,x_{*}^{m+2l+\gamma}\left[\partial_{x}^{m}\partial_{t}^{l}f\left(x,\,t\right)\right]_{\gamma;\,Q\left(x_{*},\,t_{*};\,\frac{\delta}{m+2l+1}x_{*}\right)}
\]
\begin{equation}
\leq C\left(n,\,\Lambda,\,\delta,\,m,\,l\right)x_{*}^{4\lambda_{2}}\,x_{*}^{2\lambda_{2}+1}\label{smooth bound source u}
\end{equation}
for any $m,\,l\in\mathbb{Z}_{+}$. Subtract (\ref{eigenfunction u})
from (\ref{linearized eq u}) to get 
\[
\left(\partial_{t}-\partial_{xx}^{2}-\frac{2\left(n-1\right)}{x}\partial_{x}-\frac{2\left(n-1\right)}{x^{2}}\right)\left(u\left(x,\,t\right)-k\left(-t\right)^{\lambda_{2}+\frac{1}{2}}\varphi_{2}\left(\frac{x}{\sqrt{-t}}\right)\right)\,=\,\frac{f\left(x,\,t\right)}{x^{2}}
\]
Then by the rescaling argument, together with (\ref{smooth bound source u})
and Schauder estimates, we get
\[
x_{*}^{m+2l}\left\Vert \partial_{x}^{m}\partial_{t}^{l}\left(u\left(x,\,t\right)-k\left(-t\right)^{\lambda_{2}+\frac{1}{2}}\varphi_{2}\left(\frac{x}{\sqrt{-t}}\right)\right)\right\Vert _{L^{\infty}\left(Q\left(x_{*},\,t_{*};\,\frac{\delta}{m+2l+2}x_{*}\right)\right)}
\]
\[
\leq C\left(n,\,\Lambda,\,\delta,\,m,\,l\right)\left\Vert u\left(x,\,t\right)-k\left(-t\right)^{\lambda_{2}+\frac{1}{2}}\varphi_{2}\left(\frac{x}{\sqrt{-t}}\right)\right\Vert _{L^{\infty}\left(Q\left(x_{*},\,t_{*};\,\frac{\delta}{m+2l+1}x_{*}\right)\right)}
\]
\[
+C\left(n,\,\Lambda,\,\delta,\,m,\,l\right)\,\sum_{i=0}^{m}\sum_{j=0}^{l}x_{*}^{i+2j}\left\Vert \partial_{x}^{i}\partial_{t}^{j}f\left(x,\,t\right)\right\Vert _{L^{\infty}\left(Q\left(x_{*},\,t_{*};\,\frac{\delta}{m+2l+1}x_{*}\right)\right)}
\]
\[
+C\left(n,\,\Lambda,\,\delta,\,m,\,l\right)\,\sum_{i=0}^{m}\sum_{j=0}^{l}x_{*}^{i+2j+\gamma}\left[\partial_{x}^{i}\partial_{t}^{j}f\left(x,\,t\right)\right]_{\gamma;\,Q\left(x_{*},\,t_{*};\,\frac{\delta}{m+2l+1}x_{*}\right)}
\]
\[
\leq C\left(n,\,\Lambda,\,\delta,\,m,\,l\right)\left(\left(-t_{0}\right)^{\varsigma\lambda_{2}}+x_{*}^{4\lambda_{2}}\right)x_{*}^{2\lambda_{2}+1}
\]
for any $m,\,l\in\mathbb{Z}_{+}$.
\end{proof}
Below we use (\ref{eq v}), (\ref{a priori bound v}), (\ref{C^0 v bound intermediate}),
(\ref{C^0 v bound tip}) and the regularity theory to show (\ref{C^infty v bound intermediate})
and (\ref{C^infty v bound tip}).
\begin{prop}
\label{C^infty v}If $\beta\gg1$ (depending on $n$, $\Lambda$),
$s_{0}\gg1$ (depending on $n$, $\Lambda$, $\beta$), there hold
(\ref{C^infty v bound intermediate}) and (\ref{C^infty v bound tip}).
\end{prop}
\begin{proof}
By (\ref{a priori bound v}), we have 
\begin{equation}
y^{i}\left|\partial_{y}^{i}v\left(y,\,s\right)\right|\leq\Lambda e^{-\lambda_{2}s}\left(y^{\alpha}+y^{2\lambda_{2}+1}\right)\leq C\left(n,\,\Lambda\right)e^{-\lambda_{2}s}y^{\alpha}\label{C^infty v1}
\end{equation}
for $\beta e^{-\sigma s}\leq y\leq3$, $s_{0}\leq s\leq\mathring{s}$.
In particular, we may assume that 
\[
\max\left\{ \left|\frac{v\left(y,\,s\right)}{y}\right|,\,\left|\partial_{y}v\left(y,\,s\right)\right|\right\} \leq C\left(n,\,\Lambda\right)e^{-\lambda_{2}s}y^{\alpha-1}\leq\frac{1}{3}
\]
for $\beta e^{-\sigma s}\leq y\leq3$, $s_{0}\leq s\leq\mathring{s}$,
provided that $\beta\gg1$ (depending on $n$, $\Lambda$).

Now given $0<\delta\ll1$ and fix $\left(y_{*},\,s_{*}\right)$ so
that 
\[
\frac{3}{2}\beta e^{-\sigma s_{*}}\leq y_{*}\leq2,\qquad s_{0}+\delta^{2}y_{*}^{2}\leq s_{*}\leq\mathring{s}
\]
From (\ref{eq v}), we have
\[
\partial_{s}v\,-\,\frac{1}{1+\left(\partial_{y}v\right)^{2}}\,\partial_{yy}^{2}v\,-\,\frac{1}{y}\left(\frac{2\left(n-1\right)}{1-\left(\frac{v}{y}\right)^{2}}-\frac{y^{2}}{2}\right)\partial_{y}v\,-\,\frac{1}{y^{2}}\left(\frac{2\left(n-1\right)}{1-\left(\frac{v}{y}\right)^{2}}+\frac{y^{2}}{2}\right)v\,=\,0
\]
By (\ref{C^infty v1}) and Krylov-Safonov H$\ddot{o}$lder estimates,
there is 
\[
\gamma=\gamma\left(n,\,\Lambda\right)\in\left(0,\,1\right)
\]
so that 

\begin{equation}
y_{*}^{\gamma}\left[v\right]_{\gamma;\,Q\left(y_{*},\,s_{*};\,\frac{\delta}{2}y_{*}\right)}\leq C\left(n,\,\delta\right)\left\Vert v\right\Vert _{L^{\infty}\left(Q\left(y_{*},\,s_{*};\,\delta y_{*}\right)\right)}\leq C\left(n,\,\Lambda,\,\delta\right)e^{-\lambda_{2}s_{*}}y_{*}^{\alpha}\label{Holder v}
\end{equation}
Differentiate (\ref{eq v}) with respect to $y$ to get 
\[
\partial_{s}\left(\partial_{y}v\right)-\frac{1}{1+\left(\partial_{y}v\right)^{2}}\,\partial_{yy}^{2}\left(\partial_{y}v\right)
\]
\[
-\frac{1}{y}\left(\frac{-2\left(\partial_{y}v\right)\left(y\,\partial_{yy}^{2}v\right)}{\left(1+\left(\partial_{y}v\right)^{2}\right)^{2}}+\frac{2\left(n-1\right)}{1-\left(\frac{v}{y}\right)^{2}}-\frac{y^{2}}{2}\right)\partial_{y}\left(\partial_{y}v\right)\,-\,\frac{1}{y^{2}}\left(\frac{4\left(n-1\right)\left(\frac{v}{y}\right)\partial_{y}v}{\left(1-\left(\frac{v}{y}\right)^{2}\right)^{2}}\right)\left(\partial_{y}v\right)
\]
\[
=\frac{1}{y^{2}}\left(\frac{-4\left(n-1\right)\left(\frac{v}{y}\right)}{\left(1-\left(\frac{v}{y}\right)^{2}\right)^{2}}\right)
\]
By (\ref{C^infty v1}) and Krylov-Safonov H$\ddot{o}$lder estimates,
we may assume that for the same $\gamma$, there holds 
\[
y_{*}^{\gamma}\left[\partial_{y}v\right]_{\gamma;\,Q\left(y_{*},\,s_{*};\,\frac{\delta}{2}y_{*}\right)}\leq C\left(n,\,\Lambda,\,\delta\right)\left(\left\Vert \partial_{y}v\right\Vert _{L^{\infty}\left(Q\left(y_{*},\,s_{*};\,\delta y_{*}\right)\right)}+\left\Vert \frac{v}{y}\right\Vert _{L^{\infty}\left(Q\left(y_{*},\,s_{*};\,\delta y_{*}\right)\right)}\right)
\]
\begin{equation}
\leq C\left(n,\,\Lambda,\,\delta\right)e^{-\lambda_{2}s_{*}}y_{*}^{\alpha-1}\label{Holder Dv}
\end{equation}
By  (\ref{C^infty v1}), (\ref{Holder v}) and (\ref{Holder Dv}),
applying Schauder $C^{2,\,\gamma}$ estimates to (\ref{eq v}) yields
\begin{equation}
y_{*}^{2+\gamma}\left[\partial_{yy}^{2}v\right]_{\gamma;\,Q\left(y_{*},\,s_{*};\,\frac{\delta}{3}y_{*}\right)}\leq C\left(n,\,\Lambda,\,\delta\right)\left\Vert v\right\Vert _{L^{\infty}\left(Q\left(y_{*},\,s_{*};\,\frac{\delta}{2}y_{*}\right)\right)}\leq C\left(n,\,\Lambda,\,\delta\right)e^{-\lambda_{2}s_{*}}y_{*}^{\alpha}\label{Holder DDv}
\end{equation}
Then by the bootstrap argument, one could show that 
\[
y_{*}^{m}\left\Vert \partial_{y}^{m}v\left(y,\,s\right)\right\Vert _{L^{\infty}\left(Q\left(y_{*},\,s_{*};\,\frac{\delta}{m+1}y_{*}\right)\right)}+y_{*}^{m+\gamma}\left[\partial_{y}^{m}v\left(y,\,s\right)\right]_{\gamma;\,Q\left(y_{*},\,s_{*};\,\frac{\delta}{m+1}y_{*}\right)}
\]
\begin{equation}
\leq C\left(n,\,\Lambda,\,\delta,\,m\right)e^{-\lambda_{2}s_{*}}y_{*}^{\alpha}\label{C^infty bound v}
\end{equation}
for all $m\in\mathbb{Z}_{+}$. Furthermore, by (\ref{eq v}) and (\ref{C^infty bound v}),
we get 
\[
y_{*}^{m+2}\left\Vert \partial_{y}^{m}\partial_{s}v\left(y,\,s\right)\right\Vert _{L^{\infty}\left(Q\left(y_{*},\,s_{*};\,\frac{\delta}{m+3}y_{*}\right)\right)}+y_{*}^{m+2+\gamma}\left[\partial_{y}^{m}\partial_{s}v\left(y,\,s\right)\right]_{\gamma;\,Q\left(y_{*},\,s_{*};\,\frac{\delta}{m+3}y_{*}\right)}
\]
\[
\leq C\left(n,\,\Lambda,\,\delta,\,m\right)e^{-\lambda_{2}s_{*}}y_{*}^{\alpha}
\]
for all $m\geq0$. Diffrentiating (\ref{eq v}) with respect to $s$
and using the above estimates gives
\[
y_{*}^{m+4}\left\Vert \partial_{y}^{m}\partial_{s}^{2}v\left(y,\,s\right)\right\Vert _{L^{\infty}\left(Q\left(y_{*},\,s_{*};\,\frac{\delta}{m+5}y_{*}\right)\right)}+y_{*}^{m+4+\gamma}\left[\partial_{y}^{m}v\left(y,\,s\right)\right]_{\gamma;\,Q\left(y_{*},\,s_{*};\,\frac{\delta}{m+5}y_{*}\right)}
\]
\[
\leq C\left(n,\,\Lambda,\,\delta,\,m\right)e^{-\lambda_{2}s_{*}}y_{*}^{\alpha}
\]
Continuing this process and using induction yields 
\[
y_{*}^{m+2l}\left\Vert \partial_{y}^{m}\partial_{s}^{l}v\left(y,\,s\right)\right\Vert _{L^{\infty}\left(Q\left(y_{*},\,s_{*};\,\frac{\delta}{m+2l+1}y_{*}\right)\right)}+y_{*}^{m+2l+\gamma}\left[\partial_{y}^{m}\partial_{s}^{l}v\left(y,\,s\right)\right]_{\gamma;\,Q\left(y_{*},\,s_{*};\,\frac{\delta}{m+2l+1}y_{*}\right)}
\]
\begin{equation}
\leq C\left(n,\,\Lambda,\,\delta,\,m,\,l\right)e^{-\lambda_{2}s_{*}}y_{*}^{\alpha}\label{smooth bound v}
\end{equation}
for any $m,\,l\in\mathbb{Z}_{+}$.

If $e^{-\vartheta\sigma s_{*}}\leq y_{*}\leq2$, recall that by Proposition
\ref{linear operator}, there holds 
\[
\left(\partial_{s}+\mathcal{L}\right)\left(\frac{k}{c_{2}}e^{-\lambda_{2}s}\varphi_{2}\left(y\right)\right)=0
\]
That is, 
\begin{equation}
\left(\partial_{s}-\partial_{yy}^{2}+\frac{1}{y}\left(2\left(n-1\right)-\frac{y^{2}}{2}\right)\partial_{y}-\frac{1}{y^{2}}\left(2\left(n-1\right)+\frac{y^{2}}{2}\right)\right)\left(\frac{k}{c_{2}}e^{-\lambda_{2}s}\varphi_{2}\left(y\right)\right)=0\label{eigenfunction v}
\end{equation}
In addition, from (\ref{eq v}) we have 
\begin{equation}
\left(\partial_{s}-\partial_{yy}^{2}+\frac{1}{y}\left(2\left(n-1\right)-\frac{y^{2}}{2}\right)\partial_{y}-\frac{1}{y^{2}}\left(2\left(n-1\right)+\frac{y^{2}}{2}\right)\right)v\left(y,\,s\right)=\frac{h\left(y,\,s\right)}{y^{2}}\label{linearized eq v}
\end{equation}
where 
\[
h\left(y,\,s\right)=-\frac{\left(\partial_{y}v\right)^{2}}{1+\left(\partial_{y}v\right)^{2}}\left(y^{2}\,\partial_{yy}^{2}v\right)+\frac{2\left(n-1\right)\left(\frac{v}{y}\right)^{2}}{1-\left(\frac{v}{y}\right)^{2}}\left(y\,\partial_{y}v\right)+\frac{2\left(n-1\right)\left(\frac{v}{y}\right)^{2}}{1-\left(\frac{v}{y}\right)^{2}}\,v
\]
Notice that by (\ref{smooth bound v}), the function $h\left(y,\,s\right)$
satisfies 
\[
y_{*}^{m+2l}\left\Vert \partial_{y}^{m}\partial_{s}^{l}h\left(y,\,s\right)\right\Vert _{L^{\infty}\left(Q\left(y_{*},\,s_{*};\,\frac{\delta}{m+2l+1}y_{*}\right)\right)}\,+\,y_{*}^{m+2l+\gamma}\left[\partial_{y}^{m}\partial_{s}^{l}h\left(y,\,s\right)\right]_{\gamma;\,Q\left(y_{*},\,s_{*};\,\frac{\delta}{m+2l+1}y_{*}\right)}
\]
\[
\leq C\left(n,\,\Lambda,\,\delta,\,m,\,l\right)\left(e^{-\lambda_{2}s_{*}}y_{*}^{\alpha-1}\right)^{2}\left(e^{-\lambda_{2}s_{*}}y_{*}^{\alpha}\right)
\]
\[
=C\left(n,\,\Lambda,\,\delta,\,m,\,l\right)\left(e^{-\lambda_{2}s_{*}}y_{*}^{\alpha-2}\right)^{2}\left(e^{-\lambda_{2}s_{*}}y_{*}^{\alpha+2}\right)
\]
\begin{equation}
=C\left(n,\,\Lambda,\,\delta,\,m,\,l\right)e^{-\varkappa s_{*}}\left(e^{-\lambda_{2}s_{*}}y_{*}^{\alpha+2}\right)\label{smooth bound source v1}
\end{equation}
for any $m,\,l\in\mathbb{Z}_{+}$. Then we substract (\ref{eigenfunction v})
from (\ref{linearized eq v}) to get 
\[
\left(\partial_{s}-\partial_{yy}^{2}+\frac{1}{y}\left(2\left(n-1\right)-\frac{y^{2}}{2}\right)\partial_{y}-\frac{1}{y^{2}}\left(2\left(n-1\right)+\frac{y^{2}}{2}\right)\right)\left(v-\frac{k}{c_{2}}e^{-\lambda_{2}s}\varphi_{2}\left(y\right)\right)=\frac{h}{y^{2}}
\]
By (\ref{smooth bound source v1}) and Schauder estimates, we get
\[
y_{*}^{m+2l}\left\Vert \partial_{y}^{m}\partial_{s}^{l}\left(v\left(y,\,s\right)-\frac{k}{c_{2}}e^{-\lambda_{2}s}\varphi_{2}\left(y\right)\right)\right\Vert _{L^{\infty}\left(Q\left(y_{*},\,s_{*};\,\frac{\delta}{m+2l+2}y_{*}\right)\right)}
\]
\[
\leq C\left(n,\,\Lambda,\,\delta,\,m,\,l\right)\left\Vert v\left(y,\,s\right)-\frac{k}{c_{2}}e^{-\lambda_{2}s}\varphi_{2}\left(y\right)\right\Vert _{L^{\infty}\left(Q\left(y_{*},\,s_{*};\,\frac{\delta}{m+2l+1}y_{*}\right)\right)}
\]
 
\[
+C\left(n,\,\Lambda,\,\delta,\,m,\,l\right)\,\sum_{i=0}^{m}\sum_{j=0}^{l}y_{*}^{i+2j}\left\Vert \partial_{y}^{i}\partial_{s}^{j}h\right\Vert _{L^{\infty}\left(Q\left(y_{*},\,s_{*};\,\frac{\delta}{m+2l+1}y_{*}\right)\right)}
\]
\[
+C\left(n,\,\Lambda,\,\delta,\,m,\,l\right)\,\sum_{i=0}^{m}\sum_{j=0}^{l}y_{*}^{i+2j+\gamma}\left[\partial_{y}^{i}\partial_{s}^{j}h\right]_{\gamma;\,Q\left(y_{*},\,s_{*};\,\frac{\delta}{m+2l+1}y_{*}\right)}
\]
\[
\leq C\left(n,\,\Lambda,\,\delta,\,m,\,l\right)e^{-\varkappa s_{*}}\left(e^{-\lambda_{2}s_{*}}y_{*}^{\alpha+2}\right)
\]
for any $m,\,l\in\mathbb{Z}_{+}$. 

If $\frac{3}{2}\beta e^{-\sigma s_{*}}\leq y_{*}\leq e^{-\vartheta\sigma s_{*}}$,
notice that 
\[
\partial_{\tau}\psi_{k}\left(z\right)=0=\frac{1}{1+\left(\partial_{z}\psi_{k}\left(z\right)\right)^{2}}\,\partial_{zz}^{2}\psi_{k}\left(z\right)\,+\,2\left(n-1\right)\frac{z\,\partial_{z}\psi_{k}\left(z\right)+\psi_{k}\left(z\right)}{z^{2}-\psi_{k}^{2}\left(z\right)}
\]
Let 
\begin{equation}
\breve{v}\left(y,\,s\right)=e^{-\sigma s}\,\psi_{k}\left(e^{\sigma s}y\right)\label{v asymptote}
\end{equation}
then we have 
\[
\partial_{s}\breve{v}+\sigma\left(-y\,\partial_{y}\breve{v}+\breve{v}\right)=\frac{1}{1+\left(\partial_{y}\breve{v}\right)^{2}}\,\partial_{yy}^{2}\breve{v}\,+\,2\left(n-1\right)\frac{y\,\partial_{y}\breve{v}+\breve{v}}{y^{2}-\breve{v}^{2}}
\]
Then we subtract the above equation from (\ref{eq v}) to get 
\begin{equation}
\partial_{s}\left(v-\breve{v}\right)-a\left(y,\,s\right)\partial_{yy}^{2}\left(v-\breve{v}\right)-\frac{1}{y}\,b\left(y,\,s\right)\partial_{z}\left(v-\breve{v}\right)-\frac{1}{y^{2}}\,c\left(y,\,s\right)\left(v-\breve{v}\right)=\frac{1}{y^{2}}\,f\left(y,\,s\right)\label{eq Cauchy v}
\end{equation}
where 
\[
a\left(z,\,\tau\right)=\frac{1}{1+\left(\partial_{y}v\right)^{2}}
\]
\[
b\left(z,\,\tau\right)=\frac{-\left(y\,\partial_{yy}^{2}\breve{v}\right)\left(\partial_{y}v+\partial_{y}\breve{v}\right)}{\left(1+\left(\partial_{y}v\right)^{2}\right)\left(1+\left(\partial_{y}\breve{v}\right)^{2}\right)}+\frac{2\left(n-1\right)}{1-\left(\frac{v}{y}\right)^{2}}-\frac{y^{2}}{2}
\]
\[
c\left(z,\,\tau\right)=\frac{2\left(n-1\right)\left(\partial_{y}\breve{v}+\frac{\breve{v}}{y}\right)\left(\frac{v}{y}+\frac{\breve{v}}{y}\right)}{\left(1-\left(\frac{v}{y}\right)^{2}\right)\left(1-\left(\frac{\breve{v}}{y}\right)^{2}\right)}+\frac{2\left(n-1\right)}{1-\left(\frac{v}{y}\right)^{2}}+\frac{y^{2}}{2}
\]
\[
f\left(z,\,\tau\right)=\left(\frac{1}{2}+\sigma\right)y^{2}\left(-y\,\partial_{y}\breve{v}+\breve{v}\right)
\]
Note that by Lemma \ref{order psi} and (\ref{v asymptote}), we have
\begin{equation}
y^{m}\left|\partial_{y}^{m}\breve{v}\left(y,\,s\right)\right|\leq C\left(n,\,m\right)e^{-\lambda_{2}s}y^{\alpha}\label{bound coefficients v}
\end{equation}
for $y\geq\beta$, which yields
\[
y_{*}^{m+2l}\left\Vert \partial_{y}^{m}\partial_{s}^{l}f\left(y,\,s\right)\right\Vert _{L^{\infty}\left(Q\left(y_{*},\,s_{*};\,\frac{\delta}{m+2l+1}y_{*}\right)\right)}\,+\,y_{*}^{m+2l+\gamma}\left[\partial_{y}^{m}\partial_{s}^{l}f\left(y,\,s\right)\right]_{\gamma;\,Q\left(y_{*},\,s_{*};\,\frac{\delta}{m+2l+1}y_{*}\right)}
\]
\[
\leq C\left(n,\,\delta,\,m,\,l\right)\left(e^{-\lambda_{2}s_{*}}y_{*}^{\alpha+2}\right)
\]
\begin{equation}
\leq C\left(n,\,\delta,\,m,\,l\right)e^{-2\vartheta\sigma s_{*}}\left(e^{-\lambda_{2}s_{*}}y_{*}^{\alpha}\right)\label{smooth bound source v2}
\end{equation}
since $\frac{3}{2}\beta e^{-\sigma s_{*}}\leq y_{*}\leq e^{-\vartheta\sigma s_{*}}$.
Thus, by (\ref{smooth bound v}), (\ref{bound coefficients v}), (\ref{smooth bound source v2})
and applying Schauder estimates to (\ref{eq Cauchy v}), we get 
\[
y_{*}^{m+2l}\left\Vert \partial_{y}^{m}\partial_{s}^{l}\left(v\left(y,\,s\right)-\breve{v}\left(y,\,s\right)\right)\right\Vert _{L^{\infty}\left(Q\left(y_{*},\,s_{*};\,\frac{\delta}{m+2l+2}y_{*}\right)\right)}
\]
\[
\leq C\left(n,\,\Lambda,\,\delta,\,m,\,l\right)\left\Vert v\left(y,\,s\right)-\breve{v}\left(y,\,s\right)\right\Vert _{L^{\infty}\left(Q\left(y_{*},\,s_{*};\,\frac{\delta}{m+2l+1}y_{*}\right)\right)}
\]
 
\[
+C\left(n,\,\Lambda,\,\delta,\,m,\,l\right)\,\sum_{i=0}^{m}\sum_{j=0}^{l}y_{*}^{i+2j}\left\Vert \partial_{y}^{i}\partial_{s}^{j}f\right\Vert _{L^{\infty}\left(Q\left(y_{*},\,s_{*};\,\frac{\delta}{m+2l+1}y_{*}\right)\right)}
\]
\[
+C\left(n,\,\Lambda,\,\delta,\,m,\,l\right)\,\sum_{i=0}^{m}\sum_{j=0}^{l}y_{*}^{i+2j+\gamma}\left[\partial_{y}^{i}\partial_{s}^{j}f\right]_{\gamma;\,Q\left(y_{*},\,s_{*};\,\frac{\delta}{m+2l+1}y_{*}\right)}
\]
\[
\leq C\left(n,\,\Lambda,\,\delta,\,m,\,l\right)\left(\beta^{\alpha-3}e^{-2\varrho\sigma\left(s_{*}-s_{0}\right)}e^{-\lambda_{2}s_{*}}y_{*}^{\alpha}\,+\,e^{-2\vartheta\sigma s_{*}}\left(e^{-\lambda_{2}s_{*}}y_{*}^{\alpha}\right)\right)
\]
\[
\leq C\left(n,\,\Lambda,\,\delta,\,m,\,l\right)\beta^{\alpha-3}e^{-2\varrho\sigma\left(s_{*}-s_{0}\right)}e^{-\lambda_{2}s_{*}}y_{*}^{\alpha}
\]
provided that $s_{0}\gg1$ (depending on $n$, $\beta$). Notice that
$0<\varrho<\vartheta$.
\end{proof}
Next, we would like to prove (\ref{C^2 w' bound}). The $C^{0}$ estimate
is already shown in Proposition \ref{C^0 w'}. Below we would prove
the first and second derivatives estimates in Lemma \ref{C^1 w'}
and Lemma \ref{C^2 w'}, respectively. Before that, notice that by
(\ref{a priori bound w}) we have 
\begin{equation}
z^{i}\left|\partial_{z}^{i}w\left(z,\,\tau\right)\right|\,\leq\,\Lambda\left(z^{\alpha}+\frac{z^{2\lambda_{2}+1}}{\left(2\sigma\tau\right)^{2}}\right)\,\leq\,C\left(n,\,\Lambda\right)z{}^{\alpha},\qquad i\in\left\{ 0,\,1,\,2\right\} \label{a priori bound w intermediate}
\end{equation}
for $\beta\leq z\leq\left(2\sigma\tau\right)^{\frac{1}{2}\left(1-\vartheta\right)}$,
$\tau_{0}\leq\tau\leq\mathring{\tau}$; in particular, we have
\begin{equation}
\max\left\{ \left|\frac{w\left(z,\,\tau\right)}{z}\right|,\,\left|\partial_{z}w\left(z,\,\tau\right)\right|\right\} \,\leq\,\frac{1}{3}\label{preliminary bound w intermediate}
\end{equation}
for $\beta\leq z\leq\left(2\sigma\tau\right)^{\frac{1}{2}\left(1-\vartheta\right)}$,
$\tau_{0}\leq\tau\leq\mathring{\tau}$, provided that $\beta\gg1$
(depending on $n$, $\Lambda$). In the following lemma, we show how
to transform the above estimates for $w\left(z,\,\tau\right)$ to
$\hat{w}\left(z,\,\tau\right)$ via the projected curve $\bar{\Gamma}_{\tau}$
defined in (\ref{projected Gamma}). This lemma is useful since it
provides the ``boundary values'' for estimating $\hat{w}\left(z,\,\tau\right)$
in the rescaled tip region.
\begin{lem}
If $\beta\gg1$ (depending on $n$, $\Lambda$) and $\tau_{0}\gg1$
(depending on $n$, $\Lambda$, $\rho$, $\beta$), there hold 
\begin{equation}
\left|\partial_{z}\hat{w}\left(z,\,\tau\right)-1\right|\,\leq\,C\left(n,\,\Lambda\right)z{}^{\alpha-1}\label{Dw' intermediate}
\end{equation}
 
\begin{equation}
\left|\partial_{zz}^{2}\hat{w}\left(z,\,\tau\right)\right|\,\leq\,C\left(n,\,\Lambda\right)z{}^{\alpha-2}\label{DDw' intermediate}
\end{equation}
for $2\beta\leq z\leq\frac{1}{2}\left(2\sigma\tau\right)^{\frac{1}{2}\left(1-\vartheta\right)}$,
$\tau_{0}\leq\tau\leq\mathring{\tau}$. 
\end{lem}
\begin{proof}
Let's first parametrize the projected curve $\bar{\Gamma}_{\tau}$
by 
\[
Z_{\tau}=\left(\left(z-w\left(z,\,\tau\right)\right)\frac{1}{\sqrt{2}},\,\left(z+w\left(z,\,\tau\right)\right)\frac{1}{\sqrt{2}}\right)
\]
In this parametrization, there hold 
\[
N_{\bar{\Gamma}_{\tau}}\cdot\mathbf{e}=\frac{-\partial_{z}w\left(z,\,\tau\right)}{\sqrt{1+\left(\partial_{z}w\left(z,\,\tau\right)\right)^{2}}}
\]
 
\[
A_{\bar{\Gamma}_{\tau}}=\frac{\partial_{zz}^{2}w\left(z,\,\tau\right)}{\left(1+\left(\partial_{z}w\left(z,\,\tau\right)\right)^{2}\right)^{\frac{3}{2}}}
\]
where $N_{\bar{\Gamma}_{\tau}}$ and $A_{\bar{\Gamma}_{\tau}}$ are
the (upward) unit normal vector and normal curvature of $\bar{\Gamma}_{\tau}$
at $Z_{\tau}$, respectively, and
\[
\mathbf{e}=\left(\frac{1}{\sqrt{2}},\,\frac{1}{\sqrt{2}}\right)
\]
By (\ref{a priori bound w intermediate}) and (\ref{preliminary bound w intermediate}),
we get 
\[
z\,\leq\,\left|Z_{\tau}\right|=\sqrt{z^{2}+\left(w\left(z,\,\tau\right)\right)^{2}}\,\leq\,\sqrt{\frac{10}{9}}z
\]
 
\begin{equation}
\left|N_{\bar{\Gamma}_{\tau}}\cdot\mathbf{e}\right|\leq C\left(n\right)\Lambda\left|Z_{\tau}\right|{}^{\alpha-1}\label{normal Gamma intermediate}
\end{equation}
 
\begin{equation}
\left|A_{\bar{\Gamma}_{\tau}}\right|\,\leq\,C\left(n\right)\Lambda\left|Z_{\tau}\right|{}^{\alpha-2}\label{curvature Gamma intermediate}
\end{equation}
for $\beta\leq z\leq\left(2\sigma\tau\right)^{\frac{1}{2}\left(1-\vartheta\right)}$,
$\tau_{0}\leq\tau\leq\mathring{\tau}$.

Now we reparametrize $\bar{\Gamma}_{\tau}$ as 
\[
Z_{\tau}=\left(z,\,\hat{w}\left(z,\,\tau\right)\right)
\]
In that case, we have 
\[
\left|Z_{\tau}\right|=\sqrt{z^{2}+\left(\hat{w}\left(z,\,\tau\right)\right)^{2}}
\]
 
\begin{equation}
N_{\bar{\Gamma}_{\tau}}\cdot\mathbf{e}=\frac{1-\partial_{z}\hat{w}\left(z,\,\tau\right)}{\sqrt{2\left(1+\left(\partial_{z}\hat{w}\left(z,\,\tau\right)\right)^{2}\right)}}\label{normal Gamma'}
\end{equation}
 
\begin{equation}
A_{\bar{\Gamma}_{\tau}}=\frac{\partial_{zz}^{2}\hat{w}\left(z,\,\tau\right)}{\left(1+\left(\partial_{z}\hat{w}\left(z,\,\tau\right)\right)^{2}\right)^{\frac{3}{2}}}\label{curvature Gamma'}
\end{equation}
Note that by (\ref{psi'}), (\ref{k}) and (\ref{C^0 w' bound}),
there holds 
\begin{equation}
\frac{1}{C\left(n\right)}\,\leq\,\frac{\left|Z_{\tau}\right|}{z}\,\leq\,C\left(n\right)\label{ratio w' preliminary bound}
\end{equation}
for $2\beta\leq z\leq\frac{1}{2}\left(2\sigma\tau\right)^{\frac{1}{2}\left(1-\vartheta\right)}$,
provided that $\beta\gg1$ (depending on $n$) and $\tau_{0}\gg1$
(depending on $n$, $\Lambda$, $\rho$, $\beta$). Moreover, by (\ref{normal Gamma intermediate})
we may assume
\[
\left|N_{\bar{\Gamma}_{\tau}}\cdot\mathbf{e}\right|\,\leq\,\frac{1}{100\sqrt{2}}
\]
for $2\beta\leq z\leq\frac{1}{2}\left(2\sigma\tau\right)^{\frac{1}{2}\left(1-\vartheta\right)}$.
Since
\[
\lim_{p\rightarrow\pm\infty}\,\frac{1-p}{\sqrt{2\left(1+p^{2}\right)}}=\mp\frac{1}{\sqrt{2}}
\]
it follows, by (\ref{normal Gamma'}), that 
\begin{equation}
\left|\partial_{z}\hat{w}\left(z,\,\tau\right)\right|\,\leq\,C\label{derivative w' preliminary bound}
\end{equation}
for $2\beta\leq z\leq\frac{1}{2}\sqrt{2\sigma\tau}$. The conclusion
follows by (\ref{normal Gamma intermediate}), (\ref{curvature Gamma intermediate}),
(\ref{normal Gamma'}), (\ref{curvature Gamma'}), (\ref{ratio w' preliminary bound})
and (\ref{derivative w' preliminary bound}). 
\end{proof}
\begin{rem}
Note that for the last lemma, when $\tau=\tau_{0}$, by (\ref{initial w intermediate})
we have
\[
N_{\bar{\Gamma}_{\tau_{0}}}\cdot\mathbf{e}=\frac{-\partial_{z}w\left(z,\,\tau_{0}\right)}{\sqrt{1+\left(\partial_{z}w\left(z,\,\tau_{0}\right)\right)^{2}}}>0
\]
for $\frac{1}{2}\beta\leq z\leq\left(2\sigma\tau\right)^{\frac{1}{2}\left(1-\vartheta\right)}$,
$\tau_{0}\leq\tau\leq\mathring{\tau}$. Consequently, by the same
argument and (\ref{normal Gamma'}), we can show that
\begin{equation}
0\,\leq\,1-\partial_{z}\hat{w}\left(z,\,\tau_{0}\right)\,\leq\,C\left(n,\,\Lambda\right)z{}^{\alpha-1}\label{Dw' intermediate initial}
\end{equation}
for $\frac{1}{2}\beta\leq z\leq\left(2\sigma\tau_{0}\right)^{\frac{1}{2}\left(1-\vartheta\right)}$.
\end{rem}
Below we use (\ref{eq w'}), (\ref{initial w'}), (\ref{Dw' intermediate})
and the maximum principle to show the first derivative estimate in
(\ref{C^2 w' bound}). 
\begin{lem}
\label{C^1 w'}If $\beta\gg1$ (depending on $n$, $\Lambda$), there
holds 
\begin{equation}
0\,\leq\,\partial_{z}\hat{w}\left(z,\,\tau\right)\,\leq\,1+\beta^{\alpha-2}\label{C^1 w' bound}
\end{equation}
for $0\leq z\leq\beta^{2}$, $\tau_{0}\leq\tau\leq\mathring{\tau}$. 
\end{lem}
\begin{proof}
By differentiating (\ref{eq w'}), we get 
\begin{equation}
\partial_{\tau}\left(\partial_{z}\hat{w}\right)=\frac{1}{1+\left(\partial_{z}\hat{w}\right)^{2}}\,\partial_{zz}^{2}\left(\partial_{z}\hat{w}\right)\label{eq Dw'}
\end{equation}
\[
+\left(\frac{n-1}{z}-\frac{2\,\partial_{z}\hat{w}\,\partial_{zz}^{2}\hat{w}}{\left(1+\left(\partial_{z}\hat{w}\right)^{2}\right)^{2}}-\frac{\frac{1}{2}+\sigma}{2\sigma\tau}z\right)\partial_{z}\left(\partial_{z}\hat{w}\right)+\left(n-1\right)\left(\frac{1}{\hat{w}^{2}}-\frac{1}{z^{2}}\right)\left(\partial_{z}\hat{w}\right)
\]
Notice that for the last term on the RHS of (\ref{eq Dw'}), by (\ref{eq psi'})
and (\ref{C^0 w' bound}), there holds 
\begin{equation}
\hat{w}\left(z,\,\tau\right)\,>\,z\quad\Leftrightarrow\quad\frac{1}{\hat{w}^{2}\left(z,\,\tau\right)}-\frac{1}{z^{2}}\,<\,0\label{C^1 w' positive}
\end{equation}
for $0\leq z\leq\beta^{2}$, $\tau_{0}\leq\tau\leq\mathring{\tau}$. 

Let 
\[
\left(\partial_{z}\hat{w}\right)_{\min}\left(\tau\right)\,=\,\min_{0\leq z\leq\beta^{2}}\partial_{z}\hat{w}\left(z,\,\tau\right)
\]
Then $\left(\partial_{z}\hat{w}\right)_{\min}\left(\tau_{0}\right)\geq0$
by (\ref{initial w'}). We claim that 
\begin{equation}
\left(\partial_{z}\hat{w}\right)_{\min}\left(\tau\right)\geq0\label{C^1 w' claim}
\end{equation}
for $\tau_{0}\leq\tau\leq\mathring{\tau}$. To prove that, we use
a contradiction argument. Suppose that there is $\tau_{1}^{*}>\tau_{0}$
so that 
\[
\left(\partial_{z}\hat{w}\right)_{\min}\left(\tau_{1}^{*}\right)<0
\]
Let $\tau_{0}^{*}>\tau_{0}$ be the first time after which $\left(\partial_{z}\hat{w}\right)_{\min}$
stays negative all the way up to $\tau_{1}^{*}$. By continuity, we
have 
\[
\left(\partial_{z}\hat{w}\right)_{\min}\left(\tau_{0}^{*}\right)\geq0
\]
Note that by (\ref{eq w'}) and (\ref{Dw' intermediate}), the negative
minimum of $\partial_{z}\hat{w}\left(z,\,\tau\right)$ for each time-slice
must be attained in $\left(0,\,\beta^{2}\right)$, provided that $\beta\gg1$
(depending on $n$, $\Lambda$). Applying the maximum principle to
(\ref{eq Dw'}) (and noting (\ref{C^1 w' positive})) yields
\[
\partial_{\tau}\left(\partial_{z}\hat{w}\right)_{\min}\,\geq\,\left(n-1\right)\left(\frac{1}{\hat{w}^{2}}-\frac{1}{z^{2}}\right)\left(\partial_{z}\hat{w}\right)_{\min}\,\geq\,0
\]
for $\tau_{0}^{*}\leq\tau<\tau_{1}^{*}$. It follows that 
\[
\left(\partial_{z}\hat{w}\right)_{\min}\left(\tau_{0}^{*}\right)\,\leq\,\left(\partial_{z}\hat{w}\right)_{\min}\left(\tau_{1}^{*}\right)<0
\]
which is a contradiction.

Next, let 
\[
\left(\partial_{z}\hat{w}\right)_{\max}\left(\tau\right)\,=\,\max_{0\leq z\leq\beta^{2}}\partial_{z}\hat{w}\left(z,\,\tau\right)
\]
Then 
\[
\left(\partial_{z}\hat{w}\right)_{\max}\left(\tau_{0}\right)\,\leq\,1
\]
by (\ref{initial w'}) and (\ref{Dw' intermediate initial}). We claim
that 
\[
\left(\partial_{z}\hat{w}\right)_{\max}\left(\tau\right)\,\leq\,1+\beta^{\alpha-2}
\]
for $\tau_{0}\leq\tau\leq\mathring{\tau}$. Suppose the contrary,
then there is $\tau_{1}^{*}>\tau_{0}$ so that 
\[
\left(\partial_{z}\hat{w}\right)_{\max}\left(\tau_{1}^{*}\right)\,>\,1+\beta^{\alpha-2}
\]
Let $\tau_{0}^{*}>\tau_{0}$ be the first time after which $\left(\partial_{z}\hat{w}\right)_{\max}$
is greater than $1+\beta^{\alpha-2}$ all the way up to $\tau_{1}^{*}$.
By continuity, we have 
\[
\left(\partial_{z}\hat{w}\right)_{\max}\left(\tau_{0}^{*}\right)\,\leq\,1+\beta^{\alpha-2}
\]
Notice that by (\ref{Dw' intermediate}), there holds 
\[
\partial_{z}\hat{w}\left(\beta^{2},\,\tau\right)\,\leq\,1+C\left(n,\,\Lambda\right)\beta^{2\left(\alpha-1\right)}<1+\beta^{\alpha-2}
\]
provided that $\beta\gg1$(depending on $n$, $\Lambda$). Thus, the
maximum of $\partial_{z}\hat{w}\left(z,\,\tau\right)$ for each time-slice
which is greater than $1+\beta^{\alpha-2}$ must be attained in $\left(0,\,\beta^{2}\right)$,
provided that $\beta\gg1$ (depending on $n$, $\Lambda$). Applying
the maximum principle to (\ref{eq Dw'}) (and using (\ref{C^1 w' positive})
and (\ref{C^1 w' claim})) yields
\[
\partial_{\tau}\left(\partial_{z}\hat{w}\right)_{\max}\leq0
\]
for $\tau_{0}^{*}\leq\tau<\tau_{1}^{*}$. It follows that 
\[
\left(\partial_{z}\hat{w}\right)_{\max}\left(\tau_{0}^{*}\right)\geq\left(\partial_{z}\hat{w}\right)_{\max}\left(\tau_{1}^{*}\right)>1+\beta^{\alpha-2}
\]
which is a contradiction. 
\end{proof}
Then we start to show the second derivative estimate in (\ref{C^2 w' bound}).
Note that the second fundamental form of $\Gamma_{\tau}$ (in the
parametrization of (\ref{w'})) is given by 
\begin{equation}
A_{\Gamma_{\tau}}=\frac{1}{\sqrt{1+\left|\partial_{z}\hat{w}\right|^{2}}}\left(\begin{array}{ccc}
\frac{\partial_{zz}^{2}\hat{w}}{1+\left|\partial_{z}\hat{w}\right|^{2}}\\
 & \frac{\partial_{z}\hat{w}}{z}\,I_{n-1}\\
 &  & \frac{-1}{\hat{w}}\,I_{n-1}
\end{array}\right)\label{second fundamental form Gamma}
\end{equation}
By (\ref{C^0 w' bound}) and (\ref{C^1 w' bound}), to estimate $\partial_{zz}^{2}\hat{w}\left(z,\,\tau\right)$
is equivalent to estimate $A_{\Gamma_{\tau}}$. In the following lemma,
we derive an evolution equation of $A_{\Gamma_{\tau}}$ and use that,
together with (\ref{initial w'}), (\ref{DDw' intermediate}) and
the maximum principle, to show that $A_{\Gamma_{\tau}}$ can be estimated
for a short period of time. 
\begin{lem}
\label{short time C^2 w'}If $\beta\gg1$ (depending on $n$, $\Lambda$),
then there is $\delta>0$ (depending on $n$) so that the second fundamental
form of $\Gamma_{\tau}$ satisfies
\[
\max_{\Gamma_{\tau}\cap B\left(O;\,3\beta\right)}\,\left|A_{\Gamma_{\tau}}\right|\,\leq\,C\left(n\right)
\]
for $\tau_{0}\leq\tau\leq\min\left\{ \tau_{0}+\delta,\,\mathring{\tau}\right\} $.
In particular, there holds 
\[
\left|\partial_{zz}^{2}\hat{w}\left(z,\,\tau\right)\right|\leq C\left(n\right)
\]
for $0\leq z\leq3\beta$, $\tau_{0}\leq\tau\leq\min\left\{ \tau_{0}+\delta,\,\mathring{\tau}\right\} $.
\end{lem}
\begin{proof}
By (\ref{initial w'}), (\ref{C^0 w' bound}), (\ref{Dw' intermediate}),
(\ref{DDw' intermediate}) and (\ref{second fundamental form Gamma}),
the second fundamental form of $\Gamma_{\tau}$ satisfies
\begin{equation}
\mathfrak{C}\,\,\equiv\,\,\left|A_{\Gamma_{\tau}}\right|_{\max}^{2}\left(\tau_{0}\right)\,+\,\max_{Z_{\tau}\in\Gamma_{\tau},\,\left|Z_{\tau}\right|=3\beta}\,\left|A_{\Gamma_{\tau}}\left(Z_{\tau}\right)\right|^{2}\,\,\leq\,\,C\left(n\right)\label{short time C^2 w'1}
\end{equation}
provided that $\beta\gg1$ (depending on $n$, $\Lambda$). By reparametrization
of the flow, we may derive an evolution equation for $A_{\Gamma_{\tau}}$
as follows: 
\begin{equation}
\left(\partial_{\tau}-\triangle_{\Gamma_{\tau}}\right)\left|A_{\Gamma_{\tau}}\right|^{2}\,=\,-2\left|\nabla_{\Gamma_{\tau}}A_{\Gamma_{\tau}}\right|^{2}\,+\,2\left|A_{\Gamma_{\tau}}\right|^{4}\,-\,\frac{1+2\sigma}{2\sigma\tau}\left|A_{\Gamma_{\tau}}\right|^{2}\label{eq A(Z)}
\end{equation}
Let
\[
h\left(\tau\right)=\max_{\Gamma_{\tau}\cap B\left(O;\,3\beta\right)}\,\left|A_{\Gamma_{\tau}}\right|^{2}
\]
If $h\left(\tau\right)\,\leq\,\mathfrak{C}$ for $\tau_{0}\leq\tau\leq\mathring{\tau}$,
then we are done. Otherwise, there is $\tau_{1}^{*}>\tau_{0}$ so
that 
\[
h\left(\tau_{1}^{*}\right)\,>\,\mathfrak{C}
\]
Let $\tau_{0}^{*}>\tau_{0}$ be the first time after which $h$ is
greater than $\mathfrak{C}$ all the way up to $\tau_{1}^{*}$. By
continuity, we have 
\begin{equation}
h\left(\tau_{0}^{*}\right)\,\leq\,\mathfrak{C}\label{short time C^2 w'2}
\end{equation}
Note that the maximum for each time-slice must be attained in the
interior of $\Gamma_{\tau}\cap B\left(O;\,3\beta\right)$. By applying
the maximum principle to (\ref{eq A(Z)}), we get 
\[
\partial_{\tau}h\left(\tau\right)\,\leq\,2\,h^{2}\left(\tau\right)
\]
for $\tau_{0}^{*}\leq\tau\leq\tau_{1}^{*}$, which implies 
\begin{equation}
h\left(\tau_{1}^{*}\right)\,\leq\,\frac{h\left(\tau_{0}^{*}\right)}{1-2\left(\tau_{1}^{*}-\tau_{0}^{*}\right)h\left(\tau_{0}^{*}\right)}\label{short time C^2 w'3}
\end{equation}
Thus, by (\ref{short time C^2 w'1}), (\ref{short time C^2 w'2})
and (\ref{short time C^2 w'3}), there is $\delta=\delta\left(n\right)$
so that 
\[
h\left(\tau\right)\leq2\mathfrak{C}
\]
for $\tau_{0}^{*}\leq\tau\leq\min\left\{ \tau_{0}^{*}+\delta,\,\tau_{1}^{*}\right\} $.
For this choice of $\delta>0$, we claim that
\[
h\left(\tau\right)\leq2\mathfrak{C}
\]
for $\tau_{0}\leq\tau\leq\min\left\{ \tau_{0}+\delta,\,\mathring{\tau}\right\} $;
otherwise, we may get a contradiction by the above argument. Then
the conclusion follows immediately by (\ref{C^0 w' bound}), (\ref{C^1 w' bound})
and (\ref{second fundamental form Gamma}).
\end{proof}
In the following lemma, we use Ecker-Huisken interior estimate for
MCF to estimate $A_{\Gamma_{\tau}}$ for $\tau_{0}+\delta\leq\tau\leq\mathring{\tau}$.
Combining with Lemma \ref{short time C^2 w'}, we then get the second
derivative estimate in (\ref{C^2 w' bound}).
\begin{lem}
\label{C^2 w'}If $\beta\gg1$ (depending on $n$, $\Lambda$), there
holds 
\[
\left|\partial_{zz}^{2}\hat{w}\left(z,\,\tau\right)\right|\leq C\left(n\right)
\]
for $0\leq z\leq3\beta$, $\tau_{0}\leq\tau\leq\mathring{\tau}$.
\end{lem}
\begin{proof}
By Lemma \ref{short time C^2 w'}, there is $\delta=\delta\left(n\right)$
so that 
\[
\left|\partial_{zz}^{2}\hat{w}\left(z,\,\tau\right)\right|\leq C\left(n\right)
\]
for $0\leq z\leq3\beta$, $\tau_{0}\leq\tau\leq\min\left\{ \tau_{0}+\delta,\,\mathring{\tau}\right\} $.
Hence, to prove the lemma, we have to consider the case when $\mathring{\tau}-\tau_{0}>\delta$. 

Fix $\tau_{0}+\delta\leq\tau_{*}\leq\mathring{\tau}$ and let 
\[
\varXi_{\iota}=\left(2\sigma\tau_{*}\right)^{\frac{1}{2}+\frac{1}{4\sigma}}\,\Sigma_{-\left(2\sigma\tau_{*}\right)^{\frac{-1}{2\sigma}}\left(1-\frac{\iota}{2\sigma\tau_{*}}\right)}
\]
\[
=\left\{ \left(r\nu,\,\hat{h}\left(r,\,\iota\right)\omega\right)\biggr|\,r\geq0,\,\nu\in\mathbb{S}^{n-1},\,\omega\in\mathbb{S}^{n-1}\right\} 
\]
where
\[
\hat{h}\left(r,\,\iota\right)=\left(2\sigma\tau_{*}\right)^{\frac{1}{2}+\frac{1}{4\sigma}}\,\hat{u}\left(\frac{r}{\left(2\sigma\tau_{*}\right)^{\frac{1}{2}+\frac{1}{4\sigma}}},\,-\left(2\sigma\tau_{*}\right)^{\frac{-1}{2\sigma}}\left(1-\frac{\iota}{2\sigma\tau_{*}}\right)\right)
\]
Then $\left\{ \varXi_{\iota}\right\} $ defines a MCF for $-\left(2\sigma\tau_{*}\right)\left(\left(\frac{\tau_{*}}{\tau_{0}}\right)^{\frac{1}{2\sigma}}-1\right)\leq\iota\leq0$.
Note that 
\[
\varXi_{0}=\left(2\sigma\tau_{*}\right)^{\frac{1}{2}+\frac{1}{4\sigma}}\,\Sigma_{-\left(2\sigma\tau_{*}\right)^{\frac{-1}{2\sigma}}}=\Gamma_{\tau_{*}}
\]
and 
\[
\left(2\sigma\tau_{*}\right)\left(\left(\frac{\tau_{*}}{\tau_{0}}\right)^{\frac{1}{2\sigma}}-1\right)\geq\frac{\delta}{2}
\]
provided that $\tau_{0}\gg1$(depending on $n$). By (\ref{u'w'}),
we may rewrite $\hat{h}\left(r,\,\iota\right)$ as 
\[
\hat{h}\left(r,\,\iota\right)=\left(1-\frac{\iota}{2\sigma\tau_{*}}\right)^{\frac{1}{2}+\sigma}\,\hat{w}\left(\frac{r}{\left(1-\frac{\iota}{2\sigma\tau_{*}}\right)^{\frac{1}{2}+\sigma}},\,\frac{\tau_{*}}{\left(1-\frac{\iota}{2\sigma\tau_{*}}\right)^{2\sigma}}\right)
\]
By (\ref{C^0 w' bound}) and (\ref{C^1 w' bound}), we have 
\begin{equation}
\hat{h}\left(r,\,\iota\right)\,\geq\,\frac{\hat{\psi}\left(0\right)}{2}\label{C^2  bound w'1}
\end{equation}
 
\begin{equation}
\left|\partial_{r}\hat{h}\left(r,\,\iota\right)\right|=\left|\partial_{z}\hat{w}\,\left(\frac{r}{\left(1-\frac{\iota}{2\sigma\tau_{*}}\right)^{\frac{1}{2}+\sigma}},\,\frac{\tau_{*}}{\left(1-\frac{\iota}{2\sigma\tau_{*}}\right)^{2\sigma}}\right)\right|\,\leq\,\frac{4}{3}\label{C^2  bound w'2}
\end{equation}
for $0\leq r\leq4\beta$, $-\frac{\delta}{2}\leq\iota\leq0$, provided
that $\tau_{0}\gg1$ (depending on $n$). Note that the unit normal
vector of $\varXi_{\iota}$ at $\mathcal{X}_{\iota}\left(r,\,\nu,\,\omega\right)=\left(r\nu,\,\hat{h}\left(r,\,\iota\right)\omega\right)$
is given by

\[
N_{\varXi_{\iota}}\left(r,\,\nu,\,\omega\right)=\frac{\left(-\partial_{r}\hat{h}\left(r,\,\iota\right)\,\nu,\,\,\omega\right)}{\sqrt{1+\left(\partial_{r}\hat{h}\left(r,\,\iota\right)\right)^{2}}}
\]
which satisfies
\begin{equation}
\left(N_{\varXi_{\iota}}\left(r,\,\nu,\,\omega\right)\cdot\mathbf{e}\right)^{-1}=\frac{\sqrt{1+\left(\partial_{r}\hat{h}\left(r,\,\iota\right)\right)^{2}}}{\left(\vec{0},\,\omega\right)\cdot\mathbf{e}}\label{C^2  bound w'3}
\end{equation}
where 
\[
\mathbf{e}=\left(\overset{\left(\textrm{2n-1}\right)\,\textrm{copies}}{\overbrace{0,\cdots,\,0}},\,1\right),\qquad\vec{0}=\left(\overset{n\,\textrm{copies}}{\overbrace{0,\cdots,\,0}}\right)
\]
Now fix $0\leq z_{*}\leq3\beta$ and let 
\[
\mathcal{X}_{*}=\left(z_{*}\nu_{*},\,\hat{h}\left(z_{*},\,0\right)\,\omega_{*}\right)=\left(z_{*}\nu_{*},\,\hat{w}\left(z_{*},\,\tau_{*}\right)\,\omega_{*}\right)
\]
where $\nu_{*}=\omega_{*}=\left(\overset{\left(\textrm{n-1}\right)\,\textrm{copies}}{\overbrace{0,\cdots,\,0}},\,1\right)$,
we claim that 
\begin{equation}
\left(N_{\varXi_{\iota}}\left(r,\,\nu,\,\omega\right)\cdot\mathbf{e}\right)^{-1}\leq\frac{5\sqrt{2}}{3}\label{gradient claim}
\end{equation}
for $\mathcal{X}_{\iota}\in\varXi_{\iota}\cap B^{2n}\left(\mathcal{X}_{*},\,\frac{\hat{\psi}\left(0\right)}{2\sqrt{2}}\right)$,
$-\frac{\delta}{2}\leq\iota\leq0$. Then by the curvature estimate
in \cite{EH}, the second fundamental form of $\Gamma_{\tau_{*}}$
at $\mathcal{X}_{*}$ satisfies 
\[
\left|A_{\Gamma_{\tau_{*}}}\left(\mathcal{X}_{*}\right)\right|\,=\,\left|A_{\varXi_{0}}\left(\mathcal{X}_{*}\right)\right|\,\leq\,C\left(n\right)\left(\frac{2\sqrt{2}}{\hat{\psi}\left(0\right)}+\sqrt{\frac{2}{\delta}}\right)=C\left(n\right)
\]
It follows that 
\[
\frac{\left|\partial_{zz}^{2}\hat{w}\left(z_{*},\,\tau_{*}\right)\right|}{\left(1+\left(\partial_{z}\hat{w}\left(z_{*},\,\tau_{*}\right)\right)^{2}\right)^{\frac{3}{2}}}\,\leq\,\left|A_{\Gamma_{\tau}}\left(\mathcal{X}_{*}\right)\right|\,\leq\,C\left(n\right)
\]
Now let's come back to (\ref{gradient claim}). First notice that
for each 
\[
\mathcal{X}_{\iota}\left(r,\,\nu,\,\omega\right)\,\in\,\varXi_{\iota}\cap B^{2n}\left(\mathcal{X}_{*},\,\frac{\hat{\psi}\left(0\right)}{2\sqrt{2}}\right),\qquad-\frac{\delta}{2}\leq\iota\leq0
\]
there holds 
\[
\hat{h}\left(r,\,\iota\right)\,\sqrt{1-\left(\left(\vec{0},\,\omega\right)\cdot\mathbf{e}\right)^{2}}\,\leq\,\left|\mathcal{X}_{\iota}\left(r,\,\nu,\,\omega\right)-\mathcal{X}_{*}\right|\,\leq\,\frac{\hat{\psi}\left(0\right)}{2\sqrt{2}}
\]
which, together with (\ref{C^2  bound w'1}), implies 
\begin{equation}
\left(\vec{0},\,\omega\right)\cdot\mathbf{e}\,\geq\,\frac{1}{\sqrt{2}}\label{C^2  bound w'4}
\end{equation}
Then (\ref{gradient claim}) follows by (\ref{C^2  bound w'2}), (\ref{C^2  bound w'3})
and (\ref{C^2  bound w'4}).
\end{proof}
Below we use (\ref{eq psi'}), (\ref{eq w'}), (\ref{C^2 w' bound}),
(\ref{C^0 w' bound}) and the standard regularity theory for parabolic
equations to prove (\ref{C^infty w' bound}).
\begin{prop}
\label{C^infty w'} If $\beta\gg1$ (depending on $n$) and $\tau_{0}\gg1$
(depending on $n$, $\beta$), there holds (\ref{C^infty w' bound}).
\end{prop}
\begin{proof}
Firstly, let $\hat{\boldsymbol{w}}\left(\boldsymbol{z},\,\tau\right)$
and $\hat{\boldsymbol{\psi}}_{k}\left(\boldsymbol{z}\right)$ be radially
symmetric functions so that 
\[
\hat{\boldsymbol{w}}\left(\boldsymbol{z},\,\tau\right)=\hat{w}\left(z,\,\tau\right)\Bigr|_{z=\left|\boldsymbol{z}\right|},\qquad\hat{\boldsymbol{\psi}}_{k}\left(\boldsymbol{z}\right)=\left.\hat{\psi}_{k}\left(z\right)\right|_{z=\left|\boldsymbol{z}\right|}
\]
where $\boldsymbol{z}=\left(\boldsymbol{z}_{1},\cdots,\,\boldsymbol{z}_{n}\right)$.
Note that 
\[
\partial_{\boldsymbol{z}_{i}}\hat{\boldsymbol{w}}\,=\,\partial_{z}\hat{w}\,\frac{\boldsymbol{z}_{i}}{\left|\boldsymbol{z}\right|}
\]
\[
\partial_{\boldsymbol{z}_{i}\boldsymbol{z}_{j}}^{2}\hat{\boldsymbol{w}}\,=\,\partial_{zz}^{2}\hat{w}\,\frac{\boldsymbol{z}_{i}\,\boldsymbol{z}_{j}}{\left|\boldsymbol{z}\right|^{2}}\,+\,\partial_{z}\hat{w}\,\frac{\left|\boldsymbol{z}\right|^{2}\delta_{ij}\,-\,\boldsymbol{z}_{i}\,\boldsymbol{z}_{j}}{\left|\boldsymbol{z}\right|^{3}}
\]
Then by (\ref{C^0 w' bound tip}), (\ref{C^1 w' bound}) and (\ref{C^2 w' bound}),
there hold 
\begin{equation}
\left\{ \begin{array}{c}
\left|\hat{\boldsymbol{w}}\left(\boldsymbol{z},\,\tau\right)-\hat{\boldsymbol{\psi}}_{k}\left(\boldsymbol{z}\right)\right|\leq C\left(n\right)\beta^{\alpha-3}\left(\frac{\tau}{\tau_{0}}\right)^{-\varrho}\\
\\
\left|\nabla\hat{\boldsymbol{w}}\left(\boldsymbol{z},\,\tau\right)\right|\leq1+\beta^{\alpha-2}\\
\\
\left|\nabla^{2}\hat{\boldsymbol{w}}\left(\boldsymbol{z},\,\tau\right)\right|\leq C\left(n\right)
\end{array}\right.\label{C^infty w'1}
\end{equation}
for $\boldsymbol{z}\in B\left(O;\,3\beta\right)$, $\tau_{0}\leq\tau\leq\mathring{\tau}$,
$m\in\mathbb{Z}_{+}$, where 
\[
\nabla=\left(\partial_{\boldsymbol{z}_{1}},\cdots,\,\partial_{\boldsymbol{z}_{n}}\right)
\]
Also, by (\ref{eq psi'}) and Lemma \ref{order psi'}, we get
\begin{equation}
\left\Vert \nabla^{m}\hat{\boldsymbol{\psi}}_{k}\right\Vert _{L^{\infty}}\,\leq\,C\left(n,\,m\right)\label{C^infty w'2}
\end{equation}
for all $m\geq1$. In addition, from (\ref{eq psi'}) and (\ref{eq w'}),
we have 
\[
\partial_{\tau}\hat{w}\,=\,\frac{\sqrt{1+\left(\partial_{z}\hat{w}\right)^{2}}}{z^{n-1}}\,\partial_{z}\left(\frac{z^{n-1}}{\sqrt{1+\left(\partial_{z}w\right)^{2}}}\,\partial_{z}\hat{w}\right)\,-\,\frac{n-1}{\hat{w}}\,+\,\frac{\frac{1}{2}+\sigma}{2\sigma\tau}\left(-z\,\partial_{z}\hat{w}+\hat{w}\right)
\]
and
\[
\partial_{\tau}\hat{\psi}_{k}\,=0\,=\,\frac{\sqrt{1+\left(\partial_{z}\hat{\psi}_{k}\right)^{2}}}{z^{n-1}}\,\partial_{z}\left(\frac{z^{n-1}}{\sqrt{1+\left(\partial_{z}\hat{\psi}_{k}\right)^{2}}}\,\partial_{z}\hat{\psi}_{k}\right)\,-\,\frac{n-1}{\hat{\psi}_{k}}
\]
which yield 
\[
\partial_{\tau}\hat{\boldsymbol{w}}=\sqrt{1+\left|\nabla\hat{\boldsymbol{w}}\right|^{2}}\,\,\nabla\cdot\frac{\nabla\hat{\boldsymbol{w}}}{\sqrt{1+\left|\nabla\hat{\boldsymbol{w}}\right|^{2}}}\,-\,\frac{n-1}{\hat{\boldsymbol{w}}}\,+\,\frac{\frac{1}{2}+\sigma}{2\sigma\tau}\left(-\boldsymbol{z}\cdot\nabla\hat{\boldsymbol{w}}+\hat{\boldsymbol{w}}\right)
\]
\begin{equation}
=\sum_{i,\,j=1}^{n}\left(\delta_{ij}-\frac{\partial_{\boldsymbol{z}_{i}}\hat{\boldsymbol{w}}\,\partial_{\boldsymbol{z}_{j}}\hat{\boldsymbol{w}}}{1+\left|\nabla\hat{\boldsymbol{w}}\right|^{2}}\right)\partial_{\boldsymbol{z}_{i}\boldsymbol{z}_{j}}^{2}\hat{\boldsymbol{w}}\,-\,\sum_{i=1}^{n}\left(\frac{\frac{1}{2}+\sigma}{2\sigma\tau}\boldsymbol{z}_{i}\right)\partial_{\boldsymbol{z}_{i}}\hat{\boldsymbol{w}}\,+\,\left(\frac{\frac{1}{2}+\sigma}{2\sigma\tau}\right)\hat{\boldsymbol{w}}\,-\,\frac{n-1}{\hat{\boldsymbol{w}}}\label{eq w'(z)}
\end{equation}
and 
\[
\partial_{\tau}\hat{\boldsymbol{\psi}}_{k}=0=\sqrt{1+\left|\nabla\hat{\boldsymbol{\psi}}_{k}\right|^{2}}\,\,\nabla\cdot\frac{\nabla\hat{\boldsymbol{\psi}}_{k}}{\sqrt{1+\left|\nabla\hat{\boldsymbol{\psi}}_{k}\right|^{2}}}\,-\,\frac{\left(n-1\right)}{\hat{\boldsymbol{\psi}}_{k}}
\]
\begin{equation}
=\sum_{i,\,j=1}^{n}\left(\delta_{ij}-\frac{\partial_{\boldsymbol{z}_{i}}\hat{\boldsymbol{\psi}}_{k}\,\partial_{\boldsymbol{z}_{j}}\hat{\boldsymbol{\psi}}_{k}}{1+\left|\nabla\hat{\boldsymbol{\psi}}_{k}\right|^{2}}\right)\partial_{\boldsymbol{z}_{i}\boldsymbol{z}_{j}}^{2}\hat{\boldsymbol{\psi}}_{k}\,-\,\frac{n-1}{\hat{\boldsymbol{\psi}}_{k}}\label{eq psi(z)}
\end{equation}
Then we subtract (\ref{eq psi(z)}) from (\ref{eq w'(z)}) to get
\[
\partial_{\tau}\left(\hat{\boldsymbol{w}}-\hat{\boldsymbol{\psi}}_{k}\right)-\sum_{i,\,j=1}^{n}\left(\delta_{ij}-\frac{\partial_{\boldsymbol{z}_{i}}\hat{\boldsymbol{w}}\,\partial_{\boldsymbol{z}_{j}}\hat{\boldsymbol{w}}}{1+\left|\nabla\hat{\boldsymbol{w}}\right|^{2}}\right)\partial_{\boldsymbol{z}_{i}\boldsymbol{z}_{j}}^{2}\left(\hat{\boldsymbol{w}}-\hat{\boldsymbol{\psi}}_{k}\right)
\]
 
\[
-\sum_{q=1}^{n}\left(\frac{\sum_{i,\,j=1}^{n}\partial_{\boldsymbol{z}_{i}}\hat{\boldsymbol{\psi}}_{k}\,\partial_{\boldsymbol{z}_{j}}\hat{\boldsymbol{\psi}}_{k}\,\partial_{\boldsymbol{z}_{i}\boldsymbol{z}_{j}}^{2}\hat{\boldsymbol{\psi}}_{k}\,\left(\partial_{\boldsymbol{z}_{q}}\hat{\boldsymbol{w}}+\partial_{\boldsymbol{z}_{q}}\hat{\boldsymbol{\psi}}_{k}\right)}{\left(1+\left|\nabla\hat{\boldsymbol{w}}\right|^{2}\right)\left(1+\left|\nabla\hat{\boldsymbol{\psi}}_{k}\right|^{2}\right)}\right)\partial_{\boldsymbol{z}_{q}}\left(\hat{\boldsymbol{w}}-\hat{\boldsymbol{\psi}}_{k}\right)
\]
\[
+\sum_{q=1}^{n}\left(\frac{\sum_{i=1}^{n}\partial_{\boldsymbol{z}_{i}\boldsymbol{z}_{q}}^{2}\hat{\boldsymbol{\psi}}_{k}\left(\partial_{\boldsymbol{z}_{i}}\hat{\boldsymbol{w}}+\partial_{\boldsymbol{z}_{i}}\hat{\boldsymbol{\psi}}_{k}\right)}{1+\left|\nabla\hat{\boldsymbol{w}}\right|^{2}}+\frac{\frac{1}{2}+\sigma}{2\sigma\tau}\boldsymbol{z}_{q}\right)\partial_{\boldsymbol{z}_{q}}\left(\hat{\boldsymbol{w}}-\hat{\boldsymbol{\psi}}_{k}\right)
\]
\[
-\left(\frac{n-1}{\hat{\boldsymbol{w}}\,\hat{\boldsymbol{\psi}}_{k}}\,+\,\frac{\frac{1}{2}+\sigma}{2\sigma\tau}\right)\left(\hat{\boldsymbol{w}}-\hat{\boldsymbol{\psi}}_{k}\right)
\]
\begin{equation}
=\frac{\frac{1}{2}+\sigma}{2\sigma\tau}\left(-\boldsymbol{z}\cdot\nabla\hat{\boldsymbol{\psi}}_{k}\left(\boldsymbol{z}\right)+\hat{\boldsymbol{\psi}}_{k}\left(\boldsymbol{z}\right)\right)\,\equiv\,\boldsymbol{f}\left(\boldsymbol{z},\,\tau\right)\label{eq Cauchy w'(z)}
\end{equation}
Note that by (\ref{asymptotic psi'}), we have 
\begin{equation}
\left|\nabla^{m}\partial_{\tau}^{l}\boldsymbol{f}\left(\boldsymbol{z},\,\tau\right)\right|\,\leq\,C\left(n,\,m,\,l\right)\tau^{-1}\label{C^infty w'3}
\end{equation}
for $\boldsymbol{z}\in\mathbb{R}^{n}$, $\tau_{0}\leq\tau\leq\mathring{\tau}$,
$m\geq0$.

Now fix $0<\delta\ll1$ and $\boldsymbol{z}_{*}\in B\left(O;\,2\beta\right)$,
$\tau_{0}+\delta^{2}\leq\tau_{*}\leq\mathring{\tau}$. By (\ref{C^infty w'1}),
(\ref{C^infty w'2}) and Krylov-Safonov H$\ddot{o}$lder estimate
(applying to (\ref{eq Cauchy w'(z)})), there is 
\[
\gamma=\gamma\left(n\right)\in\left(0,\,1\right)
\]
so that 
\begin{equation}
\delta^{\gamma}\left[\hat{\boldsymbol{w}}-\hat{\boldsymbol{\psi}}_{k}\right]_{\gamma;\,Q\left(\boldsymbol{z}_{*},\,\tau_{*};\,\frac{1}{2}\delta\right)}\label{Holder w'}
\end{equation}
\[
\leq C\left(n\right)\left(\left\Vert \hat{\boldsymbol{w}}-\hat{\boldsymbol{\psi}}_{k}\right\Vert _{L^{\infty}\left(Q\left(\boldsymbol{z}_{*},\,\tau_{*};\,\delta\right)\right)}+\delta^{2}\left\Vert \boldsymbol{f}\right\Vert _{L^{\infty}\left(Q\left(\boldsymbol{z}_{*},\,\tau_{*};\,\delta\right)\right)}\right)\,\leq\,C\left(n\right)
\]
provided that $\beta\gg1$ (depending on $n$) and $\tau_{0}\gg1$
(depending on $n$, $\beta$). Next, for each $p\in\left\{ 1,\cdots,\,n\right\} $,
differentiate (\ref{eq w'(z)}) with respect to $\boldsymbol{z}_{p}$
to get 
\[
\partial_{\tau}\left(\partial_{\boldsymbol{z}_{p}}\hat{\boldsymbol{w}}\right)=\triangle\left(\partial_{\boldsymbol{z}_{p}}\hat{\boldsymbol{w}}\right)-\frac{\nabla^{2}\left(\partial_{\boldsymbol{z}_{p}}\hat{\boldsymbol{w}}\right)\left(\nabla\hat{\boldsymbol{w}},\,\nabla\hat{\boldsymbol{w}}\right)}{1+\left|\nabla\hat{\boldsymbol{w}}\right|^{2}}
\]
\[
+\,\frac{1}{2}\,\biggl\langle\left(\frac{\biggl\langle\nabla\ln\left(1+\left|\nabla\hat{\boldsymbol{w}}\right|^{2}\right),\,\nabla\hat{\boldsymbol{w}}\biggr\rangle}{1+\left|\nabla\hat{\boldsymbol{w}}\right|^{2}}\,\nabla\hat{\boldsymbol{w}}\,-\,\nabla\ln\left(1+\left|\nabla\hat{\boldsymbol{w}}\right|^{2}\right)\right),\;\,\nabla\left(\partial_{\boldsymbol{z}_{p}}\hat{\boldsymbol{w}}\right)\biggr\rangle
\]
\[
-\left\langle \frac{\frac{1}{2}+\sigma}{2\sigma\tau}\,\boldsymbol{z},\;\,\nabla\left(\partial_{\boldsymbol{z}_{p}}\hat{\boldsymbol{w}}\right)\right\rangle +\frac{n-1}{\hat{\boldsymbol{w}}^{2}}\left(\partial_{\boldsymbol{z}_{p}}\hat{\boldsymbol{w}}\right)
\]
 
\[
=\sum_{i,\,j=1}^{n}\left(\delta_{ij}-\frac{\partial_{\boldsymbol{z}_{i}}\hat{\boldsymbol{w}}\,\partial_{\boldsymbol{z}_{j}}\hat{\boldsymbol{w}}}{1+\left|\nabla\hat{\boldsymbol{w}}\right|^{2}}\right)\partial_{\boldsymbol{z}_{i}\boldsymbol{z}_{j}}^{2}\left(\partial_{\boldsymbol{z}_{p}}\hat{\boldsymbol{w}}\right)
\]
\[
+\sum_{q=1}^{n}\left(\frac{\sum_{i,\,j=1}^{n}\partial_{\boldsymbol{z}_{i}}\hat{\boldsymbol{w}}\,\partial_{\boldsymbol{z}_{j}}\hat{\boldsymbol{w}}\,\partial_{\boldsymbol{z}_{q}}\hat{\boldsymbol{w}}\,\partial_{\boldsymbol{z}_{i}\boldsymbol{z}_{j}}^{2}\hat{\boldsymbol{w}}\,-\,\sum_{i=1}^{n}\left(1+\left|\nabla\hat{\boldsymbol{w}}\right|^{2}\right)\partial_{\boldsymbol{z}_{i}}\hat{\boldsymbol{w}}\,\partial_{\boldsymbol{z}_{i}\boldsymbol{z}_{q}}^{2}\hat{\boldsymbol{w}}}{\left(1+\left|\nabla\hat{\boldsymbol{w}}\right|^{2}\right)^{2}}\right)\partial_{\boldsymbol{z}_{q}}\left(\partial_{\boldsymbol{z}_{p}}\hat{\boldsymbol{w}}\right)
\]
\[
-\sum_{q=1}^{n}\frac{\frac{1}{2}+\sigma}{2\sigma\tau}\,\boldsymbol{z}_{q}\,\partial_{\boldsymbol{z}_{q}}\left(\partial_{\boldsymbol{z}_{p}}\hat{\boldsymbol{w}}\right)\,+\,\frac{n-1}{\hat{\boldsymbol{w}}^{2}}\left(\partial_{\boldsymbol{z}_{p}}\hat{\boldsymbol{w}}\right)
\]
Then by (\ref{C^infty w'1}) and Krylov-Safonov H$\ddot{o}$lder estimates,
we may assume that for the same exponent $\gamma$, there holds 
\begin{equation}
\delta^{1+\gamma}\left[\nabla\hat{\boldsymbol{w}}\right]_{\gamma;\,Q\left(\boldsymbol{z}_{*},\,\tau_{*};\,\frac{1}{2}\delta\right)}\,\leq\,C\left(n\right)\delta\left\Vert \nabla\hat{\boldsymbol{w}}\right\Vert _{L^{\infty}\left(Q\left(\boldsymbol{z}_{*},\,\tau_{*};\,\delta\right)\right)}\,\leq\,C\left(n\right)\label{Holder Dw'}
\end{equation}
Therefore, by (\ref{C^infty w'1}), (\ref{C^infty w'2}), (\ref{Holder w'})
and (\ref{Holder Dw'}), we can apply Schauder $C^{2,\,\gamma}$ estimates
to (\ref{eq Cauchy w'(z)}) to get 
\[
\delta\left\Vert \nabla\left(\hat{\boldsymbol{w}}-\hat{\boldsymbol{\psi}}_{k}\right)\right\Vert _{L^{\infty}\left(Q\left(\boldsymbol{z}_{*},\,\tau_{*};\,\frac{1}{3}\delta\right)\right)}\,+\,\delta^{2}\left\Vert \nabla^{2}\left(\hat{\boldsymbol{w}}-\hat{\boldsymbol{\psi}}_{k}\right)\right\Vert _{L^{\infty}\left(Q\left(\boldsymbol{z}_{*},\,\tau_{*};\,\frac{1}{3}\delta\right)\right)}
\]
\[
+\,\delta^{2+\gamma}\left[\nabla^{2}\left(\hat{\boldsymbol{w}}-\hat{\boldsymbol{\psi}}_{k}\right)\right]_{\gamma;\,Q\left(\boldsymbol{z}_{*},\,\tau_{*};\,\frac{1}{3}\delta\right)}
\]
\[
\leq C\left(n\right)\left(\left\Vert \hat{\boldsymbol{w}}-\hat{\boldsymbol{\psi}}_{k}\right\Vert _{L^{\infty}\left(Q\left(\boldsymbol{z}_{*},\,\tau_{*};\,\frac{1}{2}\delta\right)\right)}+\delta^{2}\left\Vert \boldsymbol{f}\right\Vert _{L^{\infty}\left(Q\left(\boldsymbol{z}_{*},\,\tau_{*};\,\frac{1}{2}\delta\right)\right)}+\delta^{2+\gamma}\Bigl[\boldsymbol{f}\Bigr]_{\gamma;\,Q\left(\boldsymbol{z}_{*},\,\tau_{*};\,\frac{1}{2}\delta\right)}\right)
\]
\begin{equation}
\leq C\left(n\right)\left(\beta^{\alpha-3}\left(\frac{\tau_{*}}{\tau_{0}}\right)^{-\varrho}+\tau_{*}^{-1}\right)\,\leq\,C\left(n\right)\beta^{2\left(\alpha-1\right)}\left(\frac{\tau_{*}}{\tau_{0}}\right)^{-\varrho}\label{Holder DDw'}
\end{equation}
provided that $\tau_{0}\gg1$ (depending on $n$, $\beta$).

The conclusion follows by using the bootstrap argument on (\ref{eq Cauchy w'(z)})
and repeatedly differentiating equations with respect to $\tau$. 
\end{proof}

\section{\label{degree Lambda}Determining the constant $\Lambda$ }

In this section, we would finish the proof of Proposition \ref{degree}
and Proposition \ref{uniform estimates}. What's left is to show (\ref{convexity})
and choose $\Lambda=\Lambda\left(n\right)\gg1$ so that (\ref{smaller a priori bound u})
holds. To this end, it suffices to show that 
\begin{enumerate}
\item In the $\mathbf{outer}$ $\mathbf{region}$, the function $u\left(x,\,t\right)$
defined in (\ref{u}) satisfies 
\begin{equation}
x^{i}\left|\partial_{x}^{i}u\left(x,\,t\right)\right|\leq C\left(n\right)x^{2\lambda_{2}+1}\qquad\forall\;\,i\in\left\{ 0,\,1,\,2\right\} \label{Lambda outer}
\end{equation}
\begin{equation}
\partial_{xx}^{2}u\left(x,\,t\right)\geq0\label{convexity outer}
\end{equation}
for $\sqrt{-t}\leq x\leq\rho$, $t_{0}\leq t\leq\mathring{t}$;
\item In the $\mathbf{intermediate}$ $\mathbf{region}$, if we perform
the type $\mathrm{I}$ rescaling, the type $\mathrm{I}$ rescaled
function $v\left(y,\,s\right)$ defined in (\ref{v}) satisfies 
\begin{equation}
y^{i}\left|\partial_{y}^{i}v\left(y,\,s\right)\right|\leq C\left(n\right)e^{-\lambda_{2}s}y^{\alpha}\qquad\forall\;\,i\in\left\{ 0,\,1,\,2\right\} \label{Lambda intermediate}
\end{equation}
\begin{equation}
\partial_{yy}^{2}v\left(y,\,s\right)\geq0\label{convexity intermediate}
\end{equation}
for $2\beta e^{-\sigma s}\leq y\leq1$, $s_{0}<s\leq\mathring{s}$;
\item Near the $\mathbf{tip}$ $\mathbf{region}$, if we perform the type
$\mathrm{II}$ rescaling, the type $\mathrm{II}$ rescaled function
$w\left(z,\,\tau\right)$ defined in (\ref{w}) satisfies 
\begin{equation}
z^{i}\left|\partial_{z}^{i}w\left(z,\,\tau\right)\right|\leq C\left(n\right)z^{\alpha}\qquad\forall\;\,i\in\left\{ 0,\,1,\,2\right\} \label{Lambda tip}
\end{equation}
for $\beta\leq z\leq2\beta$, $\tau_{0}\leq\tau\leq\mathring{\tau}$.
In addition, the type $\mathrm{II}$ rescaled function $\hat{w}\left(z,\,\tau\right)$
defined in (\ref{w'}) satisfies
\begin{equation}
\partial_{zz}^{2}\hat{w}\left(z,\,\tau\right)\geq0\label{convexity tip}
\end{equation}
for $0\leq z\leq5\beta$, $\tau_{0}\leq\tau\leq\mathring{\tau}$.
\end{enumerate}
Note that (\ref{Lambda intermediate}) is equivalent to 
\[
x^{i}\left|\partial_{x}^{i}u\left(x,\,t\right)\right|\leq C\left(n\right)\left(-t\right)^{2}x^{\alpha}\qquad\forall\;\,i\in\left\{ 0,\,1,\,2\right\} 
\]
for $2\beta\left(-t\right)^{\frac{1}{2}+\sigma}\leq x\leq\sqrt{-t}$,
$t_{0}\leq t\leq\mathring{t}$ (see (\ref{uv}) and (\ref{eigenvalues})).
Also, (\ref{Lambda tip}) is equivalent to
\[
x^{i}\left|\partial_{x}^{i}u\left(x,\,t\right)\right|\leq C\left(n\right)\left(-t\right)^{2}x^{\alpha}\qquad\forall\;\,i\in\left\{ 0,\,1,\,2\right\} 
\]
for $\beta\left(-t\right)^{\frac{1}{2}+\sigma}\leq x\leq2\beta\left(-t\right)^{\frac{1}{2}+\sigma}$,
$t_{0}\leq t\leq\mathring{t}$ (see (\ref{sigma}) and (\ref{uw})).
Moreover, by (\ref{convexity outer}), (\ref{convexity intermediate}),
(\ref{convexity tip}) and rescaling, we can show (\ref{convexity}),
i.e. the projected curve $\bar{\Sigma}_{t}$ is convex in $B\left(O;\,\rho\right)$
for $t_{0}\leq t\leq\mathring{t}$.

Recall that in Remark \ref{Lambda C^0 }, we already show the $C^{0}$
estimates in (\ref{Lambda outer}) and (\ref{Lambda intermediate}).
As for the derivatives, notice that the smooth estimates in Proposition
\ref{degree} does not imply (\ref{Lambda outer}), (\ref{Lambda intermediate})
and (\ref{Lambda tip}), since those estimates doest not extend to
the initial time. Therefore, in this section we compensate that by
showing how to estimate the quantities in (\ref{Lambda outer}), (\ref{Lambda intermediate})
and (\ref{Lambda tip}) from the initial time to some extent. The
idea is to derive evolution equations for these quantities and use
the following lemma (see Lemma \ref{short-time continuity}), together
with (\ref{initial v}), (\ref{initial u intermediate}) and (\ref{initial w intermediate}),
to show that they can be bounded in terms of $n$ for a short period
of time. Below is the lemma which we would use to prove the derivatives
estimates in (\ref{Lambda outer}) and (\ref{Lambda intermediate}).
\begin{lem}
\label{short-time continuity}Suppose that $h\left(r,\,\iota\right)$
is a function which satisfies 
\[
\partial_{\iota}h\,-\,a\left(r,\,\iota\right)\partial_{rr}^{2}h\,-\,b\left(r,\,\iota\right)\partial_{r}h\,=\,f\left(r,\,\iota\right)
\]
for $\frac{1}{2}\leq r\leq\frac{3}{2}$, $0\leq\iota\leq\mathcal{T}$,
with 
\[
a\left(r,\,\iota\right)>0
\]
\[
\max\Bigl\{\left|a\left(r,\,\iota\right)\right|,\,\left|b\left(r,\,\iota\right)\right|\Bigr\}\,\leq\,M
\]
for $\frac{1}{2}\leq r\leq\frac{3}{2}$, $0\leq\iota\leq\mathcal{T}$,
where $\mathcal{T},\,M>0$ are constants. Then there hold
\[
h\left(r,\,\iota\right)\,\leq\,\max_{\frac{1}{2}\leq\mathrm{r}\leq\frac{3}{2}}h\left(\mathrm{r},\,0\right)\,+\,C\left(M\right)\iota\left(\left\Vert h\right\Vert _{L^{\infty}\left(\left[\frac{1}{2},\,\frac{3}{2}\right]\times\left[0,\,\mathcal{T}\right]\right)}+\left\Vert f\right\Vert _{L^{\infty}\left(\left[\frac{1}{2},\,\frac{3}{2}\right]\times\left[0,\,\mathcal{T}\right]\right)}\right)
\]
 
\[
h\left(r,\,\iota\right)\,\geq\,\min_{\frac{1}{2}\leq\mathrm{r}\leq\frac{3}{2}}h\left(\mathrm{r},\,0\right)\,-\,C\left(M\right)\iota\left(\left\Vert h\right\Vert _{L^{\infty}\left(\left[\frac{1}{2},\,\frac{3}{2}\right]\times\left[0,\,\mathcal{T}\right]\right)}+\left\Vert f\right\Vert _{L^{\infty}\left(\left[\frac{1}{2},\,\frac{3}{2}\right]\times\left[0,\,\mathcal{T}\right]\right)}\right)
\]
for $\frac{3}{4}\leq r\leq\frac{5}{4}$, $0\leq\iota\leq\mathcal{T}$.
\end{lem}
\begin{proof}
Let $\eta\left(r\right)$ be a smooth function so that 
\[
\chi_{\left[\frac{3}{4},\,\frac{5}{4}\right]}\leq\eta\leq\chi_{\left[\frac{1}{2},\,\frac{3}{2}\right]}
\]
and $\eta\left(r\right)$ vanishes at $\frac{1}{2}$ and $\frac{3}{2}$
to infinite order. Note that by Lemma \ref{localization}, we may
assume that
\[
\frac{\left(\partial_{r}\eta\left(r\right)\right)^{2}}{\eta\left(r\right)}\,+\,\left|\partial_{r}\eta\left(r\right)\right|\,+\,\left|\partial_{rr}^{2}\eta\left(r\right)\right|\,\lesssim\,1
\]
It follows that
\begin{equation}
\partial_{\iota}\left(\eta h\right)-a\left(r,\,\iota\right)\partial_{rr}^{2}\left(\eta h\right)-b\left(r,\,\iota\right)\partial_{r}\left(\eta h\right)\label{eq short-time continuity}
\end{equation}
\[
=\eta\,f\left(r,\,\iota\right)\,-\,\left(a\left(r,\,\iota\right)\partial_{rr}^{2}\eta\,+\,b\left(r,\,\iota\right)\partial_{r}\eta\right)h-2\,a\left(r,\,\iota\right)\,\partial_{r}\eta\,\partial_{r}h
\]
For the last term on RHS of (\ref{eq short-time continuity}), if
we evaluate it at any maximum point of $\eta\left(r\right)h\left(r,\,\iota\right)$
for each time-slice, either $\eta=0$ and hence
\begin{equation}
\partial_{r}\eta=0\,\Rightarrow\,-2\,a\left(r,\,\iota\right)\,\partial_{r}\eta\,\partial_{r}h=0\label{short-time continuity1}
\end{equation}
or $0<\eta\leq1$, in which case we have 
\[
\partial_{r}\left(\eta h\right)=0\,\Leftrightarrow\,\eta\,\partial_{r}h+h\,\partial_{r}\eta=0
\]
which yields 
\begin{equation}
-2\,a\left(r,\,\iota\right)\,\partial_{r}\eta\,\partial_{r}h=2\,a\left(r,\,\iota\right)\,\frac{\left(\partial_{r}\eta\right)^{2}}{\eta}h\label{short-time continuity2}
\end{equation}
Now let 
\[
\left(\eta h\right)_{\max}\left(\iota\right)=\max_{r}\,\left(\eta\left(r\right)h\left(r,\,\iota\right)\right)
\]
By (\ref{short-time continuity1}) and (\ref{short-time continuity2}),
if we apply the maximum principle to (\ref{eq short-time continuity}),
we get 
\[
\partial_{\iota}\left(\eta h\right)_{\max}\leq C\left(M\right)\left(\left\Vert h\right\Vert _{L^{\infty}\left(\left[\frac{1}{2},\,\frac{3}{2}\right]\times\left[0,\,\mathcal{T}\right]\right)}\,+\,\left\Vert f\right\Vert _{L^{\infty}\left(\left[\frac{1}{2},\,\frac{3}{2}\right]\times\left[0,\,\mathcal{T}\right]\right)}\right)
\]
which implies 
\[
\left(\eta h\right)_{\max}\left(\iota\right)\,\leq\,\left(\eta h\right)_{\max}\left(0\right)\,+\,C\left(M\right)\iota\left(\left\Vert h\right\Vert _{L^{\infty}\left(\left[\frac{1}{2},\,\frac{3}{2}\right]\times\left[0,\,\mathcal{T}\right]\right)}\,+\,\left\Vert f\right\Vert _{L^{\infty}\left(\left[\frac{1}{2},\,\frac{3}{2}\right]\times\left[0,\,\mathcal{T}\right]\right)}\right)
\]
Similarly, if we define 
\[
\left(\eta h\right)_{\min}\left(\iota\right)=\min_{r}\,\left(\eta\left(r\right)h\left(r,\,\iota\right)\right)
\]
then we have 
\[
\left(\eta h\right)_{\min}\left(\iota\right)\,\geq\,\left(\eta h\right)_{\min}\left(0\right)\,-\,C\left(M\right)\iota\left(\left\Vert h\right\Vert _{L^{\infty}\left(\left[\frac{1}{2},\,\frac{3}{2}\right]\times\left[0,\,\mathcal{T}\right]\right)}\,+\,\left\Vert f\right\Vert _{L^{\infty}\left(\left[\frac{1}{2},\,\frac{3}{2}\right]\times\left[0,\,\mathcal{T}\right]\right)}\right)
\]
\end{proof}
To prove the derivatives estimates in (\ref{Lambda outer}), we divide
the region into two parts: $\frac{3}{4}\rho\leq x\leq\rho$ and $\sqrt{-t}\leq x\leq\frac{3}{4}\rho$.
In the following proposition, we show (\ref{Lambda outer}) for $\frac{3}{4}\rho\leq x\leq\rho$
by using (\ref{eq u}), (\ref{initial u intermediate}), (\ref{C^2 outside u bound})
and Lemma \ref{short-time continuity}.
\begin{prop}
If $\left|t_{0}\right|\ll1$ (depending on $n$, $\Lambda$, $\rho$,
$\beta$), then there hold
\begin{equation}
\frac{1}{2}\varUpsilon_{2}\left(2\lambda_{2}+1\right)x^{2\lambda_{2}}\,\leq\,\partial_{x}u\left(x,\,t\right)\,\leq\,\frac{3}{2}\varUpsilon_{2}\left(2\lambda_{2}+1\right)x^{2\lambda_{2}}\label{Lambda outer Du}
\end{equation}
 
\begin{equation}
\partial_{xx}^{2}u\left(x,\,t\right)\,\leq\,\frac{3}{2}\varUpsilon_{2}\left(2\lambda_{2}+1\right)\left(2\lambda_{2}\right)x^{2\lambda_{2}-1}\label{Lambda outer DDu}
\end{equation}
 
\begin{equation}
\partial_{xx}^{2}u\left(x,\,t\right)\,\geq\,\frac{1}{2}\varUpsilon_{2}\left(2\lambda_{2}+1\right)\left(2\lambda_{2}\right)x^{2\lambda_{2}-1}\,>0\label{convexity outer u}
\end{equation}
 for $\frac{3}{4}\rho\leq x\leq\frac{5}{4}\rho$, $t_{0}\leq t\leq\mathring{t}$.
\end{prop}
\begin{proof}
Let 
\[
h\left(r,\,\iota\right)=\left.x^{-2\lambda_{2}}\partial_{x}u\left(x,\,t\right)\right|_{x=r\rho,\,t=t_{0}+\iota\rho^{2}}
\]
From (\ref{eq u}), we derive 
\[
\partial_{\iota}h-a\left(r,\,\iota\right)\partial_{rr}^{2}h-b\left(r,\,\iota\right)\partial_{r}h=f\left(r,\,\iota\right)
\]
where 
\[
a\left(r,\,\iota\right)=\left.\frac{1}{1+\left(\partial_{x}u\left(x,\,t\right)\right)^{2}}\right|_{x=r\rho,\,t=t_{0}+\iota\rho^{2}}
\]
 
\[
b\left(r,\,\iota\right)=\left.\frac{1}{r}\left(\frac{-2x\left(\partial_{x}u\left(x,\,t\right)\right)\left(\partial_{xx}^{2}u\left(x,\,t\right)\right)}{\left(1+\left(\partial_{x}u\left(x,\,t\right)\right)^{2}\right)^{2}}+\frac{2\left(n-1\right)}{1-\left(\frac{u\left(x,\,t\right)}{x}\right)^{2}}\right)\right|_{x=r\rho,\,t=t_{0}+\iota\rho^{2}}
\]
 
\[
f\left(r,\,\iota\right)=\left.\frac{\rho^{-2\lambda_{2}+1}}{r^{2\lambda_{2}+1}}\left(\left(\frac{2\lambda_{2}}{1+\left(\partial_{x}u\left(x,\,t\right)\right)^{2}}\right)\left(\partial_{xx}^{2}u\left(x,\,t\right)\right)\right)\right|_{x=r\rho,\,t=t_{0}+\iota\rho^{2}}
\]
\[
+\left.\frac{\rho^{-2\lambda_{2}}}{r^{2\lambda_{2}+2}}\left(2\lambda_{2}\left(\frac{-2x\left(\partial_{x}u\left(x,\,t\right)\right)\left(\partial_{xx}^{2}u\left(x,\,t\right)\right)}{\left(1+\left(\partial_{x}u\left(x,\,t\right)\right)^{2}\right)^{2}}+\frac{2\left(n-1\right)}{1-\left(\frac{u\left(x,\,t\right)}{x}\right)^{2}}\right)\right)\left(\partial_{x}u\left(x,\,t\right)\right)\right|_{x=r\rho,\,t=t_{0}+\iota\rho^{2}}
\]
\[
+\left.\frac{\rho^{-2\lambda_{2}}}{r^{2\lambda_{2}+2}}\left(-\frac{2\lambda_{2}\left(2\lambda_{2}+1\right)}{1+\left(\partial_{x}u\left(x,\,t\right)\right)^{2}}\right)\left(\partial_{x}u\left(x,\,t\right)\right)\right|_{x=r\rho,\,t=t_{0}+\iota\rho^{2}}
\]
\[
+\left.\frac{\rho^{-2\lambda_{2}-1}}{r^{2\lambda_{2}+3}}\left(\left(\frac{4\left(n-1\right)\left(\left(\partial_{x}u\left(x,\,t\right)\right)^{2}-1\right)}{\left(1-\left(\frac{u\left(x,\,t\right)}{x}\right)^{2}\right)^{2}}\right)\left(u\left(x,\,t\right)\right)\right)\right|_{x=r\rho,\,t=t_{0}+\iota\rho^{2}}
\]
It follows, by (\ref{C^2 outside u bound}) and Lemma \ref{short-time continuity},
that
\[
\min_{\frac{1}{2}\leq\mathrm{r}\leq\frac{3}{2}}h\left(\mathrm{r},\,0\right)\,-\,C\left(n,\,\rho\right)\iota\,\leq\,h\left(r,\,\iota\right)\,\leq\,\max_{\frac{1}{2}\leq\mathrm{r}\leq\frac{3}{2}}h\left(\mathrm{r},\,0\right)\,+\,C\left(n,\,\rho\right)\iota
\]
for $\frac{3}{4}\leq r\leq\frac{5}{4}$. Undoing the change of variables,
we get 
\[
x_{*}^{-2\lambda_{2}}\partial_{x}u\left(x_{*},\,t\right)\,\leq\,\max_{\frac{1}{2}\rho\leq x\leq\frac{3}{2}\rho}\left(x^{-2\lambda_{2}}\partial_{x}u\left(x,\,t_{0}\right)\right)\,+\,C\left(n,\,\rho\right)\frac{t-t_{0}}{\rho^{2}}
\]
 
\[
x_{*}^{-2\lambda_{2}}\partial_{x}u\left(x_{*},\,t\right)\,\geq\,\min_{\frac{1}{2}\rho\leq x\leq\frac{3}{2}\rho}\left(x^{-2\lambda_{2}}\partial_{x}u\left(x,\,t_{0}\right)\right)\,-\,C\left(n,\,\rho\right)\frac{t-t_{0}}{\rho^{2}}
\]
for $\frac{3}{4}\rho\leq x_{*}\leq\frac{5}{4}\rho$, $t_{0}\leq t\leq\mathring{t}$.
Therefore, if $\left|t_{0}\right|\ll1$ (depending on $n$, $\Lambda$,
$\rho$, $\beta$), then (\ref{Lambda outer Du}) follows immediately
from the above, (\ref{initial u intermediate}) and (\ref{coefficients}).

For the second derivative, note that we have the following evolution
equation:
\[
\partial_{t}\left(x^{-2\lambda_{2}+1}\partial_{xx}^{2}u\right)-\frac{1}{1+\left(\partial_{x}u\right)^{2}}\,\partial_{xx}^{2}\left(x^{-2\lambda_{2}+1}\partial_{xx}^{2}u\right)
\]
\[
-\frac{1}{x}\left(\frac{-6x\left(\partial_{x}u\right)\left(\partial_{xx}^{2}u\right)}{\left(1+\left(\partial_{x}u\right)^{2}\right)^{2}}+\frac{2\left(n-1\right)}{1-\left(\frac{u}{x}\right)^{2}}+\frac{2\left(2\lambda_{2}-1\right)}{1+\left(\partial_{x}u\right)^{2}}\right)\partial_{x}\left(x^{-2\lambda_{2}+1}\partial_{xx}^{2}u\right)
\]
\[
=\frac{1}{x^{2\lambda_{2}+1}}\left(\frac{-2x^{2}\left(\partial_{xx}^{2}u\right)^{2}\left(1-3\left(\partial_{x}u\right)^{2}\right)}{\left(1+\left(\partial_{x}u\right)^{2}\right)^{3}}+\frac{12\left(n-1\right)\left(\frac{u}{x}\right)\partial_{x}u}{\left(1-\left(\frac{u}{x}\right)^{2}\right)^{2}}-\frac{2\left(n-1\right)\left(1+\left(\frac{u}{x}\right)^{2}\right)}{\left(1-\left(\frac{u}{x}\right)^{2}\right)^{2}}\right)\left(\partial_{xx}^{2}u\right)
\]
\[
+\frac{2\lambda_{2}-1}{x^{2\lambda_{2}+1}}\left(\left(\frac{-6x\left(\partial_{x}u\right)\left(\partial_{xx}^{2}u\right)}{\left(1+\left(\partial_{x}u\right)^{2}\right)^{2}}+\frac{2\left(n-1\right)}{1-\left(\frac{u}{x}\right)^{2}}\right)+\frac{2\lambda_{2}-2}{1+\left(\partial_{x}u\right)^{2}}\right)\left(\partial_{xx}^{2}u\right)
\]
\[
+\frac{1}{x^{2\lambda_{2}+2}}\left(\frac{4\left(n-1\right)\left(\left(\partial_{x}u\right)^{2}-1\right)\left(1+3\left(\frac{u}{x}\right)^{2}\right)}{\left(1-\left(\frac{u}{x}\right)^{2}\right)^{3}}\right)\left(\partial_{x}u\right)
\]
\[
+\frac{1}{x^{2\lambda_{2}+3}}\left(\frac{4\left(n-1\right)\left(1-\left(\partial_{x}u\right)^{2}\right)\left(\left(\frac{u}{x}\right)^{2}+3\right)}{\left(1-\left(\frac{u}{x}\right)^{2}\right)^{3}}\right)\left(u\right)
\]
By the same argument (as for the first derivative), we can show (\ref{Lambda outer DDu})
and (\ref{convexity outer u}).
\end{proof}
Now we show the derivatives estimates in (\ref{Lambda outer}) for
$\sqrt{-t}\leq x\leq\frac{3}{4}\rho$ by using (\ref{eq u}), (\ref{a priori bound u}),
(\ref{initial u intermediate}), (\ref{C^infty u bound}) and Lemma
\ref{short-time continuity}.
\begin{prop}
If $0<\rho\ll1$ (depending on $n$, $\Lambda$) and $\left|t_{0}\right|\ll1$
(depending on $n$, $\Lambda$, $\rho$, $\beta$), then there hold
\begin{equation}
2\left(\alpha+2\varUpsilon_{1}\left(\alpha+2\right)+\varUpsilon_{2}\left(2\lambda_{2}+1\right)\right)\,x^{2\lambda_{2}}\,\,\leq\,\,\partial_{x}u\left(x,\,t\right)\,\,\leq\,\,2\varUpsilon_{2}\left(2\lambda_{2}+1\right)\,x^{2\lambda_{2}}\label{Lambda Du}
\end{equation}
 
\begin{equation}
\partial_{xx}^{2}u\left(x,\,t\right)\,\,\leq\,\,2\left(\alpha\left(\alpha-1\right)+2\varUpsilon_{1}\left(\alpha+2\right)\left(\alpha+1\right)+\varUpsilon_{2}\left(2\lambda_{2}+1\right)\left(2\lambda_{2}\right)\right)\,x^{2\lambda_{2}-1}\label{Lambda DDu}
\end{equation}
 
\begin{equation}
\partial_{xx}^{2}u\left(x,\,t\right)\,\geq\,\frac{1}{2}\varUpsilon_{2}\left(2\lambda_{2}+1\right)\left(2\lambda_{2}\right)x^{2\lambda_{2}-1}\,>0\label{convexity u}
\end{equation}
for $\sqrt{-t}\leq x\leq\frac{3}{4}\rho$, $t_{0}\leq t\leq\mathring{t}$.
\end{prop}
\begin{proof}
First, fix $x_{*}\in\left[\frac{2}{3}\sqrt{-t_{0}},\,\frac{3}{4}\rho\right]$
and let 
\[
h\left(r,\,\iota\right)=\left.x^{-2\lambda_{2}}\partial_{x}u\left(x,\,t\right)\right|_{x=rx_{*},\,t=t_{0}+\iota x_{*}^{2}}
\]
From (\ref{eq u}), we derive 
\[
\partial_{\iota}h-a\left(r,\,\iota\right)\partial_{rr}^{2}h-b\left(r,\,\iota\right)\partial_{r}h=f\left(r,\,\iota\right)
\]
where 
\[
a\left(r,\,\iota\right)=\left.\frac{1}{1+\left(\partial_{x}u\left(x,\,t\right)\right)^{2}}\right|_{x=rx_{*},\,t=t_{0}+\iota x_{*}^{2}}
\]
 
\[
b\left(r,\,\iota\right)=\left.\frac{1}{r}\left(\frac{-2\,\partial_{x}u\left(x,\,t\right)\left(x\,\partial_{xx}^{2}u\left(x,\,t\right)\right)}{\left(1+\left(\partial_{x}u\left(x,\,t\right)\right)^{2}\right)^{2}}+\frac{2\left(n-1\right)}{1-\left(\frac{u\left(x,\,t\right)}{x}\right)^{2}}\right)\right|_{x=rx_{*},\,t=t_{0}+\iota x_{*}^{2}}
\]
 
\[
f\left(r,\,\iota\right)=\left.\frac{1}{r^{2}}\left(\left(\frac{2\lambda_{2}}{1+\left(\partial_{x}u\left(x,\,t\right)\right)^{2}}\right)\left(x^{-2\lambda_{2}+1}\partial_{xx}^{2}u\left(x,\,t\right)\right)\right)\right|_{x=rx_{*},\,t=t_{0}+\iota x_{*}^{2}}
\]
\[
+\left.\frac{1}{r^{2}}\left(2\lambda_{2}\left(\frac{-2\,\partial_{x}u\left(x,\,t\right)\left(x\,\partial_{xx}^{2}u\left(x,\,t\right)\right)}{\left(1+\left(\partial_{x}u\left(x,\,t\right)\right)^{2}\right)^{2}}+\frac{2\left(n-1\right)}{1-\left(\frac{u\left(x,\,t\right)}{x}\right)^{2}}\right)\right)\left(x^{-2\lambda_{2}}\partial_{x}u\left(x,\,t\right)\right)\right|_{x=rx_{*},\,t=t_{0}+\iota x_{*}^{2}}
\]
\[
+\left.\frac{1}{r^{2}}\left(-\frac{2\lambda_{2}\left(2\lambda_{2}+1\right)}{1+\left(\partial_{x}u\left(x,\,t\right)\right)^{2}}\right)\left(x^{-2\lambda_{2}}\partial_{x}u\left(x,\,t\right)\right)\right|_{x=rx_{*},\,t=t_{0}+\iota x_{*}^{2}}
\]
\[
+\left.\frac{1}{r^{2}}\left(\left(\frac{4\left(n-1\right)\left(\left(\partial_{x}u\left(x,\,t\right)\right)^{2}-1\right)}{\left(1-\left(\frac{u\left(x,\,t\right)}{x}\right)^{2}\right)^{2}}\right)\left(x^{-2\lambda_{2}-1}u\left(x,\,t\right)\right)\right)\right|_{x=rx_{*},\,t=t_{0}+\iota x_{*}^{2}}
\]
Notice that by (\ref{a priori bound u}) we have
\[
\max\left\{ \left|\frac{u\left(x,\,t\right)}{x}\right|,\,\left|\partial_{x}u\left(x,\,t\right)\right|,\,\left|x\,\partial_{xx}^{2}u\left(x,\,t\right)\right|\right\} \,\leq\,C\left(n,\,\Lambda\right)x^{2\lambda_{2}}\,\leq\,\frac{1}{3}
\]
 
\[
x^{-2\lambda_{2}-1+i}\left|\partial_{x}^{i}u\left(x,\,t\right)\right|\,\leq\,C\left(n,\,\Lambda\right),\qquad i\in\left\{ 0,\,1,\,2\right\} 
\]
for $\frac{1}{2}\sqrt{-t}\leq x\leq\rho$, $t_{0}\leq t\leq\mathring{t}$,
provided that $0<\rho\ll1$ (depending on $n$, $\Lambda$) . It follows,
by Lemma \ref{short-time continuity}, that 
\[
\min_{\frac{1}{2}\leq\mathrm{r}\leq\frac{3}{2}}h\left(\mathrm{r},\,0\right)\,-\,C\left(n,\,\Lambda\right)\iota\,\leq\,h\left(r,\,\iota\right)\,\leq\,\max_{\frac{1}{2}\leq\mathrm{r}\leq\frac{3}{2}}h\left(\mathrm{r},\,0\right)\,+\,C\left(n,\,\Lambda\right)\iota
\]
which implies
\[
x_{*}^{-2\lambda_{2}}\partial_{x}u\left(x_{*},\,t\right)\,\leq\,\max_{\frac{1}{2}\sqrt{-t_{0}}\leq x\leq\rho}\left(x^{-2\lambda_{2}}\partial_{x}u\left(x,\,t_{0}\right)\right)\,+\,C\left(n,\,\Lambda\right)\frac{t-t_{0}}{\rho^{2}}
\]
 
\[
x_{*}^{-2\lambda_{2}}\partial_{x}u\left(x_{*},\,t\right)\,\geq\,\min_{\frac{1}{2}\sqrt{-t_{0}}\leq x\leq\rho}\left(x^{-2\lambda_{2}}\partial_{x}u\left(x,\,t_{0}\right)\right)\,-\,C\left(n,\,\Lambda\right)\frac{t-t_{0}}{\rho^{2}}
\]
for $t_{0}\leq t\leq t_{0}+\delta^{2}x_{*}^{2}$. Thus, by (\ref{initial u intermediate})
and (\ref{coefficients}), we can choose $0<\delta\ll1$ (depending
on $n$, $\Lambda$) so that
\begin{equation}
2\left(\alpha+2\varUpsilon_{1}\left(\alpha+2\right)+\varUpsilon_{2}\left(2\lambda_{2}+1\right)\right)x^{2\lambda_{2}}\,\leq\,\partial_{x}u\left(x,\,t\right)\,\leq\,2\varUpsilon_{2}\left(2\lambda_{2}+1\right)x^{2\lambda_{2}}\label{Lambda outer1}
\end{equation}
for $\left(x,\,t\right)$ satisfying $\sqrt{-t}\leq x\leq\frac{3}{4}\rho$,
$t_{0}\leq t\leq t_{0}+\delta^{2}x^{2}$, provided that $\left|t_{0}\right|\ll1$
(depending on $n$, $\Lambda$, $\rho$, $\beta$).

On the other hand, by this choice of $\delta=\delta\left(n,\,\Lambda\right)$,
(\ref{C^infty u bound}) implies
\[
\left|\partial_{x}\left(u\left(x,\,t\right)-\frac{k}{c_{2}}\left(-t\right)^{\lambda_{2}+\frac{1}{2}}\varphi_{2}\left(\frac{x}{\sqrt{-t}}\right)\right)\right|\,\leq\,C\left(n,\,\Lambda\right)\rho^{4\lambda_{2}}x^{2\lambda_{2}}
\]
for $\left(x,\,t\right)$ satisfying $\sqrt{-t}\leq x\leq\frac{3}{4}\rho$,
$t_{0}+\delta^{2}x^{2}\leq t\leq\mathring{t}$, where 
\[
\partial_{x}\left(\frac{k}{c_{2}}\left(-t\right)^{\lambda_{2}+\frac{1}{2}}\varphi_{2}\left(\frac{x}{\sqrt{-t}}\right)\right)\,=\,kx^{2\lambda_{2}}\left(\varUpsilon_{2}\left(2\lambda_{2}+1\right)+2\varUpsilon_{1}\left(\alpha+2\right)\left(\frac{-t}{x^{2}}\right)+\alpha\left(\frac{-t}{x^{2}}\right)^{2}\right)
\]
It follows, by (\ref{k}), that 
\begin{equation}
2\left(\alpha+2\varUpsilon_{1}\left(\alpha+2\right)+\varUpsilon_{2}\left(2\lambda_{2}+1\right)\right)x^{2\lambda_{2}}\,\leq\,\partial_{x}u\left(x,\,t\right)\,\leq\,2\varUpsilon_{2}\left(2\lambda_{2}+1\right)x^{2\lambda_{2}}\label{Lambda outer2}
\end{equation}
for $\left(x,\,t\right)$ satisfying $\sqrt{-t}\leq x\leq\frac{3}{4}\rho$,
$t_{0}\leq t\leq t_{0}+\delta^{2}x^{2}$, provided that $\left|t_{0}\right|\ll1$
(depending on $n$, $\Lambda$, $\rho$, $\beta$). Then (\ref{Lambda Du})
follows immediately from (\ref{Lambda outer1}) and (\ref{Lambda outer2}).

As for the second derivatives, we have the evolution equation: 
\[
\partial_{t}\left(x^{-2\lambda_{2}+1}\partial_{xx}^{2}u\right)-\frac{1}{1+\left(\partial_{x}u\right)^{2}}\,\partial_{xx}^{2}\left(x^{-2\lambda_{2}+1}\partial_{xx}^{2}u\right)
\]
\[
-\frac{1}{x}\left(\frac{-6\,\partial_{x}u\left(x\,\partial_{xx}^{2}u\right)}{\left(1+\left(\partial_{x}u\right)^{2}\right)^{2}}+\frac{2\left(n-1\right)}{1-\left(\frac{u}{x}\right)^{2}}+\frac{2\left(2\lambda_{2}-1\right)}{1+\left(\partial_{x}u\right)^{2}}\right)\partial_{x}\left(x^{-2\lambda_{2}+1}\partial_{xx}^{2}u\right)
\]
\[
=\frac{1}{x^{2}}\left(\frac{-2\left(x\,\partial_{xx}^{2}u\right)^{2}\left(1-3\left(\partial_{x}u\right)^{2}\right)}{\left(1+\left(\partial_{x}u\right)^{2}\right)^{3}}+\frac{12\left(n-1\right)\left(\frac{u}{x}\right)\partial_{x}u}{\left(1-\left(\frac{u}{x}\right)^{2}\right)^{2}}-\frac{2\left(n-1\right)\left(1+\left(\frac{u}{x}\right)^{2}\right)}{\left(1-\left(\frac{u}{x}\right)^{2}\right)^{2}}\right)\left(x^{-2\lambda_{2}+1}\partial_{xx}^{2}u\right)
\]
\[
+\frac{1}{x^{2}}\left(\left(2\lambda_{2}-1\right)\left(\frac{-6\,\partial_{x}u\left(x\,\partial_{xx}^{2}u\right)}{\left(1+\left(\partial_{x}u\right)^{2}\right)^{2}}+\frac{2\left(n-1\right)}{1-\left(\frac{u}{x}\right)^{2}}\right)+\frac{\left(2\lambda_{2}-1\right)\left(2\lambda_{2}-2\right)}{1+\left(\partial_{x}u\right)^{2}}\right)\left(x^{-2\lambda_{2}+1}\partial_{xx}^{2}u\right)
\]
\[
+\frac{1}{x^{2}}\left(\frac{4\left(n-1\right)\left(\left(\partial_{x}u\right)^{2}-1\right)\left(1+3\left(\frac{u}{x}\right)^{2}\right)}{\left(1-\left(\frac{u}{x}\right)^{2}\right)^{3}}\right)\left(x^{-2\lambda_{2}}\partial_{x}u\right)
\]
\[
+\frac{1}{x^{2}}\left(\frac{4\left(n-1\right)\left(1-\left(\partial_{x}u\right)^{2}\right)\left(\left(\frac{u}{x}\right)^{2}+3\right)}{\left(1-\left(\frac{u}{x}\right)^{2}\right)^{3}}\right)\left(x^{-2\lambda_{2}-1}u\right)
\]
By a similar argument, we can deduce (\ref{Lambda DDu}) and (\ref{convexity u}).
\end{proof}
In the following proposition, we prove (\ref{Lambda intermediate})
by using (\ref{eq v}), (\ref{a priori bound v}), (\ref{initial v}),
(\ref{C^infty v bound intermediate}), (\ref{C^infty v bound tip})
and Lemma \ref{short-time continuity}.
\begin{prop}
If $\beta\gg1$ (depending on $n$, $\Lambda$) and $s_{0}\gg1$ (depending
on $n$, $\Lambda$, $\rho$, $\beta$), then there hold 
\begin{equation}
2\left(\alpha+8\varUpsilon_{1}\left(\alpha+2\right)+16\varUpsilon_{2}\left(2\lambda_{2}+1\right)\right)\,e^{-\lambda_{2}s}y^{\alpha-1}\,\,\leq\,\,\partial_{y}v\left(y,\,s\right)\,\,\leq\,\,\frac{1}{2}\alpha e^{-\lambda_{2}s}y^{\alpha-1}\label{Lambda Dv}
\end{equation}
 
\begin{equation}
\partial_{yy}^{2}v\left(y,\,s\right)\,\,\leq\,\,2\left(\alpha\left(\alpha-1\right)+8\varUpsilon_{1}\left(\alpha+2\right)\left(\alpha+1\right)+16\varUpsilon_{2}\left(2\lambda_{2}+1\right)\left(2\lambda_{2}\right)\right)\,e^{-\lambda_{2}s}y^{\alpha-2}\label{Lambda DDv}
\end{equation}
 
\begin{equation}
\partial_{yy}^{2}v\left(y,\,s\right)\,\,\geq\,\,\frac{1}{2}\left(\alpha\left(\alpha-1\right)\right)e^{-\lambda_{2}s}y^{\alpha-2}\,>0\label{convexity v}
\end{equation}
for $2\beta e^{-\sigma s}\leq y\leq1$, $s_{0}<s\leq\mathring{s}$.
\end{prop}
\begin{proof}
Firstly, for each $y_{*}\in\left[\frac{5}{3}\beta e^{-\sigma s_{0}},\,1\right]$,
let 
\[
h\left(r,\,\iota\right)=\left.e^{\lambda_{2}s}y^{-\alpha+1}\partial_{y}v\left(y,\,s\right)\right|_{y=ry_{*},\,s=s_{0}+\iota y_{*}^{2}}
\]
From (\ref{eq v}), we derive 
\[
\partial_{\iota}h-a\left(r,\,\iota\right)\partial_{rr}^{2}h-b\left(r,\,\iota\right)\partial_{r}h=f\left(r,\,\iota\right)
\]
where 
\[
a\left(r,\,\iota\right)=\left.\frac{1}{1+\left(\partial_{y}v\left(y,\,s\right)\right)^{2}}\right|_{y=ry_{*},\,s=s_{0}+\iota y_{*}^{2}}
\]
 
\[
b\left(r,\,\iota\right)=\left.\frac{1}{r}\left(\frac{-2\left(\partial_{y}v\left(y,\,s\right)\right)\left(y\,\partial_{yy}^{2}v\left(y,\,s\right)\right)}{\left(1+\left(\partial_{y}v\right)^{2}\right)^{2}}+\frac{2\left(n-1\right)}{1-\left(\frac{v\left(y,\,s\right)}{y}\right)^{2}}-\frac{y^{2}}{2}\right)\right|_{y=ry_{*},\,s=s_{0}+\iota y_{*}^{2}}
\]
 
\[
f\left(r,\,\iota\right)=\left.\frac{1}{r^{2}}\left(\frac{2\left(\alpha-1\right)}{1+\left(\partial_{y}v\left(y,\,s\right)\right)^{2}}\right)\left(e^{\lambda_{2}s}y^{-\alpha+2}\partial_{yy}^{2}v\left(y,\,s\right)\right)\right|_{y=ry_{*},\,s=s_{0}+\iota y_{*}^{2}}
\]
\[
+\left.\frac{\alpha-1}{r^{2}}\left(\frac{-2\left(\partial_{y}v\left(y,\,s\right)\right)\left(y\,\partial_{yy}^{2}v\left(y,\,s\right)\right)}{\left(1+\left(\partial_{y}v\left(y,\,s\right)\right)^{2}\right)^{2}}+\frac{2\left(n-1\right)}{1-\left(\frac{v\left(y,\,s\right)}{y}\right)^{2}}\right)\left(e^{\lambda_{2}s}y^{-\alpha+1}\partial_{y}v\left(y,\,s\right)\right)\right|_{y=ry_{*},\,s=s_{0}+\iota y_{*}^{2}}
\]
\[
+\left.\frac{1}{r^{2}}\left(\frac{-\alpha\left(\alpha-1\right)}{1+\left(\partial_{y}v\left(y,\,s\right)\right)^{2}}-\frac{\alpha-1}{2}y^{2}+\lambda_{2}y^{2}\right)\left(e^{\lambda_{2}s}y^{-\alpha+1}\partial_{y}v\left(y,\,s\right)\right)\right|_{y=ry_{*},\,s=s_{0}+\iota y_{*}^{2}}
\]
\[
+\left.\frac{1}{r^{2}}\left(\frac{4\left(n-1\right)\left(\left(\partial_{y}v\left(y,\,s\right)\right)^{2}-1\right)}{\left(1-\left(\frac{v\left(y,\,s\right)}{y}\right)^{2}\right)^{2}}\right)\left(e^{\lambda_{2}s}y^{-\alpha}v\left(y,\,s\right)\right)\right|_{y=ry_{*},\,s=s_{0}+\iota y_{*}^{2}}
\]
Notice that by (\ref{a priori bound v}) we have
\[
\max\left\{ \left|\frac{v\left(y,\,s\right)}{y}\right|,\,\left|\partial_{y}v\left(y,\,s\right)\right|,\,\left|y\,\partial_{yy}^{2}v\left(y,\,s\right)\right|\right\} \,\leq\,C\left(n,\,\Lambda\right)e^{-\lambda_{2}s}y^{\alpha-1}\,\leq\,\frac{1}{3}
\]
 
\[
e^{\lambda_{2}s}y^{-\alpha+i}\left|\partial_{y}^{i}v\left(y,\,s\right)\right|\,\leq\,C\left(n,\,\Lambda\right)\qquad\forall\;\,i\in\left\{ 0,\,1,\,2\right\} 
\]
for $\frac{3}{2}\beta e^{-\sigma s}\leq y\leq2$, $s_{0}\leq s\leq\mathring{s}$,
provided that $\beta\gg1$ (depending on $n$, $\Lambda$). Then by
Lemma \ref{short-time continuity} and (\ref{a priori bound v}),
we get
\[
\min_{\frac{1}{2}\leq\mathrm{r}\leq\frac{3}{2}}h\left(\mathrm{r},\,0\right)\,-\,C\left(n,\,\Lambda\right)\iota\,\leq\,h\left(r,\,\iota\right)\,\leq\,\max_{\frac{1}{2}\leq\mathrm{r}\leq\frac{3}{2}}h\left(\mathrm{r},\,0\right)\,+\,C\left(n,\,\Lambda\right)\iota
\]
which implies
\[
e^{\lambda_{2}s}y_{*}^{-\alpha+1}\partial_{y}v\left(y_{*},\,s\right)\,\leq\,\max_{\beta e^{-\sigma s}\leq y\leq2}\left(e^{\lambda_{2}s_{0}}y^{-\alpha+1}\partial_{y}v\left(y,\,s_{0}\right)\right)\,+\,C\left(n,\,\Lambda\right)\frac{s-s_{0}}{y_{*}^{2}}
\]
 
\[
e^{\lambda_{2}s}y_{*}^{-\alpha+1}\partial_{y}v\left(y_{*},\,s\right)\,\geq\,\min_{\beta e^{-\sigma s}\leq y\leq2}\left(e^{\lambda_{2}s_{0}}y^{-\alpha+1}\partial_{y}v\left(y,\,s_{0}\right)\right)\,-\,C\left(n,\,\Lambda\right)\frac{s-s_{0}}{y_{*}^{2}}
\]
for $s_{0}\leq s\leq s_{0}+\delta^{2}y_{*}^{2}$. It follows, by (\ref{initial v})
and (\ref{coefficients}), that we can choose $0<\delta\ll1$ (depending
on $n$, $\Lambda$) so that
\begin{equation}
2\left(\alpha+8\varUpsilon_{1}\left(\alpha+2\right)+16\varUpsilon_{2}\left(2\lambda_{2}+1\right)\right)\,e^{-\lambda_{2}s}y^{\alpha-1}\,\,\leq\,\,\partial_{y}v\left(y,\,s\right)\,\,\leq\,\,\frac{1}{2}\alpha e^{-\lambda_{2}s}y^{\alpha-1}\label{Lambda intermediate1}
\end{equation}
for $\left(y,\,s\right)$ satisfying $2\beta e^{-\sigma s}\leq y\leq1$,
$s_{0}\leq s\leq s_{0}+\delta^{2}y^{2}$, provided that $s_{0}\gg1$
(depending on $n$, $\Lambda$, $\rho$, $\beta$).

On the other hand, by the above choice of $\delta=\delta\left(n,\,\Lambda\right)$,
(\ref{C^infty v bound intermediate}) and (\ref{C^infty v bound tip})
yield
\[
\left|\partial_{y}\left(v\left(y,\,s\right)-\frac{k}{c_{2}}e^{-\lambda_{2}s}\,\varphi_{2}\left(y\right)\right)\right|\,\leq\,C\left(n,\,\Lambda\right)e^{-\varkappa s}\left(e^{-\lambda_{2}s}y^{\alpha+1}\right)
\]
for $\left(y,\,s\right)$ satisfying $e^{-\vartheta\sigma s}\leq y\leq1$,
$s_{0}+\delta^{2}y^{2}\leq s\leq\mathring{s}$, and 
\[
\left|\partial_{y}\left(v\left(y,\,s\right)-e^{-\sigma s}\,\psi_{k}\left(e^{\sigma s}y\right)\right)\right|\,\leq\,C\left(n,\,\Lambda\right)\beta^{\alpha-2}e^{-2\varrho\sigma\left(s-s_{0}\right)}\left(e^{-\lambda_{2}s}y^{\alpha-1}\right)
\]
for $\left(y,\,s\right)$ satisfying $2\beta e^{-\sigma s}\leq y\leq e^{-\vartheta\sigma s}$,
$s_{0}+\delta^{2}y^{2}\leq s\leq\mathring{s}$. Note that 
\[
\partial_{y}\left(\frac{k}{c_{2}}e^{-\lambda_{2}s}\,\varphi_{2}\left(y\right)\right)=ke^{-\lambda_{2}s}y^{\alpha-1}\left(\alpha+2\varUpsilon_{1}\left(\alpha+2\right)y^{2}+\varUpsilon_{2}\left(2\lambda_{2}+1\right)y^{4}\right)
\]
 
\[
\partial_{y}\left(e^{-\sigma s}\,\psi_{k}\left(e^{\sigma s}y\right)\right)=ke^{-\lambda_{2}s}y^{\alpha-1}\left(\alpha+O\left(\left(e^{\sigma s}y\right)^{-2\left(1-\alpha\right)}\right)\right)
\]
 It follows, by (\ref{k}), that 
\begin{equation}
2\left(\alpha+8\varUpsilon_{1}\left(\alpha+2\right)+16\varUpsilon_{2}\left(2\lambda_{2}+1\right)\right)\,e^{-\lambda_{2}s}y^{\alpha-1}\,\,\leq\,\,\partial_{y}v\left(y,\,s\right)\,\,\leq\,\,\frac{1}{2}\alpha e^{-\lambda_{2}s}y^{\alpha-1}\label{Lambda intermediate2}
\end{equation}
for $\left(y,\,s\right)$ satisfying $2\beta e^{-\sigma s}\leq y\leq1$,
$s_{0}+\delta^{2}y^{2}\leq s\leq\mathring{s}$, provided that $\beta\gg1$
(depending on $n$, $\Lambda$) and $s_{0}\gg1$ (depending on $n$,
$\Lambda$). Then (\ref{Lambda Dv}) follows from (\ref{Lambda intermediate1})
and (\ref{Lambda intermediate2}).

As for the second derivative, we derive the following evolution equation:
\[
\partial_{s}\left(e^{\lambda_{2}s}y^{-\alpha+2}\partial_{yy}^{2}v\right)-\frac{1}{1+\left(\partial_{y}v\right)^{2}}\,\partial_{yy}^{2}\left(e^{\lambda_{2}s}y^{-\alpha+2}\partial_{yy}^{2}v\right)
\]
\[
-\frac{1}{y}\left(\frac{-6\left(\partial_{y}v\right)\left(y\,\partial_{yy}^{2}v\right)}{\left(1+\left(\partial_{y}v\right)^{2}\right)^{2}}+\frac{2\left(n-1\right)}{1-\left(\frac{v}{y}\right)^{2}}-\frac{y^{2}}{2}+\frac{2\left(\alpha-2\right)}{1+\left(\partial_{y}v\right)^{2}}\right)\partial_{y}\left(e^{\lambda_{2}s}y^{-\alpha+2}\partial_{yy}^{2}v\right)
\]
\[
=\frac{1}{y^{2}}\left(\frac{-2\left(y\,\partial_{yy}^{2}v\right)^{2}\left(1-3\left(\partial_{y}v\right)^{2}\right)}{\left(1+\left(\partial_{y}v\right)^{2}\right)^{3}}-\frac{y^{2}}{2}+\lambda_{2}y^{2}\right)\left(e^{\lambda_{2}s}y^{-\alpha+2}\partial_{yy}^{2}v\right)
\]
\[
+\frac{2\left(n-1\right)}{y^{2}}\left(\frac{4\left(\frac{v}{y}\right)\partial_{y}v-1-\left(\frac{v}{y}\right)^{2}}{\left(1-\left(\frac{v}{y}\right)^{2}\right)^{2}}\right)\left(e^{\lambda_{2}s}y^{-\alpha+2}\partial_{yy}^{2}v\right)
\]
\[
+\frac{\alpha-2}{y^{2}}\left(\frac{-6\left(\partial_{y}v\right)\left(y\,\partial_{yy}^{2}v\right)}{\left(1+\left(\partial_{y}v\right)^{2}\right)^{2}}+\frac{2\left(n-1\right)}{1-\left(\frac{v}{y}\right)^{2}}-\frac{y^{2}}{2}+\frac{\alpha-3}{1+\left(\partial_{y}v\right)^{2}}\right)\left(e^{\lambda_{2}s}y^{-\alpha+2}\partial_{yy}^{2}v\right)
\]
\[
+\frac{1}{y^{2}}\left(\frac{4\left(n-1\right)\left(\frac{v}{y}\right)\left(y\,\partial_{yy}^{2}v\right)}{\left(1-\left(\frac{v}{y}\right)^{2}\right)^{2}}-\frac{4\left(n-1\right)\left(1-\left(\partial_{y}v\right)^{2}\right)\left(1-3\left(\frac{v}{y}\right)^{2}\right)}{\left(1-\left(\frac{v}{y}\right)^{2}\right)^{3}}\right)\left(e^{\lambda_{2}s}y^{-\alpha+1}\partial_{y}v\right)
\]
\[
+\frac{1}{y^{2}}\left(\frac{4\left(n-1\right)\left(1-\left(\partial_{y}v\right)^{2}\right)\left(3+\left(\frac{v}{y}\right)^{2}\right)}{\left(1-\left(\frac{v}{y}\right)^{2}\right)^{3}}\right)\left(e^{\lambda_{2}s}y^{-\alpha}v\right)
\]
Using the same argument as for the first derivative, (\ref{Lambda DDv})
and (\ref{convexity v}) can be proved.
\end{proof}
Note that by (\ref{uw}) and (\ref{Lambda intermediate}), we get
\begin{equation}
z^{i}\left|\partial_{z}^{i}w\left(z,\,\tau\right)\right|\leq C\left(n\right)z^{\alpha}\qquad\forall\;\,i\in\left\{ 0,\,1,\,2\right\} \label{Lambda boundary}
\end{equation}
for $2\beta\leq z\leq\sqrt{2\sigma\tau}$, $\tau_{0}<\tau\leq\mathring{\tau}$.
Also, by (\ref{convexity outer u}), (\ref{convexity u}), (\ref{convexity v})
and rescaling, the projected curve $\bar{\Gamma}_{\tau}$ (see (\ref{projected Gamma}))
is convex in the corresponding rescaled region. More explicitly, we
have
\begin{equation}
\partial_{zz}^{2}\hat{w}\left(z,\,\tau\right)\geq0\label{convexity boundary}
\end{equation}
for $3\beta\leq z\leq\rho\left(2\sigma\tau\right)^{\frac{1}{2}+\frac{1}{4\sigma}}$,
$\tau_{0}\leq\tau\leq\mathring{\tau}$. Below we prove (\ref{convexity tip})
by using (\ref{eq psi'}), (\ref{eq w'}), (\ref{C^0 w' bound}),
(\ref{C^1 w' bound}) and (\ref{convexity boundary}).
\begin{lem}
If $\beta\gg1$ (depending on $n$, $\Lambda$) and $\tau_{0}\gg1$
(depending on $n$, $\Lambda$, $\rho$, $\beta$), there holds (\ref{convexity tip}).
\end{lem}
\begin{proof}
From (\ref{eq w'}), we deduce that 
\begin{equation}
\partial_{\tau}\left(\partial_{zz}^{2}\hat{w}\right)=\frac{1}{1+\left(\partial_{z}\hat{w}\right)^{2}}\,\partial_{zz}^{2}\left(\partial_{zz}^{2}\hat{w}\right)\label{eq DDw'}
\end{equation}
\[
+\left(\frac{n-1}{z}-\frac{6\left(\partial_{z}\hat{w}\right)\left(\partial_{zz}^{2}\hat{w}\right)}{\left(1+\left(\partial_{z}\hat{w}\right)^{2}\right)^{2}}-\frac{\frac{1}{2}+\sigma}{2\sigma\tau}z\right)\partial_{z}\left(\partial_{zz}^{2}\hat{w}\right)-\frac{2-6\left(\partial_{z}\hat{w}\right)^{2}}{1+\left(\partial_{z}\hat{w}\right)^{2}}\left(\partial_{zz}^{2}\hat{w}\right)^{3}
\]
\[
+\left(\left(n-1\right)\left(\frac{1}{\hat{w}^{2}}-\frac{2}{z^{2}}\right)-\frac{\frac{1}{2}+\sigma}{2\sigma\tau}\right)\left(\partial_{zz}^{2}\hat{w}\right)+2\left(n-1\right)\left(\frac{1}{z^{3}}-\frac{\partial_{z}\hat{w}}{\hat{w}^{3}}\right)\partial_{z}\hat{w}
\]
Notice that the last term on the RHS is positive, i.e. 
\begin{equation}
2\left(n-1\right)\left(\frac{1}{z^{3}}-\frac{\partial_{z}\hat{w}\left(z,\,\tau\right)}{\hat{w}^{3}\left(z,\,\tau\right)}\right)\partial_{z}\hat{w}\left(z,\,\tau\right)\,>0\label{postive term1}
\end{equation}
for $0\leq z\leq5\beta$, $\tau_{0}\leq\tau\leq\mathring{\tau}$,
since by (\ref{eq psi'}), (\ref{k}), (\ref{C^0 w' bound}) and (\ref{C^1 w' bound}),
we have
\[
\left(\frac{\hat{w}\left(z,\,\tau\right)}{z}\right)^{3}\geq\left(\frac{\hat{\psi}_{1-2\beta^{\alpha-3}}\left(z\right)}{z}\right)^{3}\geq\,\left(1\,+\,2^{\frac{\alpha+1}{2}}\left(1-2\beta^{\alpha-3}\right)\left(5\beta\right)^{\alpha-1}\right)^{3}
\]
\begin{equation}
>\,1+\beta^{\alpha-2}\,\geq\,\partial_{z}\hat{w}\left(z,\,\tau\right)\label{positive term2}
\end{equation}
for $0\leq z\leq5\beta$, $\tau_{0}\leq\tau\leq\mathring{\tau}$,
provided that $\beta\gg1$ (depending on $n$, $\Lambda$) and $\tau_{0}\gg1$
(depending on $n$, $\Lambda$, $\rho$, $\beta$).

Now let
\[
\left(\partial_{zz}^{2}\hat{w}\right)_{\min}\left(\tau\right)=\min_{0\leq z\leq5\beta}\partial_{zz}^{2}\hat{w}\left(z,\,\tau\right)
\]
Note that by (\ref{initial w'}) we have
\[
\left(\partial_{zz}^{2}\hat{w}\right)_{\min}\left(\tau_{0}\right)>0
\]
Now we would like to prove 
\[
\left(\partial_{zz}^{2}\hat{w}\right)_{\min}\left(\tau\right)\geq0
\]
for $\tau_{0}\leq\tau\leq\mathring{\tau}$ by contradiction. Suppose
that $\left(\partial_{zz}^{2}\hat{w}\right)_{\min}\left(\tau\right)$
fails to be non-negative for all $\tau_{0}\leq\tau\leq\mathring{\tau}$,
there must be $\tau_{1}^{*}>\tau_{0}$ so that 
\[
\left(\partial_{zz}^{2}\hat{w}\right)_{\min}\left(\tau_{1}^{*}\right)<0
\]
Let $\tau_{0}^{*}\geq\tau_{0}$ be the first time after which $\left(\partial_{zz}^{2}\hat{w}\right)_{\min}$
is negative all the way up to $\tau_{1}^{*}$. By continuity, we have
\[
\left(\partial_{zz}^{2}\hat{w}\right)_{\min}\left(\tau_{0}^{*}\right)\geq0
\]
On the other hand, by (\ref{C^1 w' bound}) and (\ref{convexity boundary}),
there hold
\[
\partial_{zz}^{2}\hat{w}\left(0,\,\tau\right)\,=\,\lim_{z\searrow0}\frac{\partial_{z}\hat{w}\left(z,\,\tau\right)}{z}\,\geq0
\]
 
\[
\partial_{zz}^{2}\hat{w}\left(5\beta,\,\tau\right)>0
\]
for $\tau_{0}\leq\tau\leq\mathring{\tau}$. As a result, the negative
minimum of $\partial_{zz}^{2}\hat{w}\left(z,\,\tau\right)$ for each
time-slice must be achieved in $\left(0,\,5\beta\right)$. Then by
the maximum principle (applying to (\ref{eq DDw'})), (\ref{C^0 w' bound}),
(\ref{postive term1}) and (\ref{positive term2}), we get 
\[
\partial_{\tau}\left(\partial_{zz}^{2}\hat{w}\right)_{\min}\geq\left(-\frac{2-6\left(\partial_{z}\hat{w}\right)^{2}}{1+\left(\partial_{z}\hat{w}\right)^{2}}\left(\partial_{zz}^{2}\hat{w}\right)_{\min}^{2}+\left(\left(n-1\right)\left(\frac{1}{\hat{w}^{2}}-\frac{2}{z^{2}}\right)-\frac{\frac{1}{2}+\sigma}{2\sigma\tau}\right)\right)\left(\partial_{zz}^{2}\hat{w}\right)_{\min}
\]
\[
\geq\left(6\left(\partial_{z}\hat{w}\right)^{2}\left(\partial_{zz}^{2}\hat{w}\right)_{\min}^{2}\right)\left(\partial_{zz}^{2}\hat{w}\right)_{\min}\,\geq\,6\left(1+\beta^{\alpha-2}\right)^{2}\left(\partial_{zz}^{2}\hat{w}\right)_{\min}^{3}
\]
for $\tau_{0}^{*}<\tau\leq\tau_{1}^{*}$. It follows that $\left(\partial_{zz}^{2}\hat{w}\right)_{\min}\left(\tau_{0}^{*}\right)<0$,
which is a contradiction.
\end{proof}
Recall that by the admissible conditions (see Section \ref{admissible}),
the projected curve $\bar{\Gamma}_{\tau}$ (see (\ref{projected Gamma}))
is a graph over $\mathcal{\bar{C}}$ outside $B\left(O;\,\beta\right)$.
By (\ref{convexity tip}) and also the admissible conditions, we also
know that inside $B\left(O;\,\beta\right)$, $\bar{\Gamma}_{\tau}$
is a convex curve which intersects orthogonally with the vertical
ray $\left\{ \left.\left(0,\,z\right)\right|\,z>0\right\} $, i.e.
$\partial_{z}\hat{w}\left(0,\,\tau\right)=0$. Furthermore, by (\ref{eq psi'})
and (\ref{C^0 w' bound}), $\bar{\Gamma}_{\tau}$ lies above $\mathcal{\bar{C}}$
and tends to it as $z\nearrow\beta$. Therefore, we conclude that
$\bar{\Gamma}_{\tau}$ is ``entirely'' a graph over $\mathcal{\bar{C}}$
and 
\begin{equation}
\bar{\Gamma}_{\tau}=\left\{ \left.\left(z,\,\hat{w}\left(z,\,\tau\right)\right)\right|\,z\geq0\right\} \label{projected Gamma entire}
\end{equation}
\[
=\left\{ \left.\left(\left(z-w\left(z,\,\tau\right)\right)\frac{1}{\sqrt{2}},\,\left(z+w\left(z,\,\tau\right)\right)\frac{1}{\sqrt{2}}\right)\right|\,z\geq\frac{\hat{w}\left(0,\,\tau\right)}{\sqrt{2}}\right\} 
\]

\begin{rem}
For the admissible conditions in Section \ref{admissible}, we only
require the function $w\left(z,\,\tau\right)$ (see (\ref{w})) is
defined for $z\gtrsim\beta$. However, by the convexity (see (\ref{convexity tip}))
and the above argument, we find the domain of definition for $w\left(z,\,\tau\right)$
is given by 
\[
\frac{\hat{w}\left(0,\,\tau\right)}{\sqrt{2}}\leq z<\infty
\]
On the other hand, by (\ref{k}) and (\ref{C^0 w' bound}), we may
assume that inside $B\left(O;\,5\beta\right)$, $\bar{\Gamma}_{\tau}$
is bounded between $\mathcal{\bar{M}}_{\frac{1}{2}}$ and $\mathcal{\bar{M}}_{\frac{3}{2}}$,
provided that $\beta\gg1$ (depending on $n$) and $\tau_{0}\gg1$
(depending on $n$, $\Lambda$, $\rho$, $\beta$). In particular,
we have
\[
\sup_{\tau_{0}\leq\tau\leq\mathring{\tau}}\,\frac{\hat{w}\left(0,\,\tau\right)}{\sqrt{2}}\,<\,\frac{\hat{\psi}_{2}\left(0\right)}{\sqrt{2}}
\]
which means $w\left(z,\,\tau\right)$ is defined for $z\geq\frac{\hat{\psi}_{2}\left(0\right)}{\sqrt{2}}$,
$\tau_{0}\leq\tau\leq\mathring{\tau}$. In addition, since $\bar{\Gamma}_{\tau}$
is a convex curve which lies below $\mathcal{\bar{M}}_{\frac{3}{2}}$
and tends to $\mathcal{\bar{C}}$, we deduce that 
\begin{equation}
0\,\leq\,w\left(z,\,\tau\right)\,\leq\,\psi_{\frac{3}{2}}\left(z\right)\,\leq\,\frac{\psi_{\frac{3}{2}}\left(\frac{\hat{\psi}_{2}\left(0\right)}{\sqrt{2}}\right)}{\frac{\hat{\psi}_{2}\left(0\right)}{\sqrt{2}}}z\label{ratio estimate}
\end{equation}
for $\frac{\hat{\psi}_{2}\left(0\right)}{\sqrt{2}}\leq z\leq5\beta$,
$\tau_{0}\leq\tau\leq\mathring{\tau}$. Note that the slope of the
linear function on the RHS satisfies 
\[
0<\frac{\psi_{\frac{3}{2}}\left(\frac{\hat{\psi}_{2}\left(0\right)}{\sqrt{2}}\right)}{\frac{\hat{\psi}_{2}\left(0\right)}{\sqrt{2}}}<\frac{\psi_{2}\left(\frac{\hat{\psi}_{2}\left(0\right)}{\sqrt{2}}\right)}{\frac{\hat{\psi}_{2}\left(0\right)}{\sqrt{2}}}=1
\]
\end{rem}
Lastly, in order to prove (\ref{Lambda tip}), we need the following
two lemmas, which provide smooth estimates of the function $w\left(z,\,\tau\right)$
in the rescaled tip region.
\begin{lem}
If $\beta\gg1$ (depending on $n$, $\Lambda$) and $\tau_{0}\gg1$
(depending on $n$, $\Lambda$, $\rho$, $\beta$), there holds 
\begin{equation}
\left\{ \begin{array}{c}
\left|w\left(z,\,\tau\right)-\psi_{k}\left(z\right)\right|\,\leq\,C\left(n\right)\beta^{\alpha-3}\left(\frac{\tau}{\tau_{0}}\right)^{-\varrho}\\
\\
-1\,\leq\,\partial_{z}w\left(z,\,\tau\right)\,\leq\,\frac{1}{3}\\
\\
0\,\leq\,\partial_{zz}^{2}w\left(z,\,\tau\right)\,\leq\,C\left(n\right)
\end{array}\right.\label{C^2 w bound}
\end{equation}
for $\frac{\hat{\psi}_{2}\left(0\right)}{\sqrt{2}}\leq z\leq3\beta$,
$\tau_{0}\leq\tau\leq\mathring{\tau}$. 
\end{lem}
\begin{proof}
By (\ref{C^0 w' bound}), inside $B\left(O;\,5\beta\right)$, the
projected curve $\bar{\Gamma}_{\tau}$ is bounded between $\mathcal{\bar{M}}_{\left(1-\beta^{\alpha-3}\left(\frac{\tau}{\tau_{0}}\right)^{-\varrho}\right)k}$
and $\mathcal{\bar{M}}_{\left(1-\beta^{\alpha-3}\left(\frac{\tau}{\tau_{0}}\right)^{-\varrho}\right)k}$,
which implies
\[
\psi_{\left(1-\beta^{\alpha-3}\left(\frac{\tau}{\tau_{0}}\right)^{-\varrho}\right)k}\left(z\right)\,\leq\,w\left(z,\,\tau\right)\,\leq\,\psi_{\left(1+\beta^{\alpha-3}\left(\frac{\tau}{\tau_{0}}\right)^{-\varrho}\right)k}\left(z\right)
\]
for $\frac{\hat{\psi}_{2}\left(0\right)}{\sqrt{2}}\leq z\leq3\beta$,
$\tau_{0}\leq\tau\leq\mathring{\tau}$. Then by (\ref{eq psi}), (\ref{k})
and using a similar argument as in the proof of Proposition \ref{C^0 w'},
we can derive the $C^{0}$ estimate of (\ref{C^2 w bound}). 

As for the first derivative, note that by (\ref{a priori bound w}),
(\ref{convexity tip}), (\ref{convexity boundary}) and the admissible
conditions in Section \ref{admissible}, $\bar{\Gamma}_{\tau}$ is
a convex curve which intersects orthogonally with the vertical ray
$\left\{ \left.\left(0,\,z\right)\right|\,z>0\right\} $. Thus, we
have 
\begin{equation}
\begin{array}{c}
\partial_{zz}^{2}w\left(z,\,\tau\right)\geq0\\
\\
\partial_{z}w\left(z,\,\tau\right)\,\geq\,\partial_{z}w\left(\frac{\hat{w}\left(0,\,\tau\right)}{\sqrt{2}},\,\tau\right)=-1\\
\\
\partial_{z}w\left(z,\,\tau\right)\,\leq\,\partial_{z}w\left(3\beta,\,\tau\right)\,\leq\,C\left(n,\,\Lambda\right)\beta^{\alpha-1}\leq\frac{1}{3}
\end{array}\label{C^2 w bound1}
\end{equation}
for $\frac{\hat{\psi}_{2}\left(0\right)}{\sqrt{2}}\leq z\leq3\beta$,
$\tau_{0}\leq\tau\leq\mathring{\tau}$, provided that $\beta\gg1$
(depending on $n$, $\Lambda$). 

Lastly, for the second derivative, notice that by (\ref{C^2 w' bound}),
the normal curvature of $\bar{\Gamma}_{\tau}$ (in terms of $\hat{w}\left(z,\,\tau\right)$)
satisfies
\begin{equation}
\left|A_{\bar{\Gamma}_{\tau}}\right|\,=\,\frac{\left|\partial_{zz}^{2}\hat{w}\left(z,\,\tau\right)\right|}{\left(1+\left(\partial_{z}\hat{w}\left(z,\,\tau\right)\right)^{2}\right)^{\frac{3}{2}}}\,\leq\,C\left(n\right)\label{C^2 w bound2}
\end{equation}
for $0\leq z\leq3\beta$, $\tau_{0}\leq\tau\leq\mathring{\tau}$.
Now if we reparametrize $\bar{\Gamma}_{\tau}$ by means of $w\left(z,\,\tau\right)$,
the normal curvature is then given by 
\begin{equation}
A_{\bar{\Gamma}_{\tau}}=\frac{\partial_{zz}^{2}w\left(z,\,\tau\right)}{\left(1+\left(\partial_{z}w\left(z,\,\tau\right)\right)^{2}\right)^{\frac{3}{2}}}\label{C^2 w bound3}
\end{equation}
The second derivative estimate in (\ref{C^2 w' bound}) follows from
(\ref{C^2 w bound1}), (\ref{C^2 w bound2}) and (\ref{C^2 w bound3}).
\end{proof}
The following lemma can be regarded as a counterpart of Proposition
\ref{C^2 w'}.
\begin{lem}
If $\beta\gg1$ (depending on $n$, $\Lambda$) and $\left|\tau_{0}\right|\gg1$
(depending on $n$, $\Lambda$, $\rho$, $\beta$), then for any $0<\delta\ll1$,
$m,\,l\in\mathbb{Z}_{+}$, there holds
\begin{equation}
\delta^{m+2l}\left|\partial_{z}^{m}\partial_{\tau}^{l}\left(w\left(z,\,\tau\right)-\psi_{k}\left(z\right)\right)\right|\,\leq\,C\left(n,\,m,\,l\right)\beta^{\alpha-3}\left(\frac{\tau}{\tau_{0}}\right)^{-\varrho}\label{C^infty w bound}
\end{equation}
for $\left(z,\,\tau\right)$ satisfying $\hat{\psi}_{2}\left(0\right)\leq z\leq2\beta$,
$\tau_{0}+\delta^{2}\leq\tau\leq\mathring{\tau}$. 
\end{lem}
\begin{proof}
By mimicking the proof of Proposition \ref{C^infty w'} and using
(\ref{eq psi}), (\ref{eq w}), (\ref{ratio estimate}), (\ref{C^2 w bound})
and Lemma (\ref{order psi}), we can deduce (\ref{C^infty w bound}).
\end{proof}
Below we show that the $C^{0}$estimate of (\ref{Lambda tip}) follows
directly from the $C^{0}$ estimate of (\ref{C^2 w bound}).
\begin{prop}
If $\beta\gg1$ (depending on $n$) and $\tau_{0}\gg1$ (depending
on $n$, $\Lambda$, $\rho$, $\beta$), there holds 
\begin{equation}
\left|w\left(z,\,\tau\right)\right|\,\leq\,C\left(n\right)z^{\alpha}\label{Lambda w}
\end{equation}
for $2\hat{\psi}_{2}\left(0\right)\leq z\leq2\beta$, $\tau_{0}\leq\tau\leq\mathring{\tau}$.
\end{prop}
\begin{proof}
By Lemma \ref{order psi}, (\ref{k}) and (\ref{C^2 w bound}), we
have 
\[
z^{-\alpha}\left|w\left(z,\,\tau\right)\right|\,\leq\,z^{-\alpha}\left|\psi_{k}\left(z\right)\right|\,+\,z^{-\alpha}\left|w\left(z,\,\tau\right)-\psi_{k}\left(z\right)\right|
\]
\[
\leq\,z^{-\alpha}\left|\psi_{k}\left(z\right)\right|\,+\,\left(2\beta\right)^{-\alpha}\left|w\left(z,\,\tau\right)-\psi_{k}\left(z\right)\right|
\]
\[
\leq\,C\left(n\right)\left(1+\beta^{-3}\left(\frac{\tau}{\tau_{0}}\right)^{-\varrho}\right)\,\leq\,C\left(n\right)
\]
for $2\hat{\psi}_{2}\left(0\right)\leq z\leq2\beta$, $\tau_{0}\leq\tau\leq\mathring{\tau}$,
provided that $\beta\gg1$ (depending on $n$).
\end{proof}
In the following proposition, we show the first derivative estimate
of (\ref{Lambda tip}) by using the maximum principle and (\ref{C^infty w bound}).
\begin{prop}
If $\beta\gg1$ (depending on $n$) and $\tau_{0}\gg1$ (depending
on $n$, $\Lambda$, $\rho$, $\beta$), there holds 
\begin{equation}
\left|\partial_{z}w\left(z,\,\tau\right)\right|\,\leq\,C\left(n\right)z^{\alpha-1}\label{Lambda Dw}
\end{equation}
for $2\hat{\psi}_{2}\left(0\right)\leq z\leq2\beta$, $\tau_{0}\leq\tau\leq\mathring{\tau}$.
\end{prop}
\begin{proof}
From (\ref{eq w}), we derive
\begin{equation}
\partial_{\tau}\left(z^{-\alpha+1}\partial_{z}w\right)-\frac{1}{1+\left(\partial_{z}w\right)^{2}}\,\partial_{zz}^{2}\left(z^{-\alpha+1}\partial_{z}w\right)\label{eq weighted Dw}
\end{equation}
\[
-\left(\frac{-2\,\partial_{z}w\,\partial_{zz}^{2}w}{\left(1+\left(\partial_{z}w\right)^{2}\right)^{2}}+\frac{2\left(n-1\right)}{z\left(1-\left(\frac{w}{z}\right)^{2}\right)}-\left(\frac{1}{2}+\sigma\right)\frac{z}{2\sigma\tau}\right)\partial_{z}\left(z^{-\alpha+1}\partial_{z}w\right)
\]
\[
=z^{-\alpha}\left(\frac{2\left(\alpha-1\right)}{1+\left(\partial_{z}w\right)^{2}}\,\partial_{zz}^{2}w-\frac{4\left(n-1\right)\left(1-\left(\partial_{z}w\right)^{2}\right)}{z^{2}\left(1-\left(\frac{w}{z}\right)^{2}\right)^{2}}\,w\right)
\]
\[
+\left(\alpha-1\right)z^{-\alpha}\left(\frac{-2\left(\partial_{z}w\right)\left(\partial_{zz}^{2}w\right)}{\left(1+\left(\partial_{z}w\right)^{2}\right)^{2}}+\frac{2\left(n-1\right)}{z\left(1-\left(\frac{w}{z}\right)^{2}\right)}-\left(\frac{1}{2}+\sigma\right)\frac{z}{2\sigma\tau}-\frac{\alpha}{z\left(1+\left(\partial_{z}w\right)^{2}\right)}\right)\left(\partial_{z}w\right)
\]
Let 
\[
M_{\textrm{boundary}}=\max_{\tau_{0}\leq\tau\leq\mathring{\tau}}\left\{ \left.z^{-\alpha+1}\partial_{z}w\left(z,\,\tau\right)\right|_{z=2\hat{\psi}_{2}\left(0\right)},\,\left.z^{-\alpha+1}\partial_{z}w\left(z,\,\tau\right)\right|_{z=2\beta}\right\} 
\]
\[
M_{\textrm{initial}}=\max_{2\hat{\psi}_{2}\left(0\right)\leq z\leq2\beta}z^{-\alpha+1}\partial_{z}w\left(z,\,\tau_{0}\right)
\]
Then by (\ref{Lambda boundary}) and (\ref{C^2 w bound}), we have
\[
M_{\textrm{boundary}}\,\leq\,C\left(n\right)
\]
By (\ref{initial w}), we have 
\[
M_{\textrm{initial}}\leq C\left(n\right)
\]
Let 
\[
h\left(\tau\right)=\max_{2\hat{\psi}_{2}\left(0\right)\leq z\leq2\beta}z^{-\alpha+1}\partial_{z}w\left(z,\,\tau\right)
\]
and 
\[
M=\max\left\{ M_{\textrm{boundary}},\,M_{\textrm{initial}}\right\} 
\]
If $h\left(\tau\right)\leq M$ for $\tau_{0}\leq\tau\leq\mathring{\tau}$,
then we are done. Otherwise, there is $\tau_{1}^{*}>\tau_{0}$ for
which
\[
h\left(\tau_{1}^{*}\right)>M
\]
Let $\tau_{0}^{*}$ be the first time after which $h$ is greater
than $M$ all the way upto time $\tau_{1}^{*}$. By continuity, we
have 
\[
h\left(\tau_{0}^{*}\right)\leq M
\]
Applying the maximum principle to (\ref{eq weighted Dw}) (and using
(\ref{ratio estimate}) and (\ref{C^2 w bound})) yields
\[
\partial_{\tau}h\,\leq\,C\left(n\right)\beta^{-\alpha}
\]
which implies that 
\[
h\left(\tau\right)\,\leq\,M+C\left(n\right)\beta^{-\alpha}\left(\tau-\tau_{0}^{*}\right)
\]
 for $\tau_{0}^{*}\leq\tau\leq\tau_{1}^{*}$. Now choose $0<\varepsilon\ll1$
(depending on $n$) so that 
\[
h\left(\tau\right)\,\leq\,M+1
\]
for $\tau_{0}^{*}\leq\tau\leq\tau_{0}^{*}+\varepsilon\beta^{\alpha}$.
By the above argument, we claim that
\begin{equation}
\max_{2\hat{\psi}_{2}\left(0\right)\leq z\leq2\beta}z^{-\alpha+1}\partial_{z}w\left(z,\,\tau\right)\,\leq\,M+1\label{Lambda Dw1}
\end{equation}
for $\tau_{0}\leq\tau\leq\tau_{0}+\varepsilon\beta^{\alpha}$; otherwise,
we would get a contradiction by the above argument.

On the other hand, by (\ref{C^infty w bound}) we have
\[
\left(\varepsilon\beta^{\alpha}\right)^{\frac{1}{2}}\left|\partial_{z}\left(w\left(z,\,\tau\right)-\psi_{k}\left(z\right)\right)\right|\,\leq\,C\left(n\right)\beta^{\alpha-3}\left(\frac{\tau}{\tau_{0}}\right)^{-\varrho}
\]
for $\hat{\psi}_{2}\left(0\right)\leq z\leq2\beta$, $\tau_{0}+\varepsilon\beta^{\alpha}\leq\tau\leq\mathring{\tau}$.
It follows, by (\ref{alpha}), (\ref{k}) and Lemma \ref{order psi},
that 
\[
z^{-\alpha+1}\partial_{z}w\left(z,\,\tau\right)\,\leq\,z^{-\alpha+1}\partial_{z}\psi_{k}\left(z\right)\,+\,C\left(n\right)\left(\varepsilon\beta^{\alpha}\right)^{-\frac{1}{2}}\beta^{\alpha-3}\left(\frac{\tau}{\tau_{0}}\right)^{-\varrho}z^{-\alpha+1}
\]
\begin{equation}
\leq\,z^{-\alpha+1}\partial_{z}\psi_{k}\left(z\right)\,+C\left(n\right)\beta^{-2-\frac{\alpha}{2}}\,\leq\,C\left(n\right)\label{Lambda Dw2}
\end{equation}
for $\hat{\psi}_{2}\left(0\right)\leq z\leq2\beta$, $\tau_{0}+\varepsilon\beta^{\alpha}\leq\tau\leq\mathring{\tau}$
, provided that $\beta\gg1$ (depending on $n$) and $\tau_{0}\gg1$
(depending on $n$, $\Lambda$, $\rho$, $\beta$). Note that $\varepsilon=\varepsilon\left(n\right)$.

Combining (\ref{Lambda Dw1}) with (\ref{Lambda Dw2}) yields 
\[
\partial_{z}w\left(z,\,\tau\right)\,\leq\,C\left(n\right)z^{\alpha-1}
\]
for $\hat{\psi}_{2}\left(0\right)\leq z\leq2\beta$, $\tau_{0}\leq\tau\leq\mathring{\tau}$.
By a similar argument, we can show that 
\[
\partial_{z}w\left(z,\,\tau\right)\,\geq\,-C\left(n\right)z^{\alpha-1}
\]
\end{proof}
Next, given any constant $p$, from (\ref{eq w}) we derive the following
evolution equation in order to estimate the second derivative of (\ref{Lambda tip}).
\begin{equation}
\partial_{\tau}\left(z^{-p+2}\partial_{zz}^{2}w\right)-\frac{1}{1+\left(\partial_{z}w\right)^{2}}\,\partial_{zz}^{2}\left(z^{-p+2}\partial_{zz}^{2}w\right)\label{eq weighted DDw}
\end{equation}
\[
-\left(\frac{-6\left(\partial_{z}w\right)\left(\partial_{zz}^{2}w\right)}{\left(1+\left(\partial_{z}w\right)^{2}\right)^{2}}+\frac{2\left(n-1\right)}{z\left(1-\left(\frac{w}{z}\right)^{2}\right)}-\left(\frac{1}{2}+\sigma\right)\frac{z}{2\sigma\tau}+\frac{2\left(p-2\right)}{z\left(1+\left(\partial_{z}w\right)^{2}\right)}\right)\partial_{z}\left(z^{-p+2}\partial_{zz}^{2}w\right)
\]
\[
=\left(\frac{-2\left(1-3\left(\partial_{z}w\right)^{2}\right)\left(\partial_{zz}^{2}w\right)^{2}}{\left(1+\left(\partial_{z}w\right)^{2}\right)^{3}}+\frac{12\left(n-1\right)\left(\frac{w}{z}\right)\partial_{z}w}{z^{2}\left(1-\left(\frac{w}{z}\right)^{2}\right)^{2}}\right)\left(z^{-p+2}\partial_{zz}^{2}w\right)
\]
\[
-\left(\frac{2\left(n-1\right)\left(1+\left(\frac{w}{z}\right)^{2}\right)}{z^{2}\left(1-\left(\frac{w}{z}\right)^{2}\right)^{2}}+2\left(n-1\right)\left(\frac{1}{2}+\sigma\right)\frac{1}{2\sigma\tau}\right)\left(z^{-p+2}\partial_{zz}^{2}w\right)
\]
\[
+\left(p-2\right)\left(\frac{-6\left(\partial_{z}w\right)\left(\partial_{zz}^{2}w\right)}{z\left(1+\left(\partial_{z}w\right)^{2}\right)^{2}}+\frac{2\left(n-1\right)}{z^{2}\left(1-\left(\frac{w}{z}\right)^{2}\right)}-\left(\frac{1}{2}+\sigma\right)\frac{1}{2\sigma\tau}+\frac{p-3}{z^{2}\left(1+\left(\partial_{z}w\right)^{2}\right)}\right)\left(z^{-p+2}\partial_{zz}^{2}w\right)
\]
\[
+\frac{1}{z^{2}}\left(\frac{4\left(n-1\right)}{\left(1-\left(\frac{w}{z}\right)^{2}\right)^{3}}\left(\left(\partial_{z}w\right)^{2}+3\left(\frac{w}{z}\right)^{2}\left(\partial_{z}w\right)^{2}-1-3\left(\frac{w}{z}\right)^{2}\right)\right)\left(z^{-p+1}\partial_{z}w\right)
\]
\[
+\frac{1}{z^{2}}\left(\frac{4\left(n-1\right)}{\left(1-\left(\frac{w}{z}\right)^{2}\right)^{3}}\left(1-\left(\partial_{z}w\right)^{2}\right)\left(3+\left(\frac{w}{z}\right)^{2}\right)\right)\left(z^{-p}w\right)
\]
The following lemma is essential for the derivation of the second
derivative estimates in (\ref{Lambda tip}), and its proof is very
similar to the one in the previous lemma 
\begin{lem}
If $\tau_{0}\gg1$ (depending on $n$, $\Lambda$, $\rho$, $\beta$),
there holds 
\[
\left|z\,\partial_{zz}^{2}w\left(z,\,\tau\right)\right|\leq C\left(n\right)
\]
for $2\hat{\psi}_{2}\left(0\right)\leq z\leq2\beta$, $\tau_{0}\leq\tau\leq\mathring{\tau}$.
\end{lem}
\begin{proof}
Let 
\[
M_{\textrm{boundary}}=\max_{\tau_{0}\leq\tau\leq\mathring{\tau}}\left\{ \left.z\,\partial_{zz}^{2}w\left(z,\,\tau\right)\right|_{z=2\hat{\psi}_{2}\left(0\right)},\,\left.z\,\partial_{zz}^{2}w\left(z,\,\tau\right)\right|_{z=2\beta}\right\} 
\]
\[
M_{\textrm{initial}}=\max_{2\hat{\psi}_{2}\left(0\right)\leq z\leq2\beta}z\,\partial_{zz}^{2}w\left(z,\,\tau_{0}\right)
\]
By (\ref{initial w}), (\ref{Lambda boundary}) and (\ref{C^2 w bound}),
we have 
\[
M=\max\left\{ M_{\textrm{boundary}},\,M_{\textrm{initial}}\right\} \,\leq\,C\left(n\right)
\]
Define
\[
h\left(\tau\right)=\max_{2\hat{\psi}_{2}\left(0\right)\leq z\leq2\beta}z\,\partial_{zz}^{2}w\left(z,\,\tau\right)
\]
If $h\left(\tau\right)\leq M$ for $\tau_{0}\leq\tau\leq\mathring{\tau}$,
then we are done. Otherwise, there is $\tau_{1}^{*}>\tau_{0}$ for
which
\[
h\left(\tau_{1}^{*}\right)>M
\]
Let $\tau_{0}^{*}$ be the first time after which $h$ is greater
than $M$ all the way upto time $\tau_{1}^{*}$. By continuity, we
have 
\[
h\left(\tau_{0}^{*}\right)\leq M
\]
Applying the maximum principle to (\ref{eq weighted DDw}) with $p=1$
(and using (\ref{ratio estimate}) and (\ref{C^2 w bound})) yields
\[
\partial_{\tau}h\left(\tau\right)\,\leq\,C\left(n\right)\left(h\left(\tau\right)+1\right)
\]
which implies that 
\[
h\left(\tau\right)\,\leq\,C\left(n\right)^{\tau-\tau_{0}^{*}}\left(M+C\left(n\right)\right)\,\leq\,2\left(M+C\left(n\right)\right)
\]
for $\tau_{0}^{*}\leq\tau\leq\tau_{0}^{*}+\varepsilon$, where $0<\varepsilon=\varepsilon\left(n\right)\ll1$.
Thus, we claim that 
\begin{equation}
\max_{2\hat{\psi}_{2}\left(0\right)\leq z\leq2\beta}z\,\partial_{zz}^{2}w\left(z,\,\tau\right)\,\leq\,2\left(M+C\left(n\right)\right)\label{Lambda DDw lemma1}
\end{equation}
for $\tau_{0}\leq\tau\leq\tau_{0}+\varepsilon$; otherwise, we would
get a contradiction by the above argument.

On the other hand, by (\ref{C^infty w bound}) we have
\[
\varepsilon\left|\partial_{zz}^{2}\left(w\left(z,\,\tau\right)-\psi_{k}\left(z\right)\right)\right|\,\leq\,C\left(n\right)\beta^{\alpha-3}\left(\frac{\tau}{\tau_{0}}\right)^{-\varrho}
\]
for $2\hat{\psi}_{2}\left(0\right)\leq z\leq2\beta$, $\tau_{0}+\varepsilon\leq\tau\leq\mathring{\tau}$,
which, together with (\ref{alpha}), (\ref{k}) and Lemma \ref{order psi},
implies 
\begin{equation}
z\,\partial_{zz}^{2}w\left(z,\,\tau\right)\,\leq\,z\,\partial_{zz}^{2}\psi_{k}\left(z\right)\,+\,C\left(n\right)\varepsilon^{-1}\beta^{\alpha-3}\left(\frac{\tau}{\tau_{0}}\right)^{-\varrho}z\,\leq\,C\left(n\right)\label{Lambda DDw lemma2}
\end{equation}
for $2\hat{\psi}_{2}\left(0\right)\leq z\leq2\beta$, $\tau_{0}+\varepsilon\leq\tau\leq\mathring{\tau}$
(since $\varepsilon=\varepsilon\left(n\right)$).

By (\ref{Lambda DDw lemma1}) and (\ref{Lambda DDw lemma2}), we get
\[
z\,\partial_{zz}^{2}w\left(z,\,\tau\right)\,\leq\,C\left(n\right)
\]
for $2\hat{\psi}_{2}\left(0\right)\leq z\leq2\beta$, $\tau_{0}\leq\tau\leq\mathring{\tau}$.
Similarly, by a similar argument, we can show that 
\[
z\,\partial_{zz}^{2}w\left(z,\,\tau\right)\,\geq\,-C\left(n\right)
\]
\end{proof}
Now we are ready to show the second derivative estimate of (\ref{Lambda tip})
with the help of the previous lemma.
\begin{prop}
If $\tau_{0}\gg1$ (depending on $n$), there holds 
\[
\left|\partial_{zz}^{2}w\left(z,\,\tau\right)\right|\,\leq\,C\left(n\right)z^{\alpha-2}
\]
for $2\hat{\psi}_{2}\left(0\right)\leq z\leq2\beta$, $\tau_{0}\leq\tau\leq\mathring{\tau}$.
\end{prop}
\begin{proof}
Let 
\[
M_{\textrm{boundary}}=\max_{\tau_{0}\leq\tau\leq\mathring{\tau}}\left\{ \left.z^{-\alpha+2}\partial_{zz}^{2}w\left(z,\,\tau\right)\right|_{z=2\hat{\psi}_{2}\left(0\right)},\,\left.z^{-\alpha+2}\partial_{zz}^{2}w\left(z,\,\tau\right)\right|_{z=2\beta}\right\} 
\]
\[
M_{\textrm{initial}}=\max_{2\hat{\psi}_{2}\left(0\right)\leq z\leq2\beta}z^{-\alpha+2}\partial_{zz}^{2}w\left(z,\,\tau_{0}\right)
\]
By (\ref{initial w}), (\ref{Lambda boundary}) and (\ref{C^2 w bound}),
we have
\[
M=\max\left\{ M_{\textrm{boundary}},\,M_{\textrm{initial}}\right\} \leq C\left(n\right)
\]
 Define 
\[
h\left(\tau\right)=\max_{2\hat{\psi}_{2}\left(0\right)\leq z\leq2\beta}z^{-\alpha+2}\partial_{zz}^{2}w\left(z,\,\tau\right)
\]
If $h\left(\tau\right)\leq M$ for $\tau_{0}\leq\tau\leq\mathring{\tau}$,
then we are done. Otherwise, there is $\tau_{1}^{*}>\tau_{0}$ for
which
\[
h\left(\tau_{1}^{*}\right)>M
\]
Let $\tau_{0}^{*}$ be the first time after which $h$ is greater
than $M$ all the way upto time $\tau_{1}^{*}$. By continuity, we
have 
\[
h\left(\tau_{0}^{*}\right)\leq M
\]
By applying the maximum principle to (\ref{eq weighted DDw}) with
$p=\alpha$ and using (\ref{ratio estimate}), (\ref{C^2 w bound}),
(\ref{Lambda w}) and (\ref{Lambda Dw}), we get
\[
\partial_{\tau}h\left(\tau\right)\,\leq\,C\left(n\right)\left(h\left(\tau\right)+1\right)
\]
which implies that 
\[
h\left(\tau\right)\,\leq\,C\left(n\right)^{\tau-\tau_{0}}\left(M+C\left(n\right)\right)\,\leq\,2\left(M+C\left(n\right)\right)
\]
for $\tau_{0}^{*}\leq\tau\leq\tau_{0}^{*}+\varepsilon$, where $0<\varepsilon=\varepsilon\left(n\right)\ll1$.
Thus, we infer that 
\begin{equation}
\max_{2\hat{\psi}_{2}\left(0\right)\leq z\leq2\beta}z^{-\alpha+2}\partial_{zz}^{2}w\left(z,\,\tau\right)\,\leq\,2\left(M+C\left(n\right)\right)\label{Lambda DDw1}
\end{equation}
for $\tau_{0}\leq\tau\leq\tau_{0}+\varepsilon$, since otherwise,
we would get a contradiction by the above argument.

On the other hand, by (\ref{C^infty w bound}) we have
\[
\varepsilon\left|\partial_{zz}^{2}\left(w\left(z,\,\tau\right)-\psi_{k}\left(z\right)\right)\right|\,\leq\,C\left(n\right)\beta^{\alpha-3}\left(\frac{\tau}{\tau_{0}}\right)^{-\varrho}
\]
for $2\hat{\psi}_{2}\left(0\right)\leq z\leq2\beta$, $\tau_{0}+\varepsilon\leq\tau\leq\mathring{\tau}$,
which, together with (\ref{k}) and Lemma \ref{order psi}, implies
\[
z^{-\alpha+2}\partial_{zz}^{2}w\left(z,\,\tau\right)\,\leq\,z^{-\alpha+2}\partial_{zz}^{2}\psi_{k}\left(z\right)\,+\,C\left(n\right)\beta^{\alpha-3}\left(\frac{\tau}{\tau_{0}}\right)^{-\varrho}z^{-\alpha+2}
\]
\begin{equation}
\leq\,z^{-\alpha+2}\partial_{zz}^{2}\psi_{k}\left(z\right)\,+\,C\left(n\right)\beta^{-1}\,\leq\,C\left(n\right)\label{Lambda DDw2}
\end{equation}
for $2\hat{\psi}_{2}\left(0\right)\leq z\leq2\beta$, $\tau_{0}+\varepsilon\leq\tau\leq\mathring{\tau}$,
provided that $\beta\gg1$ (depending on $n$). Notice that $\varepsilon=\varepsilon\left(n\right)$.

Combining (\ref{Lambda DDw1}) with (\ref{Lambda DDw2}) yields 
\[
\partial_{zz}^{2}w\left(z,\,\tau\right)\,\leq\,C\left(n\right)z^{\alpha-2}
\]
for $2\hat{\psi}_{2}\left(0\right)\leq z\leq2\beta$, $\tau_{0}\leq\tau\leq\mathring{\tau}$.
Likewise, by a similar argument, we can show 
\[
\partial_{zz}^{2}w\left(z,\,\tau\right)\,\geq\,-C\left(n\right)z^{\alpha-2}
\]
for $2\hat{\psi}_{2}\left(0\right)\leq z\leq2\beta$, $\tau_{0}\leq\tau\leq\mathring{\tau}$.
\end{proof}

\vspace{0.3in}
\email{
\noindent Department of Mathematics, Rutgers University - Hill Center for the Mathematical Sciences 
110 Frelinghuysen Rd., Piscataway, NJ 08854-8019\\\\
E-mail addresses: \textsf{showhow@math.rutgers.edu}\\ 
\indent \hspace{.91in} \textsf{natasas@math.rutgers.edu}
}


\begin{thebibliography}{EMT}
\bibitem[C]{C}A. Cooper, A characterization of the singular time
of the mean curvature flow. Proc. Amer. Math. Soc. 139 (2011), no.
8, 2933\textendash 2942. 

\bibitem[EH]{EH}K. Ecker, G. Huisken, Interior estimates for hypersurfaces
moving by mean curvature. Invent. Math. 105 (1991), no. 3, 547\textendash 569.

\bibitem[EMT]{EMT}J. Enders, R. M$\ddot{u}$ller, P. Topping, On
type-I singularities in Ricci flow. Comm. Anal. Geom. 19 (2011), no.
5, 905\textendash 922. 

\bibitem[H]{H}G. Huisken, Flow by mean curvature of convex surfaces
into spheres. J. Differential Geom. 20 (1984), no. 1, 237\textendash 266.

\bibitem[K]{K}S. Kandanaarachchi, On the extension of axially symmetric
volume flow and mean curvature flow. arXiv:1301.1125.

\bibitem[LS]{LS}N. Le, N. Sesum, The mean curvature at the first
singular time of the mean curvature flow. Ann. Inst. H. Poincar$\acute{e}$
Anal. Non Lin$\acute{e}$aire 27 (2010), no. 6, 1441\textendash 1459. 

\bibitem[LW]{LW}H. Li, B. Wang, The extension problem of the mean
curvature flow. arXiv:1608.02832.

\bibitem[V]{V}J. J. L. Vel$\acute{a}$zquez, Curvature blow-up in
perturbations of minimal cones evolving by mean curvature flow, Ann.
Scuola Norm. Sup. Pisa Cl. Sci. (4) 21 (1994), no. 4, 595\textendash 628.
\end{thebibliography}
\end{document}